\documentclass[11pt]{amsart}

\usepackage{amssymb,amsmath,amscd,amsfonts,a4,psfrag,graphicx,latexsym,epsfig,color, relsize}

\usepackage{yhmath} 
\usepackage{comment}

\usepackage{hyperref} 

\usepackage{bbm,dsfont}
\usepackage{array}

\usepackage[active]{srcltx}
\usepackage[all]{xy}

\usepackage{enumitem}

\entrymodifiers={!!<0pt,0.7ex>+}

\everymath{\displaystyle}

\newtheorem{teo}{Theorem}[section]
\newtheorem{lem}[teo]{Lemma}
\newtheorem{cor}[teo]{Corollary}

\newtheorem{prop}[teo]{Proposition}
\newtheorem{defi}[teo]{Definition}

\newtheorem{conj}[teo]{Conjecture}
\newtheorem{remark}[teo]{Remark}

\newcommand{\fonc}[5]{                     
            \begin{array}{rcll}#1 :& #2 & \longrightarrow & #3 \\   %
                         &#4 &\longmapsto & #5          %
            \end{array}}

\allowdisplaybreaks 

\newcommand{\e}{\epsilon}

\newcommand{\mc}{\mathbb{C}}
\newcommand{\mz}{\mathbb{Z}}
\newcommand{\mn}{\mathbb{N}}

\newcommand{\Gg}{\mathcal{G}}
\newcommand{\Hh}{\mathcal{H}}
\newcommand{\Ll}{\mathcal{L}}
\newcommand{\Oo}{\mathcal{O}}
\newcommand{\Tt}{\mathcal{T}}
\newcommand{\Zz}{\mathcal{Z}}

\newcommand{\qD}{q_{\scalebox{0.5}{D}}}
\newcommand{\qDi}{q_{\scalebox{0.5}{D}i}}
\newcommand{\eD}{\epsilon_{\scalebox{0.5}{D}}}
\newcommand{\eDi}{\epsilon_{\scalebox{0.5}{D}i}}
\newcommand{\AD}{A_{\scalebox{0.5}{D}}}
\newcommand{\Q}{{\scalebox{0.5}{${\rm Q}$}}}

\newcommand{\Y}{{\scalebox{0.5}{Y}}}
\newcommand{\D}{{\scalebox{0.5}{D}}}

\newcommand{\Pup}{{\scalebox{0.5}{${\rm P}$}}}

\newcommand{\Lam}{{\scalebox{0.5}{$\Lambda$}}}

\newcommand{\smallbullet}{%
\vcenter{\hbox{\,\scalebox{0.5}{$\bullet$}\,}}
}

\title[Quantum graph algebras at roots of unity]{On the structure and representations of \\ quantum graph algebras at roots of unity}

\author[St\'ephane Baseilhac, Matthieu Faitg, Philippe Roche]{St\'ephane Baseilhac\textsuperscript{1}, Matthieu Faitg\textsuperscript{2}, Philippe Roche\textsuperscript{3}}

\begin{document}
\maketitle

\begin{center}
\textsuperscript{1,3}IMAG, Univ Montpellier, CNRS, Montpellier, France.

\textsuperscript{2}Univ Toulouse, CNRS, IMT, Toulouse, France.

\medskip

\textsuperscript{1}stephane.baseilhac@umontpellier.fr, \textsuperscript{2}matthieu.faitg@univ-tlse3.fr, \textsuperscript{3}philippe.roche@umontpellier.fr
\end{center}

\begin{abstract} We study the specializations $\Ll_{g,n}^\e$ at roots of unity $\e$ of odd order of the graph algebras, associated to a simply-connected complex semi-simple algebraic group $G$ and a compact oriented surface $\Sigma_{g,n}^{\circ}$ with genus $g$, $n$ punctures, and one boundary component. We prove that the central localizations of $\Ll_{g,n}^\e$, and its subalgebra $\Ll_{g,n}^{u_\e}$ of invariant elements under the coadjoint action of a small quantum group, are central simple algebras of PI degrees that we compute. Also, we describe their centers, and show they are integrally closed rings.   
\end{abstract} 
\medskip



\tableofcontents

\section{Introduction}
Let $G$ be a simply-connected complex semi-simple algebraic group, and $\mathfrak{g}$ its Lie algebra. Let $\Sigma_{g,n}^{\circ}$ be a compact oriented surface of genus $g$, with $n$ punctures and one boundary component.

The quantum graph algebra $\Ll_{g,n}$ is a quantization of the algebra of functions on the variety of representations ${\rm Hom}_{\rm Grp}\bigl(\pi_1(\Sigma_{g,n}^{\circ}),G\bigr)$ endowed with the Fock-Rosly Poisson structure \cite{FR}. It was introduced in the mid '$90$s by Alekseev-Grosse-Schomerus \cite{Al, AGS1,AGS2,AS} and Buffenoir-Roche \cite{BuR1,BuR2}, and has been actively studied in recent years in connection with non-semisimple TQFTs (see, e.g., the survey \cite{JordanSurvey}). $\Ll_{g,n}$ is a module-algebra over the Drinfeld-Jimbo quantum envelopping algebra $U_q$ of $\mathfrak{g}$, and can be naturally identified with $\Oo_q^{\otimes (2g+n)}$ as a module with a coadjoint action of $U_q$, where $\Oo_q$ is the quantum coordinate algebra of $G$.  There is an explicit isomorphism (called ``holonomy'') between $\Ll_{g,n}$ and the stated $\mathfrak{g}$-skein algebra of $\Sigma_{g,n}^{\circ,\bullet}$, the surface $\Sigma_{g,n}^{\circ}$ with one marked point on the boundary component \cite{BFR, FaitgHol}. The subalgebra $\Ll_{g,n}^{{\scalebox{0.6}{$U_q$}}}$ of $U_q$-invariant elements of $\Ll_{g,n}$ quantizes the character variety ${\rm Hom}_{\rm Grp}\bigl(\pi_1(\Sigma_{g,n}^{\circ}),G\bigr)/\!\!/ G$, and the holonomy isomorphism identifies it with the $\mathfrak{g}$-skein algebra of $\Sigma_{g,n}^{\circ}$. Thus, we see that $\Ll_{g,n}$ provides tools from quantum group theory to study the skein algebras and their representations.  

Studying these representations when $q$ is specialized to roots of unity is a challenging problem, both from the perspectives of quantum group theory and quantum topology, where these representations are expected to form state spaces of TQFTs defined for manifolds endowed with $G$-characters of the fundamental group.

In this paper we tackle this problem for the specializations $\Ll_{g,n}^{\e}$ of $\Ll_{g,n}$ at roots of unity $\e$ of odd order (plus some mild assumptions), and its subalgebra $\Ll_{g,n}^{u_\e}$ of invariant elements under the action of a small quantum group $u_\e$. The algebra $\Ll_{g,n}^{u_\e}$ should not be confused with the specialization of $\Ll_{g,n}^{{\scalebox{0.6}{$U_q$}}}$ at $\e$, which it contains strictly. That specialization was denoted by $\mathcal{M}_{g,n}^{A,\e}$ in \cite{BR2}. It can be shown that $\mathcal{M}_{g,n}^{A,\e} \cong (\mathcal{L}_{g,n}^{u_\e})^{U(\mathfrak{g})}$; details will appear in \cite{BF}.

\smallskip

Our main results are that $\Ll_{g,n}^{\e}$ and $\Ll_{g,n}^{u_\e}$ are domains whose central localizations are central simple algebras, the computation of their PI-degrees, and a complete description of their centers, which are integrally closed rings. From classical tools of PI ring theory we can deduce the maximal dimension of the irreducible representations of $\mathfrak{A} =\Ll_{g,n}^{\e}$ or $\Ll_{g,n}^{u_\e}$, and identify the {\it Azumaya locus} of $\mathfrak{A}$ (that is, the set of central characters of the irreducible representations of $\mathfrak{A}$ whose dimension is maximal) with the complement of the discriminant subvariety of ${\rm MaxSpec}(\mathcal{Z}(\mathfrak{A}))$.

An intermediate result within our constructions is a surjective morphism from  $\Ll_{g,n}^{\e}$ to the quantum graph algebra $\Ll_{g,n}(u_\e)$ associated to a small quantum group $u_\e$ (see Th.\,\ref{LgnLgnpetit}). It is known (\cite{FaitgMCG, CF}) that $\Ll_{g,n}(u_\e)$ is the algebra of endomorphisms of the state space of Kerler--Lyubashenko's TQFT for the modular category $u_\e\text{-}\mathrm{mod}$ (\cite{KeLy}), so one can expect this result will have interesting consequences in quantum topology.

\smallskip

In an Appendix we extend some of our results for $\Ll_{g,n}^{\e}$ to specializations $\Ll_{g,n}^{\eD}$ at more general roots of unity $\eD$. Ideas are similar, but results are different in general.
\smallskip

In the case of $\mathfrak{g} = \mathfrak{sl}_{m+1}$ and $n=0$, the PI-degree of $\Ll_{g,0}^\e$ that we compute could be derived from \cite[Th.\,4]{KaruoWang}, by using the specialization of the holonomy isomorphism of \cite[Th.\,2]{BFR} (see \cite{BF}), mentioned above, between $\Ll_{g,n}$ and the stated $\mathfrak{sl}_{m+1}$-skein algebra of $\Sigma_{g,n}^{\circ,\bullet}$. Also, a combinatorial description of the center of that stated skein algebra has been obtained in \cite[Th.\,2]{KaruoWang}.

As compared to skein theoretic tools like those used in \cite{KaruoWang}, an advantage of our treatment of $\Ll_{g,n}^{\e}$ is that we can deal with all complex simple Lie algebras $\mathfrak{g}$, simultaneously and in an uniform way. No presentation of $\Ll_{g,n}^{\e}$ by generators and relations is required, and tools from invariant theory can be used to study subalgebras of invariant elements like $\Ll_{g,n}^{u_\e}$, for instance to describe the structure of its center.

\smallskip

For $n=0$, the algebra $\Ll_{g,0}^\e$ and the specialization of $\Ll_{g,0}^{{\scalebox{0.6}{$U_q$}}}$ at $\e$ have been studied in \cite{GJS} using coends at roots of unity, factorization homology, and the theory of Poisson orders of \cite{BrGo}. In particular, the center $\mathcal{Z}(\Ll_{g,0}^\e)$ was computed explicitly (that is, identified with $\mathcal{Z}_0(\Ll_{g,0}^\e)$ in the notations below), and it was shown that $\Ll_{g,0}^\e$ is a finitely generated $\mathcal{Z}(\Ll_{g,0}^\e)$-module, and that the Azumaya locus of $\Ll_{g,0}^\e$ is the preimage of the big Bruhat cell $G^0\subset G$ by the quantum moment map $\mu_{g,n}^\e$ (see \cite[Th.\,1.1]{GJS}).

A main difference between \cite{GJS} and the present paper is that we define and study the algebras $\Ll_{g,n}^{\e}$ and $\Ll_{g,n}^{u_\e}$ by means of twists, co-R-matrices, and specializations of integral forms of quantum groups (following our previous works \cite{BR1,BR2,BFR}). Thus our approach lies entirely within Hopf algebra and quantum group theory. This allows us to treat all cases $n\geq 0$ and specializations at more general roots of unity $\eD$, and to obtain additional structure results. In particular, in our approach a key r\^ole is played by (variants of) the Alekseev morphism, and its interaction with the quantum moment map $\mu_{g,n}^\e$. This is crucial in our proofs that $\Ll_{g,n}^{\e}$ and $\Ll_{g,n}^{u_\e}$ are domains, whose central localizations are central simple algebras, and to eventually compute their PI-degrees. We have tried to let no point unclear.

\smallskip

In other forthcoming works, we will study the specializations at roots of unity of the algebra $\Ll_{g,n}^{U_q}$, maximal order issues, and  provide concrete indecomposable representations of these algebras when $\mathfrak{g}=\mathfrak{sl}_2$. 

\medskip

\noindent {\bf Statements and plan of the paper.} The proofs of our main results are contained in \S \ref{secdefLgn}--\ref{sec:proofLgn3}. Arguments that follow fairly directly from the existing literature or can be obtained after simple adaptations are detailed in the appendices \ref{sec:qKilling}--\ref{Lgndomain}. In \S \ref{sectionnotations}-\ref{sectionPreliminaires} we introduce mainly the material on graph algebras and quantum groups, and prove also some important preliminary results about them.

Section \ref{sectionnotations} introduces notations, and basic results on specializations of $\mc[q^{\pm 1}]$-modules and the central localization of domains, used throughout the paper.

Section \ref{sectiongraphalg} recalls basic properties of the graph algebra $\Ll_{g,n}(H)$ associated with a Hopf algebra $H$ with invertible antipode. Guided by the case $H=U_q^\Q(\mathfrak{g})$ and its specializations, we discuss with some extent the weakest structures on $H$ required to define $\Ll_{g,n}(H)$, that we call a $(H,\Phi^\pm)$ datum. Such a datum implies the existence of a co-R-matrix which is used to define the product in $\mathcal{L}_{g,n}(H)$. Also, we prove in this setup a number of general facts on the quantum moment map (QMM) $\mu_{g,n}$ (\S \ref{subsecQMMgeneral}) and the Alekseev morphism $\Phi_{g,n}$ (\S \ref{subsecAlekseev}), in particular concerning sufficient conditions for their injectivity. These will play a key role in the sequel.

Section \ref{sec:Uq}-\ref{specialisOA} contain background results on quantum groups. Let $G$ be a complex simple and simply-connected Lie group with Lie algebra $\mathfrak{g}$, and $\Lambda$ a lattice such that $Q\subset \Lambda \subset P$, where $Q$ is the root lattice and $P$ the weight lattice of $\mathfrak{g}$. Let $A=\mc[q,q^{-1}]$ (the generalization to semi-simple $G$ is straightforward). We use the integral forms $U_A^\Lam$ (De Concini-Kac-Procesi \cite{DCK,DCP}) and $\Gamma_A^\Lam$ (De Concini-Lyubashenko \cite{DC-L}) of the Drinfeld-Jimbo quantum group $U_q^\Lam:=U_q^\Lam(\mathfrak{g})$ associated to $\mathfrak{g}$ and $\Lambda$, and the integral form $\Oo_A$ (Lusztig \cite{Lusztig}) of the quantum coordinate algebra $\Oo_q:=\Oo_q(G)$. We recall some facts on these algebras, their specializations $U_\e^\Lam$, $\Gamma_\e^\Lam$ and $\Oo_\e$ at an odd $l$-th root of unity $\e$ (satisfying some mild assumptions, see \S\ref{secnotationq}), the quantum Frobenius epimorphism $\mathbb{F}\mathrm{r}_\e \colon \Gamma_\e^\Q \to U(\mathfrak{g})$ and its dual monomorphism $\mathbb{F}\mathrm{r}_\e^* : \mathcal{O}(G) \to \mathcal{O}_\e(G)$, and the small quantum group $u_\e^\Lam$ (also called Frobenius-Lusztig kernel).\\
\indent Here are some important features which are recalled in detail. The inclusion map $U_A^\Q \to \Gamma_A^\Q$ has a remarkable degeneracy at the specialization to $q=\e$, as it acquires a big kernel, the quotient of $U_\e^\Q$ by this kernel being precisely the small quantum group $u_\e^\Q$. The specialization $\langle \text{-},\text{-} \rangle_{\e}^\Q:{\mathcal O}_\epsilon\times \Gamma_\epsilon^\Q\to {\mathbb C}$ of the evaluation pairing is non-degenerate, a known fact for which we provide a proof in Lemma \ref{nondegpairepsD} using Lusztig's \cite{Lusztig} idempotent completion $\stackrel{.}{\mathbf{U}}$ of $U_q^\Q$ (recalled in Appendix \ref{appLusztigModif}). Also, by using an embedding of $\Gamma_\e^\Lam$ in the specialization $\hat{\mathbf{U}}_\e$ of the canonical basis completion of $\stackrel{.}{\mathbf{U}}$ (App.\,\ref{appLusztigModif}), we provide an alternative proof (Prop.\,\ref{pipropO}) of the following important exact sequence of Hopf algebras (\cite[Th. III.7.10]{BG}, \cite{AG}), where $\pi$ is the restriction of the evaluation map to $u_\e^\Q\subset  \Gamma_\epsilon^\Q$:
\begin{equation}\label{seqmapsdualintro}
\xymatrix@C=3em{1 \ar[r] & \Oo(G)\ar[r]^-{\ \mathbb{F}\mathrm{r}_\e^{*}} & \Oo_\e \ar[r]^-{\pi} & (u_\epsilon^\Q)^{*} \ar[r] & 1}.
\end{equation}
This preliminary material allows us in \S\ref{sec:qtsroot1} to relate the co-quasitriangular structures of $\Oo_\e$ and $(u_\e^\Q)^*$. 
 
Section \ref{subsecQuantumGraphAlg} gives the definition of the graph algebra $\Ll_{g,n} := \Ll_{g,n}(U_q^\Q)$, its integral form $\Ll_{g,n}^{\AD}$, and specialization $\Ll_{g,n}^\e$; these definitions use the results of \S \ref{sectionPreliminaires} to specify the datum $(H,\Phi^\pm)$ from \S \ref{sectiongraphalg}. As a $\mathbb{C}$-vector space we have $\Ll_{g,n}^\e = \Oo_\e^{\otimes (2g+n)}$ by definition.

It should be noted that the co-R-matrix on $\Oo_q$ takes values in the field $\mc(\qD)$, which is an extension of $\mathbb{C}(q)$ by a formal variable $\qD$ satisfying $(\qD)^D=q$, where $D$ is the smallest positive integer such that $DP\subset Q$ (the values of $D$, depending on $\mathfrak{g}$, are recalled in \S \ref{subsecLieAlg}). So the theory needs to fix $\qD$: the algebra $\Ll_{g,n}$ is defined over $\mc(\qD)$, and $\Ll_{g,n}^{\AD}$ over $\AD := \mc[\qD,\qD^{-1}]$. As for specializations, we consider only roots of unity $\e$ of odd order $l$ such that ${\rm gcd}(l,D)=1$; hence there exists $\bar{D} \in \mathbb{N}$ such that $\bar{D}D\equiv 1 \pmod{l}$ and $\eD := \e^{\bar{D}}$ is a $D$-th root of $\e$ such that $\eD^l =1$. In the main body of the text we focus on the specialization of $\Ll_{g,n}^{\AD}$ for this peculiar choice $\eD = \e^{\bar{D}}$ of $D$-th root of $\e$, as it simplifies all statements. Being determined by $\e$, we denote this specialization $\Ll_{g,n}^\e$.

We consider the generalizations of Th.\,\ref{LgnLgnpetit} and Th.\,\ref{Lgn} below to a primitive $D$-th root $\eD$ of $\e$ in Appendix \ref{appDlRoot}, where we write the corresponding specialization $\Ll_{g,n}^{\eD}$. Proofs of analogs of  Th.\,\ref{Lgn2} and Th.\,\ref{Lgn3} would require additional work; for instance the results of De Concini-Kac-Procesi on the center $\mathcal{Z}(U_\e^\Pup)$ used in \S \ref{sec:proofLgn3} do not hold in general for $\mathcal{Z}(U_\e^\Lam)$ with $\Lambda$ arbitrary (see \cite[Rk.\,(a) on p.\,179]{DC-K-P1}).

\smallskip

In Section \ref{subsecStructResLgnEps} we show that the image of the Frobenius embedding
\begin{equation}\label{Frobembintro}(\mathbb{F}\mathrm{r}^*_\e)^{\otimes (2g+n)} \colon \Oo(G)^{\otimes (2g+n)}\to \Ll_{g,n}^\e
\end{equation}
is a central subalgebra, where $\mathbb{F}\mathrm{r}^*_\e : \mathcal{O}(G) \to \Oo_\e(G)$ is the dual of the Frobenius morphism $\mathbb{F}\mathrm{r}\colon \Gamma_\e^\Q \to U(\mathfrak{g})$. We denote $$\mathcal{Z}_0(\mathcal{L}_{g,n}^\e) := \mathbb{F}\mathrm{r}^*_\e\bigl( \Oo(G) \bigr)^{\otimes (2g+n)}.$$ On the other hand we have the graph algebra $\Ll_{g,n}(u_\e^\Q)$ which can be defined as in \S \ref{subsecDefLgnH} by means of the standard quasitriangular structure of the small quantum group $u_\e^\Q$, coming from the Drinfeld double construction. Our next result shows that $\Ll_{g,n}(u_\e^\Q)$ is the quotient of $\Ll_{g,n}^{\e}$ by the ideal generated by the augmentation ideal of $\mathcal{Z}_0(\mathcal{L}_{g,n}^\e) \cong \mathcal{O}(G)^{\otimes (2g+n)}$. More precisely, let $\Oo^+(G) = \ker(\varepsilon) \cap \Oo(G)$ be the augmentation ideal of $\mathcal{O}(G)$, which consists of the functions vanishing at the unit element of $G$. Using the results of \S \ref{sectionPreliminaires} we  deduce the following generalization of the exact sequence \eqref{seqmapsdualintro}:

\begin{teo}\label{LgnLgnpetit} {\rm (see Th.\,\ref{centralextLgn})} The algebra $\Ll_{g,n}^{\e}$ is a central extension of the $l^{(2g+n)\dim(\mathfrak{g})}$-dimensional algebra $\mathcal{L}_{g,n}(u_\epsilon^\Q)$ by $\Oo^+(G)^{\otimes (2g+n)}$, where $l$ is the order of the root of unity $\e$.
\end{teo}

In Appendix \ref{appDlRoot}, Cor.\,\ref{centralextLgneD}, we obtain an analogous result for $\Ll_{g,n}^{\eD}$ where $\eD$ is a primitive $Dl$-th root of unity; the central extension is by $\Oo^+(G^\Q)^{\otimes 2g}\otimes \Oo^+(G)^{\otimes n}$, where $G^\Q$ is the adjoint group. For $\mathfrak{g}=\mathfrak{sl}_{m+1}$ and $(g,n)=(0,1)$ a similar quotient of $\Ll_{g,n}^{\e}$ has been obtained in \cite{CooLau} by means of a presentation by generators and relations.

One can show that the extension in Th.\,\ref{LgnLgnpetit} is compatible with natural actions \cite{AS,FaitgMCG} of the mapping class group ${\rm MCG}(\Sigma_{g,n}^\circ)$ on $\mathcal{O}(G)^{\otimes (2g+n)}$, $\Ll_{g,n}^{\e}$ and $\Ll_{g,n}(u_\e^\Q)$. This will be developed in \cite{BF}, among other topologically-flavored results. 

\smallskip

By factorisability of the Hopf algebra $u_\epsilon^\Q$ (because the order $l$ of $\e$ is odd, see \cite{Lyub95}) we can deduce from Th.\,\ref{LgnLgnpetit} that $\Ll_{g,n}^{\e}$ has an irreducible representation of dimension $l^{g.\mathrm{dim}(\mathfrak{g})+nN}$ (Prop.\,\ref{irrep}). Also, the Alekseev morphism $\Phi_{g,n}^\e$  on $\Ll_{g,n}^{\e}$ from \S \ref{subsecAlekseev}, defined by the datum $\Phi^\pm$ from \S \ref{subsecQuantumGraphAlg}, is injective (Prop.\,\ref{AlekseevInjRootOf1}; to be precise this statement proves injectivity of the {\it modified} Alekseev morphism on $\Ll_{g,n}^{\e}$ discussed below, but the same arguments work for $\Phi_{g,n}^\e$), and from this we show that $\Ll_{g,n}^{\e}$ is a domain (Prop.\,\ref{propLgndomain}). We can thus consider the central localization
$$Q(\Ll_{g,n}^\e) := Q(\mathcal{Z}(\Ll_{g,n}^\e))  \otimes_{\mathcal{Z}(\Ll_{g,n}^\e)} \Ll_{g,n}^\e.$$
We prove:
\begin{teo}\label{Lgn} {\rm (see Prop.\,\ref{Z0Lgn}, Th.\,\ref{Lgnteo1})} (i) The algebra $\Ll_{g,n}^{\e}$ is a free $\mathcal{Z}_0(\mathcal{L}_{g,n}^\e)$-module of rank $l^{(2g+n).\mathrm{dim}(\mathfrak{g})}$, and $Q(\Ll_{g,n}^{\e})$ is a division algebra and a central simple algebra.

\noindent (ii) $Q(\Ll_{g,n}^\e)$ has PI-degree $l^{g.\mathrm{dim}(\mathfrak{g})+nN}$, where $N$ is the number of positive roots of $\mathfrak{g}$ .
\end{teo}
The (easier) case $g=0$ was obtained for $\Ll_{0,n}^{\eD}$, $\eD$ arbitrary, in \cite[Th.\,4.9 and Th.\,5.2]{BR2}. In Appendix \ref{appDlRoot}, Prop.\,\ref{Z0LgneD}, we show that Th.\,\ref{Lgn} (i) can be extended to $\Ll_{g,n}^{\eD}$ also when $g>0$, where $\eD$ is a primitive $Dl$-th root of unity. If $D$ is odd we can show that $Q(\Ll_{g,n}^\e)$ has PI-degree $D^gl^{g.\mathrm{dim}(\mathfrak{g})+nN}$ (see Th.\,\ref{Lgnteo1eD}). For $D$ even we expect that $Q(\Ll_{g,n}^\e)$ has PI-degree $(D/2)^gl^{g.\mathrm{dim}(\mathfrak{g})+nN}$, which we can prove for $\mathfrak{g}=\mathfrak{sl}_2$ (see Conj.\,\ref{Lgnteo1eDconj}).
\smallskip

Recall that $\Gamma_\e^\Q$ is the specialization of the De Concini-Lyubashenko integral form $\Gamma_A^\Q$. Fundamental in the theory is a coadjoint action of $\Gamma^\Q_\e$ on $\Ll_{g,n}^\e$, which makes it a module algebra. The inclusion of the De Concini-Kac-Procesi integral form $U_A^\Q \subset \Gamma_A^\Q$ factorizes after specialization as $U_\e^\Q \twoheadrightarrow u_\e^\Q \subset \Gamma_\e^\Q$, and we consider the subalgebra $\mathcal{L}_{g,n}^{u_\e}$ of invariant elements for the coadjoint action of the small quantum group $u_\e^\Q$. We stress that $\mathcal{L}_{g,n}^{u_\e}$ strictly contains the subalgebra $\mathcal{L}_{g,n}^{\Gamma_\e^{{\scalebox{0.35}{Q}}}}$ (see \eqref{factoraction} for their precise relation), whose properties will be studied in \cite{BF}. 

The sections \ref{subsecModifiedAlekseev}-\ref{secQMMandInv} are devoted to a modification $$\widehat{\Phi}_{g,n}^\e\colon \Ll_{g,n}^\e \to \widehat{\mathcal{H}}^{\otimes g}_\e \otimes (U_\e^\Pup)^{\otimes n}$$ of the {\em Alekseev morphism} $\Phi_{g,n}^\e\colon \Ll_{g,n}^\e \to \mathcal{H}\mathcal{H}^{\otimes g}_\e \otimes (U_\e^\Pup)^{\otimes n}$ (where $\widehat{\mathcal{H}}_\e$ and $\mathcal{H}\mathcal{H}_\e$ are the specializations to $\e$ of certain extensions of the Heisenberg double of $\Oo_q$), and to the {\em quantum moment map} (QMM) $\mu_{g,n}^\e\colon \Phi_{0,1}^\e(\Ll_{0,1}^\e) \to \Ll_{g,n}^\e$; both play a key role in the proof of Th.\,\ref{Lgn2} below. Besides being injective (Prop.\,\ref{AlekseevInjRootOf1}) and equivariant (Prop.\,\ref{propModifAlekseev}) like $\Phi_{g,n}^\e$, a key property of $\widehat{\Phi}_{g,n}^\e$ is a certain compatibility with the QMM $\mu_{g,n}^\e$ which is not satisfied by $\Phi_{g,n}^\e$ (Prop.\,\ref{propAlekseevModifAndQMM}). It allows us to prove that $\mathcal{L}_{g,n}^{u_\e}$ is the centralizer of the image of $\mu_{g,n}^{\e}$ in $\mathcal{L}_{g,n}^{\e}$ (Th.\,\ref{centralizerMu}). We show that $\mu_{g,n}^\e$ is injective by a ``reduction to the classical QMM'' (Prop.\,\ref{QMMinExactSeq}); for completeness we show also that $\mu_{g,n}$, the QMM defined over $\mc(q)$, is injective (Prop.\,\ref{muinjq}). From this, Theorem \ref{Lgn}, and the double centralizer theorem, we can prove :

\begin{teo}\label{Lgn2} {\rm (see Th.\,\ref{Lgnteo2})} If $(g,n) \neq (0,1)$, the central localization $Q(\Ll_{g,n}^{u_\e})$ is a division algebra, and a central simple algebra of PI-degree $l^{g.\mathrm{dim}(\mathfrak{g})+N(n-1)-m}$, where $l$ is the order of $\e$, $N$ is the number of positive roots of $\mathfrak{g}$ and $m$ is the rank of $\mathfrak{g}$.
\end{teo}
When $(g,n) = (0,1)$ the algebra $\Ll_{0,1}^{u_\e}$ is Abelian and its structure is described in Cor.\,\ref{FrisoG-inv}. The case $g=0$ was obtained in \cite[Theorem 5.4]{BR2} by a similar procedure. In this case the QMM is much simpler, being essentially given by the iterated coproduct $\Delta^{(n)}$ on $U_\e^\Pup$. Also, the modified morphism $\widehat{\Phi}_{g,n}^\e$ is unnecessary, because the compatibility of $\Phi_{0,n}^\e$ with $\mu_{0,n}^{\e}=\Delta^{(n)}$ is much easier, and injectivity of the latter is obvious.

The proof of Th.\,\ref{Lgn2} relies heavily on the structure of the centers of $\Ll_{g,n}^\e$ and $\Ll_{g,n}^{u_\e}$, namely the fact that the multiplication map $j$ in Th.\,\ref{Lgn3} below is an isomorphism, and $j_2$ is injective. To keep clear the line of reasoning of Th.\,\ref{Lgn2} we have preferred to postpone the proof of these facts to \S \ref{sec:proofLgn3}. We introduce a central subalgebra $\mathcal{Z}_1(\mathcal{L}_{g,n}^{\e}) $ of $\Ll_{g,n}^{\e}$, which is essentially the specialization of the center of the subalgebra $1^{\otimes 2g} \otimes \mathcal{L}_{0,n}^{\e} \subset \Ll_{g,n}$. By using the quantum Killing form (see App.\,\ref{sec:qKilling}) we show in Lem.\,\ref{Z0Z1Phimai25} that $\Phi_{0,1}^\e$ restricts to  isomorphisms $\mathcal{Z}_1(\Ll_{0,1}^\e) \cong \mathcal{Z}_1(U_\e^\Pup)$ and $\mathcal{Z}_1(\Ll_{0,1}^\e) \cap \mathcal{Z}_0(\Ll_{0,1}^\e) \cong \mathcal{Z}_1(U_\e^\Pup) \cap \mathcal{Z}_0(U_\e^\Pup)$, the latter intersection having been described by De Concini-Kac-Procesi in their study of the center $\mathcal{Z}(U_\e^\Pup)$ of $U_\e^\Pup$ in \cite{DC-K-P1,DCP}. Let $\mathfrak{d}_{g,n}^{\e} := \mu_{g,n}^\e \circ \Phi_{0,1}^\e : \Ll_{0,1}^\e \to \Ll_{g,n}^{\e}$, and $(\mathcal{Z}_0 \cap \mathcal{Z}_1)(\mathcal{L}_{g,n}^{\e}) := \mathcal{Z}_0(\mathcal{L}_{g,n}^{\e}) \cap \mathcal{Z}_1(\mathcal{L}_{g,n}^{\e})$. Then, relying strongly on results of De Concini-Kac-Procesi we can prove:

\begin{teo}\label{Lgn3} {\rm (Th.\,\ref{Lgnteo3}, Cor.\,\ref{corZZ0})} 1. The rings $\mathcal{Z}(\Ll_{g,n}^{\e})$ and $\mathcal{Z}(\Ll_{g,n}^{u_\e})$ are integrally closed and Noetherian.

\noindent 2. The multiplication maps
\begin{align*}
j\colon \mathcal{Z}_0(\mathcal{L}_{g,n}^{\e}) \otimes_{(\mathcal{Z}_0 \cap \mathcal{Z}_1)(\mathcal{L}_{g,n}^{\e})} \mathcal{Z}_1(\mathcal{L}_{g,n}^{\e}) & \to \mathcal{Z}(\Ll_{g,n}^{\e})\\ j_2\colon \mathcal{Z}(\Ll_{g,n}^{\e}) \otimes_{\mathfrak{d}_{g,n}^{\e}\left((\mathcal{Z}_0\cap \mathcal{Z}_1)(\Ll_{0,1}^\e)\right)} \mathfrak{d}_{g,n}^{\e}\bigl( \mathcal{Z}_1(\Ll_{0,1}^\e) \bigr)& \to \mathcal{Z}(\Ll_{g,n}^{u_\e})
\end{align*} are isomorphisms. 

\noindent 3. $\mathcal{Z}(\Ll_{g,n}^{\e})$ and $\mathcal{Z}(\Ll_{g,n}^{u_\e})$ coincide with the trace rings of $\Ll_{g,n}^{\e}$ and $\Ll_{g,n}^{u_\e}$ respectively. Moreover they are free $\mathcal{Z}_0(\Ll_{g,n}^{\e})$-modules of respective ranks $l^{mn}$ and $l^{m(n+1)}$.
\end{teo}
We also provide a simple consequence of this result for the Frobenius embedding \eqref{Frobembintro} of $\bigl(\mathcal{O}(G)^{\otimes (2g+n)}\bigr)^{G}$, which is the coordinate ring of the variety of $G$-characters of $\pi_1(\Sigma_{g,n}^{\circ})$. Namely, setting $\mathcal{Z}_0^{\rm inv}  := (\mathbb{F}\mathrm{r}_{\epsilon}^*)^{(2g+n)}\bigl[ \bigl(\mathcal{O}(G)^{\otimes (2g+n)}\bigr)^{G} \bigr]$ we prove in Cor.\,\ref{FrisoG-inv} that the following multiplication map is an isomorphisms of algebras
$$\mathcal{Z}_0^{\rm inv}  \otimes_{\mathcal{Z}_0^{\rm inv} \cap \mathcal{Z}_1(\mathcal{L}_{g,n}^{\e})}  \mathcal{Z}_1(\mathcal{L}_{g,n}^{\e}) \to \mathcal{Z}(\mathcal{L}_{g,n}^{\e})^{\Gamma_\e}.$$
This result will be developed in \cite{BF}.
\medskip

\noindent {\bf Applications to representations.} As a consequence of Th.\,\ref{Lgn} and \ref{Lgn2} (see Th.\,\ref{unicityThm}), the simple modules of $\mathfrak{A} := \Ll_{g,n}^{\e}$ or $\Ll_{g,n}^{u_\e}$ have dimension equal or less than the PI-degree $d_\mathfrak{A}$ of $\mathfrak{A}$ described in these statements, each $\chi \in \mathrm{MaxSpec}(\mathcal{Z})$, where $\mathcal{Z}:= \mathcal{Z}(\mathfrak{A})$ is the center, is the central character of some simple finite dimensional $\mathfrak{A}$-module, and there is a ``Unicity theorem for representations'', that is, there is an open and dense Zariski subset $\mathcal{S}$ of $\mathrm{MaxSpec}(\mathcal{Z})$ such that every $\chi \in \mathcal{S}$ is the central character of a unique simple $\mathfrak{A}$-module (up to isomorphism) which has dimension $d_\mathfrak{A}$.

Th.\,\ref{Lgn3} provides an algebraic description of $\mathcal{S}$ by means of the following general facts from the theory of PI rings (see, e.g., \cite[Chap.\,13]{MC-R}, \cite[Chap.\,6]{Rowen}, \cite{Artin}). Let $\mathfrak{A}$ be a division PI algebra, with center $\Zz$ a Noetherian and integrally closed ring. By Th.\,\ref{Lgn}, \ref{Lgn2} and \ref{Lgn3} we can take
$$\mathfrak{A} = \Ll_{g,n}^\e\ {\rm or}\ \Ll_{g,n}^{u_\e}$$ with PI-degree $d_\mathfrak{A}$ as described in these statements. The {\it discriminant ideal} ${\rm Disc}(\mathfrak{A})$ is the ideal of $\Zz$ generated by the elements ${\rm det}\bigl[\bigl(t_{\mathrm{red}}(x_ix_j)\bigr)_{1\leq i,j\leq d_\mathfrak{A}}\bigr]$, where $x_1,\ldots x_{d_\mathfrak{A}}\in \mathfrak{A}$ and $t_{\mathrm{red}}: \mathfrak{A} \rightarrow \Zz$ is the reduced trace map of $\mathfrak{A}$ (see \cite[\S 10]{Reiner}). The set $\mathcal{D}(\mathfrak{A}) \subset {\rm MaxSpec}(\Zz)$ of zeros of ${\rm Disc}(\mathfrak{A})$ is a closed (proper) subvariety of ${\rm MaxSpec}(\Zz)$, called the discriminant variety. Now, given a central character $\chi\in {\rm MaxSpec}(\Zz)$, put $\mathfrak{A}^\chi := \mathfrak{A}/{\rm Ker}(\chi)\mathfrak{A}$. Since $\mathfrak{A}$ is a domain, Nakayama's lemma (in the form of \cite[Cor.\,2.5]{AMacD}) implies that $\mathfrak{A}^\chi\ne 0$. Then we have (see e.g. \cite[Lem.\,3.7]{DCK}, or \cite[Th.\,I.13.5]{BG} for a more general statement):
\begin{cor}\label{disAz} $\mathfrak{A}^\chi$ is isomorphic to $M_{d_\mathfrak{A}}(\mc)$ if and only if $\chi\notin \mathcal{D}(\mathfrak{A})$, and if $\chi\in \mathcal{D}(\mathfrak{A})$ then every irreducible representation of $\mathfrak{A}^\chi$ has dimension less than ${d_\mathfrak{A}}$. 
\end{cor} 
The set $\mathcal{A}_{\mathfrak{A}} := {\rm MaxSpec}(\Zz)\setminus \mathcal{D}(\mathfrak{A})$ is the {\it Azumaya locus} of $\mathfrak{A}$ \cite[III.1.6-III.1.7]{BG}. Loci of irreducible representations of dimension $< d$ can be reached by using lower discriminant ideals (see the Main Theorem of \cite{BY}).

It is in general a difficult problem to describe Azumaya loci. However $\mathfrak{A} = \Ll_{g,n}^\e$ and $\Ll_{g,n}^{u_\e}$ have a natural structure of $\mathcal{Z}_0(\Ll_{g,n}^\e)$-Poisson order described in \cite[\S 2.1-2.2]{BrGo}, which gives ${\rm MaxSpec}\bigl(\Zz_0(\Ll_{g,n}^\e)\bigr)$ a structure of Poisson variety that may be used to study $\mathcal{A}_{\mathfrak{A}}$ geometrically. The Poisson structure coincides with the one defined by the Fock-Rosly Poisson bracket (\cite{Frolov}; see also \cite[Prop. 4.3]{GJS}, and \cite{BR1} for the detailed case of $\mathcal{Z}_0(\Ll_{0,n}^\e(\mathfrak{sl}_2)$), and by \cite[Th.\,4.2]{BrGo} the set $\mathcal{A}_{\mathfrak{A}}$ is a union of symplectic leaves of the pull-back of that Poisson bracket by the restriction map ${\rm MaxSpec}(\Zz(\Ll_{g,n}^\e))\to {\rm MaxSpec}(\Zz_0(\Ll_{g,n}^\e))$. This viewpoint has been used in \cite{GJS} (case $n=0$) and \cite{KaruoKorinman} (case of $\mathfrak{g}=\mathfrak{sl}_2$).

\medskip

\noindent \textbf{Acknowledgements.} We are grateful to Joao-Pedro Dos Santos for valuable discussions on topics in algebraic geometry. Part of this work was done during the postdoctoral contract of M.F. at IMT Toulouse, funded by the CIMI Labex ANR 11-LABX-0040.

\section{General preliminaries} \label{sectionnotations}

In all this paper $\Bbbk$ denotes a field and except otherwise stated ``algebra'' means ``associative $\Bbbk$-algebra''. The center of an algebra $A$ is denoted $\mathcal{Z}(A)$.

\subsection{Simple Lie algebras}\label{subsecLieAlg}
We fix a complex finite dimensional simple Lie algebra $\mathfrak{g}$ and denote by $U(\mathfrak{g})$ its universal enveloping algebra. 

Denote by $m$ the rank of $\mathfrak{g}$. We fix a Cartan subalgebra $\mathfrak{h}\subset \mathfrak{g}$, and a basis of simple roots $\alpha_1,\cdots, \alpha_m\in \mathfrak h^*$. Let $C:=(a_{ij}) \in \mathrm{Mat}_m(\mathbb{Z})$ be the Cartan matrix of $\mathfrak{g}$. Let $d_i\in \{1,2,3\}$ the smallest positive integers such that $(d_ia_{ij})$ is a symmetric matrix. We denote by $(\cdot , \cdot )$ the scalar product on $\mathfrak h^*$ defined by the Killing form of $\mathfrak{g}$, normalized so that $(\alpha_i,\alpha_j) := d_ia_{ij}$. Then $(\alpha,\alpha)=2$ for the short roots $\alpha$. The set of positive roots is denoted $\phi^+=\{\beta_1,\cdots,\beta_N\}$, so ${\rm dim}({\mathfrak{g}})=2N+m$, and we have the Cartan decomposition $\mathfrak{g}=\mathfrak{n}^-\oplus\mathfrak{h}\oplus\mathfrak{n}^+$. We put $\rho:=\textstyle \frac{1}{2}\sum_{i=1}^N\beta_i$, and the coroot of a simple root $\alpha_i$ is $\check{\alpha}_i := 2\alpha_i/(\alpha_i,\alpha_i)$.

Let $P=\textstyle \sum_{i=1}^m \mathbb{Z}. \varpi_i\subset{\mathfrak h}^*$ be the lattice of integral weights, with the fundamental weights $\varpi_i$ satisfying $(\varpi_i,\check{\alpha}_j) = \delta_{i,j}$, $P_+ = \textstyle \sum_{i=1}^m \mathbb{N}. \varpi_i$ the set of integral dominant weights, and $Q = \textstyle \sum_{i=1}^m \mathbb{Z}. \alpha_i$, $\check{Q} = \textstyle \sum_{i=1}^m \mathbb{Z}. \check{\alpha}_i$ are the the root and coroot lattices.

We denote by $D$ the smallest positive integer such that $DP\subset Q$; it is also the smallest positive integer such that $D(\lambda,\varpi_i)\in d_i\mz$ for all $\lambda\in P$, $i\in \{1,\ldots,m\}$. Note that $D$ is the exponent of the group $P/Q$. The entries of $C^{-1}$ belong to $D^{-1}\mn$. The values of max$(d_i)$, $D$ and $P/Q$, are recalled in the following table (where $\mz_n:
=\mz/n\mz$):
\medskip

\begin{center}
\begin{tabular}{||c|| c | c | c |c |c |c |c |c|c|c ||} 
\hline 
type of $\mathfrak{g}$ & $A_{n\geq 1}$ & $B_{n\geq 2}$ & $C_{n\geq 3}$ & \begin{tabular}{c}$D_{n\geq 5}$ \\ $n$ odd \end{tabular} & \begin{tabular}{c}$D_{n\geq 4}$ \\ $n$ even \end{tabular} & $E_6$ & $E_7$ & $E_8$ & $F_4$ & $G_2$ \\ 
\hline
\rule{0cm}{0.6cm} \raisebox{0.15cm}{$\mathfrak{g}$} & \raisebox{0.1cm}{$\mathfrak{sl}_{n+1}$} & \raisebox{0.1cm}{$\mathfrak{so}_{2n+1}$} & \raisebox{0.1cm}{$\mathfrak{sp}_{2n}$} & \raisebox{0.1cm}{$\mathfrak{so}_{2n}$} & \raisebox{0.1cm}{$\mathfrak{so}_{2n}$} & \raisebox{0.1cm}{$\mathfrak{e}_6$} & \raisebox{0.1cm}{$\mathfrak{e}_7$} & \raisebox{0.1cm}{$\mathfrak{e}_8$} & \raisebox{0.1cm}{$\mathfrak{f}_4$} & \raisebox{0.1cm}{$\mathfrak{g}_2$}\\
\hline 
\rule{0cm}{0.6cm} \raisebox{0.1cm}{$\max(d_i)$} & \raisebox{0.1cm}{$1$} &  \raisebox{0.1cm}{$2$} &  \raisebox{0.1cm}{$2$} &  \raisebox{0.1cm}{$1$} & \raisebox{0.1cm}{$1$} & \raisebox{0.1cm}{$1$} & \raisebox{0.1cm}{$1$} & \raisebox{0.1cm}{$1$} & \raisebox{0.1cm}{$2$} & \raisebox{0.1cm}{$3$} \\
\hline
\rule{0cm}{0.6cm} \raisebox{0.1cm}{$D$} &
\raisebox{0.1cm}{$n+1$} & \raisebox{0.1cm}{$2$} & \raisebox{0.1cm}{$2$} & \raisebox{0.1cm}{$4$} &  \raisebox{0.1cm}{$2$} &
\raisebox{0.1cm}{$3$} & \raisebox{0.1cm}{$2$} & \raisebox{0.1cm}{$1$} & \raisebox{0.1cm}{$1$}& \raisebox{0.1cm}{$1$}\\ 
\hline
\rule{0cm}{0.6cm} \raisebox{0.1cm}{$P/Q$} &
\raisebox{0.1cm}{$\mz_{n+1}$} & \raisebox{0.1cm}{$\mz_2$} & \raisebox{0.1cm}{$\mz_2$} & \raisebox{0.1cm}{$\mz_4$} & \raisebox{0.1cm}{$\mz_2\times \mz_2$} &
\raisebox{0.1cm}{$\mz_3$} & \raisebox{0.1cm}{$\mz_2$} & \raisebox{0.1cm}{$\{0\}$}& \raisebox{0.1cm}{$\{0\}$}& \raisebox{0.1cm}{$\{0\}$} \\ 
\hline
\end{tabular}
\end{center}
\medskip

\noindent Any intermediate lattice $Q\subset \Lambda\subset P$ has rank $m = \dim(\mathfrak{h})$. 

\indent We denote by $W$ the Weyl group of $\mathfrak{g}$, by $w_0$ the longest element of $W$, and by ${\mathcal B}(\mathfrak{g})$ the braid group of $\mathfrak{g}$.

\indent We denote by $G^{\Lam}$ the connected simple complex algebraic group with Lie algebra $\mathfrak{g}$ such that its maximal torus $T^{\Lam}\subset G^{\Lam}$ associated with $\mathfrak{h}$ has character group ${\rm Hom}_{\rm Ab}(T^{\Lam},\mc^\times)\cong \Lambda$. In particular $G := G^{\Pup}$ is the simply-connected group, and $G^\Q ={\rm Image}(Ad\colon G \to {\rm Aut}(\mathfrak{g}))$ is the adjoint group. We denote by $B_\pm$ the Borel subgroup of $G$ with Lie algebra $\mathfrak{b}_\pm$, and by $T\subset G$ the corresponding maximal torus. 

\subsection{Parameters \texorpdfstring{$q$}{q}, \texorpdfstring{$q_{\scalebox{0.5}{D}}$}{qD}, \texorpdfstring{$\e$}{e} and \texorpdfstring{$\e_{\scalebox{0.5}{D}}$}{eD}}\label{secnotationq}
We let $q$, $\qD$ be indeterminates such that $\qD^D=q$, and denote $$A=\mc[q,q^{-1}]\ , \AD =\mc[\qD,\qD^{-1}]\ ,\ q_i=q^{d_i}\ ,\ \qDi=\qD^{d_i}.$$

By $\mathbb{C}(q)$ we denote the field of rational fractions in the indeterminate $q$. We put $[0]_q !:=1$, $[ x; 0]_{q}:=1$, $(x; 0)_{q}:=1$ and for $k,t\in \mathbb{N}\setminus \{0\}$, 
$$\begin{array}{ccc} [ x; t ]_{q}:= \prod_{s=1}^t \frac{xq^{-s+1}-x^{-1}q^{s-1}}{q^s-q^{-s}} & , & (x; t )_{q}:= \prod_{s=1}^t \frac{xq^{-s+1}-1}{q^s-1}\\ \left[ k \right]_q:=\frac{q^k-q^{-k}}{q-q^{-1}} & , & [k]_q!:=[k]_q\cdots [1]_q.
\end{array}$$
For $m,t\in \mn$ such that $m\geq t$ we have $[ q^m; t ]_{q}$, $(q^m; t )_{q} \in A$. They are quantum versions of the standard binomial coefficient $\begin{pmatrix} m \\ t \end{pmatrix}$. Note that $[q^m;t]_q = (q^m;t)_q = 0$ if $m < t$, still with $m,t \in \mathbb{N}$.

\indent We fix an integer $l\geq 3$ and let $\e$ be a primitive $l$-th root of unity. For simplicity, in all the paper we require that:
\begin{itemize}
\item $l$ is odd, and in case $\mathfrak{g}$ is of type $G_2$, $l$ is coprime to $3$. Thus ${\rm gcd}(l,d_i)=1$, and $l$ is the smallest positive integer such that $\epsilon_i^l=1$ for all $i$. 
\item ${\rm gcd}(l,D)=1$; because $l$ is odd, this is a constraint only in type $A_{n}$ for $n$ even and type $E_6$.
\end{itemize}
In particular, there exists a unique integer $\overline{D}$ modulo $l$ such that $D\overline{D} \equiv 1 \pmod{l}$. We set $\e_{\scalebox{0.5}{D}} = \epsilon^{\overline{D}}$, so that $\eD^D = \e$. Note that $\eD$ satisfies $\eD^l =1$. Also we put $\e_i:=\e^{d_i}$ and $\eDi:=\eD^{d_i}$. A different choice of $\eD$ is discussed in App.\,\ref{appDlRoot}.

\subsection{Specialization of \texorpdfstring{$A$}{A}-modules}\label{generalRmkSpe}
Let $z \in \mathbb{C}^\times$ and $A = \mathbb{C}[q^{\pm 1}]$. Then $\mathbb{C}$ can be seen as an $A$-module $\mathbb{C}_z$ by means of the evaluation morphism $q \mapsto z$. If $X$ is an $A$-module, we denote by $X_z = X \otimes_A \mathbb{C}_z$ the {\em specialization of $X$ at $q=z$}; it is a $\mathbb{C}$-vector space. For any $x \in X$, we let $x_{|z} = x \otimes_A 1$. The map $X \to X_z$ given by $x \mapsto x_{|z}$ is surjective by definition and satisfies $(P \cdot x)_{|z} = P(z)x_{|z}$ for all Laurent polynomial $P \in A$ and $x \in X$.

\begin{lem}\label{lemmaSpecialisation}
1. We have $X_z \cong X\big/(q-z)X$ as $\mathbb{C}$-vector spaces. It follows that the kernel of the specialization map $x \mapsto x_{|z}$ is $(q-z)X$.
\\2. If the elements $\{ x_i \}$ are an $A$-basis of $X$, then the elements $\{ (x_i)_{|z} \}$ are a $\mathbb{C}$-basis of $X_z$.
\end{lem}
\begin{proof}
1. We check that $X\big/(q-z)X$ satisfies the universal property of $X \otimes_A \mathbb{C}_z$. Let $u : X \times \mathbb{C}_z \to X\big/(q-z)X$, $(x,c) \mapsto cx + (q-z)X$ which is obviously an $A$-balanced map, i.e. $u(P\cdot x,c) = u\bigl( x, P(z)c \bigr)$ for all $P \in A$. Let $h : X \times \mathbb{C}_z \to G$ be another morphism of abelian groups which is $A$-balanced. Define $\bar{h} : X\big/(q-z)X \to G$ by $\bar{h}\bigl(x + (q-z)X \bigr) = h(x,1)$. It is easy to check that $\bar{h}$ is well-defined and satisfies $\bar{h} \circ u = h$. By surjectivity of $u$ this factorization of $h$ is unique, which is the desired universal property. With this model of $X \otimes_A \mathbb{C}_z$ we have $x \otimes_A 1 = u(x,1) = x + (q-z)X$, whence the second claim.
\\2. Assume we have a vanishing linear combination $\textstyle \sum_i \lambda_i (x_i)_{|z} = 0$ with $\lambda_i \in \mathbb{C}$. Viewing the $\lambda_i$'s as constant polynomials in $A$, this reads $\textstyle \bigl( \sum_i \lambda_i \cdot x_i \bigr)_{|z} = 0$. By item 1 it follows that $\textstyle \sum_i \lambda_i \cdot x_i = (q-z)x'$ for some $x' \in X$. Write $x' = \textstyle \sum_i P_i \cdot x_i$ with $P_i \in A$. Identifying coefficients in the basis $\{x_i\}$ gives $\lambda_i = (q-z)P_i$. Hence $\lambda_i = 0$ for all $i$, proving that the family $\{ (x_i)_{|z} \}$ is linearly independent. It is a generating family because the specialization map is surjective.
\end{proof}

\indent Let $Y$ be another $A$-module and $f \in \mathrm{Hom}_A(X,Y)$. Then we have the $\mathbb{C}$-linear map $f_z = f \otimes_A \mathrm{id}_{\mathbb{C}_z} : X_z \to Y_z$. It holds $f_z(x_{|z}) = f(x)_{|z}$, which is another way of defining $f_z$. We record the basic properties:
\begin{equation}\label{speSurjAndIso}
f \text{ surjective} \: \implies \: f_z \text{ surjective}, \qquad f \text{ isomorphism} \: \implies \: f_z \text{ isomorphism}.
\end{equation}
The first assertion is clear (more generally any tensor product functor is right exact); for the second it suffices to note that $(f^{-1})_z = (f_z)^{-1}$ by functoriality. Injectivity is not preserved in general; so one must be careful with specializations of submodules, because the specialization of the inclusion map is in general no longer injective.

\smallskip

\indent Note that, using associativity of $\otimes$ and the fact that $\mathbb{C}_z$ is an $(A,\mathbb{C})$-bimodule,
\[ (X \otimes_A Y)_z = X \otimes_A Y \otimes_A \mathbb{C}_z = X \otimes_A Y_z = X \otimes_A \mathbb{C}_z \otimes_{\mathbb{C}} Y_z = X_z \otimes_{\mathbb{C}} Y_z. \]
Hence $(x \otimes_A y)_{|z}$ is identified with $x_{|z} \otimes y_{|z}$, where we simply write $\otimes$ instead of $\otimes_{\mathbb{C}}$.

\indent We will use pairings over $A$, which are $A$-bilinear maps of the form $b : X \times Y \to A$. By the above remarks, the specialized version $b_z = b \otimes_A \mathrm{id}_{\mathbb{C}_z}$ is the $\mathbb{C}$-bilinear map $b_z : X_{|z} \times Y_{|z} \to \mathbb{C}$ defined by $b_z(x_{|z}, y_{|z}) = b(x,y)_{|z}$.

\subsection{Central localization of a domain}\label{subsecCentralLoc}
A {\em domain} (a.k.a. ring without zero divisors) is a ring in which $xy=0$ implies $x=0$ or $y=0$. If $\mathcal{Z}$ is a commutative domain, we denote by $Q(\mathcal{Z})$ its field of fractions.

\indent Let $R$ be a domain and denote by $\mathcal{Z} := \mathcal{Z}(R)$ its center. The {\em central localization} of $R$ is
\[ Q(R) := Q(\mathcal{Z}) \otimes_{\mathcal{Z}} R. \]
The functor $Q(\mathcal{Z}) \otimes_{\mathcal{Z}} - : \mathcal{Z}\text{-}\mathrm{Mod} \to Q(\mathcal{Z})\text{-}\mathrm{Vect}$ is isomorphic to the localization functor at $\mathcal{Z}\setminus\! \{0\}$, see e.g. \cite[Prop.\,1.10.18]{Rowen}. Hence $Q(R)$ can also be defined as the localization $(\mathcal{Z}\setminus \{0\})^{-1}R$, and their natural ring structures are isomorphic \cite[Cor.\,1.10.18']{Rowen}. Explicitly, $\textstyle \frac{z}{z'} \otimes_{\mathcal{Z}} r$ is identified with the fraction $\textstyle \frac{zr}{z'}$. Note that $Q(R)$ has a natural structure of $Q(\mathcal{Z})$-algebra, and we denote by $\bigl[ Q(R):Q(\mathcal{Z}) \bigr]$ its dimension over that field.

\begin{lem}\label{lemTrucsGenerauxQR}
1. The ring morphism $q_R : R \to Q(R)$, $r \mapsto 1 \otimes_{\mathcal{Z}} r$ is injective.
\\2. The center of $Q(R)$ is $Q( \mathcal{Z})$, or more precisely $Q( \mathcal{Z}) \otimes _{\mathcal{Z}} 1$.
\\3. If $Z_0 \subset \mathcal{Z}$ is a subring such that $Q(\mathcal{Z})$ is a finite extension of $Q(Z_0)$, then $Q(R) \cong Q(Z_0) \otimes_{Z_0} R$.
\\4. If $R$ is finitely generated as a $\mathcal{Z}$-module then $\bigl[ Q(R):Q(\mathcal{Z}) \bigr] < \infty$. If moreover $R$ is free over $\mathcal{Z}$ then $\mathrm{rank}_{\mathcal{Z}}(F) = \bigl[ Q(R):Q(\mathcal{Z}) \bigr]$.
\end{lem}
\begin{proof}
1. Let $0 \neq r \in R$. The $\mathcal{Z}$-linear map $m_r : \mathcal{Z} \to R$, $z \mapsto zr$ is injective because $R$ is a domain. The functor $Q(\mathcal{Z}) \otimes_{\mathcal{Z}} -$ being exact \cite[Th.\,3.1.20]{Rowen}, the resulting $Q(\mathcal{Z})$-linear map $\widetilde{m}_r : Q(\mathcal{Z}) \to Q(R)$, $\textstyle \frac{z}{z'} \mapsto \frac{z}{z'} \otimes_{\mathcal{Z}} r$ is injective. Hence $\widetilde{m}_r(1) = 1 \otimes_{\mathcal{Z}} r \neq 0$.
\\2. Note that a general element $\textstyle \sum_{i=1}^n \frac{a_i}{b_i} \otimes_{\mathcal{Z}} r_i$ can be rewritten as $\textstyle \frac{1}{b_1\ldots b_n} \otimes_{\mathcal{Z}} \bigl( \sum_{i=1}^n r_i a_i\prod_{j \neq i}b_j \bigr)$ and hence we are reduced to elements of the form $z^{-1} \otimes_{\mathcal{Z}} r$. Assume such an element is central. Then $1 \otimes _{\mathcal{Z}} r$ is central as well, which by item 1 implies that $r$ is central in $R$. But then $z^{-1} \otimes_{\mathcal{Z}} r = z^{-1}r \otimes_{\mathcal{Z}} 1$ and we are done.
\\3. The key is to prove that $Q(\mathcal{Z}) \cong Q(Z_0) \otimes_{Z_0} \mathcal{Z}$. For this let us first show that $Q(Z_0) \otimes_{Z_0} \mathcal{Z}$ is a field, which is easily seen to be equivalent to the claim that any non-zero element $1 \otimes_{Z_0} z$ has an inverse. Let $P = \alpha_0 + \alpha_1X + \ldots + \alpha_dX^d \in Q(Z_0)[X]$ be the minimal polynomial of $z \in Q(\mathcal{Z})$ over $Q(Z_0)$; note that $\alpha_0 \neq 0$ by minimality of $P$. In $ Q(Z_0) \otimes_{Z_0} \mathcal{Z}$ we can rewrite the equality $P(z) = 0$ as $(1 \otimes_{Z_0} z)\bigl[ -(\alpha_0^{-1} \otimes_{Z_0} \bigl( \alpha_d z^{d-1} + \alpha_{d-1}z^{d-2} + \ldots + \alpha_1 \bigr) \bigr] = 1$, giving the inverse of $1 \otimes_{Z_0} z$. We deduce that the morphism $Q(Z_0) \otimes_{Z_0} \mathcal{Z} \to Q(\mathcal{Z})$, $\frac{z_0}{z'_0} \otimes_{Z_0} z \mapsto \frac{z_0z}{z'_0}$ is surjective; it is also injective because its source is a field. As a result
\[ Q(\mathcal{Z}) \otimes_{\mathcal{Z}} R \cong \bigl( Q(Z_0) \otimes_{Z_0} \mathcal{Z} \bigr) \otimes_{\mathcal{Z}} R = Q(Z_0) \otimes_{Z_0} \bigl( \mathcal{Z} \otimes_{\mathcal{Z}} R \bigr) \cong Q(Z_0) \otimes_{Z_0} R. \]
4. If $\{r_i\}_{i=1}^n \subset R$ is a generating family over $\mathcal{Z}$ then any $z^{-1} \otimes_{\mathcal{Z}} r \in Q(R)$ can be written as $z^{-1} \otimes_{\mathcal{Z}} \bigl( \textstyle \sum_{i=1}^n z_ir_i \bigr) = \textstyle \sum_{i=1}^n z^{-1}z_i \otimes_{\mathcal{Z}} r_i$, proving that $\{1 \otimes_{\mathcal{Z}} r_i\}_{i=1}^n \subset Q(R)$ is a generating family over $Q(\mathcal{Z})$. If $\{r_i\}_{i=1}^n$ is a basis then so is $\{1 \otimes_{\mathcal{Z}} r_i\}_{i=1}^n$ thanks to item 1.
\end{proof}

\indent The central localization is especially interesting when $Q(R)$ is finite over $Q(\mathcal{Z})$.
\begin{lem}\label{lemQRCSA}
If $\bigl[ Q(R):Q(\mathcal{Z}) \bigr] < \infty$ then $Q(R)$ is a division algebra. In particular, it is a central simple algebra over the field $Q(\mathcal{Z})$ and $[Q(R) : Q(\mathcal{Z})]=d^2$ for some integer $d$ called the {\em PI-degree} of $Q(R)$.
\end{lem}
\begin{proof}
The first claim uses the easy fact that a finite-dimensional algebra without zero divisors is a division algebra. In particular $Q(R)$ is a simple $Q(\mathcal{Z})$-algebra, and its center is $Q(\mathcal{Z})$ by Lemma \ref{lemTrucsGenerauxQR}(2). The last claim is a general fact: if $S$ is a finite-dimensional central simple algebra over a field $F$, then $[S : F] = d^2$. This follows from the existence of a splitting field, i.e. a field extension $\mathbb{F}$ of $F$ such that $\mathbb{F} \otimes_F S \cong M_d(\mathbb{F})$; see e.g. \cite[Cor.\,2.3.25]{Rowen} or \cite[\S 13.3]{MC-R}.
\end{proof}

Although it looks a bit technical the next lemma will allow us to perform a ``reduction to the classical case'' in \S\ref{subsecPIdegInv}, when $R$ is a specialization of a ``quantum algebra'':
\begin{lem}\label{lemmeTechniqueInjectivite}
Let $R, S$ be domains, $Z_0 \subset \mathcal{Z}(R)$ be a central subring and $f : R \to S$ be a ring morphism. Assume that $\bigl[ Q(R) : Q(Z_0) \bigl] < \infty$ and $f\bigl( Z_0 \setminus \{0\} \bigr) \subset \mathcal{Z}(S) \setminus \{ 0 \}$. Then $f$ is injective.
\end{lem}
\begin{proof}
The dimension assumption implies $\bigl[ Q\bigl( \mathcal{Z}(R) \bigr) : Q(Z_0) \bigl] < \infty$ and thus $Q(R)$ $=$ $Q(Z_0)$ $\otimes_{Z_0} R$ thanks to Lemma \ref{lemTrucsGenerauxQR}(3). Due to the assumption on $f$ the following map is well-defined:
\[ \widetilde{f} : Q(R) = Q(Z_0) \otimes_{Z_0} R \to Q(S), \quad  \frac{z_0}{z'_0} \otimes_{Z_0} r \mapsto \frac{f(z_0)}{f(z'_0)} \otimes_{\mathcal{Z}(S)} f(r). \]
The dimension assumption also gives $\bigl[ Q(R) : Q\bigl( \mathcal{Z}(R) \bigr) \bigr] < \infty$, so that $Q(R)$ is a division algebra by Lemma \ref{lemQRCSA}. Hence the ring morphism $\widetilde{f}$ is necessarily injective. Recall the embeddings $q_R : R \hookrightarrow Q(R)$ and $q_S : S \hookrightarrow Q(S)$ from Lemma \ref{lemTrucsGenerauxQR}(1). By very definition we have $q_S \circ f = \widetilde{f} \circ q_R$, which forces injectivity of $f$.
\end{proof}

\indent Under mild assumptions on $R$ the PI-degree of $Q(R)$ defined in Lemma \ref{lemQRCSA} is of extreme relevance for the representation theory of $R$, as we now recall. Assume that $R$ is an associative algebra over an algebraically closed field $\Bbbk$, still without zero divisors. Let $V$ be a simple $R$-module which is finite-dimensional over $\Bbbk$; then by Schur's lemma if $z \in \mathcal{Z} = \mathcal{Z}(R)$ there exists $\lambda_z \in \Bbbk$ such that $z \cdot v = \lambda_z v$ for all $v \in V$. The morphism
\[ \chi_V : \mathcal{Z} \to \Bbbk, \quad z \mapsto \lambda_z. \]
is called the {\em central character} of $V$. More generally, we denote by $\mathrm{MaxSpec}(\mathcal{Z})$ the set of characters of $\mathcal{Z}$ (i.e. algebra morphisms $\mathcal{Z} \to \Bbbk$) and recall that it is endowed with the Zariski topology.

\begin{teo}\label{unicityThm}
Let $R$ be an associative algebra over an algebraically closed field $\Bbbk$. Assume that $R$ is a domain, finitely generated as an algebra, and finitely generated as a module over its center $\mathcal{Z}$. Denote by $d$ the PI-degree of $Q(R)$ defined in Lemma \ref{lemQRCSA}. Then:
\\1. Each $\chi \in \mathrm{MaxSpec}(\mathcal{Z})$ is the central character of some simple $R$-module, which is finite-dimensional over $\Bbbk$.
\\2. Any finite-dimensional simple $R$-module has dimension $\leq d$.
\\3. {\em (``Unicity theorem for representations'')} There is an open and dense Zariski subset $\mathcal{S}$ of $\mathrm{MaxSpec}(\mathcal{Z})$ such that $R\big/R\ker(\chi) \cong M_d(\Bbbk)$ for all $\chi \in \mathcal{S}$. It follows that every $\chi\in \mathcal{S}$ is the central character of a unique simple $R$-module (up to isomorphism), which has dimension $d$.
\end{teo}
\begin{proof}
See for instance \cite[\S III.1.6]{BG}. A pedagogical detailed proof is also given in \cite[\S 2]{FKL}. 
\end{proof}
\noindent Theorem \ref{unicityThm} means that ``generically'' a character of $\mathcal{Z}$ determines a unique simple $R$-module.

\section{General facts on graph algebras}\label{sectiongraphalg}
Here we explain how to define the so-called {\em graph algebra} $\mathcal{L}_{g,n}(H)$ when the Hopf algebra $H$ is not necessarily quasitriangular in the strict sense ({\it i.e.}, $H$ has an $R$-matrix in an appropriate completion of $H\otimes H$), as well as the associated notions of Alekseev morphism and quantum moment map. This discussion is of course guided by the case $H = U_q(\mathfrak{g})$.

\subsection{A substitute for quasitriangularity}\label{subsecSubstQuasi}
Let $\Bbbk$  be a field and $H = (H,\cdot,1_H,\Delta,\varepsilon,S)$ be a Hopf $\Bbbk$-algebra with invertible antipode $S$. The $n$-th iterated coproduct of $H$ is denoted $\Delta^{(n)} := (\Delta \otimes \mathrm{id})\circ \Delta^{(n-1)}$, $n\geq 2$ (and by convention $\Delta^{(1)} :={\rm Id}$, $\Delta^{(2)} := \Delta$). We use Sweedler's coproduct notation: $\Delta^{(n)}(x) = \textstyle \sum_{(x)} x_{(1)}\otimes \ldots \otimes x_{(n)}$, sometimes with implicit summation.

\smallskip

\indent If $V$ is a finite-dimensional $H$-module, then for any $v \in V$ and $f \in V^*$ the linear form $f(? \cdot v) : H \to \Bbbk$ given by $h \mapsto f(h \cdot v)$ is called a {\em matrix coefficient} of $V$. Denote by $H^{\circ}$ the {\em restricted dual} of $H$ (a.k.a. finite dual), which is the subspace of $H^*$ generated by matrix coefficients of finite-dimensional $H$-modules. Endow this space with the usual Hopf structure $(H^\circ,\star,{\rm 1}_{H^\circ},\Delta_{H^\circ},\varepsilon_{H^\circ},S_{H^\circ})$, which is dual to that of $H$:
\begin{equation}\label{usualHopfDual}
\begin{array}{c}
(\varphi \star \psi)(x) = \sum_{(x)}\varphi(x_{(1)}) \psi(x_{(2)}), \quad \varphi(xy) = \sum_{(\varphi)}\varphi_{(1)}(x) \varphi_{(2)}(y) = \bigl\langle \Delta_{H^{\circ}}(\varphi), x\otimes y \bigr\rangle\\[.3em] \varepsilon_{H^\circ}(\varphi) = \varphi(1_H), \quad 1_{H^\circ} = \varepsilon, \quad S_{H^\circ}(\varphi) = \varphi \circ S
\end{array}
\end{equation}
for all $\varphi,\psi \in H^\circ$ and $x,y \in H$. The coregular actions $\rhd$, $\lhd$ of $H$ on $H^\circ$ are defined by
\begin{equation}\label{coregActions}
\forall \, \varphi \in H^\circ, \:\: \forall \, h,x \in H, \quad
(h \rhd \varphi)(x) = \varphi(xh), \quad (\varphi \lhd h)(x) = \varphi(hx).
\end{equation}
which can also be written as $h \rhd \varphi =  \textstyle \sum_{(\varphi)}\varphi_{(1)}\,\varphi_{(2)}(h)$, $\varphi \lhd h =  \textstyle \sum_{(\varphi)}\varphi_{(1)}(h)\varphi_{(2)}$.

\begin{remark}\label{remarkRestDualForC}
{\em More generally, if $\mathcal{C}$ is a full subcategory of $H\text{-}\mathrm{mod}$ (the finite-dimensional $H$-modules) which is stable by monoidal product, finite direct sums, and duals then there is a Hopf algebra $H^\circ_{\mathcal{C}} \subset H^\circ$ spanned by the matrix coefficients of objects in $\mathcal{C}$; note that this space is stable by the coregular actions of $H$. One can replace $H^\circ$ with such an $H^\circ_{\mathcal{C}}$ everywhere in the discussion below.}
\end{remark}

\indent We now make three assumptions on the Hopf algebra $H$.

\smallskip

\noindent \textit{Assumption 1:} $H^\circ$ separates the points of $H$, \textit{i.e.} the duality pairing $\langle \: , \: \rangle : H^\circ \times H \to \Bbbk$ is non-degenerate on the right (it is non-degenerate on the left by definition of $H^\circ$). We write indifferently $\varphi(h)$ or $\langle \varphi, h \rangle$.

\smallskip

\noindent \textit{Assumption 2:} We are given a morphism of Hopf algebras $\Phi^+ : H^{\circ} \to H^{\mathrm{cop}}$ (where $H^{\mathrm{cop}}$ is $H$ with opposite coproduct) such that
\begin{equation}\label{quasiCocommPhi}
\forall \, \varphi \in H^{\circ}, \:\: \forall\, h \in H, \quad \sum_{(h)} h_{(1)}\Phi^+\bigl( \varphi \lhd h_{(2)} \bigr) = \sum_{(h)} \Phi^+\bigl( h_{(1)} \rhd \varphi \bigr) h_{(2)}.
\end{equation}

\smallskip

\noindent \textit{Assumption 3:} There exists a $\Bbbk$-linear map $\Phi^- : H^{\circ} \to H$ such that
\begin{equation}\label{PhipmInverses}
\forall \, \varphi, \psi \in H^{\circ}, \quad \varphi\bigl( \Phi^-(\psi) \bigr) = \psi\bigl( \Phi^+(S(\varphi)) \bigr).
\end{equation}

\indent Several comments are in order.  First, since $H^\circ$ separates the points of $H$, it is easy to see that $\Phi^-$ has exactly the same properties as $\Phi^+$: it is a morphism of Hopf algebras $H^{\circ} \to H^{\mathrm{cop}}$ which satisfies \eqref{quasiCocommPhi}.

\indent It is a straightforward exercise to check from Assumption 2 that the bilinear form
\begin{equation}\label{defCoRMat}
\mathcal{R} : H^{\circ} \otimes H^{\circ} \to \Bbbk, \quad \mathcal{R}(\varphi \otimes \psi) = \psi\bigl( \Phi^+(\varphi) \bigr)
\end{equation}
is a co-R-matrix in the sense of Majid \cite[Def.\,2.2.1]{Majid}, \textit{i.e.} we have
\begin{align}
&\sum_{(\varphi),(\psi)}\mathcal{R}(\varphi_{(1)} \otimes \psi_{(1)})\varphi_{(2)} \star \psi_{(2)} = \sum_{(\varphi),(\psi)}\mathcal{R}(\varphi_{(2)} \otimes \psi_{(2)})\psi_{(1)}\star \varphi_{(1)}, \label{braidedComm}\\
&\mathcal{R}(\varphi \star \psi \otimes \eta) = \sum_{(\eta)} \mathcal{R}\bigl(\varphi \otimes \eta_{(1)}\bigr)\mathcal{R}\bigl(\psi \otimes \eta_{(2)}\bigr),\label{coTriang1}\\
&\mathcal{R}(\varphi \otimes \psi \star \eta) = \sum_{(\varphi)} \mathcal{R}\bigl(\varphi_{(1)} \otimes \eta\bigr)\mathcal{R}\bigl(\varphi_{(2)} \otimes \psi\bigr)\label{coTriang2}
\end{align}
for all $\varphi,\psi,\eta \in H^\circ$. The equality \eqref{braidedComm} comes from \eqref{quasiCocommPhi} while \eqref{coTriang1}--\eqref{coTriang2} directly follow from the fact that $\Phi^+$ is a morphism of Hopf algebra and the definition of $\star$. Note that these identities imply
\begin{equation}\label{coReps}\mathcal{R}\bigl(\varphi\otimes 1_{H^\circ}\bigr) = \mathcal{R}\bigl(1_{H^\circ} \otimes \varphi\bigr) = \varepsilon_{H^\circ}(\varphi).
\end{equation}
The map $\mathcal{R}^{-1} : H^\circ \otimes H^\circ \to \Bbbk$ given by $\mathcal{R}^{-1}(\varphi \otimes \psi) = \psi\bigl( \Phi^+(S(\varphi)) \bigr) = \mathcal{R}\bigl( S(\varphi) \otimes \psi \bigr)$ is the {\em convolution-inverse} of $\mathcal{R}$, meaning that
\[ \sum_{(\varphi),(\psi)} \mathcal{R}(\varphi_{(1)} \otimes \psi_{(1)})\mathcal{R}^{-1}(\varphi_{(2)} \otimes \psi_{(2)}) = \sum_{(\varphi),(\psi)} \mathcal{R}^{-1}(\varphi_{(1)} \otimes \psi_{(1)})\mathcal{R}(\varphi_{(2)} \otimes \psi_{(2)}) = \varepsilon_{H^\circ}(\varphi)\varepsilon_{H^\circ}(\psi) \]
for all $\varphi,\psi \in H^\circ$. We note however that the existence of the map $\Phi^+$ is stronger than the existence of a co-R-matrix; indeed, $\mathcal{R}$ yields the map $H^\circ \to (H^\circ)^*$ given by $\varphi \mapsto \mathcal{R}(\varphi \otimes \text{-})$ but it does not necessarily takes values in $H \subset (H^\circ)^*$.

\indent When $H$ is quasitriangular with R-matrix $R \in H^{\otimes 2}$, we put $R^+:=R$, $R^- := (R^{fl})^{-1}$, where $R^{fl}$ is $R$ with tensorands permuted. Note that $R^-$ is also an $R$-matrix for $H$. Then there is the canonical choice
\begin{equation}\label{LfuncFromRmat}
\Phi^+(\varphi) = (\varphi \otimes \mathrm{id})(R), \qquad \Phi^-(\varphi) = (\mathrm{id} \otimes \varphi)(R^{-1})=(\varphi \otimes \mathrm{id})(R^-).
\end{equation}
\noindent When $H$ is finite-dimensional, we have $H^\circ = H^*$ and the existence of $\Phi^+$ is equivalent to the existence of a co-R-matrix, which is itself equivalent to the existence of an R-matrix. Thus, in this case, $\Phi^+$ is an alternative definition of quasitriangularity in terms of Hopf algebra morphisms $H^* \to H$ satisfying \eqref{quasiCocommPhi}; this was noted by Radford in \cite{radfordQuasitriang}.

\smallskip

\indent The quantum enveloping algebra $U_q(\mathfrak{g})$ and its dual quantum function algebra $\mathcal{O}_q(G)$ form an example of the above setup, as will be explained in \S\ref{sec:Oq}.

\smallskip

\indent Finally, let us record the following commutation relations, for any signs $\sigma, \tau \in \{\pm\}$ 
\begin{align}
\begin{split}\label{RLLdeguise}
\sum_{(\varphi),(\psi)} \mathcal{R}^{(\sigma)}\bigl( \varphi_{(1)} \otimes \psi_{(1)} \bigr) \Phi^{\sigma}(\varphi_{(2)}) \Phi^\tau(\psi_{(2)}) &=  \sum_{(\varphi),(\psi)} \mathcal{R}^{(\sigma)}\bigl( \varphi_{(2)} \otimes \psi_{(2)} \bigr) \Phi^\tau(\psi_{(1)}) \Phi^{\sigma}(\varphi_{(1)})\\
\sum_{(\varphi),(\psi)} \mathcal{R}^{(\sigma)}\bigl( \varphi_{(1)} \otimes \psi_{(1)} \bigr) \Phi^{\tau}(\varphi_{(2)}) \Phi^\tau(\psi_{(2)}) &=  \sum_{(\varphi),(\psi)} \mathcal{R}^{(\sigma)}\bigl( \varphi_{(2)} \otimes \psi_{(2)} \bigr) \Phi^\tau(\psi_{(1)}) \Phi^{\tau}(\varphi_{(1)})
\end{split}
\end{align}
where we use the notation
\begin{equation}\label{coRMatPM}
\mathcal{R}^{(+)} = \mathcal{R} \quad \text{ and } \quad \mathcal{R}^{(-)}(\varphi \otimes \psi) = \mathcal{R}^{-1}(\psi \otimes \varphi) = \mathcal{R}\bigl(S(\psi) \otimes \varphi\bigr).
\end{equation}
The second equality in \eqref{RLLdeguise} is obvious as $\Phi^\tau$ is an algebra morphism and $\mathcal{R}^{\sigma}$ satisfies \eqref{braidedComm}; the first is independant of the second in case $\sigma$ and $\tau$ are distinct, and in this case it follows by taking $h=\Phi^\sigma(\varphi)$ in the property \eqref{quasiCocommPhi} for $\Phi^\tau$ (replacing $\varphi$ with $\psi$). With the notation introduced in \eqref{coRMatPM} we have
\begin{equation}\label{coRMatSign}
\forall \, \varphi, \psi \in H^\circ, \quad \mathcal{R}^{(\sigma)}(\varphi \otimes \psi) = \psi\bigl( \Phi^{\sigma}(\varphi) \bigr)
\end{equation}
and we note that $\mathcal{R}^{(\sigma)}$ is a co-R-matrix for all $\sigma \in \{\pm\}$. 

\subsection{Definition of \texorpdfstring{$\mathcal{L}_{g,n}(H)$}{the graph algebra}}\label{subsecDefLgnH}
\indent Given a datum $(H,\Phi^{\pm})$ as in the previous section, we define an algebra $\mathcal{L}_{g,n}(H)$ for any integers $g,n \geq 0$ not both equal to $0$. As a $\Bbbk$-vector space, $\mathcal{L}_{g,n}(H) = (H^{\circ})^{\otimes (2g+n)}$. Let $b_1,a_1,\ldots,b_g,a_g,m_{g+1},\ldots,m_{g+n}$ be formal symbols\footnote{These symbols have to be interpreted as the standard generators of the fundamental group of the oriented surface of genus $g$ with $n$ punctures and one boundary component.} and label each tensorand by one of these symbols as follows:
\begin{equation}\label{labelCopiesHcirc}
\mathcal{L}_{g,n}(H) = \underset{(b_1)}{H^{\circ}} \otimes \underset{(a_1)}{H^{\circ}} \otimes \ldots \otimes \underset{(b_g)}{H^{\circ}} \otimes \underset{(a_g)}{H^{\circ}} \otimes \underset{(m_{g+1})}{H^{\circ}} \otimes \ldots \otimes \underset{(m_{g+n})}{H^{\circ}}.
\end{equation}
In this way we get $\Bbbk$-linear embeddings $\mathfrak{i}_{b_i}, \mathfrak{i}_{a_i}, \mathfrak{i}_{m_{g+j}} : H^{\circ} \to \mathcal{L}_{g,n}(H)$ defined by $\mathfrak{i}_{b_1}(\varphi) = \varphi \otimes 1_{H^\circ} \otimes \ldots \otimes 1_{H^\circ}$ {\it etc}. An associative product in $\mathcal{L}_{g,n}(H)$ is defined by the following four rules, which use the co-R-matrix $\mathcal{R} : H^\circ \otimes H^\circ \to \Bbbk$ from \eqref{defCoRMat}.

\smallskip

\noindent (i) Elements in $(H^\circ)^{\otimes (2g+n)}$ can be written as linear combinations of ``ordered monomials'':
\begin{equation}\label{monomialRelLgn}
\varphi_1 \otimes \ldots \otimes \varphi_{2g+n} = \mathfrak{i}_{b_1}(\varphi_1)\mathfrak{i}_{a_1}(\varphi_2) \ldots \mathfrak{i}_{b_g}(\varphi_{2g-1})\mathfrak{i}_{a_g}(\varphi_{2g})\mathfrak{i}_{m_{g+1}}(\varphi_{2g+1}) \ldots \mathfrak{i}_{m_{g+n}}(\varphi_{2g+n}).
\end{equation}

\noindent (ii) For all $s \in \{ b_1,a_1, \ldots, b_g,a_g,m_{g+1}, \ldots, m_{g+n} \}$, we have a ``fusion relation'':
\begin{equation}\label{fusionRelL01}
\mathfrak{i}_s(\varphi) \mathfrak{i}_s(\psi) = \sum_{(\varphi),(\psi)}\mathcal{R}\bigl( \psi_{(1)} \otimes S(\varphi_{(3)}) \bigr) \mathcal{R}\bigl( \psi_{(3)} \otimes \varphi_{(2)} \bigr) \, \mathfrak{i}_s\bigl(\varphi_{(1)} \star \psi_{(2)} \bigr)
\end{equation}
where we use the usual Hopf algebra structure on $H^{\circ}$ recalled in \eqref{usualHopfDual}.

\smallskip

\noindent (iii) For all $1 \leq i \leq g$, we have an exchange relation:
\begin{align}
\begin{split}\label{L01prodsept25}
&\mathfrak{i}_{a_i}(\varphi)\mathfrak{i}_{b_i}(\psi)=\\
&\sum_{(\varphi),(\psi)}\mathcal{R}\bigl( \psi_{(1)} \otimes S(\varphi_{(5)}) \bigr) \mathcal{R}\bigl( \psi_{(2)} \otimes \varphi_{(1)} \bigr) \mathcal{R}\bigl( \psi_{(5)} \otimes \varphi_{(4)} \bigr) \mathcal{R}\bigl( \varphi_{(2)} \otimes \psi_{(4)} \bigr) \, \mathfrak{i}_{b_i}(\psi_{(3)}) \mathfrak{i}_{a_i}(\varphi_{(3)}).
\end{split}
\end{align}
(iv) Put $\mathrm{ind}(b_i) = \mathrm{ind}(a_i) = i$, $\mathrm{ind}(m_{g+j}) = g+j$. Then for all $s,t \in \{ b_1,a_1, \ldots, b_g,a_g$, $m_{g+1}, \ldots, m_{g+n} \}$ such that $\mathrm{ind}(t) > \mathrm{ind}(s)$, we have another exchange relation:
\begin{align}
\begin{split}\label{braidedTensProdComm}
&\mathfrak{i}_t(\varphi)\mathfrak{i}_s(\psi) =\\
&\sum_{(\varphi),(\psi)} \mathcal{R}\bigl( \psi_{(1)} \otimes S(\varphi_{(5)}) \bigr) \mathcal{R}\bigl( \psi_{(2)} \otimes \varphi_{(1)} \bigr) \mathcal{R}\bigl( S(\psi_{(4)}) \otimes \varphi_{(2)} \bigr) \mathcal{R}\bigl( \psi_{(5)} \otimes \varphi_{(4)} \bigr) \, \mathfrak{i}_s(\psi_{(3)}) \mathfrak{i}_t(\varphi_{(3)}).\!\!\!\!\!\!
\end{split}
\end{align}
\indent The algebra $\mathcal{L}_{g,n}(H)$ is called the {\em graph algebra} of the surface $\Sigma_{g,n}^{\circ}$ (oriented, genus $g$, $n$ punctures and one boundary circle). The rules (ii)--(iv) allow one to compute the product of any pair of elements of $\mathcal{L}_{g,n}(H)$ by \eqref{monomialRelLgn}.

Note that $\mathcal{L}_{0,1}(H)$ is just the vector space $H^{\circ}$ equipped with the product given by taking $\mathfrak{i}_s = \mathrm{id}_ {H^\circ}$ in the rule (ii); it is a version of Majid's braided dual construction, also called transmutation \cite{Majid0,Majid}. The linear maps $\mathfrak{i}_t : \mathcal{L}_{0,1}(H) \to \mathcal{L}_{g,n}(H)$ are morphisms of algebras for any $t \in \bigl\{ b_1,a_1,\ldots,b_g,a_g,m_{g+1},\ldots,m_{g+n} \bigr\}$. 

\begin{remark}
{\rm 1. When $H$ is quasitriangular and the datum $\Phi^{\pm}$ is chosen as in \eqref{LfuncFromRmat}, the algebra $\mathcal{L}_{g,n}(H)$ defined above is exactly the same as in \cite{BFR}.}
\end{remark}

\indent The (right) coadjoint action $\mathrm{coad}^r$ of any $h \in H$ on any $\varphi \in H^{\circ}$ is

\begin{equation}\label{defCoad}
\mathrm{coad}^r(h)(\varphi) = \sum_{(h)} S(h_{(2)}) \rhd \varphi \lhd h_{(1)} = \sum_{(\varphi)}\bigl\langle \varphi_{(1)} \star S(\varphi_{(3)}),h \bigr\rangle \,  \varphi_{(2)}
\end{equation}
with the coregular actions \eqref{coregActions}. Thanks to the iterated coproduct $\Delta^{(2g+n)} : H \to H^{\otimes (2g+n)}$ we get an action of $H$ on $\mathcal{L}_{g,n}(H) = (H^\circ)^{\otimes (2g+n)}$, still denoted by $\mathrm{coad}^r$:
\begin{equation}\label{coadLgn}
\mathrm{coad}^r(h)(\varphi_1 \otimes \ldots \otimes \varphi_{2g+n}) = \sum_{(h)} \mathrm{coad}^r(h_{(1)})(\varphi_1) \otimes \ldots \otimes \mathrm{coad}^r(h_{(2g+n)})(\varphi_{2g+n}).
\end{equation}
{\em Equipped with this action, $\mathcal{L}_{g,n}(H)$ becomes a right $H$-module-algebra}:
\[ \forall \, h \in H, \:\: \forall \, a,b \in \mathcal{L}_{g,n}(H), \quad \mathrm{coad}^r(h)(ab) = \sum_{(h)} \mathrm{coad}^r(h_{(1)})(a)\,\mathrm{coad}^r(h_{(2)})(b). \]
The subspace of invariant elements forms therefore a subalgebra called the {\em moduli algebra} and denoted by $\Ll^H_{g,n}(H)$ (the notation $\Ll_{g,n}^{\mathrm{inv}}(H)$ is also used in some other papers on graph algebras):
\begin{equation}\label{moduliAlgH}
\Ll^H_{g,n}(H) = \bigl\{ x \in \mathcal{L}_{g,n}(H) \,\big| \, \forall \, h \in H, \:\: \mathrm{coad}^r(h)(x) = \varepsilon(h)x \bigr\}.
\end{equation}

\indent It is convenient to introduce the left coaction $\delta : \mathcal{L}_{0,1}(H) \to H^\circ \otimes \mathcal{L}_{0,1}(H)$  which is dual to the right action $\mathrm{coad}^r : \mathcal{L}_{0,1}(H) \otimes H \to \mathcal{L}_{0,1}(H)$
 \begin{equation}\label{defCoactCoad}
\delta(\varphi) = \sum_{(\varphi)} \bigl( \varphi_{(1)} \star S(\varphi_{(3)}) \bigr) \otimes \varphi_{(2)}, \quad \text{written as } \delta(\varphi) = \sum_{[\varphi]} \varphi_{[1]} \otimes \varphi_{[2]}
\end{equation}
with a Sweedler type notation for coactions. More generally we have a left coaction $\delta : \mathcal{L}_{g,n}(H) \to H^\circ \otimes \mathcal{L}_{g,n}(H)$   dual to $\mathrm{coad}^r$, still denoted $\delta(x) = \textstyle \sum_{[x]} x_{[1]} \otimes x_{[2]}$; it yields a structure of left $H^\circ$-comodule-algebra. For instance in $\mathcal{L}_{1,0}(H)$ it is
\[ \delta(\beta \otimes \alpha) = \sum_{[\beta],[\alpha]} (\alpha_{[1]} \star \beta_{[1]}) \otimes \beta_{[2]} \otimes \alpha_{[2]}. \]
In particular $\varphi\in \Ll^H_{g,n}(H)$ if, and only if,
\begin{equation}\label{coinvinv}
\sum_{[\varphi]} \varphi_{[1]} \otimes \varphi_{[2]} = \varepsilon \otimes \varphi.
\end{equation}
This coaction can be used to rewrite the exchange relation \eqref{braidedTensProdComm} as follows. As vector spaces we have the decomposition $\mathcal{L}_{g,n}(H) =  \mathcal{L}_{1,0}(H)^{\otimes g} \otimes \mathcal{L}_{0,1}(H)^{\otimes n}$, which gives tensor-wise embeddings of algebras $\mathfrak{j}_k : \mathcal{L}_{1,0}(H) \to \mathcal{L}_{g,n}(H)$ for $k=1,\ldots,g$ and $\mathfrak{j}_{g+l} = \mathfrak{i}_{m_{g+l}} : \mathcal{L}_{0,1}(H) \to \mathcal{L}_{g,n}(H)$ for $l=1,\ldots,n$; the relation with notations introduced after \eqref{labelCopiesHcirc} is $\mathfrak{i}_{a_k} = \mathfrak{j}_k \circ \mathfrak{i}_a$ and $\mathfrak{i}_{b_k} = \mathfrak{j}_k \circ \mathfrak{i}_b$, where $\mathfrak{i}_b,\mathfrak{i}_a$ are embeddings of $\mathcal{L}_{0,1}(H) = H^\circ$ into $\mathcal{L}_{1,0}(H) = H^\circ \otimes H^\circ$. Then by \eqref{monomialRelLgn} we have
\begin{equation}\label{monomialFactoBraided}
x_1 \otimes \ldots \otimes x_g \otimes \varphi_1 \otimes \ldots \otimes \varphi_n = \mathfrak{j}_1(x_1) \ldots \mathfrak{j}_g(x_g)\mathfrak{j}_{g+1}(\varphi_1)\ldots \mathfrak{j}_{g+n}(\varphi_n)
\end{equation}
and using basic properties of the co-R-matrix $\mathcal{R}$, \eqref{braidedTensProdComm} is easily seen to be equivalent to
\begin{equation}\label{braidedTensProdLgn}
\forall \,1 \leq k < l \leq g+n, \quad \mathfrak{j}_l(y) \mathfrak{j}_k(x) = \sum_{[x],[y]}\mathcal{R}\bigl( x_{[1]} \otimes y_{[1]} \bigr) \, \mathfrak{j}_k(x_{[2]}) \mathfrak{j}_l(y_{[2]}).
\end{equation}
The properties \eqref{monomialFactoBraided}--\eqref{braidedTensProdLgn} mean that $\mathcal{L}_{g,n}(H) = \mathcal{L}_{1,0}(H)^{\widetilde{\otimes}\,g} \,\widetilde{\otimes}\, \mathcal{L}_{0,1}(H)^{\widetilde{\otimes}\,n}$, where $\widetilde{\otimes}$ is the {\em braided tensor product} in the category of $H^\circ$-comodules endowed with the braiding defined by $\mathcal{R}$; see \cite[\S 4.1]{BFR} for details in the quasitriangular case.

\begin{lem}\label{Z1general}
1. We have $\mathcal{L}_{0,1}^H(H) \subset \mathcal{Z}\bigl( \mathcal{L}_{0,1}(H) \bigr)$.
\\2. The subspace $\bigl( 1_{\mathcal{L}_{1,0}(H)} \bigr)^{\otimes g} \otimes \mathcal{L}_{0,1}^H(H)^{\otimes n}$ is a central subalgebra of $\mathcal{L}_{g,n}(H)$.
\end{lem}
\begin{proof}
1. Note that the product in $\mathcal{L}_{0,1}(H)$ can be rewritten as
\begin{equation}\label{productL01InTermsOfCoaction}
\varphi\psi = \sum_{(\varphi),[\psi]} \mathcal{R}\bigl( \psi_{[1]} \otimes S(\varphi_{(2)}) \bigr) \,\varphi_{(1)} \star \psi_{[2]} \quad\text{and}\quad \varphi\psi = \sum_{[\varphi],(\psi)} \mathcal{R}\bigl( \psi_{(1)} \otimes \varphi_{[1]} \bigr) \, \psi_{(2)} \star \varphi_{[2]}
\end{equation}
thanks to the basic properties \eqref{coTriang1}--\eqref{coTriang2} of a co-R-matrix and also \eqref{braidedComm} for the second equality. By \eqref{coinvinv} and \eqref{coReps} the first equality gives $\psi\varphi = \psi \star \varphi$ (we switch $\varphi$ and $\psi$) while the second reduces to $\varphi\psi = \psi \star \varphi$, whence the claim.

\noindent 2. Fix $1 \leq l \leq n$ and $\varphi \in \mathcal{L}_{0,1}^H(H)$. By \eqref{monomialFactoBraided} it suffices to show that $\mathfrak{j}_{g+l}(\varphi)$ commutes with $\mathfrak{j}_k(x)$ for all $1 \leq k \leq g+n$. For $k\ne g+l$ the commutation is clear from \eqref{braidedTensProdLgn} and \eqref{coReps}. When $k=g+l$, by the first item $\mathfrak{j}_{g+l}(\varphi)$ commutes with $\mathfrak{j}_{g+l}(\psi)$. 
\end{proof}
\noindent The inclusion in Lemma \ref{Z1general}(1) is an equality for $H = U_q(\mathfrak{g})$ \cite[Prop.\,6.19]{BR1} and also for $H$ finite-dimensional and factorizable \cite[Th.\,3.7]{FaitgSL2Z}. We remark that it is not true that $\mathcal{L}_{1,0}^H(H) \subset \mathcal{Z}\bigl( \mathcal{L}_{1,0}(H) \bigr)$, which explains that Lemma \ref{Z1general}(2) only concerns the ``$\mathcal{L}_{0,n}$ part'' of $\mathcal{L}_{g,n}(H)$.

\smallskip

\indent It should be noted that the definition of $\mathcal{L}_{g,n}$ only uses the co-R-matrix $\mathcal{R} : H^\circ \otimes H^\circ \to \Bbbk$; the morphisms $\Phi^\pm$ will enter in the game only in the next sections. Hence the $\mathcal{L}_{g,n}$ construction applies to any co-quasitriangular Hopf algebra $(O,\mathcal{R})$, not necessarily of the form $H^\circ$. We could denote the resulting algebra as $\mathcal{L}_{g,n}[O]$, which is $O^{\otimes (2g+n)}$ as a vector space.\footnote{This notation is only used here and in the proof of Prop.\,\ref{piMorphBetweenLgn}.} It generalizes the above definition because $\mathcal{L}_{g,n}(H) = \mathcal{L}_{g,n}[H^\circ]$. This construction is ``functorial'':

\begin{lem}\label{lemmaLgnIsFunctorial}
Let $(O_1,\mathcal{R}_1)$ and $(O_2, \mathcal{R}_2)$ be co-quasitriangular Hopf algebras. Let $f : O_1 \to O_2$ be a morphism of Hopf algebras such that $\mathcal{R}_1 = \mathcal{R}_2 \circ (f \otimes f)$. Then $f^{\otimes (2g+n)} : \mathcal{L}_{g,n}[O_1] \to \mathcal{L}_{g,n}[O_2]$ is a morphism of algebras.
\end{lem}
\begin{proof}
For $r=1,2$, let $\mathfrak{i}_{b_i}^r, \mathfrak{i}_{a_i}^r, \mathfrak{i}_{m_{g+j}}^r : O_r \to \mathcal{L}_{g,n}[O_r]$ be the embeddings from which the product of $\mathcal{L}_{g,n}[O_r]$ is defined. They satisfy $f^{\otimes (2g+n)} \circ \mathfrak{i}^1_s = \mathfrak{i}^2_s \circ f$. By \eqref{monomialRelLgn}, it suffices to check that $f^{\otimes (2g+n)}\bigl( \mathfrak{i}^1_s(\varphi)\mathfrak{i}^1_t(\psi) \bigr) = \mathfrak{i}^2_s\bigl( f(\varphi)\bigr) \mathfrak{i}^2_t\bigl(f(\psi) \bigr)$ for all $\varphi,\psi \in F_1$ and all $s,t \in \{b_i,a_i,m_{g+j} \}$. This is straightforward and left to the reader.
\end{proof}

\subsection{Quantum moment map}\label{subsecQMMgeneral}
We continue with the datum $(H,\Phi^\pm)$ as in \S\ref{subsecSubstQuasi}. The following map is of extreme importance in the theory of graph algebras
\begin{equation}\label{RSDphi}
\Phi_{0,1} : \mathcal{L}_{0,1}(H) \to H, \quad \varphi \mapsto \sum_{(\varphi)} \Phi^+(\varphi_{(1)})\Phi^-\bigl( S(\varphi_{(2)}) \bigr)
\end{equation}
where we use the usual coproduct \eqref{usualHopfDual} in $H^\circ$.  $\Phi_{0,1}$ is a morphism of $H$-module-algebras provided $H$ is equipped with the right adjoint action
\begin{equation}\label{defAdr}
\forall \, h,x \in H, \quad \mathrm{ad}^r(h)(x) = \sum_{(h)} S(h_{(1)})xh_{(2)}.
\end{equation}
This is shown thanks to \eqref{RLLdeguise} and \eqref{quasiCocommPhi}. When $H$ is quasitriangular with R-matrix $\textstyle R = \sum_{(R)} R_{(1)} \otimes R_{(2)}$ and $\Phi^\pm$ is chosen as in \eqref{LfuncFromRmat}, this reads $\Phi_{0,1}(\varphi) = (\varphi \otimes \mathrm{id})(RR^{fl})$, where as usual $\textstyle R^{fl} = \sum_{(R)} R_{(2)} \otimes R_{(1)}$ is the flip of $R$.

\smallskip

\indent Now let us make the following assumptions on $(H,\Phi^\pm)$:

\smallskip

\indent \textbullet ~$H^\circ$ is {\em coribbon}, meaning that there is a  linear form $\mathsf{v} : H^\circ \to \Bbbk$ satisfying $\mathsf{v} \circ S = \mathsf{v}$ and\footnote{These axioms are dual to those of a ribbon element \cite[\S 4.2.C]{CP}, in the sense that if $v \in H$ is a ribbon element, then $\mathsf{v} = \langle -,v \rangle : H^\circ \to \Bbbk$ is a coribbon element.}
 \begin{align}
&\textstyle \sum_{(\varphi)} \mathsf{v}(\varphi_{(1)})\varphi_{(2)} = \sum_{(\varphi)}\varphi_{(1)}\mathsf{v}(\varphi_{(2)})\label{coribbonCocentral}\\
&\textstyle \mathsf{v}(\varphi\psi) = \sum_{(\varphi),(\psi)}\mathcal{R}^{-1}\bigl( \psi_{(3)} \otimes \varphi_{(3)} \bigr) \mathcal{R}^{-1}\bigl( \varphi_{(2)} \otimes \psi_{(2)} \bigr) \mathsf{v}(\varphi_{(1)}) \, \mathsf{v}(\psi_{(1)})\label{axiomCoribbonProduct}\\
&\mathsf{v}^2(\varphi) = \textstyle \sum_{(\varphi)}\mathsf{u}(\varphi_{(1)}) \mathsf{u}\bigl( S(\varphi_{(2)}) \bigr) \quad \text{where } \mathsf{u}(\varphi) = \sum_{(\varphi)}\mathcal{R}\bigl( \varphi_{(2)}\otimes S(\varphi_{(1)}) \bigr) \label{axiomCorribonSquare}
\end{align}
where $\mathsf{v}^2 = (\mathsf{v} \otimes \mathsf{v}) \circ \Delta$ and $\mathsf{u}$ is called the {\em co-Drinfeld element}. These axioms imply that $\mathsf{v}$ is convolution-invertible because so is $\mathsf{u}$ \cite[Prop.\,2.2.4]{Majid} and it holds $\mathsf{v}(1_{H^\circ}) = 1$.

\smallskip

\indent \textbullet ~The morphism $\Phi_{0,1}$ is injective.

\smallskip

\noindent We introduce auxiliary maps from which the so-called quantum moment map will be defined. First let
\begin{equation}\label{frakD10}
\mathfrak{d}_{1,0} : \mathcal{L}_{0,1}(H) \to \mathcal{L}_{1,0}(H), \quad \varphi \mapsto \sum_{(\varphi)} \mathsf{v}^2(\varphi_{(1)})\, \mathfrak{i}_b(\varphi_{(2)}) \mathfrak{i}_a\bigl( S_{\mathcal{L}}(\varphi_{(3)})\bigr) \mathfrak{i}_b\bigl( S_{\mathcal{L}}(\varphi_{(4)}) \bigr)\mathfrak{i}_a(\varphi_{(5)})
\end{equation}
where we use the usual coproduct \eqref{usualHopfDual} in $H^\circ$ and $S_{\mathcal{L}}$ is the antipode of $\mathcal{L}_{0,1}(H)$, uniquely defined by $\textstyle \sum_{(\varphi)} \varphi_{(1)}S_{\mathcal{L}}(\varphi_{(2)}) = \sum_{(\varphi)} S_{\mathcal{L}}(\varphi_{(1)})\varphi_{(2)} = \varphi(1_H)$. Explicitly:
\begin{equation}\label{antipodeL01}
\forall \, \varphi \in \mathcal{L}_{0,1}(H), \quad S_{\mathcal{L}}(\varphi) = \textstyle \sum_{(\varphi)} \mathcal{R}^{-1}\bigl( \varphi_{(4)} \otimes \varphi_{(1)} \bigr) \mathsf{u}^{-1}(\varphi_{(2)}) \, S_{H^\circ}(\varphi_{(3)})
\end{equation}
with $\mathsf{u}$ from \eqref{axiomCorribonSquare} (see the discussion after \cite[eq. (3.24)]{BFR}). Now observe that as a vector space $\mathcal{L}_{g,n}(H)$ is $\mathcal{L}_{1,0}(H)^{\otimes g} \otimes \mathcal{L}_{0,1}(H)^{\otimes n}$. Hence we can define
\begin{equation}\label{pseudoQMMgnH}
\mathfrak{d}_{g,n} : \mathcal{L}_{0,1}(H) \to \mathcal{L}_{g,n}(H), \quad \varphi \mapsto \sum_{(\varphi)} \mathfrak{d}_{1,0}(\varphi_{(1)}) \otimes \ldots \otimes \mathfrak{d}_{1,0}(\varphi_{(g)}) \otimes \varphi_{(g+1)} \otimes \ldots \otimes \varphi_{(g+n)}.
\end{equation}
The map $\mathfrak{d}_{g,n}$ is easily seen to be a morphism of $H$-modules with respect to the right actions $\mathrm{coad}^r$.

\begin{remark}
{\rm Let $X$ be a finite-dimensional $H$-module and $x^i(?\cdot x_j) \in H^\circ$ be the matrix coefficients of $X$ in some basis $(x_j)$. Then the matrix $\overset{X}{C} = \bigl[ \mathfrak{d}_{g,n}\bigl( x^i(?\cdot x_j) \bigr) \bigr]_{i,j}$ with coefficients in $\mathcal{L}_{g,n}(H)$ equals the matrix from \cite[Def.\,7.11]{BFR}. The use of the antipode $S_{\mathcal{L}}$ corresponds to matrix inverse in the notations of \cite{BFR}.}
\end{remark}

\indent Denote by $H'$ the image of the morphism $\Phi_{0,1}$ from \eqref{RSDphi}. The following map is called {\em quantum moment map} and will play an important role
\begin{equation}\label{defQMMGen}
\mu_{g,n} : H' = \Phi_{0,1}\bigl( \mathcal{L}_{0,1}(H) \bigr) \xrightarrow{\quad\Phi_{0,1}^{-1}\quad} \mathcal{L}_{0,1}(H) \xrightarrow{\quad\mathfrak{d}_{g,n}\quad} \mathcal{L}_{g,n}(H).
\end{equation}
From the definitions we have
\begin{equation}\label{coproduitSurPhi01}
\forall \, \varphi \in H^\circ, \quad \Delta\bigl( \Phi_{0,1}(\varphi) \bigr) = \sum_{(\varphi)} \Phi_{0,1}(\varphi_{(2)}) \otimes \Phi^+(\varphi_{(1)})\Phi^-\bigl( S(\varphi_{(3)}) \bigr)
\end{equation}
which implies that $\Delta(H') \subset H' \otimes H$ (right coideal), and thus equality \eqref{mudef} below makes sense.
\begin{teo} {\em \cite[Th.\,7.14]{BFR}}
$\mu_{g,n}$ is a morphism of algebras satisfying the quantum moment map property, i.e.
\begin{equation}\label{mudef}
\forall h' \in H', \:\: \forall \, x \in \mathcal{L}_{g,n}(H), \quad \mu_{g,n}(h')\,x = \sum_{(h')} \, \mathrm{coad}^r\bigl(  S^{-1}(h'_{(2)}) \bigr)(x) \, \mu_{g,n}(h'_{(1)}).
\end{equation}
\end{teo}
\noindent This can also be written as $\textstyle x\mu_{g,n}(h') = \sum_{(h')} \mu_{g,n}(h_{(1)}')\mathrm{coad}^r(h_{(2)}')(x)$.

\noindent Note in particular that $\mathfrak{d}_{g,n} : \mathcal{L}_{0,1}(H) \to \mathcal{L}_{g,n}(H)$ is a morphism of algebras.

\smallskip

\indent The question of whether the quantum moment map is injective reduces to a ground case:

\begin{lem}\label{lemmaInjQMMtrivial}
1. $\mu_{g,n}$ is always injective for all $g$ and $n \neq 0$.
\\2. If $\mu_{1,0}$ is injective then $\mu_{g,0}$ is injective for all $g$.
\end{lem}
\begin{proof}
By very definition, $\mu_{g,n}$ is injective if and only if $\mathfrak{d}_{g,n}$ is injective. One can check that $(\varepsilon_{H^\circ})^{\otimes 2}\bigl( \mathfrak{d}_{1,0}(\varphi) \bigr) = \varepsilon_{H^\circ}(\varphi)$, where we use that $\mathcal{L}_{1,0}(H) = H^\circ \otimes H^\circ$ as a vector space. Hence if $n > 0$ we have $\bigl[ \bigl((\varepsilon_{H^\circ})^{\otimes 2}\bigr)^{\otimes g} \otimes (\varepsilon_{H^\circ})^{\otimes (n-1)} \otimes \mathrm{id}_{H^\circ}\bigr] \bigl( \mathfrak{d}_{g,n}(\varphi) \bigr) = \varphi$, thus proving claim 1. For claim 2, note similarly that $\bigl[ \bigl((\varepsilon_{H^\circ})^{\otimes 2}\bigr)^{\otimes (g-1)} \otimes \mathrm{id}_{(H^\circ)^{\otimes 2}}\bigr] \bigl( \mathfrak{d}_{g,n}(\varphi) \bigr) = \mathfrak{d}_{1,0}(\varphi)$.
\end{proof}

\subsection{\texorpdfstring{$\mathcal{L}_{1,0}(H)$}{The handle algebra} and the Heisenberg double}\label{subsecHeisenbergL10}
\indent The {\em Heisenberg double} $\mathcal{H}(H^\circ)$ is the vector space $H^\circ \otimes H$ endowed with the algebra structure defined by the following rules:

\indent (i) $H^\circ \otimes 1_H$ and $1_{H^\circ} \otimes H$ are subalgebras of $\mathcal{H}(H^{\circ})$ isomorphic to $H^\circ$ and $H$. We write $\varphi$ instead of $\varphi \otimes 1_H$, and $a$ instead of $1_{H^\circ} \otimes a$ (where $1_{H^\circ} = \varepsilon$).

\indent (ii) With these identifications, we have the relations 
\begin{equation}\label{productHeisenberg}
\forall \, a \in H, \:\: \forall \, \varphi \in H^\circ, \quad \varphi a = \varphi \otimes a, \quad a\varphi = \sum_{(a)}(a_{(1)} \rhd \varphi) \, a_{(2)}.
\end{equation}
If we define a right action of $H$ on $\mathcal{H}(H^\circ)$ by
\begin{equation}\label{modAlgStructHeisenberg}
(\varphi a) \cdot h = \sum_{(h)} \bigl( S(h_{(2)}) \rhd \varphi \lhd h_{(3)} \bigr) S(h_{(1)})ah_{(4)}
\end{equation}
then $\mathcal{H}(H^\circ)$ becomes a right $H$-module algebra \cite[\S 3.3]{BFR}.  There is a representation
\begin{equation}\label{HeisenbergRep}
\rho : \mathcal{H}(H^\circ) \to \mathrm{End}_\Bbbk(H^\circ), \quad \rho(\varphi a)(\psi) = \varphi \star (a \rhd \psi) \qquad (\forall \, \psi \in H^\circ).
\end{equation}
The morphism $\rho$ is injective, without any assumption on $H$ \cite[Lem.\,9.4.2]{Mo}.

\smallskip

\indent Recall from \eqref{monomialRelLgn} that any element in $\mathcal{L}_{1,0}(H)$ is a sum of products $\mathfrak{i}_b(\beta)\mathfrak{i}_a(\alpha)$ for $\alpha,\beta \in \mathcal{L}_{0,1}(H)$. There is a morphism of right $H$-module-algebras $\Phi_{1,0} : \mathcal{L}_{1,0}(H) \to \mathcal{H}(H^{\circ})$ uniquely characterized by
\begin{equation}\label{defPhi10}
\forall \, \varphi \in H^\circ, \quad \Phi_{1,0}\bigl( \mathfrak{i}_a(\varphi) \bigr) = \Phi_{0,1}(\varphi), \qquad \Phi_{1,0}\bigl( \mathfrak{i}_b(\varphi) \bigr) = \sum_{(\varphi)}\Phi^+(\varphi_{(1)})\varphi_{(2)}\Phi^-\bigl( S(\varphi_{(3)}) \bigr)
\end{equation}
where we use the morphism $\Phi_{0,1} : \mathcal{L}_{0,1}(H) \to H$ from \eqref{RSDphi}. In the case where $H$ is quasitriangular with $\Phi^\pm$ as in \eqref{LfuncFromRmat}, the fact that $\Phi_{1,0}$ is a morphism of right $H$-module-algebras was proven in \cite{Al} (see also \cite[\S 3.3]{BFR}). In the present general setting one must show that the defining relations of $\mathcal{L}_{1,0}(H)$ (fusion relation for $\mathfrak{i}_b$ and $\mathfrak{i}_a$ and exchange relation between $\mathfrak{i}_b$ and $\mathfrak{i}_a$) are compatible with \eqref{defPhi10}; the computations use \eqref{RLLdeguise} and the following exchange relation in $\mathcal{H}(H^{\circ})$:
\begin{equation}\label{exchangeRelHeisenberg}
\forall \, \sigma \in \{\pm\}, \:\: \forall \, \varphi,\psi \in H^{\circ}, \quad \Phi^{\sigma}(\varphi)\psi = \sum_{(\varphi),(\psi)}\mathcal{R}^{(\sigma)}\bigl(\varphi_{(2)} \otimes \psi_{(2)}\bigr)\psi_{(1)}\Phi^{\sigma}\bigl(\varphi_{(1)}\bigr)
\end{equation}
which directly follows from \eqref{productHeisenberg} and \eqref{coRMatSign}. Details are left to the reader.

\smallskip

\indent Let us now investigate how the QMM $\mu_{1,0}$ interacts with the morphism $\Phi_{1,0}$; thus we resume the assumptions on $H$ from \S\ref{subsecQMMgeneral} used to construct $\mu_{g,n}$ ($\Phi_{0,1}$ injective and existence of coribbon element $\mathsf{v}$). Recall that $H'$ denotes the image of the morphism $\Phi_{0,1}$ from \eqref{RSDphi}. For all $h' = \Phi_{0,1}(\varphi) \in H'$ we define
\begin{equation}\label{elmtsChapeau}
\widehat{h'} = \sum_{[\varphi]} \varphi_{[1]} \otimes S^{-1}\bigl( \Phi_{0,1}(\varphi_{[2]}) \bigr) \in H^\circ \otimes H = \mathcal{H}(H^\circ)
\end{equation}
using the coaction $\mathcal{L}_{0,1}(H) \to H^\circ \otimes \mathcal{L}_{0,1}(H)$ from \eqref{defCoactCoad}. 

The definition of the elements $\widehat{h'}$ in the form of \eqref{elmtsChapeauBis} below, and their remarkable properties given in Lemma \ref{lemmeChapeau}, are directly inspired from the computations in \cite[Lem.\,4.1, also see \S 2.2]{FaitgMCG} which was for finite-dimensional $H$. Note that there the elements $\widehat{h}$ were denoted $\widetilde{h}$, not to be confused with the elements $\widetilde{h}$ in \S\ref{subsecAlekseev} of the present paper. Also the equality $\mu_{1,0} \circ \Phi_{1,0} = \widehat{\mathsf{D}}_{1,0}$ from item 2 of Prop.\,\ref{propCommDiag10} was proven in \cite[eq.\,(4.22)]{FaitgThesis} in the case where $H$ is finite-dimensional.

\begin{lem}\label{lemmeChapeau}
1. With the representation $\rho : \mathcal{H}(H^\circ) \to \mathrm{End}_\Bbbk(H^\circ)$ from \eqref{HeisenbergRep} we have
\begin{equation}\label{formuleRepChapeau}
\forall \, h' \in H', \:\: \forall \, \psi \in H^\circ, \quad \rho\bigl( \widehat{h'} \bigr)(\psi) = \psi \lhd S^{-1}(h').
\end{equation}
2. In $\mathcal{H}(H^\circ)$ we have the relations
\[ \textstyle \widehat{h'h''} = \widehat{h'} \, \widehat{h''}, \qquad \widehat{h'}\,\varphi a = \sum_{(h')} \bigl( \varphi \lhd S^{-1}(h'_{(2)}) \bigr)a \, \widehat{\,h'_{(1)}} \]
for all $h', h'' \in H'$, and all $\varphi a \in \mathcal{H}(H^\circ)$. In particular $\widehat{h'} \, a = a \, \widehat{h'}$ for all $a \in H \subset \mathcal{H}(H^\circ)$.
\\3. The right action \eqref{modAlgStructHeisenberg} of $h \in H$ on $\widehat{h'} \in \mathcal{H}(H^\circ)$ is given by
\[ \widehat{h'} \cdot h = \textstyle \sum_{(h)} S(h_{(1)})h_{(4)} \, \widehat{S(h_{(2)})h'h_{(3)}}. \]
4. The map $H' \to \mathcal{H}(H^\circ)$, $h' \mapsto \widehat{h'}$ is injective.
\end{lem}
\begin{proof}
1. Write $h' = \Phi_{0,1}(\varphi)$. The proof is a direct computation (using Sweedler's notation with implicit summation for coproducts and coactions):
\begin{align*}
\rho\bigl( \widehat{h'} \bigr)(\psi) &= \varphi_{[1]} \star \bigl[ S^{-1}\bigl( \Phi_{0,1}(\varphi_{[2]})  \bigr) \rhd \psi \bigr] = \bigl\langle \psi_{(2)}, S^{-1}\bigl( \Phi_{0,1}(\varphi_{(2)}) \bigr) \bigr\rangle \, \varphi_{(1)} \star S(\varphi_{(3)}) \star \psi_{(1)}\\
&=\bigl\langle S^{-1}(\psi_{(2)}), \Phi^-\bigl( S(\varphi_{(3)}) \bigr) \bigr\rangle \, \bigl\langle S^{-1}(\psi_{(3)}), \Phi^+(\varphi_{(2)}) \bigr\rangle \, \varphi_{(1)} \star S(\varphi_{(4)}) \star \psi_{(1)}\\
&=\mathcal{R}\bigl( \psi_{(2)} \otimes S(\varphi_{(3)}) \bigr) \, \mathcal{R}^{-1}\bigl( \varphi_{(2)} \otimes \psi_{(3)} \bigr) \, \varphi_{(1)} \star S(\varphi_{(4)}) \star \psi_{(1)}\\
&=\mathcal{R}\bigl( \psi_{(1)} \otimes S(\varphi_{(4)}) \bigr) \, \mathcal{R}^{-1}\bigl( \varphi_{(2)} \otimes \psi_{(3)} \bigr) \, \varphi_{(1)} \star \psi_{(2)} \star S(\varphi_{(3)})\\
&=\mathcal{R}\bigl( \psi_{(1)} \otimes S(\varphi_{(4)}) \bigr) \, \mathcal{R}^{-1}\bigl( \varphi_{(1)} \otimes \psi_{(2)} \bigr) \, \psi_{(3)} \star \varphi_{(2)} \star S(\varphi_{(3)})\\
&=\mathcal{R}\bigl( \psi_{(1)} \otimes S(\varphi_{(2)}) \bigr) \, \mathcal{R}^{-1}\bigl( \varphi_{(1)} \otimes \psi_{(2)} \bigr) \, \psi_{(3)}\\
&=\bigl\langle S^{-1}(\psi_{(1)}), \Phi^-\bigl( S(\varphi_{(2)}) \bigr) \bigr\rangle \, \bigl\langle S^{-1}(\psi_{(2)}), \Phi^+(\varphi_{(1)}) \bigr\rangle \, \psi_{(3)}\\
&=\bigl\langle \psi_{(1)}, S^{-1}\bigl[ \Phi^+(\varphi_{(1)}) \Phi^-\bigl( S(\varphi_{(2)}) \bigr) \bigr] \bigr\rangle \, \psi_{(2)} = \psi \lhd S^{-1}\bigl( \Phi_{0,1}(\varphi) \bigr) = \psi \lhd S^{-1}(h')
\end{align*}
where the first equality is by definition of $\rho$, the second equality is by definition of the coaction \eqref{defCoactCoad} and of $\rhd$ in \eqref{coregActions}, the third is by definition of $\Phi_{0,1}$ in \eqref{RSDphi}, the fourth is by \eqref{coRMatSign}, the fifth and sixth use \eqref{braidedComm}, and the eighth is by \eqref{coRMatSign} also using that $\mathcal{R}^{-1} = \mathcal{R} \circ (\mathrm{id} \otimes S^{-1})$ and the remaining equalities are easy.
\\2. These claims are easily deduced from item 1, thanks to injectivity of $\rho$. For instance:
\begin{align*}
\rho\bigl( \widehat{h'h''} \bigr)(\psi) = \psi \lhd S^{-1}(h'h'') &= \bigl( \psi \lhd S^{-1}(h'') \bigr) \lhd S^{-1}(h')\\
&= \bigl[ \rho\bigl(\widehat{h'}\bigr) \circ \rho\bigl( \widehat{h''} \bigr) \bigr](\psi) = \rho\bigl( \widehat{h'}\,\widehat{h''} \bigr)(\psi)
\end{align*}
for all $\psi \in H^\circ$; hence $\rho\bigl( \widehat{h'h''} \bigr) = \rho\bigl( \widehat{h'}\,\widehat{h''} \bigr)$ and $\widehat{h'h''} = \widehat{h'}\,\widehat{h''}$.
\\3. It suffices to check this equality through the injective morphism $\rho : \mathcal{H}(H^\circ) \to \mathrm{End}_\Bbbk(H^\circ)$, which is a straightforward computation using \eqref{formuleRepChapeau}.
\\4. Recall that $\mathcal{H}(H^\circ) = H^\circ \otimes H$ as a vector space. By definition of a coaction it holds $\varepsilon_{H^\circ}(\varphi_{[1]})\varphi_{[2]} = \varphi$. As a result writing $h' = \Phi_{0,1}(\varphi)$ we find $(\varepsilon_{H^\circ} \otimes S)\bigl(\widehat{h'} \bigr) = h'$.
\end{proof}

Thanks to item 2 in the previous lemma we have a morphism of algebras
\[ \widehat{\mathsf{D}}_{1,0} : H' \to \mathcal{H}(H^\circ), \quad h' \mapsto \textstyle \sum_{(h')} \widehat{h'_{(1)}} \, h'_{(2)} \]
whose definition makes sense because $\Delta(H') \subset H' \otimes H$. The next result uses the right $H$-module-algebra structure \eqref{modAlgStructHeisenberg} on $\mathcal{H}(H^\circ)$.

\begin{prop}\label{propCommDiag10}
1. $\widehat{\mathsf{D}}_{1,0}$ is a quantum moment map for $\mathcal{H}(H^\circ)$, meaning that
\[ \forall \, w \in \mathcal{H}(H^\circ), \:\: \forall \, h' \in H', \quad \widehat{\mathsf{D}}_{1,0}(h')\,w = \sum_{(h')} \, \bigl( w \cdot  S^{-1}(h'_{(2)}) \bigr)\, \widehat{\mathsf{D}}_{1,0}(h'_{(1)}). \]
2. It holds $\mu_{1,0} \circ \Phi_{1,0} = \widehat{\mathsf{D}}_{1,0}$; said differently we have a commutative diagram
\[ \xymatrix@R=1.5em@C=4em{
\mathcal{L}_{0,1}(H) \ar[r]^-{\mathfrak{d}_{1,0}} \ar[d]_-{\Phi_{0,1}} & \mathcal{L}_{1,0}(H) \ar[d]^-{\Phi_{1,0}}\\
H' \ar[r]_{\widehat{\mathsf{D}}_{1,0}} \ar[ur]^-{\mu_{1,0}}& \mathcal{H}(H^\circ)
}  \]
\end{prop}
\begin{proof}
We use Sweedler's notation with implicit summation.
\\1. Obtained by a straightforward computation using the relations established in Lemma \ref{lemmeChapeau}; this is left to the reader.
\\2. A preliminary step is required. For all $\varphi \in \mathcal{L}_{0,1}(H)$ we claim that
\begin{equation}\label{elmtsChapeauBis}
\widehat{\Phi_{0,1}(\varphi)} = \mathsf{v}^2(\varphi_{(1)}) \, \varphi_{(2)} \, \Phi^+\bigl( S(\varphi_{(3)}) \bigr) \, \Phi^-(\varphi_{(4)}) \, S(\varphi_{(5)}) \in \mathcal{H}(H^\circ).
\end{equation}
where $\mathsf{v}$ is the coribbon element. In order to prove this, recall the co-Drinfeld element $\mathsf{u} : H^\circ \to \Bbbk$ from \eqref{axiomCorribonSquare}. Its convolution inverse is
\begin{equation}\label{inverseu}
\mathsf{u}^{-1}(\alpha) = \mathcal{R}\bigl( S^2(\alpha_{(2)}) \otimes \alpha_{(1)} \bigr),
\end{equation}
and by \cite[Prop.\,2.2.4]{Majid} it satisfies the key property
\begin{equation}\label{coPivotality}
\forall \, \alpha \in H^\circ, \quad S^2(\alpha) = \mathsf{u}(\alpha_{(1)})\alpha_{(2)}\mathsf{u}^{-1}(\alpha_{(3)}). 
\end{equation}
Using it and \eqref{exchangeRelHeisenberg}, straightforward computations give the first equality below, and then we obtain:
\begin{align*}
\widehat{\Phi_{0,1}(\varphi)}=\:& \mathsf{v}^2(\varphi_{(1)})\mathsf{u}^{-1}(\varphi_{(6)})\mathcal{R}\bigl(\varphi_{(3)} \otimes \varphi_{(7)}\bigr) \Bigl[\varphi_{(2)} \star S(\varphi_{(8)}) \otimes \Phi^+\bigl( S(\varphi_{(4)}) \bigr) \Phi^-(\varphi_{(5)}) \Bigr]\\
=\:& \mathsf{v}^2(\varphi_{(1)})\mathsf{u}^{-1}(\varphi_{(5)})\mathcal{R}\bigl(\varphi_{(3)} \otimes \varphi_{(7)}\bigr) \Bigl[\varphi_{(2)} \star S(\varphi_{(8)}) \otimes S^{-1}\Bigl( \Phi^-\bigl( S(\varphi_{(6)}) \bigr)\Phi^+( \varphi_{(4)}) \Bigr)\Bigr]\\
=\:& \mathsf{v}^2(\varphi_{(1)})\mathsf{u}^{-1}(\varphi_{(5)})\mathcal{R}\bigl(\varphi_{(4)} \otimes \varphi_{(6)}\bigr) \Bigl[\varphi_{(2)} \star S(\varphi_{(8)}) \otimes S^{-1}\Bigl( \Phi^+( \varphi_{(3)}) \Phi^-\bigl( S(\varphi_{(7)}) \bigr) \Bigr)\Bigr]\\
=\:& \mathsf{v}^2(\varphi_{(1)})\mathsf{u}^{-1}(\varphi_{(4)}) \mathsf{u}^{-1}\bigl( S(\varphi_{(5)}) \bigr) \Bigl[\varphi_{(2)} \star S(\varphi_{(7)}) \otimes S^{-1}\Bigl( \Phi^+( \varphi_{(3)}) \Phi^-\bigl( S(\varphi_{(6)}) \bigr) \Bigr)\Bigr]\\
=\:&\varphi_{(1)} \star S(\varphi_{(4)}) \otimes S^{-1}\Bigl( \Phi^+( \varphi_{(2)}) \Phi^-\bigl( S(\varphi_{(3)}) \bigr) \Bigr) = \varphi_{[1]} \otimes S^{-1}\bigl( \Phi_{0,1}(\varphi_{[2]}) \bigr)
\end{align*}
where the second equality uses \eqref{coPivotality} together with $\Phi^\pm \circ S = S^{-1} \circ \Phi^{\pm}$, the third uses \eqref{RLLdeguise}, the fourth uses the outcome of the following computation based on properties of $\mathsf{u}$ recalled just above and the fact that $\mathcal{R} \circ (S \otimes \mathrm{id}) = \mathcal{R} \circ (\mathrm{id} \otimes S^{-1})$:
\begin{align*}
\mathsf{u}^{-1}(\varphi_{(5)})\mathcal{R}\bigl(\varphi_{(4)} \otimes \varphi_{(6)}\bigr) &= \mathsf{u}^{-1}(\varphi_{(4)})\mathcal{R}\bigl(S^2(\varphi_{(5)}) \otimes \varphi_{(6)}\bigr)\\
&= \mathsf{u}^{-1}(\varphi_{(4)})\mathcal{R}\bigl[ S\bigl( S(\varphi_{(5)})_{(2)} \bigr) \otimes S^{-1}\bigl( S(\varphi_{(5)})_{(1)} \bigr) \bigr]\\
&= \mathsf{u}^{-1}(\varphi_{(4)})\mathcal{R}\bigl[ S^2\bigl( S(\varphi_{(5)})_{(2)} \bigr) \otimes  S(\varphi_{(5)})_{(1)} \bigr] = \mathsf{u}^{-1}(\varphi_{(4)}) \mathsf{u}^{-1}\bigl( S(\varphi_{(5)}) \bigr)
\end{align*}
the fifth uses \eqref{coribbonCocentral} and \eqref{axiomCorribonSquare}, and the last is by definitions of the coaction and of $\Phi_{0,1}$.

\indent We now use \eqref{elmtsChapeauBis} to prove the result in the Proposition. One can check that\footnote{These equalities are obvious in the matrix formalism explained in \cite[\S 3.2]{BFR}, which is why we don't give further details. }
\[ \Phi_{1,0} \circ \mathfrak{i}_a\bigl( S_{\mathcal{L}}(\varphi) \bigr) = \Phi^-(\varphi_{(1)}) \Phi^+\bigl( S(\varphi_{(2)}) \bigr), \quad \Phi_{1,0} \circ \mathfrak{i}_b\bigl( S_{\mathcal{L}}(\varphi) \bigr) = \Phi^-(\varphi_{(1)}) S(\varphi_{(2)}) \Phi^+\bigl( S(\varphi_{(3)}) \bigr). \]
From this and \eqref{defPhi10} and \eqref{elmtsChapeauBis} we find
\begin{align*}
&\Phi_{1,0} \circ \mathfrak{d}_{1,0}(\varphi) = \mathsf{v}^2(\varphi_{(1)}) \Phi_{1,0}\bigl[ \mathfrak{i}_b(\varphi_{(2)}) \, \mathfrak{i}_a\bigl( S_{\mathcal{L}}(\varphi_{(3)}) \bigr) \, \mathfrak{i}_b\bigl( S_{\mathcal{L}}(\varphi_{(4)}) \bigr) \, \mathfrak{i}_a(\varphi_{(5)}) \bigr]\\
&=\mathsf{v}^2(\varphi_{(1)})\, \bigl[ \Phi^+(\varphi_{(2)})\varphi_{(3)}\Phi^-\bigl( S(\varphi_{(4)}) \bigr)\bigr] \, \, \bigl[ \Phi^-(\varphi_{(5)}) \Phi^+\bigl( S(\varphi_{(6)}) \bigr) \bigr]\\
&\qquad\qquad\quad \bigl[ \Phi^-(\varphi_{(7)}) S(\varphi_{(8)}) \Phi^+\bigl( S(\varphi_{(9)}) \bigr) \bigr] \bigl[ \Phi^+(\varphi_{(10)}) \Phi^-\bigl( S(\varphi_{(11)}) \bigr) \bigr]\\
&= \mathsf{v}^2(\varphi_{(1)})\, \Phi^+(\varphi_{(2)})\varphi_{(3)} \Phi^+\bigl( S(\varphi_{(4)}) \bigr) \Phi^-(\varphi_{(5)}) S(\varphi_{(6)}) \Phi^-\bigl( S(\varphi_{(7)}) \bigr)\\
&=\Phi^+(\varphi_{(1)}) \widehat{\Phi_{0,1}(\varphi_{(2)})} \Phi^-\bigl( S(\varphi_{(3)}) \bigr) \\ & = \widehat{\Phi_{0,1}(\varphi_{(2)})} \Phi^+(\varphi_{(1)}) \Phi^-\bigl( S(\varphi_{(3)}) \bigr) = \widehat{\Phi_{0,1}(\varphi)_{(1)}} \Phi_{0,1}(\varphi)_{(2)} = \widehat{\mathsf{D}}_{1,0} \circ \Phi_{1,0}
\end{align*}
where the first equality uses the definition of $\mathfrak{d}_{1,0}$ in \eqref{frakD10}, the second uses the values of $\Phi_{1,0}$, the third uses that $\Phi^\pm$ is a morphism of algebras to apply the antipode axiom, the fourth uses \eqref{coribbonCocentral} and \eqref{elmtsChapeauBis}, the fifth uses Lemma \ref{lemmeChapeau}(2) and the sixth is by \eqref{coproduitSurPhi01}.
\end{proof}

\subsection{Alekseev morphism and representations}\label{subsecAlekseev}
\indent The morphisms $\Phi_{0,1} : \mathcal{L}_{0,1}(H) \to H$ and $\Phi_{1,0} : \mathcal{L}_{1,0}(H) \to \mathcal{H}(H^\circ)$ discussed previously can be combined into a morphism $\Phi_{g,n}$ called {\em Alekseev morphism}. We recall here a general way of defining $\Phi_{g,n}$, following \cite[\S 5]{BFR}. This approach has the advantage to work for any datum $(H,\Phi^\pm)$ but does not interact well with the QMM $\mu_{g,n}$ (see Rmk.\,\ref{remarqueProbleme}). This is an issue for later purposes and will force us to construct in \S\ref{subsecModifiedAlekseev} another version of $\Phi_{g,n}$ in the case $H = U_q(\mathfrak{g})$.

\smallskip

\indent  Recall the Heisenberg double $\mathcal{H}(H^\circ)$ from \S\ref{subsecHeisenbergL10}. The {\em two-sided Heisenberg double} $\mathcal{HH}(H^\circ)$ is the vector space $\mathcal{H}(H^\circ) \otimes H = (H^\circ \otimes H) \otimes H$ endowed with the product defined by the following rules introduced in \cite[\S 5.1]{BFR}:

\indent (i) $\mathcal{H}(H^\circ) \otimes 1_H$ and $1_{\mathcal{H}(H^\circ)} \otimes H$ are subalgebras of $\mathcal{HH}(H^{\circ})$ isomorphic to $\mathcal{H}(H^\circ)$ and $H$. We write $\varphi a$ instead of $(\varphi a) \otimes 1_H $ and $\widetilde{b}$ instead of $1_{\mathcal{H}(H^\circ)} \otimes b$.

\indent (ii) With these identifications, we have the relations 
\begin{equation}\label{2sidedproductHeisenberg}
\forall \, \varphi a \in \mathcal{H}(H^\circ), \:\: \forall \, b \in H, \quad  \varphi a \, \widetilde{b} = \varphi a \otimes b, \quad \widetilde{b}\,\varphi a = \sum_{(b)}\bigl( \varphi \lhd S^{-1}(b_{(2)})\bigr)a \, \widetilde{b_{(1)}}.
\end{equation}
Note in particular that $\widetilde{b}\,a = a \,\widetilde{b}$ for all $a,b \in H$. The elements $\widetilde{b}$ are a substitute for the elements $\widehat{h'}$ from Lemma \ref{lemmeChapeau} which are defined only for $h' \in H'$. The point is that we need elements satisfying the commutation relation \eqref{2sidedproductHeisenberg} for all $h$ in order to define the Alekseev morphism; since in general we do not know how $H'$ compares to $H$ we add them formally instead of using the already present elements $\widehat{h'}$. See Remark \ref{remarqueProbleme} for more comments.

\smallskip

\indent There is a right $H$-module-algebra structure on $\mathcal{HH}(H^\circ) = \mathcal{H}(H^\circ) \otimes H$ given by
\begin{equation}\label{modAlgStructTwoSidedHeis}
(X \otimes b) \cdot h = \sum_{(h)} (X \cdot h_{(1)}) S(h_{(2)})h_{(5)} \otimes S(h_{(3)})bh_{(4)}
\end{equation}
for all $X \otimes b \in \mathcal{HH}(H^\circ)$ and $h \in H$, where we use the already defined right $H$-module-algebra structure \eqref{modAlgStructHeisenberg} on $\mathcal{H}(H^\circ)$. The proof of this fact is left as an exercise for the reader; one must check that it is indeed a right action and it is compatible with the relations \eqref{2sidedproductHeisenberg}.

\smallskip

\indent For all integers $g,n \geq 0$, let $\mathsf{D}_{g,n} : H \to \mathcal{HH}(H^\circ)^{\otimes g} \otimes H^{\otimes n}$ be the algebra morphism defined by
\begin{equation}\label{dressingMap}
 \mathsf{D}_{g,n}(h) = \sum_ {(h)} \widetilde{h_{(1)}}\,h_{(2)} \otimes \ldots \otimes \widetilde{h_{(2g-1)}}\,h_{(2g)} \otimes h_ {(2g+1)} \otimes \ldots \otimes h_{(2g+n)}
\end{equation}
with the convention that $\mathsf{D}_{0,0}(h) = \varepsilon(h)$. As explained around \eqref{monomialFactoBraided} we have a decomposition $\mathcal{L}_{g,n}(H) = \mathcal{L}_{1,0}(H)^{\widetilde{\otimes}\,g} \,\widetilde{\otimes}\, \mathcal{L}_{0,1}(H)^{\widetilde{\otimes}\,n}$; this allows us to inductively define a map $\Phi_{g,n} : \mathcal{L}_{g,n}(H) \to \mathcal{HH}(H^\circ)^{\otimes g} \otimes H^{\otimes n}$ for all values of $g$ and $n$. First since $\mathcal{L}_{0,n+1}(H) = \mathcal{L}_{0,1}(H) \,\widetilde{\otimes}\, \mathcal{L}_{0,n}(H)$ we let
\begin{equation}\label{AlekseevValue1}
\Phi_{0,n+1}(\varphi \otimes w) =  \sum_{[\varphi]} \Phi_{0,1}(\varphi_{[2]}) \otimes \mathsf{D}_{0,n}\bigl( \Phi^+(\varphi_{[1]}) \bigr)\Phi_{0,n}(w)
\end{equation}
for all $\varphi \in \mathcal{L}_{0,1}(H)$ and $w \in \mathcal{L}_{0,n}(H)$, using the left $H^\circ$-coaction \eqref{defCoactCoad} which is dual to $\mathrm{coad}^r$. Second since $\mathcal{L}_{g+1,n}(H) = \mathcal{L}_{1,0}(H) \,\widetilde{\otimes}\, \mathcal{L}_{g,n}(H)$ we let
\begin{equation}\label{AlekseevValue2}
\Phi_{g+1,n}(x \otimes w) =  \sum_{[x]} \Phi_{1,0}(x_{[2]}) \otimes \mathsf{D}_{g,n}\bigl( \Phi^+(x_{[1]}) \bigr)\Phi_{g,n}(w).
\end{equation}
for all $x \in \mathcal{L}_{1,0}(H)$ and $w \in \mathcal{L}_{g,n}(H)$. Note that actually $\Phi_{g,n}$ takes values in the subalgebra $\mathcal{H}(H^\circ) \otimes \mathcal{HH}(H^\circ)^{\otimes (g-1)} \otimes H^{\otimes n}$. When $H$ is quasitriangular and the datum $\Phi^{\pm}$ is chosen as in \eqref{LfuncFromRmat}, the definition of $\Phi_{g,n}$ coincides with the one in \cite[\S 5.2]{BFR}. For general $\Phi^\pm$ the arguments leading to the main of properties of $\Phi_{g,n}$ must be rephrased without using the R-matrix; this is why we give a detailed proof of the following proposition.

\begin{prop}\label{propAlekseevMorph}
1. The following holds in the algebra $\mathcal{HH}(H^\circ)^{\otimes g} \otimes H^{\otimes n}$ endowed with the usual tensor-wise product:
\[ \forall \, h \in H, \:\: \forall \, x \in \mathcal{L}_{g,n}(H), \quad \Phi_{g,n}\bigl( \mathrm{coad}^r(h)(x) \bigr) = \sum_{(h)} \mathsf{D}_{g,n}\bigl( S(h_{(1)}) \bigr) \, \Phi_{g,n}(x) \, \mathsf{D}_{g,n}(h_{(2)}). \]
2. The map $\Phi_{g,n}$ is a morphism of right $H$-module-algebras.\footnote{The target is equipped with the usual action of $H$ on a tensor product, obtained by combining the action \eqref{modAlgStructTwoSidedHeis} on $\mathcal{HH}(H^\circ)$ and the right adjoint action \eqref{defAdr} on $H$ {\it via} iterated coproduct $H \to H^{\otimes (g+n)}$.} 
\end{prop}
\begin{proof}
In this proof we use Sweedler's notation with implict summation.
\\1. The proof is by induction. We start with the base cases \textit{i.e.} $(g,n) = (0,1)$ and $(g,n)=(1,0)$. Since $\Phi_{0,1}:\bigl( \mathcal{L}_{0,1}(H), \mathrm{coad}^r \bigr) \to (H,\mathrm{ad}^r)$ is $H$-linear it satisfies the desired property. For $\Phi_{1,0} : \mathcal{L}_{1,0}(H) \to \mathcal{H}(H^\circ)$, recall the $H$-module-algebra structure \eqref{modAlgStructHeisenberg} on $\mathcal{H}(H^\circ)$. It satisfies $\Phi_{1,0}\bigl( \mathrm{coad}^r(h)(x) \bigr) = \Phi_{1,0}(x) \cdot h$ \cite[\S 3.3]{BFR}. The desired property is then established by the following computation for all $\varphi a \in \mathcal{H}(H^\circ) \subset \mathcal{HH}(H^\circ)$ and $h \in H$
\begin{align}
\begin{split}\label{QMMHeisDansPreuve}
\mathsf{D}_{1,0}\bigl( S(h_{(1)}) \bigr) \varphi a \mathsf{D}_{1,0}(h_{(2)}) &=  \widetilde{S(h_{(2)})} S(h_{(1)}) \varphi a \, \widetilde{h_{(3)}} h_{(4)}\\
&=\bigl( S(h_{(2)}) \rhd \varphi \lhd h_{(3)} \bigr) \widetilde{S(h_{(4)})} S(h_{(1)}) \, a\, \widetilde{h_{(5)}} h_{(6)}\\
&= \bigl( S(h_{(2)}) \rhd \varphi \lhd h_{(3)} \bigr) S(h_{(1)}) a h_{(4)} = (\varphi a) \cdot h
\end{split}
\end{align}
where the second and third equalities use \eqref{2sidedproductHeisenberg}.

Let us now prove the claim for $\Phi_{0,n}$ by induction on $n$. As a preparation, note that the coaction $\mathcal{L}_{g,n}(H) \to H^\circ \otimes \mathcal{L}_{g,n}(H)$ is characterized by $\langle x_{[1]}, h \rangle \, x_{[2]} = \mathrm{coad}^r(h)(x)$ for all $h \in H$ and $x \in \mathcal{L}_{g,n}(H)$. One thus checks the two following facts:
\begin{align*}
&\forall \, h \in H, \quad (h \rhd x_{[1]}) \otimes x_{[2]} = x_{[1]} \otimes \mathrm{coad}^r(h)(x_{[2]}),\\
&\forall \, h \in H, \quad \mathrm{coad}^r(h)(x)_{[1]} \otimes \mathrm{coad}^r(h)(x)_{[2]} = \mathrm{coad}^r(h_{(1)})(x_{[1]}) \otimes \mathrm{coad}^r(h_{(2)})(x_{[2]}).
\end{align*}
The first identity is easy, and true for any coaction $\textstyle x\mapsto \sum_{[x]} x_{[1]}\otimes x_{[2]}$. The second is obtained as follows, where $a\in H$:
\begin{align*}
&\langle \mathrm{coad}^r(h_{(1)})(x_{[1]}) ,a\rangle \mathrm{coad}^r(h_{(2)})(x_{[2]})  = \langle x_{[1]} , h_{(1)}aS(h_{(2)})\rangle \mathrm{coad}^r(h_{(3)})(x_{[2]})\\
 =\:\,& \mathrm{coad}^r(h_{(3)})(\mathrm{coad}^r(h_{(1)}aS(h_{(2)}))(x)) = \mathrm{coad}^r(h_{(1)}aS(h_{(2)})h_{(3)})(x)\\
 =\:\,& \mathrm{coad}^r(ha)(x) = \mathrm{coad}^r(a)(\mathrm{coad}^r(h)(x)) = \langle \mathrm{coad}^r(h)(x)_{[1]} ,a\rangle \mathrm{coad}^r(h)(x)_{[2]}.
\end{align*}
Using these identities we have the following computation for all $h \in H$ and $x \in \mathcal{L}_{g,n}(H)$:
\begin{align}
\begin{split}\label{trucTechniquePreparation}
&\Phi^+\bigl( \mathrm{coad}^r(h)(x)_{[1]} \bigr) \otimes \mathrm{coad}^r(h)(x)_{[2]} = \Phi^+\bigl( \mathrm{coad}^r(h_{(1)})(x_{[1]}) \bigr) \otimes \mathrm{coad}^r(h_{(2)})(x_{[2]})\\
=\:&\Phi^+\bigl( S(h_{(2)}) \rhd x_{[1]} \lhd h_{(1)} \bigr) \otimes \mathrm{coad}^r(h_{(3)})(x_{[2]}) \quad \text{\footnotesize by \eqref{defCoad}}\\
=\:& S(h_{(1)})\Phi^+\bigl( h_{(2)}S(h_{(4)}) \rhd x_{[1]} \bigr) h_{(3)} \otimes \mathrm{coad}^r(h_{(5)})(x_{[2]}) \quad \text{\footnotesize by \eqref{quasiCocommPhi}}\\
=\:&S(h_{(1)})\Phi^+\bigl( x_{[1]} \bigr) h_{(3)} \otimes \mathrm{coad}^r(h_{(5)})\bigl( \mathrm{coad}^r(h_{(2)}S(h_{(4)}))(x_{[2]}) \bigr)\\
=\:&S(h_{(1)})\Phi^+\bigl( x_{[1]} \bigr) h_{(3)} \otimes \mathrm{coad}^r(h_{(2)})(x_{[2]})
\end{split}
\end{align}
where the last equality uses that $\mathrm{coad}^r$ is a right action. Hence for all $\varphi \otimes w \in \mathcal{L}_{0,n+1}(H)$ we have
\begin{align*}
&\Phi_{0,n+1}\bigl( \mathrm{coad}^r(h)(\varphi \otimes w) \bigr) = \Phi_{0,n+1}\bigl( \mathrm{coad}^r(h_{(1)})(\varphi) \otimes \mathrm{coad}^r(h_{(2)})(w) \bigr)\\
=\:&\Phi_{0,1}\bigl( \mathrm{coad}^r(h_{(1)})(\varphi)_{[2]} \bigr) \otimes \mathsf{D}_{0,n}\bigl[ \Phi^+\bigl(\mathrm{coad}^r(h_{(1)})(\varphi)_{[1]} \bigr) \bigr] \Phi_{0,n}\bigl( \mathrm{coad}^r(h_{(2)})(w) \bigr) \quad \text{\footnotesize by \eqref{AlekseevValue1}}\\
=\:&\Phi_{0,1}\bigl( \mathrm{coad}^r(h_{(2)})(\varphi_{[2]}) \bigr) \otimes \mathsf{D}_{0,n}\bigl[ S(h_{(1)}) \Phi^+(\varphi_{[1]}) h_{(3)} \bigr] \Phi_{0,n}\bigl( \mathrm{coad}^r(h_{(4)})(w) \bigr) \quad \text{\footnotesize by \eqref{trucTechniquePreparation}}\\
=\:&S(h_{(2)})\Phi_{0,1}(\varphi_{[2]})h_{(3)} \otimes \mathsf{D}_{0,n}\bigl[ S(h_{(1)}) \Phi^+(\varphi_{[1]}) h_{(4)} \bigr] \mathsf{D}_{0,n}\bigl( S(h_{(5)}) \bigr) \Phi_{0,n}(w) \mathsf{D}_{0,n}(h_{(6)})\\
=\:&S(h_{(2)})\Phi_{0,1}(\varphi_{[2]})h_{(3)} \otimes \mathsf{D}_{0,n}\bigl( S(h_{(1)}) \bigr) \mathsf{D}_{0,n}\bigl( \Phi^+(\varphi_{[1]}) \bigr) \Phi_{0,n}(w) \mathsf{D}_{0,n}(h_{(4)})\\
=\:&\mathsf{D}_{0,n+1}\bigl( S(h_{(1)}) \bigr) \bigl[ \Phi_{0,1}(\varphi_{[2]}) \otimes \mathsf{D}_{0,n}\bigl( \Phi^+(\varphi_{[1]}) \bigr) \Phi_{0,n}(w) \bigr] \mathsf{D}_{0,n+1}(h_{(2)})
\end{align*}
which by \eqref{AlekseevValue1} is the desired result. The fourth equality is by induction hypothesis and the last equality uses that $S$ is an antimorphism of coalgebras.

The proof for $\Phi_{g,n}$ is by induction on $g$ and uses exactly the same computation.
\\2. The hard part is ``algebra morphism'', which we do first. As a technical preliminary note that for all $\psi \in H^\circ$ and $v \in \mathcal{L}_{g,n}(H)$ we have
\begin{align*}
&\mathcal{R}\bigl( \psi_{(1)} \otimes v_{[1]} \bigr)  \mathrm{coad}^r\bigl( S^{-1}\bigl(\Phi^+(\psi_{(2)})\bigr) \bigr)(v_{[2]}) = \mathcal{R}\bigl( \psi_{(1)} \otimes v_{[1]} \bigr) \bigl\langle v_{[2][1]}, S^{-1}\bigl(\Phi^+(\psi_{(2)})\bigr) \bigr\rangle v_{[2][2]}\\
=\:&\mathcal{R}\bigl( \psi_{(1)} \otimes v_{[1](1)} \bigr) \bigl\langle S^{-1}(v_{[1](2)}), \Phi^+(\psi_{(2)}) \bigr\rangle v_{[2]} = \mathcal{R}\bigl( \psi_{(1)} \otimes v_{[1](1)} \bigr) \mathcal{R}\bigl( \psi_{(2)} \otimes S^{-1}(v_{[1](2)}) \bigr) v_{[2]}\\
=\:&\varepsilon(\psi) \varepsilon\bigl( v_{[1]} \bigr) v_{[2]} = \varepsilon(\psi)v
\end{align*}
by \eqref{defCoRMat} and usual properties of a co-R-matrix. The proof of the claim is again by induction on $n$ and $g$. Assume for instance that $\Phi_{g,n}$ is a morphism of algebras (the case $\Phi_{0,n}$ is similar). Let $x \otimes v, y \otimes w \in \mathcal{L}_{g+1,n}(H) = \mathcal{L}_{1,0}(H) \,\widetilde{\otimes}\, \mathcal{L}_{g,n}(H)$; it follows from \eqref{braidedTensProdLgn} that their product is
\[ (x \otimes v)(y \otimes w) = \mathcal{R}\bigl( y_{[1]} \otimes v_{[1]} \bigr) \, xy_{[2]} \otimes v_{[2]}w. \]
Hence
\begin{align*}
&\Phi_{g+1,n}\bigl( (x \otimes v)(y \otimes w) \bigr) = \mathcal{R}\bigl( y_{[1]} \otimes v_{[1]} \bigr) \, \Phi_{1,0}\bigl(x_{[2]}y_{[2][2]}\bigr) \otimes \mathsf{D}_{g,n}\bigl( \Phi^+(x_{[1]}y_{[2][1]}) \bigr) \Phi_{g,n}\bigl(v_{[2]}w\bigr)\\
=\:& \mathcal{R}\bigl( y_{[1](1)} \otimes v_{[1]} \bigr) \, \Phi_{1,0}\bigl(x_{[2]}\bigr)\Phi_{1,0}\bigl(y_{[2]}\bigr) \otimes \mathsf{D}_{g,n}\bigl( \Phi^+(x_{[1]}) \bigr)  \mathsf{D}_{g,n}\bigl( \Phi^+(y_{[1](2)}) \bigr) \Phi_{g,n}\bigl(v_{[2]}\bigr)\Phi_{g,n}(w)\\
=\:& \mathcal{R}\bigl( y_{[1](1)} \otimes v_{[1]} \bigr) \, \Phi_{1,0}\bigl(x_{[2]}\bigr)\Phi_{1,0}\bigl(y_{[2]}\bigr)\\
&\qquad\qquad\otimes \mathsf{D}_{g,n}\bigl( \Phi^+(x_{[1]}) \bigr) \Phi_{g,n}\bigl[ \mathrm{coad}^r\bigl( S^{-1}\bigl(\Phi^+(y_{[1](2)})\bigr) \bigr)(v_{[2]})\bigr] \mathsf{D}_{g,n}\bigl( \Phi^+(y_{[1](3)})\bigr) \Phi_{g,n}(w)\\
=\:&\Phi_{1,0}\bigl(x_{[2]}\bigr)\Phi_{1,0}\bigl(y_{[2]}\bigr) \otimes \mathsf{D}_{g,n}\bigl( \Phi^+(x_{[1]}) \bigr) \Phi_{g,n}\bigl(v_{[2]}\bigr) \mathsf{D}_{g,n}\bigl( \Phi^+(y_{[1]})\bigr) \Phi_{g,n}(w)\\
=\:& \Phi_{g+1,n}(x \otimes v) \Phi_{g+1,n}(y \otimes w)
\end{align*}
where the first equality is by definition of $\Phi_{g+1,n}$, the second uses the coaction property and the fact that all maps involved are morphism of algebras, the third uses item 1 of the present Proposition and the fact that $\Phi^+$ is an anti-morphism of coalgebras, the fourth uses the technical preliminary and the last is by definition of $\Phi_{g+1,n}$.
\\To prove $H$-linearity, we first show it when $(g,n)=(1,0)$. For all $X \otimes b = X \, \widetilde{b} \in \mathcal{HH}(H^\circ) = \mathcal{H}(H^\circ) \otimes H$ and $h \in H$ we have
\begin{align}
\begin{split}\label{computationProofQMM2SidedHeis}
&\mathsf{D}_{1,0}\bigl( S(h_{(1)}) \bigr) \, X \, \widetilde{b} \, \mathsf{D}_{1,0}(h_{(2)}) = \mathsf{D}_{1,0}\bigl( S(h_{(1)}) \bigr) X \mathsf{D}_{1,0}(h_{(2)}) \mathsf{D}_{1,0}\bigl( S(h_{(3)}) \bigr)  \widetilde{b} \mathsf{D}_{1,0}(h_{(4)})\\
&\overset{\eqref{QMMHeisDansPreuve}}{=} \bigl( X \cdot h_{(1)} \bigr) \,  \mathsf{D}_{1,0}\bigl( S(h_{(3)}) \bigr)  \widetilde{b} \,\mathsf{D}_{1,0}(h_{(4)}) \overset{\eqref{dressingMap}}{=} \bigl( X \cdot h_{(1)} \bigr) \, \widetilde{S(h_{(3)})} S(h_{(2)}) \widetilde{b} \,\widetilde{h_{(4)}} h_{(5)}\\
&\overset{\eqref{2sidedproductHeisenberg}}{=} \bigl( X \cdot h_{(1)} \bigr) \, S(h_{(2)})h_{(5)} \widetilde{S(h_{(3)})bh_{(4)}} \overset{\eqref{modAlgStructTwoSidedHeis}}{=} \bigl( X \, \widetilde{b} \, \bigr) \cdot h.
\end{split}
\end{align}
Since $\mathsf{D}_{g,n}(h) = \mathsf{D}_{1,0}(h_{(1)}) \otimes \ldots \otimes \mathsf{D}_{1,0}(h_{(g)}) \otimes h_{(g+1)} \otimes \ldots \otimes h_{(g+n)}$ by very definition in \eqref{dressingMap}, we conclude that $x \cdot h = \mathsf{D}_{g,n}\bigl( S(h_{(1)}) \bigr) x \mathsf{D}_{g,n}(h_{(2)})$ for all $x \in \mathcal{HH}(H^\circ)^{\otimes g} \otimes H^{\otimes n}$, so that item 1 in the present Proposition actually means $H$-linearity.
\end{proof}
\noindent The computation \eqref{computationProofQMM2SidedHeis} shows that $\mathsf{D}_{1,0}$ is a quantum moment map for $\mathcal{HH}(H^\circ)$. Actually the action of $H$ in \eqref{modAlgStructTwoSidedHeis} was defined precisely in order to have this property.

\smallskip

We record that the question of injectivity of $\Phi_{g,n}$ reduces to the base cases:
\begin{lem}\label{lemmaInjAlekseev}
If $\Phi_{0,1}$ and $\Phi_{1,0}$ are injective then $\Phi_{g,n}$ is injective for all $g,n$.
\end{lem}
\begin{proof} Recall that a tensor product over $\Bbbk$ of injective $\Bbbk$-linear maps is again injective ($\Bbbk$ is a field). We first do the case $\Phi_{0,n}$ by induction on $n$. Let
\[ T_{0,n} : H' \otimes H^{\otimes (n-1)} \to H' \otimes H^{\otimes (n-1)}, \quad \Phi_{0,1}(\alpha) \otimes w \mapsto \Phi_{0,1}(\alpha_{[2]}) \otimes \mathsf{D}_{0,n-1}\bigl( \Phi^+(\alpha_{[1]}) \bigr)w \]
It is a linear isomorphism, with $T_{0,n}^{-1}\bigl( \Phi_{0,1}(\alpha) \otimes w \bigr) = \Phi_{0,1}(\alpha_{[2]}) \otimes \mathsf{D}_{0,n-1}\bigl( \Phi^+\bigl( S^{-1}(\alpha_{[1]}) \bigr) \bigr)w$, which is seen by the coaction property $\alpha_{[1]} \otimes \alpha_{[2][1]} \otimes \alpha_{[2][2]} = \alpha_{[1](1)} \otimes \alpha_{[1](2)} \otimes \alpha_{[2]}$. Now by very definition we have $\Phi_{0,n} = T_{0,n} \circ (\Phi_{0,1} \otimes \Phi_{0,n-1})$ and induction hypothesis applies. The general case $\Phi_{g,n}$ is done similarly by induction on $g$ with $n$ fixed. Write $\mathcal{H}' = \mathrm{im}(\Phi_{1,0})$, and use the isomorphism
\begin{align*}
T_{g,n} : \mathcal{H}' \otimes \bigl( \mathcal{HH}(H^\circ)^{\otimes (g-1)} \otimes H^{\otimes n} \bigr) &\to \mathcal{H}' \otimes \bigl( \mathcal{HH}(H^\circ)^{\otimes (g-1)} \otimes H^{\otimes n} \bigr)\\
\Phi_{1,0}(x) \otimes w &\mapsto \Phi_{1,0}(x_{[2]}) \otimes \mathsf{D}_{g-1,n}\bigl( \Phi^+(x_{[1]}) \bigr)w
\end{align*}
whose inverse is $T_{g,n}^{-1}\bigl( \Phi_{1,0}(x) \otimes w \bigr) = \Phi_{1,0}(x_{[2]}) \otimes \mathsf{D}_{g,n-1}\bigl( \Phi^+\bigl( S^{-1}(x_{[1]}) \bigr) \bigr)w$. It satisfies $\Phi_{g,n} = T_{g,n} \circ \bigl( \Phi_{1,0} \otimes \Phi_{g-1,n} \bigr)$, whence the result.
\end{proof}

\begin{remark}\label{remarqueProbleme}
{\rm When $g > 0$ it is not true that $\mu_{g,n} \circ \Phi_{g,n} = \mathsf{D}_{g,n}$. Said differently, the bottom triangle in the following diagram
\[ \xymatrix@R=1.5em@C=4em{
\mathcal{L}_{0,1}(H) \ar[r]^-{\mathfrak{d}_{g,n}} \ar[d]_-{\Phi_{0,1}} & \mathcal{L}_{g,n}(H) \ar[d]^-{\Phi_{g,n}}\\
H' \ar[r]_-{\mathsf{D}_{g,n}} \ar[ur]^-{\mu_{g,n}} & \mathcal{HH}(H^\circ)^{\otimes g} \otimes H^{\otimes n}
} \]
does not commute when $g > 0$ (a proof that it commutes for $g=0$ is given in \cite[Prop.\,6.18]{BR1}). For instance in the case $g=1$ we have $\mu_{1,0} \circ \Phi_{1,0} = \widehat{\mathsf{D}}_{1,0} \neq \mathsf{D}_{1,0}$  by Prop.\,\ref{propCommDiag10}. The failure of the diagram comes from the use of elements $\widetilde{h}$ instead of elements $\widehat{h}$ to define $\mathsf{D}_{g,n}$ and hence $\Phi_{g,n}$. In general there is no solution to this problem: $\widehat{h}$ might not exist for all $h \in H$ and thus to define $\mathsf{D}_{g,n}$ we are forced to formally introduce elements $\widetilde{h}$ for all $h \in H$ with the same commutation relations as $\widehat{h'}$, leading to the extension $\mathcal{HH}(H^\circ)$. In the case where $H$ is finite-dimensional it is easy to see that $\widehat{h}$ exist for all $h \in H$ so we can define a ``hat version'' $\widehat{\mathsf{D}}_{g,n}$, yielding an alternative Alekseev morphism $\widehat{\Phi}_{g,n} : \mathcal{L}_{g,n}(H) \to \mathcal{H}(H^*)^{\otimes g} \otimes H^{\otimes n}$ such that the diagram commutes (\cite[\S 5.3]{BFR}, where $\widehat{\Phi}_{g,n}$ was denoted $\Phi_{g,n}^{\rm fin}$). In the case where $H = U_q(\mathfrak{g})$ we know precisely how $H'$ compares to $H$ and then we can use a minimal extension $\widehat{\mathcal{H}}$ instead of $\mathcal{HH}(H^\circ)$, yielding again a commutative diagram; see \S\ref{subsecModifiedAlekseev}.}
\end{remark}

\indent The faithful representation $\rho : \mathcal{H}(H^\circ) \to \mathrm{End}_\Bbbk(H^\circ)$ from \eqref{HeisenbergRep} can be extended to a morphism $\widetilde{\rho} : \mathcal{HH}(H^\circ) \to \mathrm{End}_\Bbbk(H^\circ)$ defined by
\begin{equation}\label{HeisenbergRepHH}
\widetilde{\rho}\bigl( \varphi a  \widetilde{b} \bigr)(\psi) = \varphi \star \bigl( a \rhd \psi \lhd S^{-1}(b) \bigr) \qquad (\forall \, \psi \in H^\circ).
\end{equation}
Unlike $\rho$, the morphism $\widetilde{\rho}$ is not injective: by Lemma \ref{lemmeChapeau} we have $\widetilde{\rho}(\widehat{h'}) = \widetilde{\rho}(\widetilde{h'})$ for all $h' \in H'$ but $\widehat{h'} \neq \widetilde{h'}$ since $\widetilde{h'} \not\in \mathcal{H}(H^\circ)$. Consider however the following algebra morphism
\begin{equation}\label{defFgn}
F_{g,n} : \mathcal{L}_{g,n}(H) \xrightarrow{\quad\Phi_{g,n}\quad} \mathcal{HH}(H^\circ)^{\otimes g} \otimes H^{\otimes n} \xrightarrow{\quad\widetilde{\rho}^{\otimes g} \,\otimes\, \mathrm{id}_H^{\otimes n}\quad} \mathrm{End}_\Bbbk(H^\circ)^{\otimes g} \otimes H^{\otimes n}.
\end{equation}
One can show by induction as in Lemma \ref{lemmaInjAlekseev} that if $\Phi_{0,1}$ and $\Phi_{1,0}$ are injective then $F_{g,n}$ is injective as well. Indeed, let
\[ \begin{array}{c}
 \widetilde{T}_{g,n} : \mathcal{H}' \otimes \bigl( \mathrm{End}_\Bbbk(H^\circ)^{\otimes (g-1)} \otimes H^{\otimes n} \bigr) \to \mathcal{H}' \otimes \bigl( \mathrm{End}_\Bbbk(H^\circ)^{\otimes (g-1)} \otimes H^{\otimes n} \bigr)\\[.3em]
\Phi_{1,0}(x) \otimes w \mapsto \Phi_{1,0}(x_{[2]}) \otimes \bigl( (\widetilde{\rho}^{\otimes (g-1)} \otimes \mathrm{id}_H^{\otimes n}) \circ \mathsf{D}_{g-1,n}\bigr)\bigl( \Phi^+(x_{[1]}) \bigr)w
\end{array}\]
with $\mathcal{H}' = \mathrm{im}(\Phi_{1,0})$. It is a linear  isomorphism (for the inverse, replace $x_{[1]}$ by $S^{-1}(x_{[1]})$). Then $F_{g,n} = \bigl( \rho \otimes \mathrm{id}_{\mathrm{End}_\Bbbk(H^\circ)^{\otimes (g-1)} \otimes H^{\otimes n}} \bigr) \circ \widetilde{T}_{g,n} \circ \bigl( \Phi_{1,0} \otimes F_{g-1,n} \bigr)$. Since $\rho$ is injective it follows that $F_{g,n}$ is also injective.

\indent  Assume that the Hopf algebra $H$ is finite-dimensional; then, having the morphisms $\Phi^{\pm}$ is equivalent to having an R-matrix $R \in H^{\otimes 2}$. The quasitriangular Hopf algebra $(H,R)$ is called {\em factorizable} if the map $\Phi_{0,1} : H^* \to H$, $\varphi \mapsto (\varphi \otimes \mathrm{id}_H)(RR')$ is a bijection.
\begin{prop}\label{coroLgnFacto}
If $H$ is finite-dimensional and factorizable then $F_{g,n}$ is an isomorphism.
\end{prop}
\begin{proof}
These assumptions imply that $\Phi_{1,0}$ is an isomorphism \cite[Th.\,4.8]{FaitgSL2Z}. In particular $\Phi_{0,1}$ and $\Phi_{1,0}$ are injective which as noted above implies that $F_{g,n}$ is injective. By equality of the dimensions of source and target, it is an isomorphism.
\end{proof}

\indent Les us go back to the general situation of a Hopf algebra $H$ satisfying the assumptions 1,2 and 3.

\indent The existence of the morphism $F_{g,n}$ implies that $\mathcal{L}_{g,n}(H)$ has a representation on the space $(H^\circ)^{\otimes g} \otimes V_1 \otimes \ldots \otimes V_n$ for any $H$-modules $V_1,\ldots,V_n$, which is an important point in the theory of graph algebras.

\smallskip

To conclude, we note as a side remark that the diagram in Remark \ref{remarqueProbleme} becomes commutative through the representation $\widetilde{\rho}$:
\begin{prop}\label{Fmuformula}
For all $g,n$, we have a commutative diagram
\[ \xymatrix@R=1.5em@C=4.5em{
\mathcal{L}_{0,1}(H) \ar[r]^-{\mathfrak{d}_{g,n}} \ar[d]_-{\Phi_{0,1}} & \mathcal{L}_{g,n}(H) \ar[r]^-{\Phi_{g,n}} \ar[dr]^-{F_{g,n}}& \mathcal{HH}(H^\circ)^{\otimes g} \otimes H^{\otimes n} \ar[d]^{\widetilde{\rho}^{\otimes g} \otimes \mathrm{id}_H^{\otimes n}}\\
H' \ar[r]_-{\mathsf{D}_{g,n}} \ar[ur]^-{\mu_{g,n}} & \mathcal{HH}(H^\circ)^{\otimes g} \otimes H^{\otimes n} \ar[r]_{\widetilde{\rho}^{\otimes g} \otimes \mathrm{id}_H^{\otimes n}} & \mathrm{End}_\Bbbk(H^\circ)^{\otimes g} \otimes H^{\otimes n}
} \]
\end{prop}
\noindent Equivalently, by definition of $\widetilde{\rho}$, $F_{g,n}$ and $\mathsf{D}_{g,n}$ we can write this property as
\begin{equation}\label{HmodStructThroughQMM}
(F_{g,n}\circ \mu_{g,n})(h') = \textstyle \sum_{(h')}  {\rm act}(h'_{(1)})\otimes \ldots \otimes {\rm act}(h'_{(g)})\otimes h'_{(g+1)}\otimes \ldots h'_{(g+n)}
\end{equation}
for all $h' \in H'$, where $\mathrm{act} : H \to \mathrm{End}_\Bbbk(H^\circ)$ is defined by $\textstyle \mathrm{act}(h)(\varphi) = \sum_{(h)}h_{(2)} \rhd \varphi \lhd S^{-1}(h_{(1)})$.
\begin{proof}
For the case $(g,n) = (1,0)$ note that $\rho \circ \mathsf{D}_{1,0} = \rho \circ \widehat{\mathsf{D}}_{1,0}$ by \eqref{HeisenbergRep} and Lemma \ref{lemmeChapeau}, so that Prop.\,\ref{propCommDiag10} gives the result in this case. For the general case the proof is completely similar to the one of Prop.\,\ref{propAlekseevModifAndQMM} below, where the reader will find all the details. For the induction step, the only changes are typographical: $\widehat{\Phi}_{g,n}$ must be replaced by $F_{g,n}$ and $\widehat{\mathsf{D}}_{g,n}$ must be replaced by $\mathrm{act}_{g,n} := \bigl( \widetilde{\rho}^{\otimes g} \otimes \mathrm{id}_H^{\otimes n} \bigr) \circ \mathsf{D}_{g,n}$, which is given by the right-hand side of \eqref{HmodStructThroughQMM}.
\end{proof}
The following consequence is motivated by Lemma \ref{lemmaInjQMMtrivial} regarding injectivity of $\mu_{g,n}$:
\begin{cor}\label{rmkMu10Inj} Assume that $\Phi_{1,0}$ is injective. Then $\mu_{1,0}$ is injective if and only if the linear map $H' \to \mathrm{End}_\Bbbk(H^\circ)$, $h \mapsto \mathrm{coad}^r(h)$ is injective (with the right action $\mathrm{coad}^r$ from \eqref{defCoad}).
\end{cor}
\begin{proof} Indeed $F_{1,0} \circ \mu_{1,0} = \mathrm{act}$ by Prop.\,\ref{Fmuformula}, and $F_{1,0}$ is injective because so is $\Phi_{1,0}$. Consider the linear automorphism $T$ of $H^\circ$ given by $\varphi \mapsto \textstyle \sum_{(\varphi)}\mathcal{R}\bigl( S(\varphi_{(3)}) \otimes \varphi_{(1)} \bigr)\varphi_{(2)}$. Then it can be checked that $\mathrm{coad}^r(h) = T^{-1} \circ \mathrm{act}\bigl( S(h) \bigr) \circ T$ for all $h$, whence the result.
\end{proof}

\section{Preliminaries on quantum groups}\label{sectionPreliminaires}
Here we set conventions and collect facts on quantum groups which will eventually allow us to define various instances of graph algebras associated to quantum groups: generic parameter, integral form, specialization at roots of unity and finite-dimensional quotient of these. References for the results discussed in this section are \cite{CP,BG,VY}.

\subsection{The quantum enveloping algebra \texorpdfstring{$U_q^\Lam(\mathfrak {g})$}{ }}\label{sec:Uq} Recall from \S\ref{sectionnotations} our notations regarding Lie algebras. Fix a Cartan decomposition of $\mathfrak {g}$ and a $\mathbb{Z}$-lattice $\Lambda$ such that $Q \subset \Lambda \subset P$. We denote by $U_q^\Lam := U_q^\Lam(\mathfrak{g})$ the ${\mathbb C}(q)$-algebra generated by elements $E_i, F_i, K_{\lambda}$, ($\lambda \in \Lambda$, $i=1,\ldots,m$) with relations:
\begin{eqnarray}
&&K_0 = 1_{U_q^\Lam}, \quad K_{\lambda}K_{\lambda'} = K_{\lambda + \lambda'}, \quad K_{\lambda}E_i = q^{(\lambda,\alpha_i)}E_iK_{\lambda}, \quad K_{\lambda}F_i = q^{-(\lambda,\alpha_i)}F_iK_{\lambda}, \notag \\
&&E_iF_j-F_jE_i=\delta_{i,j}\frac{K_{\alpha_i}-K_{\alpha_i}^{-1}}{q_i-q_i^{-1}},\label{EFK}\notag\\
&&\sum_{r=0}^{1-a_{ij}} (-1)^r \left[\begin{array}{c} 1-a_{ij} \\ r \end{array}\right]_{q_i} E_i^{1-a_{ij}-r}E_jE_i^{r} = 0 \quad {\rm if}\ i\ne j,\label{Serre1}\notag\\
&&\sum_{r=0}^{1-a_{ij}} (-1)^r  \left[\begin{array}{c} 1-a_{ij} \\ r \end{array}\right]_{q_i}  F_i^{1-a_{ij}-r}F_jF_i^{r} = 0 \quad {\rm if}\ i\ne j, \quad \text{where } \left[\begin{array}{c} s \\ t \end{array}\right]_q = \frac{[s]_q!}{[t]_q!\,[s-t]_q!}.\label{Serre2}\notag
\end{eqnarray}
Usual notations are $K_i := K_{\alpha_i}$, $L_i := K_{\varpi_i}$ and $q_i = q^{d_i}$ (see \S\ref{sectionnotations}). Note that if $\textstyle \lambda=\sum_{i=1}^m n_i \varpi_i \in P$ then $\textstyle K_\lambda=\prod_{i=1}^m L_i^{n_i}$ and $\textstyle K_i =\prod_{j=1}^m L_j^{a_{ji}}$. Also $L_iE_j = q_i^{\delta_{i,j}}E_jL_i$ and $L_iF_j=q_i^{-\delta_{i,j}}F_jL_i$.

\indent For $\Lambda = P$ this is the so-called {\em simply-connected} quantum group $U_q^\Pup$, generated by $E_i,F_i,L_i$. For $\Lambda = Q$ this is the so-called {\em adjoint} quantum group $U_q^\Q$, generated by $E_i,F_i,K_i$. We will mostly use these two versions, except in  Appendix \ref{appDlRoot}.

$U_q^\Lam$ is a Hopf algebra whose coproduct $\Delta$, antipode $S$, and counit $\varepsilon$ are given by
\begin{equation*}\begin{array}{c}
\Delta(K_{\lambda})=K_{\lambda} \otimes K_{\lambda}\ ,\ \Delta(E_i)=E_i\otimes K_i+1\otimes E_i\ ,\ \Delta(F_i)=F_i\otimes 1 + K_i^{-1}\otimes F_i \\ S(E_i) = -E_iK_i^{-1}\ ,\ S(F_i) = -K_iF_i\ ,\ S(K_{\lambda}) = K_{-\lambda} = K_{\lambda}^{-1}\\ \varepsilon(E_i) = \varepsilon(F_i)=0,\ \varepsilon(K_{\lambda})=1.
\end{array}
\end{equation*}
Note that any $U_q^\Lam$ is a Hopf subalgebra of $U_q^\Pup$.
The element
$$\textstyle \ell := K_{2\rho} = \prod_{j=1}^mL_j^2\in U_q^\Q (\mathfrak{g})$$ 
is pivotal for the Hopf algebra $U_q^\Lam$, i.e., it is grouplike and satisfies $S^2(x) = \ell x\ell^{-1}$ for all $x\in U_q^\Lam$.

The {\em Cartan subalgebra} of $U_q^\Lam$, denoted by $U_q^\Lam(\mathfrak{h})$, is the (commutative) subalgebra generated by the elements $K_\lambda$ for all $\lambda \in \Lambda$. We denote by $U_q(\mathfrak{n}_-)$ (resp. $U_q(\mathfrak{n}_+)$) the subalgebras of $U_q^\Lam$ generated by the elements $F_i$ (resp. $E_i$). Then the {\em Borel subalgebras} are $U_q^\Lam(\mathfrak{b}_\pm) := U_q(\mathfrak{n}_\pm)U_q^\Lam(\mathfrak{h})$. 

We fix a reduced expression $s_{i_1}\ldots s_{i_N}$ of the longest element $w_0$ of the Weyl group of $\mathfrak{g}$. It induces a total ordering of the positive roots, 
\begin{equation}\label{orderingbeta}
\beta_1 = \alpha_{i_1}, \beta_2 = s_{i_1}(\alpha_{i_2}),\ldots, \beta_N = s_{i_1}\ldots s_{i_{N-1}}(\alpha_{i_N}).
\end{equation}
The braid group ${\mathcal B}(\mathfrak{g})$ acts on $U_q^\Lam$ by means of Lusztig's algebra automorphisms $T_i$, associated to the simple roots $\alpha_i$. The root vectors of $U_q^\Lam$ with respect to the above ordering are defined by
\begin{equation}\label{rootvectdef}
E_{\beta_k} =  T_{i_1}\ldots T_{i_{k-1}}(E_{i_k})\ ,\ F_{\beta_k} =  T_{i_1}\ldots T_{i_{k-1}}(F_{i_k}).
\end{equation}
The monomials in the elements $E_{\beta}$ and $F_{\beta}$ ($\beta\in \phi^+$) form bases of $U_q(\mathfrak{n}_+)$ and $U_q(\mathfrak{n}_-)$, respectively. The elements $K_{\lambda}$, $\lambda\in \Lambda$, form a basis of $U_q^\Lam(\mathfrak{h})$, and a PBW basis of $U_q^\Lam(\mathfrak{g})$ is deduced from the polarisation $U_q^\Lam(\mathfrak{g})=U_q(\mathfrak{n}_-)U_q^\Lam(\mathfrak{h})U_q(\mathfrak{n}_+)$. 

The center ${\mathcal  Z}(U_q^\Pup)$ of $U_q^\Pup$ is isomorphic to the subalgebra of elements in $U_q^\Pup({\mathfrak h})$ invariant under (a $(\mz/2)^m$ extension of) the Weyl group; this is the quantum analog of the Harish-Chandra isomorphism.

\smallskip

\indent  Let ${\mathcal C}$ be the category of finite dimensional $U_q^\Q$-modules {\it of type $1$}, i.e. where the elements $K_i$ act semisimply with eigenvalues in $q_i^{\mathbb Z}$. The category ${\mathcal C}$ is monoidal and rigid (i.e. closed under finite direct sums, tensor products and duals), and it is semisimple. The simple objects up to isomorphism are the highest weight modules $V_\mu$ of weights $\mu\in P_+$.

Occasionally we will use the simple finite dimensional $U(\mathfrak{g})$-modules; they are parametrized by weights $\mu\in P_+$ and we denote them by $L(\mu).$  

Recall that $\qD$ is a formal variable such that $\qD^D = q$. Extending scalars to ${\mathbb C}(\qD)$, any module $V \in \mathcal{C}$ can be regarded as an $U_q^\Pup$-module by stipulating that
\begin{equation}\label{actionLi}
\forall \, \mu \in P, \quad K_\mu \cdot v_\lambda := \qD^{D(\mu,\lambda)} v_\lambda
\end{equation} 
for every weight vector $v_\lambda \in V$ of weight $\lambda\in P$ (recall that for all $\lambda,\mu \in P$, $\textstyle (\mu,\lambda)\in \frac{1}{D}\mz$). A $\mathbb{C}(\qD)$-vector space is a type 1 $U_q^\Pup$-module if $L_i$ acts semisimply with eigenvalues in $\qD^{\mathbb{Z}}$ for all $i$. Formula \eqref{actionLi} then establishes an equivalence of categories between $U_q^\Pup$-modules of type 1 and $U_q^\Q$-modules of type 1 (with scalars extended to $\mc(\qD)$).

\smallskip

\indent The category $\mathcal{C}$ is a braided category by means of Drinfeld's universal $R$-matrix. Namely, the braiding is the natural isomorphism $(\sigma_{V,W} \circ R_{V,W})_{V,W\in Ob(\mathcal{C})}$, where $\sigma_{V,W}\colon v\otimes w \mapsto w\otimes v$, and $R_{V,W}\colon V\otimes W \rightarrow V\otimes W$, $v\otimes w \mapsto R(v\otimes w)$, is defined by $R = \Theta \hat{R}$ with
\begin{equation}\label{thetadef}
\Theta(v\otimes w) = \qD^{D(\mu,\nu)} v\otimes w
\end{equation} on weight vectors $v$, $w$ of weights $\mu$, $\nu$ respectively\footnote{The braiding thus requires an extension of scalars of $\mathcal{C}$ from $\mathbb{C}(q)$ to $\mathbb{C}(\qD)$.} and
\begin{equation}\label{Rhat}
\hat{R} = \hat{R}_{\beta_N} \ldots \hat{R}_{\beta_2} \hat{R}_{\beta_1}, \quad \text{with } \: \hat{R}_{\beta} = \sum_{s=0}^{\infty} q_{\beta}^{s(s-1)/2}\frac{(q_{\beta} - q_{\beta}^{-1})^s}{[s]_{q_{\beta}}!} (E_{\beta})^s \otimes (F_{\beta})^s
\end{equation}
(the order of factors in the product $\hat{R}$ is extremely important) for all positive root $\beta \in \phi^+$, where $q_\beta := q^{(\beta,\beta)/2}$. Since the elements $E_{\beta}$, $F_{\beta}$ act nilpotently on the objects of $\mathcal{C}$, the operator $\hat{R}_{V,W} \in \mathrm{End}_{\mathbb{C}(\qD)}(V \otimes W)$ is well-defined for all $V,W\in \mathrm{Ob}(\mathcal{C})$, and thus $R_{V,W}$ is well-defined as well. Actually $R$ can be regarded as an element in a Hopf algebra which is categorical completion of $U_q^\Pup$, as defined e.g. in \cite[\S 2 and \S 3.2]{BR1}; such a completion is an instance of multiplier Hopf algebra in the sense of \cite{VD1} and \cite{VY}.

We remark that if we set
\[ E_{\beta}^{(s)} := \frac{E_{\beta}^s}{[s]_{q_{\beta}}!}, \quad F_{\beta}^{(s)} := \frac{F_{\beta}^s}{[s]_{q_{\beta}}!}, \quad \underline{E}_{\beta} := (q_\beta - q_\beta^{-1})E_\beta, \quad \underline{F}_{\beta} := (q_\beta - q_\beta^{-1})F_\beta \]
for every $\beta \in \phi^+$ and $s \in \mathbb{N}$, then we can rewrite
\begin{equation}\label{Rhatdiv}
\hat{R}_\beta = \sum_{s=0}^{\infty} q_{\beta}^{s(s-1)/2} (\underline{E}_{\beta})^s \otimes F_{\beta}^{(s)} = \sum_{s=0}^{\infty} q_{\beta}^{s(s-1)/2} E_{\beta}^{(s)} \otimes (\underline{F}_{\beta})^s
\end{equation} 
which is a power series with coefficients in $\mathbb{C}[q^{\pm 1}]$; this will be used in the discussion of integral forms in \S\ref{integralO}.

\subsection{The quantum function algebra \texorpdfstring{${\mathcal O}_q(G)$}{ }}\label{sec:Oq}
The algebra ${\mathcal O}_q:={\mathcal O}_q(G)$ is the subspace of $(U_q^\Q)^{\circ}$ spanned by the matrix coefficients of the objects in the subcategory $\mathcal{C}$ of type 1 modules; this is a Hopf algebra (Rmk.\,\ref{remarkRestDualForC}). It is also a $U_q^\Q$-bimodule by means of the left and right coregular actions $\rhd$ and $\lhd$. We have the Peter--Weyl decomposition as a bimodule, 
\begin{equation}\label{PWdecomp}
{\mathcal O}_q =\bigoplus_{\mu\in P_+}C(\mu) \quad {\rm where }\;\;C(\mu):=V_{\mu}^*\otimes V_\mu,
\end{equation}
where the left action of $U_q^\Q$ on $C(\mu)$  is on the tensorand $V_\mu$, and the right action is its transpose on $V_{\mu}^*$. Once a basis of the modules $V_\mu$ has been fixed for every $\mu\in P_+$, the set of matrix coefficients in these bases is a ${\mathbb C}(q)$-basis of ${\mathcal O}_q$. Extending scalars to ${\mathbb C}(\qD)$ and using \eqref{actionLi}, the $C(\mu)$'s become $U_q^\Pup$-bimodules.

\smallskip

\indent The duality (evaluation) bilinear form $\langle \text{-},\text{-} \rangle^\Q :{\mathcal O}_q\times U_q^\Q\to {\mathbb C}(q)$ is a Hopf pairing, and it is non degenerate.
It is non degenerate on the left by the definition of $\Oo_q$ as a restricted dual, and it is non degenerate on the right because the modules in ${\mathcal C}$ separate the points of $U_q^\Q$ (see, e.g., \cite[Th.\,3.46]{VY}). Extending scalars to $\mc(\qD)$ and thanks to \eqref{actionLi}, we obtain a non-degenerate pairing
\[ \langle \text{-},\text{-} \rangle^\Pup:{\mathcal O}_q\times U_q^\Pup \to {\mathbb C}(\qD). \]
The pairing $\langle \text{-},\text{-} \rangle^\Q$ is a quantum analog of the classical non-degenerate Hopf pairing $\langle \text{-},\text{-} \rangle :{\mathcal O}(G)\times U(\mathfrak g)\to {\mathbb C}$, defined for {\it any finite dimensional} Lie algebra (not necessarily semisimple), which is obviously non degenerate on the left, and is non degenerate on the right because of the Harish-Chandra theorem \cite[Th.\,2.5.7]{Di}.

\smallskip

We put $R^+:=R$, $R^- := (R^{fl})^{-1}$, where $R^{fl}$ is $R$ with tensorands permuted. Although the $R$-matrix does not exist in $U_q^\Pup \otimes U_q^\Pup$, the following maps are well-defined 
\begin{equation}\label{phipmdebut}
\fonc{\Phi^\pm}{\Oo_q}{(U_q^\Pup)^{\rm cop}}{\alpha}{(\alpha \otimes \mathrm{id})(R^\pm).}
\end{equation}
Indeed let $V$ be a finite-dimensional $U_q^\Q$-module with basis of weight vectors $\{ v_1, \ldots, v_s \}$ and dual basis $\{ v^1, \ldots, v^s \}$. Then $\bigl( v^i(? \cdot v_j) \otimes \mathrm{id} \bigr)(\Theta) = \delta_{i,j}K_{\mu_j}$ where $\mu_j$ is the weight of $v_j$ while $\bigl( v^i(? \cdot v_j) \otimes \mathrm{id} \bigr)(\hat{R})$ is an element in $U_q^\Pup$ by nilpotency of the action of the elements $E_{\beta_r}$ on $V$. Hence $\bigl( v^i(? \cdot v_j) \otimes \mathrm{id} \bigr)(R) := \textstyle \sum_k\bigl( v^i(? \cdot v_k) \otimes \mathrm{id} \bigr)(\Theta) \bigl( v^k(? \cdot v_j) \otimes \mathrm{id} \bigr)(\hat{R})$ gives an element in $U_q^\Pup$. The same argument applies for $R^-$.

\indent The maps $\Phi^{\pm}$ are morphisms of Hopf algebras which satisfy the assumptions from \S\ref{subsecSubstQuasi}; we stress that they take values in the simply-connected quantum group $U_q^\Pup$. Actually $\Phi^\pm$ takes values in the subalgebra $U_q^\Pup(\mathfrak{b}_\mp)^{\rm cop}$.

\indent Denote by $\Oo_q(B_\pm)$ the image of $\Oo_q$ by the Hopf epimorphism dual to the inclusion $U_q^\Q(\mathfrak{b}_\pm)\hookrightarrow U_q^\Q(\mathfrak{g})$. The maps $\Phi^\pm$ factor through these projections and yield isomorphisms of Hopf algebras $\Phi^\pm : \Oo_q(B_\pm) \overset{\sim}{\rightarrow} U_q^\Pup(\mathfrak{b}_\mp)^{\mathrm{cop}}$.

\smallskip

\indent A convenient way of manipulating the maps $\Phi^{\pm}$ without using $R$ is through pairings, as follows. Consider the pairing $\rho : U_q^\Pup(\mathfrak{b}_-)^{\rm cop}\times U_q^\Pup(\mathfrak{b}_+) \to \mc(\qD)$ defined on generators by
\begin{equation}\label{defrho}
\rho(K_\lambda, K_\mu) = \qD^{D(\lambda,\mu)}, \quad \rho(K_\lambda,E_i) = \rho(F_i, K_\lambda) = 0, \quad \rho(F_i,E_j) = \frac{\delta_{i,j}}{q_i - q_i^{-1}}
\end{equation}
for every $1 \leq i,j\leq m$, $\lambda,\mu\in P$ and extended by the Hopf pairing property. Note that the map $(x_-,x_+)\mapsto \rho\bigl(x_-,S^{-1}(x_+)\bigr)$ is the convolution inverse of $\rho$ and it holds $\rho \circ (S \otimes S) = \rho$. Also consider the so-called {\em Drinfeld pairing}
\begin{equation}\label{deftau}
\tau:U_q^\Pup(\mathfrak b_+)^{\rm cop} \times  U_q^\Pup(\mathfrak b_-) \to \mc(\qD), \quad \tau(x_+,x_-) = \rho\bigl( S(x_-), x_+ \bigr).
\end{equation}
The pairings $\rho$ and $\tau$ are non degenerate by the classical arguments (see e.g. \cite[Th.\,3.92]{VY}). We record the following formula for $\rho$ on PBW monomials:
\begin{equation}\label{expressionRho}
\rho\left( h_- F_{\beta_N}^{s_N} \ldots F_{\beta_1}^{s_1}, \,h_+ E_{\beta_N}^{t_N} \ldots E_{\beta_1}^{t_1} \right) = \rho(h_-,h_+)\prod_{i=1}^N \delta_{s_i,t_i} \, q_{\beta_i}^{-s_i(s_i - 1)/2}\,\frac{[s_i]_{q_{\beta_i}}!}{(q_{\beta_i} - q_{\beta_i}^{-1})^{s_i}}
\end{equation}
for all $h_-,h_+ \in U_q^\Pup(\mathfrak{h})$; this formula  is given e.g. in \cite[Th.\,2.4]{DC-L} (and modified here to match with our conventions). The following lemma uncovers the relation between $\rho$, $\tau$ and the $R$-matrix:
\begin{lem}\label{lemRhoVsRMat}
For all $\boldsymbol{s} = (s_N,\ldots,s_2,s_1) \in \mathbb{N}^N$ let $X_+^{\boldsymbol{s}} = E_{\beta_N}^{s_N} \ldots E_{\beta_1}^{s_1}$ and $X_-^{\boldsymbol{s}} = F_{\beta_N}^{s_N} \ldots F_{\beta_1}^{s_1}$.
\\1. We have $\widehat{R} = \sum_{\boldsymbol{s}\in \mathbb{N}^N} \frac{X_+^{\boldsymbol{s}} \otimes X_-^{\boldsymbol{s}}}{\rho(X_-^{\boldsymbol{s}}, X_+^{\boldsymbol{s}})}$, where $\widehat{R}$ is defined in \eqref{Rhat}.
\\2. For all $\alpha\in \Oo_q$ and  $x_\pm \in U_q^\Pup(\mathfrak{b}_\pm)$ it holds
\begin{equation}\label{RhoTauQgen}
\rho\bigl( \Phi^+(\alpha),x_+ \bigr) = \langle \alpha,x_+\rangle^\Pup \quad \text{and} \quad \tau\bigl( \Phi^-(\alpha),x_- \bigr) = \langle \alpha,x_-\rangle^\Pup.
\end{equation}
\end{lem}
\noindent When the second variable of $\rho$ and $\tau$ is restricted to $U_q^\Q(\mathfrak b_{\mp})$ then these pairings take values in $\mc(q)$ and the evaluation pairing $\langle\text{-},\text{-}\rangle^\Pup$ can be replaced by $\langle\text{-},\text{-}\rangle^\Q$ in formula \eqref{RhoTauQgen}.
\begin{proof}
1. Note first by comparison of \eqref{Rhat} and \eqref{expressionRho} that for all positive root $\beta \in \phi^+$ we have $\hat{R}_\beta = \textstyle \sum_{s \in \mathbb{N}} \rho(F_\beta^s,E_\beta^s)^{-1} (E_\beta)^s \otimes (F_\beta)^s$. Also by \eqref{expressionRho} we see that $\textstyle \rho(X_-^{\boldsymbol{s}}, X_+^{\boldsymbol{t}}) = \prod_{i=1}^N \rho(F_{\beta_i}^{s_i}, E_{\beta_i}^{t_i})$. Hence by definition of $\hat{R}$ one gets
\[ \hat{R} = \hat{R}_{\beta_N} \ldots \hat{R}_{\beta_1} = \sum_{s_N,\ldots,s_1 \in \mathbb{N}} \prod_{i=N}^1 \frac{(E_{\beta_i})^{s_i} \otimes (F_{\beta_i})^{s_i}}{\rho(F_{\beta_i}^{s_i},E_{\beta_i}^{s_i})} = \sum_{\boldsymbol{s} \in \mathbb{N}^N} \frac{X_+^{\boldsymbol{s}} \otimes X_-^{\boldsymbol{s}}}{\rho(X_-^{\boldsymbol{s}}, X_+^{\boldsymbol{s}})}. \]
2. Recall that $R = \Theta \hat{R}$. For each finite-dimensional $U_q^\Q(\mathfrak{g})$-module, denote by $\boldsymbol{1}_\lambda$ the projection on the subspace of vectors having weight $\lambda \in P$. With this notation, the action of $K_\mu$ on a module can be written as the formal sum $\textstyle \sum_{\lambda \in P} \qD^{D(\lambda,\mu)}\boldsymbol{1}_\lambda$ while the element $\Theta$ defined in \eqref{thetadef} is $\textstyle \sum_{\lambda,\mu \in P} \qD^{D(\lambda,\mu)} \boldsymbol{1}_\lambda \otimes \boldsymbol{1}_\mu$. If $v_\nu$ is a vector with weight $\nu$ we thus have $\textstyle \sum_{\lambda,\mu \in P} \qD^{D(\lambda,\mu)} \boldsymbol{1}_\lambda v_\nu \otimes \boldsymbol{1}_\mu = v_\nu \otimes \sum_{\mu \in P} \qD^{D(\nu,\mu)} \boldsymbol{1}_\mu  = v_\nu \otimes K_\nu$. It follows that $\textstyle \sum_{\lambda,\mu \in P} \qD^{D(\lambda,\mu)} \boldsymbol{1}_\lambda v \otimes \boldsymbol{1}_\mu = \sum_{\lambda \in P} \boldsymbol{1}_\lambda v \otimes K_\lambda$ for any vector $v$. This implies that for all $c \in \mathcal{O}_q$ we have $\textstyle (c \otimes \mathrm{id})(\Theta) = \sum_{\lambda \in P} c(\boldsymbol{1}_\lambda)K_\lambda$. Also observe from \eqref{expressionRho} that $\rho(X_-^{\boldsymbol{s}}, X_+^{\boldsymbol{t}}) = \delta_{\boldsymbol{s},\boldsymbol{t}} \rho(X_-^{\boldsymbol{s}}, X_+^{\boldsymbol{t}})$. We thus get, for all $\boldsymbol{a} \in \mathbb{N}^N$ and $\mu \in P$,
\begin{align*}
&\rho\bigl( \Phi^+(\alpha), K_\mu X_+^{\boldsymbol{a}} \bigr) = \rho\bigl( (\alpha \otimes \mathrm{id})(\Theta\hat{R}), K_\mu X_+^{\boldsymbol{a}} \bigr) = \sum_{\boldsymbol{s},\lambda} \alpha(\boldsymbol{1}_\lambda X_+^{\boldsymbol{s}}) \frac{\rho\bigl( K_\lambda X_-^{\boldsymbol{s}}, K_\mu X_+^{\boldsymbol{a}} \bigr)}{\rho\bigl( X_-^{\boldsymbol{s}}, X_+^{\boldsymbol{s}} \bigr)}\\
=\:& \sum_{\boldsymbol{s},\lambda} \qD^{D(\lambda,\mu)}\alpha(\boldsymbol{1}_\lambda X_+^{\boldsymbol{s}}) \frac{\rho\bigl( X_-^{\boldsymbol{s}}, X_+^{\boldsymbol{a}} \bigr)}{\rho\bigl( X_-^{\boldsymbol{s}}, X_+^{\boldsymbol{s}} \bigr)} = \sum_{\lambda} \qD^{D(\lambda,\mu)}\alpha(\boldsymbol{1}_\lambda X_+^{\boldsymbol{a}}) = \alpha(K_\mu X_+^{\boldsymbol{a}})
\end{align*}
where the second equality uses item 1 of the present lemma and the third equality is by \eqref{expressionRho}. The computation for $\tau$ is similar.
\end{proof}

Let $\mathcal{C}(\qD) = \mathcal{C} \otimes_{\mathbb{C}(q)} \mathbb{C}(\qD)$ and $\mathcal{O}_q(\qD) = \mathcal{O}_q \otimes_{\mathbb{C}(q)} \mathbb{C}(\qD)$ be the extension of scalars. Since the category $\mathcal{C}(\qD)$ is braided by means of the action of the R-matrix, $\mathcal{O}_q(\qD)$ is naturally co-quasitriangular. The co-R-matrix can be expressed in terms of the morphisms $\Phi^+$ as
\begin{equation}\label{coRmatOq}
\forall \, \varphi, \psi \in \mathcal{O}_q(\qD), \quad \mathcal{R}(\varphi \otimes \psi) = \bigl\langle \psi, \Phi^+(\varphi)  \bigr\rangle^\Pup
\end{equation}
which corresponds to the definition in \textsection \ref{subsecSubstQuasi}. By \eqref{RhoTauQgen}, $\mathcal{R}$ relates to the Drinfeld pairing $\tau$ by
\[ \mathcal{R}(\varphi \otimes \psi) = \tau\bigl( \Phi^-(\psi), \Phi^+(\varphi) \bigr). \]
In particular, $\mathcal{R}$ factors through the projection $\Oo_q^{\otimes 2} \to \Oo_q(B_+) \otimes \Oo_q(B_-)$.

\smallskip

\indent Moreover the category $\mathcal{C}(\qD)$ is ribbon, with ribbon natural transformation given by the action of $u \ell^{-1}$, where $u = \textstyle \sum_{(R)} S(R_{(2)})R_{(1)}$ is the Drinfeld element and $\ell = K_{2\rho}$ (although $u \not\in U_q^\Pup$, its action on objects in $\mathcal{C}(\qD)$ is well-defined). This element acts by scalar on any irreducible module; looking at the highest weight vector reveals that this action on $V_\lambda$ is $q^{-(\lambda,\lambda+2\rho)}\mathrm{id}_{V_\lambda}$ for all $\lambda \in P_+$. As a result the algebra $\mathcal{O}_q(\qD)$ is coribbon with $\mathsf{v} : \mathcal{O}_q(\qD) \to \mathbb{C}(\qD)$ given by
\begin{equation}\label{coribbonOq}
\mathsf{v}\bigl( _\lambda\phi^f_x \bigr) = \qD^{-D(\lambda,\lambda+2\rho)}f(x)
\end{equation}
where $_\lambda\phi^f_x \in \mathcal{O}_q$ is the matrix coefficient of $V_\lambda$ given by $h \mapsto f(h\cdot x)$ for $f \in V_\lambda^*$, $x \in V_\lambda$. We have to use $\qD$ in order to ensure an integer exponent, because $\textstyle (\lambda,\lambda) \in \frac{1}{D}\mathbb{Z}$.

\subsection{Integral forms of \texorpdfstring{$U_q^\Lam(\mathfrak{g})$}{the quantum enveloping algebra}}\label{integralFormsUq}
Recall that $A=\mc[q,q^{-1}]$. An integral form of a $\mc(q)$-algebra $R$ is an $A$-subalgebra $R_A$ such that the canonical map $R_A \otimes_{A} \mc(q) \to R$ is an isomorphism. Given a $\mathbb{Z}$-lattice $\Lambda$ such that $Q \subset \Lambda \subset P$ we shall use two integrals versions of $U_q^\Lam(\mathfrak{g})$, respectively called unrestricted and restricted.

\smallskip

\indent The {\em unrestricted integral form} $U_A^\Lam$ of $U_q^\Lam$ is its smallest $A$-subalgebra containing the elements
$$\underline{E}_i:=(q_i-q_i^{-1})E_i\ ,\ \underline{F}_i:=(q_i-q_i^{-1})F_i\ , \ K_\lambda\ \  \mathrm{for}\ i=1,\ldots,m, \:\lambda \in \Lambda$$ and stable under the action of the braid group ${\mathcal B}(\mathfrak{g})$ by means of Lusztig's automorphisms $T_i$. 

We denote by $U_A(\mathfrak{n}_-)$ (resp.  $U_A(\mathfrak{n}_+)$, resp. $U_A^\Lam(\mathfrak{h})$)  the smallest $A$-subalgebra of $U_A^\Lam$ containing the elements $\underline{E}_i$ (resp. $\underline{F}_i$, resp. $K_\lambda$ for $\lambda \in \Lambda$) and stable under the action of ${\mathcal B}(\mathfrak{g})$. We put $U_A^\Lam(\mathfrak{b}_\pm) := U_A(\mathfrak{n}_\pm)U_A^\Lam(\mathfrak{h})$. 

Bases of $U_A(\mathfrak{n}_+)$ and $U_A(\mathfrak{n}_-)$ are respectively formed by the monomials in the elements
\begin{equation}\label{defRootVectorsUnrestricted}
\underline{E}_{\beta}:=(q_\beta-q_\beta^{-1})E_{\beta} \quad \text{and} \quad \underline{F}_{\beta}:=(q_\beta-q_\beta^{-1})F_{\beta_k}
\end{equation}
for $\beta \in \phi_+$, where $q_\beta = q^{(\beta,\beta)/2}$ and $E_{\beta}$, $F_{\beta}$ are the root vectors \eqref{rootvectdef}. Together with the basis $K_\lambda$ ($\lambda\in \Lambda$) of $U_A^\Lam(\mathfrak{h})$, this and the polarisation $U_A^\Lam(\mathfrak{g})=U_A(\mathfrak{n}_-)U_A^\Lam(\mathfrak{h})U_A(\mathfrak{n}_+)$ yield a PBW basis of $U_A^\Lam(\mathfrak{g})$. 
\smallskip

Note that the center $\mathcal Z(U_A^\Lam)$ of $U_A^\Lam(\mathfrak{g})$ satisfies $\mathcal Z(U_A^\Lam)\otimes_A {\mathbb C}(q)=\mathcal Z(U_q^\Lam)$.

\begin{remark}
{\rm Denote by $A'$ the localisation of $A$ with respect to the elements  $q_i-q_i^{-1}, i=1,\ldots,m$. Clearly $U_A^\Lam\otimes_A A'$ is isomorphic to the $A'$-subalgebra of $U_q^\Lam$  generated by $E_i, F_i , K_{\lambda}$ ($i=1,\ldots, m$, $\lambda \in \Lambda$), with no condition involving the braid group action.

There is another unrestricted integral form of $U_q^\Lam$, simpler to define \cite{DCK}: it is the $A$-subalgebra of $U_q^\Lam$ generated by  $E_i, F_i, K_\lambda, [K_i, 0]_{q_i}$ ($i=1,\ldots,m$, $\lambda \in \Lambda$). We will not use it. It is not $A$-isomorphic to $U_A^\Lam$, but it is over $A'$. This last definition of $U_A^\Lam$ is suited to the study of the quotient (commutative) algebra $U_A/(q-1)U_A$, and its Poisson geometry \cite[\S 12.4, \S 14, \S 15]{DCP}.}
\end{remark}

\indent The {\em restricted integral form} $\Gamma_A^\Lam = \Gamma_A^\Lam(\mathfrak{g})$ of $U_q^\Lam$ (in the version of De Concini--Lyubashenko \cite{DC-L}) is the $A$-subalgebra of $U_q^\Lam$ generated by the  elements
$$ E_i^{(p)} = \frac{E_i^p}{[p]_{q_i}!},\ F_i^{(p)} = \frac{F_i^p}{[p]_{q_i}!}\ , \ (K_i,\,t)_{q_i} := \prod_{s=1}^t \frac{K_iq_i^{-s+1}-1}{q_i^s-1}, \quad K_{\lambda}^{\pm 1} $$
for all $i=1,\ldots,m$, $p, t \geq 0$ and $\lambda \in \Lambda$. Denote $\Gamma_A^\Lam({\mathfrak h}):=\Gamma_A^\Lam \cap U_q^\Lam(\mathfrak h)$ and $\Gamma_A^\Lam({\mathfrak b}_{\pm}):=\Gamma_A^\Lam \cap U_q^\Lam(\mathfrak b_{\pm})$.  In particular $\Gamma_A({\mathfrak n}_{\pm}):=\Gamma_A^\Q \cap U_q(\mathfrak n_{\pm})$. For all $\beta \in \phi^+$, introduce the notations
\begin{equation}\label{defDivPowerBeta}
E_\beta^{(p)} = \frac{E_\beta^p}{[p]_{q_\beta}!},\ F_i^{(p)} = \frac{F_\beta^p}{[p]_{q_\beta}!}
\end{equation}
with as usual $\textstyle [p]_{q_\beta}! = \prod_{i=1}^p \frac{q_\beta^i - q_\beta^{-i}}{q_\beta - q_\beta^{-1}}$. An $A$-basis of $\Gamma_A^\Q$ is formed by the elements
\begin{equation}\label{basegamma}
Y({\bf r}, {\bf t}, {\bf s})=E_{\beta_N}^{(r_N)}\ldots E_{\beta_1}^{(r_1)}\left(\prod_{i=1}^m K_i^{-\sigma(t_i)}(K_i; t_i)_{q_i} \right)F_{\beta_N}^{(s_N)}\ldots F_{\beta_1}^{(s_1)}
\end{equation}
where ${\bf r}=(r_1,...,r_N), {\bf s}=(s_1,...,s_N)\in {\mathbb N}^N$, ${\bf t}=(t_1,...,t_m)\in {\mathbb N}^m$ and $\sigma(t)$ is the integer part of $t/2$. Moreover we have an  isomorphism of $A$-modules $\Gamma_A^{\Lam} \cong A[\Lambda/Q] \otimes_A \Gamma_A^\Q$, which gives an $A$-basis of $\Gamma_A^\Lam$ for all $\Lambda$.

\begin{remark}\label{comparaisonRestrictedVersions}
{\rm The restricted integral form often appears in the literature in Lusztig's version, denoted by $U_A^{\Lam,{\rm res}}$ \cite{Lusztig0,Lusztig1}. It is the $A$-subalgebra of $U_q^\Lam$ generated by the  elements $E_i^{(p)}$, $F_i^{(p)}$ and $K_{\lambda}^{\pm 1}$ for $i=1,\ldots,m$, $p\geq 0$ and $\lambda \in \Lambda$. In particular, $U_A^{\Q,{\rm res}}$ is the smallest $A$-subalgebra of $U_q^\Q$ containing the elements $E_i^{(p)}$, $F_i^{(p)}$ and $K_i^{\pm 1}$ for all $i$ and $p$. Also $U_A^{\Q,{\rm res}}$ contains $\textstyle [ K_i; t ]_{q_i}:= \prod_{s=1}^t \frac{K_iq_i^{-s+1}-K_i^{-1}q_i^{s-1}}{q_i^s-q_i^{-s}}$ for all $t \geq 0$ but does not contain $(K_i,t)_{q_i}$. There are thus strict inclusions of algebras $U_A^\Lam \subset U_A^{\Lam,{\rm res}} \subset \Gamma_A^\Lam$. We note that suitably-defined categories of representations (\S \ref{integralO}) of $U_A^{\Q,{\rm res}}$ and $\Gamma_A^\Lam$ are equivalent. We prefer to use $\Gamma_A^\Lam$ because it is from it that De Concini--Lyubashenko defined the integral version of the quantum coordinate algebra $\mathcal{O}_q(G)$ and also because it leads to a version of the small quantum group more suitable for us (see Remark \ref{versionsSmallQG}).}
\end{remark}

\subsection {Integral form of \texorpdfstring{$\Oo_q(G)$}{the quantum function algebra}}\label{integralO}
We denote by $\mathcal{C}_A$ the category of $\Gamma_A^\Q$-modules which are free $A$-modules of finite rank and which have a basis such that for all $i$ the elements $K_i$ and $(K_i,t)_{q_i}$ act diagonally with eigenvalues of the form $q_i^m$ and $(q_i^m,t)_{q_i}$ for some $m \in \mathbb{Z}$, where $\textstyle (x,t)_{q} = \prod_{s=1}^t \frac{xq^{-s+1}-1}{q^s-1}$ for all $t \in \mathbb{N}$ and formal variables $x,q$.

The category $\mathcal{C}_A$ is rigid, monoidal, but not semisimple. Every type $1$ finite dimensional simple $U_q^\Q$-module $V_\mu$, $\mu \in P_+$, has a $\Gamma_A^\Q$-invariant full $A$-sublattice, that we denote by ${}_AV_\mu$ (``full'' means that $_AV_\mu \otimes_A \mathbb{C}(q) = V_{\mu}$). These $\Gamma_A^\Q$-modules form the simple objects of $\mathcal{C}_A$.

If we extend scalars of ${\mathcal C}_A$ to $\AD$, the objects of $\mathcal{C}_A$ can be regarded as $\Gamma_A^\Lam$-modules by means of \eqref{actionLi} for any lattice $Q\subset \Lambda \subset P$.

The Hopf algebra $\Oo_A=\Oo_A(G)$ \cite{Lusztig2} is the $A$-Hopf subalgebra of $\Oo_q$ defined as the $A$-span of the matrix coefficients $h\mapsto v^i(h \cdot v_j)$, for every module $V$ in the category $\mathcal{C}_A$, where $(v_j)$ is an $A$-basis of $V$ with dual $A$-basis $(v^i)$. It is a free $A$-module, and an integral form of $\Oo_q$ \cite{Lusztig2,Lusztig}. So $\Oo_A$ can be regarded as a restricted dual of $\Gamma_A^\Q$ with respect to the category $\mathcal{C}_A$.

 The duality (evaluation) Hopf pairing
\begin{equation}\label{pairingA}\langle \text{-}, \text{-} \rangle_A^\Q :{\mathcal O}_A\times \Gamma_A^\Q \to A
\end{equation}
is non-degenerate because it recovers the usual pairing $\langle \text{-}, \text{-} \rangle^\Q : \mathcal{O}_q \otimes U^\Q_q \to \mathbb{C}(q)$ when scalars are extended from $A$ to ${\mathbb C}(q)$.

\indent Let $\mathcal{O}_{\AD} = \mathcal{O}_A \otimes_A \AD$, $\Gamma_{\AD}^\Pup = \Gamma_A^\Pup \otimes_A \AD$ be the extension of ground ring to $\AD = \mathbb{C}[\qD^{\pm 1}]$, with $\qD^D = q$. Then we see that $\langle \text{-},\text{-} \rangle_A^\Q$ extends to a non degenerate Hopf pairing
\begin{equation}\label{defpairAD}
\langle \text{-}, \text{-} \rangle_{\AD}^\Pup :{\mathcal O}_{\AD}\times \Gamma_{\AD}^\Pup\to A_{\D}.
\end{equation}
which is the integral version of the pairing $\langle \text{-},\text{-} \rangle^\Pup$ from \S \ref{sec:Oq}.

Observe that the expression of the pairing $\rho : U_q(\mathfrak{b}_-)^{\mathrm{cop}} \times U_q(\mathfrak{b}_+) \to \mathbb{C}(\qD)$ in \eqref{expressionRho} can be rewritten as
\[ \rho\left( h_- \underline{F}_{\beta_N}^{s_N} \ldots \underline{F}_{\beta_1}^{s_1}, \,h_+ E_{\beta_N}^{(t_N)} \ldots E_{\beta_1}^{(t_1)} \right) = \rho(h_-,h_+)\prod_{i=1}^N \delta_{s_i,t_i} \, q_{\beta_i}^{-s_i(s_i - 1)/2} \]
with notations from \eqref{defRootVectorsUnrestricted} and \eqref{defDivPowerBeta}. Clearly $\rho(h_-,h_+) \in \AD$ if $h_- \in U_A^\Pup(\mathfrak{h})$ and $h_+ \in \Gamma_A^\Pup(\mathfrak{h})$. It follows that the pairing $\rho$, and hence $\tau$ as well, take values in $A$ if we restrict them as follows:
\begin{equation}\label{tauA}
\rho_A\colon U_A^\Pup(\mathfrak{b}_-)^{\rm cop}\times \Gamma_A^\Q(\mathfrak{b}_+)\to A, \qquad \tau_A\colon U_A^\Pup(\mathfrak{b}_+)^{\rm cop}\times \Gamma_A^\Q(\mathfrak{b}_-)\to A.
\end{equation}
These integral pairings are of course still non-degenerate, as restrictions of the non-degenerate pairings $\rho$ and $\tau$. This shows that the restricted integral version $\Gamma_A^\Q$ and the unrestricted integral version $U_A^\Pup$ have triangular factors in duality.

Let $\Phi_A^{\pm}$ be defined by restriction to $\Oo_A \subset \mathcal{O}_q$ of $\Phi^{\pm} : \mathcal{O}_q \to (U_q^\Pup)^{\mathrm{cop}}$ from \eqref{phipmdebut}. By \eqref{Rhat} and \eqref{Rhatdiv} it defines a map $\Phi_A^{\pm} : \mathcal{O}_A \to (U_A^\Pup)^{\mathrm{cop}}$. By restriction to integral forms, the equalities in \eqref{RhoTauQgen} become
\begin{equation}\label{identPhibracket}
\rho_A\bigl(\Phi_A^+(\alpha),x_+\bigr) = \langle \alpha,x_+\rangle_A^\Q,\qquad \tau_A\bigl(\Phi_A^-(\alpha),x_-\bigr) = \langle \alpha,x_-\rangle_A^\Q
\end{equation}
for every $\alpha\in \Oo_A$, $x_\pm\in \Gamma_A^\Q(\mathfrak{b}_\pm)$. As a result $\Phi_A^{\pm}$ factors through the projection $\mathcal{O}_A \twoheadrightarrow \mathcal{O}_A(B_\pm)$ and yields an isomorphism of Hopf algebras
\begin{equation}\label{Phi+-A}
\Phi_A^\pm\colon \Oo_A(B_\pm) \overset{\sim}{\longrightarrow} U_A^\Pup(\mathfrak{b}_\mp)^{\rm cop}.
\end{equation}

\indent Recall the co-R-matrices $\mathcal{R}^{(\pm)}$ from \eqref{RLLdeguise}. They and the coribbon element on $\mathcal{O}_q$ described in \S\ref{sec:Oq} admit integral versions:
\begin{lem}\label{lemIntegralRandCorib}
1. The restriction of the co-R-matrix $\mathcal{R}^{(\pm)} : \mathcal{O}_q(\qD)^{\otimes 2} \to \mathbb{C}(\qD)$ to $\mathcal{O}_{\AD}^{\otimes 2}$ takes values in $\AD$. Hence we have an integral version $\mathcal{R}_{\AD}^{(\pm)}  : \mathcal{O}_{\AD} \otimes \mathcal{O}_{\AD} \to \AD$. It satisfies
\begin{equation}\label{integralCoRmat}
\forall \, \varphi, \psi \in \Oo_{\AD}, \quad \mathcal{R}_{\AD}^{(\pm)} (\varphi \otimes \psi) = \bigl\langle  \psi, \widetilde{i}\bigl(\Phi^\pm_A(\varphi) \bigr) \bigr\rangle_{\AD}^\Pup
\end{equation}
where $\widetilde{i} :  U_A^\Pup \hookrightarrow \Gamma_A^\Pup$ is the natural embedding $\underline{E}_i \mapsto (q_i-q_i^{-1})E_i^{(1)}$, $\underline{F}_i \mapsto (q_i-q_i^{-1})F_i^{(1)}$, $K_\lambda \mapsto K_\lambda$.
\\2. The restriction of the coribbon element $\mathsf{v} : \mathcal{O}_q(\qD) \to \mathbb{C}(\qD)$ to $\mathcal{O}_{\AD}$ takes values in $\AD$. Hence we have an integral version $\mathsf{v}_{\AD} : \mathcal{O}_{\AD} \to \AD$.
\end{lem}
\begin{proof}
For all $\lambda \in P_+$, there is a canonical $A$-basis $\mathcal{B}_\lambda$ of $_AV_\lambda$ due to Kashiwara--Lusztig, see e.g. \cite[Chap.\,14]{CP}.
\\1. It is known that the matrix entries in the basis $\mathcal{B}_\lambda \otimes \mathcal{B}_\mu$ of the action of $R^{\pm 1}$ on ${}_AV_\lambda \otimes {}_AV_\mu$ belong to $q^{\pm (\lambda,\mu)}A \subset \AD$ \cite[Th.\,32.1.5]{Lusztig}. Moreover $\mathcal{O}_A$ is generated as an $A$-algebra by the collection of matrix coefficients of $_AV_\lambda$ in the basis $\mathcal{B}_\lambda$ for all $\lambda \in P_+$. It follows that $\mathcal{R}$ takes values in $\AD$ on such matrix coefficients and the co-R-matrix axioms \eqref{coTriang1}--\eqref{coTriang2} together with the fact that $\mathcal{O}_A$ is a Hopf $A$-algebra imply that $\mathcal{R}$ takes values in $\AD$ for all elements of $\mathcal{O}_A^{\otimes 2}$. Formula \eqref{integralCoRmat} is simply obtained by restriction of formula \eqref{coRmatOq} to integral versions.
\\2. By formula \eqref{coribbonOq} for $\mathsf{v} : \mathcal{O}_q \to \mathbb{C}(\qD)$ we see that $\mathsf{v}$ takes values in $\AD$ when applied to matrix coefficients in $\mathcal{B}_\lambda$. One then conclude as in the previous item thanks to the axiom \eqref{axiomCoribbonProduct} of $\mathsf{v}$.
\end{proof}

\subsection{Specialization of \texorpdfstring{$U_q(\mathfrak{g})$}{the quantum enveloping algebra}}\label{specializUq} Recall the notations $A = \mathbb{C}[q^{\pm 1}]$ and $\AD = \mathbb{C}[\qD^{\pm 1}]$ with $\qD^D = q$, where $D$ is the smallest integer such that $DP \subset Q$, see \S\ref{subsecLieAlg}. General facts on specialization at $q=z$ were remembered in \S\ref{generalRmkSpe}. Here we will work with specializations at a root of unity. Recall from the introduction that we use the following assumptions:
\medskip

\noindent \fbox{\textbf{Assumption:}} From now on, $l \geq 3$ is an integer such that
\begin{equation}\label{assumptionl}
\mathrm{gcd}(l,2) = 1, \qquad \mathrm{gcd}(l,D) = 1, \quad \text{and also } \mathrm{gcd}(l,3) = 1 \text{ if } \mathfrak{g} \text{ is of type } G_2.
\end{equation}
We denote by $\epsilon \in \mathbb{C}$ a primitive root of unity or order $l$. Let $\overline{D}$ be the inverse of $D$ modulo $l$, meaning that $D\overline{D} \equiv 1 \pmod{l}$; it exists because $l$ and $D$ are coprime. We define
\begin{equation}\label{choixED}
\eD = \epsilon^{\overline{D}} \qquad \text{so that } \:\: \eD^D = \e \:\: \text{ and } \:\: \eD^l = 1.
\end{equation}
The specialization of an object defined over $\AD = \mathbb{C}[\qD^{\pm 1}]$ is always done with this $\eD$; since $\eD$ depends on $\e$, the specialization of $\AD$-modules and of $\AD$-(bi)linear maps is denoted with a subscript $\e$ (instead of $\eD$). The only exception is  App.\,\ref{appDlRoot}, where we discuss the case of general $Dl$-th roots of unity.

\medskip

Let $U_\e^\Lam$ and $\Gamma_\e^\Lam$ be respectively the unrestricted and restricted specializations, of the integral forms defined in \S\ref{integralFormsUq}. By Lemma \ref{lemmaSpecialisation}, the PBW basis of $U_A^\Lam$ yields a basis of $U_\e^\Lam$,  and the basis (\ref{basegamma}) yields a basis of $\Gamma_\e^\Q$, with elements we shall denote by $Y({\bf p}, {\bf t}, {\bf n})_\e$. When $q=1$ there is an isomorphism $\Gamma_1^\Q\cong U(\mathfrak{g})$ \cite[\S 6]{DC-L}.

In $\Gamma_\epsilon^\Q$ we have the relations
\begin{equation}\label{toto}\forall \beta\in \phi^+, \:\: E_{\beta}^l=F_{\beta}^l=0 \quad \text{and} \quad K_i^{l}=1 \:\:\text{for}\; i=1,\ldots,m.
\end{equation}
The first two are obvious, and the last one comes by clearing denominators in the definition of the elements $(K_i;t)_{q_i}$, and then evaluating at $q_i=\e_i$. Therefore the specialization $i_\e:U_\e^\Q\rightarrow \Gamma_{\e}^\Q$ of the inclusion map $i:U_A^\Q\hookrightarrow \Gamma_{A}^\Q$ is not injective. This degeneracy (i.e. an injection acquiring a big kernel at $\e$) will be of fundamental importance in the sequel.

\begin{defi}\label{defSmallUeps}
The image of $i_\e\colon U_\e^\Q\rightarrow \Gamma_{\e}^\Q$ is a finite-dimensional Hopf algebra called {\it small quantum group} and denoted $u_\e^\Q(\mathfrak{g})$ or more shortly $u_\e^\Q$.
\end{defi}
\noindent This algebra is considered in \cite[App]{Lyub95}. By definition, $u_\e^\Q$ is the Hopf subalgebra of $\Gamma_\e^\Q$ generated by the elements $E_i$, $F_i$ and $K_i$ for $1 \leq i \leq m$.

\begin{remark}\label{versionsSmallQG}
\indent {\rm In $U_{\epsilon}^{\Q, {\rm res}}$ the elements $K_i^l$ are central (but not equal to $1$) and it holds $K_i^{2l} = 1$. The image of $U_\epsilon^\Q \rightarrow U_{\epsilon}^{\Q, {\rm res}}$ is thus another finite-dimensional Hopf algebra, denoted by $U_\e^{\rm fin}$ and introduced by Lusztig (\cite{Lusztig1,Lusztig2}, \cite[Chap.\,36]{Lusztig}; see also \cite[\S 9.3.B]{CP}). We will not use it.} 
\end{remark}

\indent A PBW basis of $u_\epsilon^\Q$ is formed by the elements
\begin{equation}\label{baseue}
E_{\beta_N}^{s_N}\cdots E_{\beta_1}^{s_1} K_1^{r_1}\cdots K_m^{r_m}F_{\beta_N}^{t_N}\cdots F_{\beta_1}^{t_1}
\end{equation}
where $0\leq t_j, s_j, r_i< l$ for all $i,j$; this fact is an adaptation of \cite[Th.\,8.3]{Lusztig2} to $\Gamma_\e^\Q$. Therefore
\begin{equation}\label{dlambdae}
\dim(u_\epsilon^\Q) = l^{\dim {\mathfrak g}}.
\end{equation}
Comparing the PBW basis of  $u_\e$ and $U_\e^\Q$ (see \S\ref{integralFormsUq}) we see that $u_\e^\Q  \cong U_{\epsilon}^\Q\big/[{\rm relations}\ \eqref{toto}]$. Following \cite{DCK,DCP}, for every sublattice $Q \subset \Lambda  \subset P$ introduce the central subalgebra
\begin{equation}\label{Z0Uedefsep25}
\mathcal{Z}_0(U_\e^\Lam) := \mathbb{C}\bigl[ E_\beta^l, \:F_\beta^l, \:K_\lambda^l \,\big| \, \beta \in \phi_+, \: \lambda \in \Lambda \bigr] \subset \mathcal{Z}(U_\e^\Lam).
\end{equation}
Put $\mathcal{Z}_0(U_\e^\Lam(\mathfrak{b}_\pm)) = \mathcal{Z}_0(U_\e^\Lam) \cap U_\e^\Lam(\mathfrak{b}_\pm)$, $\mathcal{Z}_0^+(U_\e^\Lam(\mathfrak{b}_\pm)) = \mathcal{Z}_0(U_\e^\Lam(\mathfrak{b}_\pm)) \cap {\rm Ker}(\varepsilon)$. Then the relations \eqref{toto} are rewritten as $u_\e^\Q \cong U_\e^\Q/\mathcal{Z}_0^+(U_\e^\Q)U_\e^\Q$ where  $\mathcal{Z}_0^+(U_\e^\Q):= \ker (\varepsilon) \cap \mathcal{Z}_0(U_\e^\Q)$. Since $\mathcal{Z}_0^+(U_\e^\Q)U_\e^\Q$ is a Hopf ideal, the surjection
\begin{equation}\label{defqmaps}
p\colon U_\e^\Q \twoheadrightarrow U_\e^\Q/\mathcal{Z}_0^+(U_\e^\Q)U_\e^\Q \overset{\sim}{\longrightarrow} u_\e^\Q
\end{equation}
 is a morphism of Hopf algebras. By definition we have 
\begin{equation}\label{injectionBecomesProjection}
\forall \, h \in U_A^\Q, \quad i(h)_{|\e} = i_\e(h_{\vert \e}) = p(h_{|\e}).
\end{equation}

\subsection{The quantum Frobenius morphism}\label{subsecDefFrobGamma} The morphism $\mathbb{F}\mathrm{r}_\e$ defined below is the version due to De Concini-Lyubashenko of Lusztig's Frobenius morphism $U_\e^{\Q,\mathrm{res}} \to U(\mathfrak{g})$, introduced in \cite{Lusztig1}; see Remarks \ref{comparaisonRestrictedVersions} and \ref{versionsSmallQG} for a comparison between $\Gamma_\e^\Q$ and $U_\e^{\Q,\mathrm{res}}$. 

The quantum Frobenius morphism is the epimorphism of Hopf algebras
\[ \mathbb{F}\mathrm{r}_\e :\Gamma_\e^\Q \to U(\mathfrak{g}) \]
given by \cite[Th.\,6.3]{DC-L}
\begin{equation}\label{etadef}
\begin{array}{l} \mathbb{F}\mathrm{r}_\e (E_i^{(p)}) = \left\lbrace\begin{array}{cl} \frac{e_i^{p/l}}{(p/l)!} & {\rm if}\ l \ {\rm divides}\ p\\ 0 & {\rm otherwise}\end{array}\right. \ ,\ \mathbb{F}\mathrm{r}_\e (F_i^{(p)}) = \left\lbrace\begin{array}{cl} \frac{f_i^{p/l}}{(p/l)!} & {\rm if}\ l \ {\rm divides}\ p\\ 0 & {\rm otherwise}\end{array}\right.\ \\ \\ \mathbb{F}\mathrm{r}_\e\bigl( (K_i;p)_{\e_i} \bigr) = \left\lbrace\begin{array}{cl} \frac{h_i(h_i-1)\ldots (h_i-(p/l)+1)}{(p/l)!} & {\rm if}\ l \ {\rm divides}\ p\\ 0 & {\rm otherwise}\end{array}\right.
\end{array}
\end{equation}
where $p\in \mn$, and $e_i$, $f_i$, $h_i$ are the standard Serre generators of $U(\mathfrak{g})$. We have $\mathbb{F}\mathrm{r}_\e(K_i^{\pm 1})=1$. $\mathbb{F}\mathrm{r}_\e$ commutes with the braid group action, so the formulas in the first line of \eqref{etadef} still hold true by replacing $E_i^{(p)}$, $F_i^{(p)}$ (resp. $e_i,f_i$) with the root vectors $E_\beta^{(p)}$, $F_\beta^{(p)}$ (resp. $e_\beta, f_\beta$) for all $\beta\in \phi^+$.

We are going to recall a number of properties of $\mathbb{F}\mathrm{r}_\e$ that will be crucial in the sequel. First note that in $\Gamma_\e^\Q$ we have 
\begin{equation}\label{relEnov25}
\begin{array}{c} (K_j; t+l)_{\e_j} = (K_j; t)_{\e_j}(K_j; l)_{\e_j} \text{for all } t\in \mn,\\[.3em]
\text{and } (K_j; t)_{\e_j} \text{ is a  polynomial of degree } t \text{ in } K_j \text{ if } 0\leq t<l.
\end{array}
\end{equation}
Also, for all $0\leq r_0<l$, $r_1\in \mn$, and $\beta\in \phi^+$ we have
\begin{equation}\label{relE2nov25}
E_\beta^{(r_0+lr_1)} = E_\beta^{(r_0)}\frac{(E_\beta^{(l)})^{r_1}}{r_1!}\ ,\ E_\beta^{(r_0)}\in \mc E_\beta^{r_0},
\end{equation}
and similarly by replacing $E_\beta$ with $F_\beta$ (see \cite{Lusztig0} and, e.g., \cite[Rk. 1 on page 300, and top of page 302]{CP}).

Now recall the basis elements of $\Gamma_A^\Q$ in \eqref{basegamma}. Consider the free $A$-submodules $J_l,\Gamma_l\subset \Gamma_A^\Q$ defined by
\begin{align} J_l := & A\left\lbrace\prod_{i=N}^1 E_{\beta_i}^{(p_i)}\left(\prod_{j=1}^m K_j^{-\sigma(t_j)}(K_j; t_j)_{q_j} \right)\prod_{k=N}^1 F_{\beta_k}^{(n_k)},\ \exists p_i,t_j,n_k \not\equiv 0 \!\!\pmod{l}\right\rbrace \label{Jldef}\\ \Gamma_l := & A\left\lbrace\prod_{i=N}^1 E_{\beta_i}^{(p_i)}\left(\prod_{j=1}^m K_j^{-\sigma(t_j)}(K_j; t_j)_{q_j} \right)\prod_{k=N}^1 F_{\beta_k}^{(n_k)},\ \mathrm{all}\ p_i,t_j,n_k \equiv 0 \!\!\pmod{l}\right\rbrace.\label{Gammaldef}\end{align}
We have $\Gamma_A^\Q = J_l \oplus \Gamma_l$ as a free $A$-module.

Specializing to $q=\e$, the observations \eqref{relEnov25}-\eqref{relE2nov25} imply that the vector space $(\Gamma_l)_\e$ has a basis given by the elements $\textstyle (\prod_{i=N}^1 (E_{\beta_i}^{(l)})^{p_i'})M(\prod_{k=N}^1 (F_{\beta_k}^{(l)})^{n_k'})$, where $p_i',n_k'\in \mn$ and $M$ is a polynomial in the elements $(K_j; l)_{\e_j}$, $j\in \{1,\ldots,m\}$. (Beware that $(\Gamma_l)_\e$ is not a subalgebra of $\Gamma_\e^\Q$ - this is immediate from the commutation relations between elements $E_{\beta}^{(l)}$ and $F_{\beta}^{(l)}$, $\beta\in \phi^+$ (see, e.g., \cite[Prop.\,3.23]{VY})).

The observations \eqref{relEnov25}-\eqref{relE2nov25} imply also that $\Gamma_\e^\Q$ is generated as an algebra by the elements $E_i,F_i, E_i^{(l)}, F_i^{(l)},K_i$ and $(K_i; l)_{\e_i}$, $i\in \{1,\ldots,m\}$. Therefore\begin{equation}\label{genGammae}\Gamma_\e^\Q \mathrm{\ is\ generated\ by\ } u_\e \mathrm{\ and\ } (\Gamma_l)_\e.\end{equation}
Moreover we have \cite[Th. 6.3 and Prop 6.5 (i)]{DC-L}: 
\begin{itemize}
\item[(i)] the map $\mathbb{F}\mathrm{r}_\e\colon (\Gamma_l)_\e \to U(\mathfrak{g})$ is bijective (and hence it is a linear isomorphism).
\item[(ii)] 
the Hopf ideal $\ker(\mathbb{F}\mathrm{r}_\e)$ is the space $(J_l)_\e$, it is generated as an ideal by the elements $E_i$, $F_i$ and $K_i-1$ ($i\in \{1,\ldots,m\}$), and it has a $\mc$-basis formed by the monomials
\begin{equation}\label{baseIidealoct24}
\left(\prod_{i=N}^1 E_{\beta_i}^{(p_i)}\right) M\left(\prod_{k=N}^1 F_{\beta_k}^{(n_k)}\right)
\end{equation}
such that there are divided powers $p_i, n_k$ which are not divisible by $l$, or $M$ is in the ideal of $\Gamma_\e^\Q(\mathfrak{h})$ generated by the elements $K_i-1$, $i\in \{1,\ldots,m\}$.
\end{itemize}
\begin{remark}\label{rkbaseuedec25}{\rm Because $S$ is an isomorphism of $\Gamma_\e^\Q(\mathfrak{b}_+)$, $\ker(\mathbb{F}\mathrm{r}_\e)$ has also a $\mc$-basis formed by the monomials $\textstyle S\!\left(\prod_{k=N}^1 E_{\beta_k}^{(n_k)}\right)M \left(\prod_{i=N}^1 F_{\beta_i}^{(p_i)}\right)$, with the same assumptions on the integers $p_i$, $n_k$ and on $M$ as in (ii) above.}
\end{remark}

It follows from \eqref{baseue} and \eqref{baseIidealoct24} that
\begin{equation}
\ker(\mathbb{F}\mathrm{r}_\e) = \text{two-sided ideal in } \Gamma_\e^\Q \text{ generated by } \ker(\varepsilon) \cap u_\e^\Q.\label{kerFrbasic}
\end{equation}
From this, the small quantum group $u_\epsilon^\Q$ is often called the {\it Frobenius-Lusztig kernel}. In particular the composition of maps   
\begin{equation}\label{seqmaps}
\xymatrix{ & u_\e^\Q\ar@{^{(}->}[r]^{{\rm incl.\ }} & \Gamma_\e^\Q \ar@{->>}[r]^{\mathbb{F}\mathrm{r}_\e\ } & U(\mathfrak{g})} 
\end{equation}
sends $h \in u_\e^\Q$ to $\varepsilon(h)1_{U(\mathfrak{g})}$.

\subsection {Specialisation of \texorpdfstring{${\mathcal O}_q(G)$}{the quantum function algebra}}\label{specialisOA} 

By specializing $q$ to $\e$, define ${\mathcal O}_\epsilon={\mathcal O}_A\otimes_A {\mathbb C}_{\epsilon}$. Since $\Oo_A$ is generated as an algebra by matrix coefficients of the modules ${}_A V_{\varpi_i}$ ($i=1,\ldots,m$) \cite{Lusztig2}, $\Oo_\e$ is generated as an algebra by matrix coefficients of the modules $V_{\varpi_i,\e}:={_AV_{\varpi_i}} \otimes_A \mathbb{C}_\e$ ($i=1,\ldots,m$). 

The Hopf pairing \eqref{pairingA} specializes to the Hopf pairing 
\begin{equation}\label{nondegpaireps}
\langle \text{-}, \text{-} \rangle_\epsilon^\Q :{\mathcal O}_\epsilon\times \Gamma_\epsilon^\Q\to {\mathbb C} \quad \text{such that } \forall \, \varphi \in \Oo_A, \:\: \forall \, x \in \Gamma_A^\Q, \:\: \langle \varphi_{|\e}, x_{|\e} \rangle_\e^\Q = \bigl( \langle \varphi, x \rangle^\Q_A \bigr)_{|\e}.
\end{equation}

\indent By specializing to $q=\e$ the pairings \eqref{tauA} one obtains non degenerate pairings (see \cite[\S 6]{DC-L}, and \cite[Th.\,2.20 and 2.29]{BR2} in the present notations)
\begin{equation}\label{taue}
\tau_\e\colon U_\e^\Pup(\mathfrak{b}_+)^{\rm cop}\times \Gamma_\e^\Q(\mathfrak{b}_-)\to \mc, \qquad \rho_\e\colon U_\e^\Pup(\mathfrak{b}_-)^{\rm cop}\times \Gamma_\e^\Q(\mathfrak{b}_+)\to \mc.
\end{equation}
By the observation \eqref{speSurjAndIso}, the isomorphisms \eqref{Phi+-A} specialize to isomorphisms 
\begin{equation}\label{Phi+-e}
\Phi_\e^\pm\colon \Oo_\e(B_\pm) \to U_\e^\Pup(\mathfrak{b}_\mp)^{\rm cop}
\end{equation}
satisfying, for every $\alpha\in \Oo_\e$ and $x_\pm\in \Gamma_\e^\Q(\mathfrak{b}_\pm)$,
\begin{equation}\label{identPhibrackete}
\rho_\e\bigl(\Phi_\e^+(\alpha),x_+\bigr) = \langle \alpha,x_+\rangle_\e^\Q, \qquad  \tau_\e\bigl(\Phi_\e^-(\alpha),x_-\bigr) = \langle \alpha,x_-\rangle_\e^\Q.
\end{equation}

\indent Recall from \eqref{choixED} the $l$-root of unity $\eD$ such that $\eD^D = \e$, and to which $\qD$ is specialized. 

\begin{lem}\label{nondegpairepsD}
The pairing $\langle \text{-},\text{-} \rangle_{\AD}^\Pup:{\mathcal O}_{\AD}\times \Gamma_{\AD}^\Pup\to A_\D$ from \eqref{defpairAD} specializes at $\qD = \eD$ to a non-degenerate Hopf pairing $\langle \text{-},\text{-} \rangle_{\e}^\Pup:{\mathcal O}_\epsilon\times \Gamma_\epsilon^\Pup\to {\mathbb C}$.
\end{lem}
\begin{proof}
It is non degenerate on the left because if $\varphi\in {\mathcal O}_\epsilon$ vanishes on $\Gamma_\epsilon^\Q$, then it is the zero linear form (by definition of $\Oo_\e$). Our proof that it is non degenerate on the right uses Lusztig's modified quantum group and is thus postponed to Appendix \ref{appLusztigModif}.
\end{proof}

Consider the dual map $\mathbb{F}\mathrm{r}_\e^{*} :U(\mathfrak{g})^{*}\to (\Gamma_\e^\Q)^{* }$ of the quantum Frobenius morphism $\mathbb{F}\mathrm{r}_\e$. By restriction to $U(\mathfrak{g})^\circ$ it becomes a monomorphism of Hopf algebras $U(\mathfrak{g})^\circ = \mathcal{O}(G) \to (\Gamma_\e^\Q)^\circ$. Actually it holds $\mathbb{F}\mathrm{r}_\e^{*}(U(\mathfrak{g})^{\circ})\subset \Oo_\e = ((\Gamma_A^\Q)^\circ)_\e \subset  (\Gamma_\e^\Q)^{\circ}$, because for every $\lambda\in P^+$ the $\Gamma_\e^\Q$-module $V_{l\lambda,\e}:= {}_AV_{l\lambda} \otimes_A \mathbb{C}_\e$ is isomorphic to the pull-back by $\mathbb{F}\mathrm{r}_\e$ of the irreducible $U(\mathfrak{g})$-module $L(\lambda)$ with highest weight $\lambda$ (see \cite[Prop.\,6.4]{DC-L}, \cite[Prop.\,11.2.11]{CP} or \cite[\S III.7.2]{BG}). This leads to the following definition:
\begin{defi}
We denote
\begin{equation}\label{Z0def}
\mathcal{Z}_0(\mathcal{O}_\e):=\mathbb{F}\mathrm{r}_\e^{*}\bigl({\mathcal O}(G)\bigr),
\end{equation}
which is a Hopf subalgebra of $\mathcal{O}_\e(G)$.
\end{defi}
\noindent Here is another characterization of this subspace:
\begin{lem}\label{lemAutreDefZ0}
$\mathcal{Z}_0(\mathcal{O}_\e) \subset \mathcal{O}_\e$ is the orthogonal of $\ker(\mathbb{F}\mathrm{r}_\e) \subset \Gamma_\e^\Q$ with respect to the non-degenerate pairing $\langle \text{-},\text{-} \rangle^\Q_\e$. Said differently, for all $\varphi \in \mathcal{O}_\e$ we have
\[ \varphi \in \mathcal{Z}_0(\mathcal{O}_\e) \: \iff \: \bigl\langle \varphi,\ker(\mathbb{F}\mathrm{r}_\e) \bigr\rangle^\Q_\e = 0. \]
\end{lem}
\begin{proof} If $\varphi \in \mathcal{Z}_0(\mathcal{O}_\e)$ then by definition in \eqref{Z0def} there exists $\psi \in \mathcal{O}(G)$ such that $\varphi = \psi \circ \mathbb{F}\mathrm{r}_\e$ and hence $\varphi$ vanishes on $\ker(\mathbb{F}\mathrm{r}_\e)$. Conversely, assume that $\varphi \in \mathcal{O}_\e$ vanishes on $\ker(\mathbb{F}\mathrm{r}_\e)$. Since $\mathbb{F}\mathrm{r}_\e$ is surjective we can define a linear form $\psi \in U(\mathfrak{g})^*$ by $\psi\bigl( \mathbb{F}\mathrm{r}_\e(h) \bigr) = \varphi(h)$ for all $h \in \Gamma_\e^\Q$. We have to show that such a $\psi$ belongs to $\Oo(G)$. Denote by ${\mathcal C}_{u_\e^\Q}$ and ${\mathcal C}_\e$ the categories of finite dimensional $u_\e^\Q$-modules and $\Gamma_\e^\Q$-modules (they are weight modules in the sense of \cite[\S 11.2]{CP}, and modules of type $1$ since $K_i^l = 1$ in $\Gamma_\e^\Q$). Denote also by ${\mathcal C}_G$ the category of finite dimensional $G$-modules (or equivalently $U(\mathfrak{g})$-modules). The sequence \eqref{seqmaps} yields a sequence of functors
\begin{equation}\label{seqfunctorL}
\xymatrix{\mathcal{C}_{G} \ar[r]^{({\mathbb F}r_\e)^\sharp} & \mathcal{C}_{\e} \ar[r]^{{\rm res}^\sharp\ } & \mathcal{C}_{u_\e^\Q}}
\end{equation}
where ${\rm {\rm res}}^\sharp$ is the restriction of modules along the inclusion $u_\e^\Q\subset \Gamma_\e^\Q$, and $({\mathbb F}r_\e)^\sharp$ pulls-back modules along ${\mathbb F}r_\e$ (so, given $M$ in $\mathcal{C}_{G}$ the action of $\Gamma_\e^\Q$ on ${\mathbb F}r_\e^\sharp(M)$ is given by the formulas \eqref{etadef}). As observed in \cite[\S 3.11]{AP}, the sequence \eqref{seqfunctorL} is exact in the middle. That is, for any module $V$ in $\mathcal{C}_{G}$ the $u_\e^\Q$-module $({\rm {\rm res}}^\sharp \circ {\mathbb F}r_\e^\sharp)(V)$ is trivial (i.e., with action given by the counit), and for any $M$ in $\mathcal{C}_{\e}$ such that ${\rm res}^\sharp(M)$ is a trivial $u_\e^\Q$-module, there is a module $V$ in $\mathcal{C}_{G}$ such that $M={\mathbb F}r_\e^\sharp(V)$. The two claims follow from the linear isomorphism $\mathbb{F}\mathrm{r}_\e\colon (\Gamma_l)_\e \to U(\mathfrak{g})$ (see (i) above \eqref{baseIidealoct24}), together with \eqref{kerFrbasic} and \eqref{seqmaps} (see \cite[Prop.\,1.5 (2)]{ArGa} for some details). To prove the lemma, take in particular the modules in $\mathcal{C}_{\e}$ induced by specialization from modules in $\mathcal{C}_{A}$. Consider the elements of $\mathcal{O}_\e$, i.e. the span of matrix coefficients of such modules, vanishing on $\ker(\mathbb{F}\mathrm{r}_\e)$.  They define through ${\rm res}^\sharp$ linear combinations of matrix coefficients of trivial $u_\e^\Q$-modules, so they belong to $\mathrm{im}(\mathbb{F}\mathrm{r}_\e^*) \cong \mathcal{O}(G)$ by exactness of \eqref{seqfunctorL}. It follows that $\psi \in \mathcal{O}(G)$.
\end{proof}
\begin{remark}{\rm By the isomorphism $V_{l\lambda,\e}\cong ({\mathbb F}\mathrm{r}_\e)^\sharp(L(\lambda))$ discussed above \eqref{Z0def}, the functor $({\mathbb F}\mathrm{r}_\e)^\sharp$ is fully faithful. }\end{remark}

The following important theorem collects results from \cite{DC-L, Enr, BGS} (for the correspondence with the present notations, see \cite[Th.\,2.29]{BR2}).

\begin{teo}  \label{DCLteo1} (1) $\mathcal{Z}_0(\Oo_\e)$ is a central Hopf subalgebra of $\Oo_\e $, generated by certain matrix coefficients of simple $\Gamma_\e^\Q$-modules of highest weights $l\lambda$, $\lambda\in P_+$, and $Q\bigl(\mathcal{Z}(\Oo_\e)\bigr)$ is an extension of $Q\bigl(\mathcal{Z}_0(\Oo_\e)\bigr)$ of degree $l^m$.

\noindent (2) $\Oo_\e$ is a domain, and it is a free $\mathcal{Z}_0(\Oo_\e)$-module of rank $l^{dim \mathfrak{g}}$. Moreover the central localization $Q(\Oo_\e) := Q\bigl(\mathcal{Z}(\Oo_\e)\bigr)\otimes_{\mathcal{Z}(\Oo_\e)} \Oo_\e$ is a central simple algebra of PI degree $l^{N}$, where $N$ is the number of positive roots of $\mathfrak{g}$. 
\end{teo}

\begin{remark}{\rm The sequence of inclusions ${\mathcal Z}_{0}({\mathcal O}_\epsilon)\subset {\mathcal Z}({\mathcal O}_\epsilon)\subset {\mathcal O}_\e$ is a PI Hopf triple in the sense of \cite{Br,BG}. This notion is a unifying concept which encompasses properties of $\Oo_\epsilon$ described in Theorem \ref{DCLteo1}.}
\end{remark}

\indent We saw in \eqref{injectionBecomesProjection} that the injection $i : U_A^\Q \to \Gamma_A^\Q$ specializes to a surjection $p : U_\e^\Q \to u_\e^\Q$. By very definition of $\langle \text{-},\text{-} \rangle_\e^\Q$, this may be written as:
\begin{equation}\label{fundamentalDegeneracy}
\forall \, \varphi \in \mathcal{O}_A, \:\: \forall \, h \in U_A^\Q \subset \Gamma_A^\Q, \quad \bigl(\langle \varphi, h \rangle_A^\Q\bigr)_{|\e} =  \bigl\langle \varphi_{|\e}, p(h_{|\e}) \bigr\rangle_{\e}^\Q.
\end{equation}

\indent Denote by $(u_\epsilon^\Q)^{*}$ the Hopf algebra dual to $u_\e^\Q$ and define a map $\pi: {\mathcal O}_\epsilon \rightarrow (u_\epsilon^\Q)^{*}$ by  restriction, {\it i.e.}
\begin{equation}\label{piPdef}
\forall \, \varphi \in {\mathcal O}_\epsilon, \:\: \forall \, x\in u_\epsilon^\Q \subset \Gamma_\epsilon^\Q, \quad \pi(\varphi)(x):=\langle \varphi,x\rangle_\e^\Q
\end{equation}
We have the sequence of maps 
\begin{equation}\label{seqmapsdual}
\xymatrix{ & \Oo(G)\ar@{^{(}->}[rr]^{\ (\mathbb{F}\mathrm{r}_\e)^{*}} & & \Oo_\e \ar[r]^{\pi} & (u_\epsilon^\Q)^{*}}.
\end{equation}
The following result is proved in e.g. \cite[Th.\,III.7.10]{BG}. A delicate point of the proof is the surjectivity of $\pi$, obtained by means of representation theory. We give a different proof, based on the canonical monomorphism $\Gamma_\e^\Q\hookrightarrow \hat{{\mathbf{U}}}_\e$ from Lemma \ref{embedGammaeps}.

\begin{prop}\label{pipropO} (1) The map $\pi$ is a surjective morphism of Hopf algebras.

\noindent (2) We have $\ker(\pi)={\mathcal O}_\epsilon {\mathcal Z}_{0}^+(\Oo_\e)$, where $\mathcal{Z}_0^+(\Oo_\e) := \ker \bigl(\varepsilon\vert_{\mathcal{Z}_0(\Oo_\e)} \bigr) = \ker(\varepsilon_{\Oo_\e}) \cap {\mathcal Z}_0({\mathcal O}_\e)$
\end{prop}

\begin{proof}
(1) Showing that $\pi$ is Hopf algebra morphism is trivial, since $u_\epsilon^\Q$ is a sub-Hopf algebra of $\Gamma_\epsilon^\Q$ and $\langle \text{-}, \text{-} \rangle_{\e}^\Q$ is a Hopf pairing. For surjectivity, let $\{x_i\}_{1 \leq i \leq n}$ be a basis of $u_\epsilon^\Q$. By the inclusions $u_\epsilon^\Q \subset \Gamma_\epsilon^\Q \subset \hat{{\mathbf{U}}}_\e$ in Lemma \ref{embedGammaeps}, we can write
$\textstyle x_i=\sum_{b\in \stackrel{.}{\mathbf{B}}}A_{ib} b_\epsilon$ with $A_{ib}\in {\mathbb C}.$ The family $(x_i)$ being free, if $(\lambda_i)\in \mc^n$ satisfies $\textstyle \sum_{i}\lambda_i A_{ib}=0$ for all $b \in \:\stackrel{.}{\mathbf{B}}$, then $\lambda_1=\ldots=\lambda_n=0$. Therefore $\textstyle\bigcap_{b\in \stackrel{.}{\mathbf{B}}}H_b=\{0\}$, where $H_b \subset  {\mathbb C}^{n}$ is the hyperplane with equation $\textstyle \sum_{i}\lambda_i A_{ib}=0$. We can extract exactly $n$ hyperplanes such that $\textstyle \bigcap_{j=1}^nH_{b_j}=\{0\},$ which amounts to say that the matrix $(A_{ib_j})$  is invertible.
Let $\overline{A}$ be the inverse of this matrix, and recall the canonical basis $\stackrel{.}{\mathbf{B}}{\!\!^*} = \{\varphi_b \}_{b \in \stackrel{.}{\mathbf{B}}}$ of $\mathcal{O}_A$, which is dual to $\stackrel{.}{\mathbf{B}}$. Then the elements $\psi_j\in {\mathcal O}_\epsilon$ defined by $\textstyle \psi_j=\sum_{k=1}^n \overline{A}_{b_k,j} \varphi_{b_k}^{\epsilon} \in \mathcal{O}_{\epsilon}$ satisfy $\langle \psi_j, x_i\rangle_\epsilon^\Q =\delta_{ij}$. Hence $\bigl\{\pi(\psi_j) \bigr\}_{1 \leq j \leq n}$ is a basis of $(u_\epsilon^\Q)^*$ and $\pi$ is surjective.

(2) We first show that $\mathcal{Z}_0^+(\mathcal{O}_\e) \subset \ker(\pi)$. Because $\pi$ is an algebra morphism it is sufficient to show that $\pi\bigl( {\mathcal Z}_{0}^+(\Oo_\e) \bigr)=\{0\}$. Let $\varphi\in {\mathcal Z}_{0}^+(\Oo_\e)$. Because ${\mathcal Z}_0({\mathcal O}_\e)$ is a sub-Hopf algebra of ${\mathcal O}_\epsilon$ and $\langle \text{-},\text{-} \rangle_{\e}^\Q$ is a Hopf pairing, it is sufficient to show that $\langle \varphi ,\text{-} \rangle_{\e}^\Q$ vanishes on the generators $E_i,F_i,K_i$ of $u_\epsilon^\Q$. We have $\varphi=\mathbb{F}\mathrm{r}_\e^{*}(\psi)$ where $\psi \in {\mathcal  O}(G)$ satisfies $\varepsilon(\psi) = \psi(1)=0$ (because $\mathbb{F}\mathrm{r}_\e^{*}$ is an injective morphism of Hopf algebras). By definitions $\langle \mathbb{F}\mathrm{r}_\e^{*}(\psi), E_i\rangle_{\e}^\Q=\langle \psi , \mathbb{F}\mathrm{r}_\e(E_i)\rangle=0$, similarly $\langle \mathbb{F}\mathrm{r}_\e^{*}(\psi), F_i\rangle_{\e}^\Q=0$, and $\langle \mathbb{F}\mathrm{r}_\e^{*}(\psi), K_i\rangle_{\e}^\Q=\langle \psi , \mathbb{F}\mathrm{r}_\e(K_i)\rangle = \psi(1) =0$ as desired.

Hence we can define the quotient morphism $\bar{\pi}: {\mathcal O}_\e/{\mathcal O}_\epsilon {\mathcal Z}_{0}^+(\Oo_\e) \rightarrow (u_\epsilon^\Q)^*$, which is still surjective. From \cite[\S III.7.7]{BG}, the algebra ${\mathcal O}_\e/{\mathcal O}_\epsilon {\mathcal Z}_{0}^+(\Oo_\e)$ has dimension $l^{\dim(\mathfrak g)}$, which is exactly $\dim\bigl( (u_\e^\Q)^*\bigr)$. By equality of dimensions, we deduce that $\bar{\pi}$ is an isomorphism, proving the claim.
\end{proof}

\subsection{Quasitriangular structures at roots of unity}\label{sec:qtsroot1}

Our goal here is to relate the specialization at $\e$ of the co-R-matrix on $\mathcal{O}_A$ (\S\ref{integralO}) with the usual R-matrix on the small quantum $u_\e^\Q$, which by finite-dimensionality is equivalent to a co-R-matrix on $(u_\e^\Q)^*$. We start by recalling this standard quasitriangular structure of $u_\e^\Q$.

\subsubsection{R-matrix for the small quantum group}\label{subsecRmatSmallQG}
\indent We explain the R-matrix of $u_\e^\Q$ by means of the Drinfeld double construction. Recall that $\tau$ denotes the Drinfeld pairing \eqref{deftau} and $\tau_{\e} : U_\e^\Pup(\mathfrak{b}_+)^{\rm cop}\times \Gamma_\e^\Q(\mathfrak{b}_-)\to \mc$ is its specialized version \eqref{taue}. Recall also the notation \eqref{Z0Uedefsep25}.
\begin{lem}\label{basicqpairingnov24} For all $x\in \mathcal{Z}_0^+(U_\e^\Q(\mathfrak{b}_+))$ and $h\in u_\e^\Q(\mathfrak{b}_-) \subset \Gamma_\e^\Q(\mathfrak{b}_-)$ we have $\tau_{\e}(x,h)=0$. Therefore we have a well-defined pairing
$$\bar{\tau}_{\e} : u_\e^\Q(\mathfrak{b}_+)^{\rm cop}\times u_\e^\Q(\mathfrak{b}_-)\to \mc \quad \text{given by } \:\: \bar{\tau}_{\e}(p(x),h) = \tau_{\e}(x,h).$$
Moreover $\bar{\tau}_{\e}$ is a non degenerate Hopf pairing.
\end{lem}
\begin{proof}
The first claim is immediate on generators, and the fact that $\mathcal{Z}_0^+(U_\e^\Q(\mathfrak{b}_+))U_\e^\Q(\mathfrak{b}_+)$ is a Hopf ideal. It implies that $\bar{\tau}_{\e}$ is well-defined; it is a Hopf pairing since $\tau_{\e}$ is a Hopf pairing and $p : U_\e^\Q \to u_\e^\Q$ is a Hopf morphism. Because $\tau_{\e}$ is non degenerate and $p$ is surjective, $\bar{\tau}_{\e}$ is non degenerate on the right, and therefore yields an injective map $u_\e^\Q(\mathfrak{b}_-) \rightarrow u_\e^\Q(\mathfrak{b}_+)^*$. By equality of dimensions of the source and target spaces this map is an isomorphism. Therefore $\bar{\tau}_{\e}$ is a non degenerate Hopf pairing.
\end{proof}

\noindent We also let $\bar{\rho}_{\e} : u_\e^\Q(\mathfrak{b}_-)^{\rm cop}\times u_\e^\Q(\mathfrak{b}_+)\to \mc$ be the pairing defined by $$\bar{\rho}_{\e}(x,y) := \bar{\tau}_{\e}(y,S^{-1}(x)) = \bar{\tau}_{\e}(S(y),x).$$
Since $S$ is an isomorphism, this pairing is non degenerate. Note that we also have, for every $x\in U_\e(\mathfrak{b}_-)$, $h\in u_\e(\mathfrak{b}_+)$:
\begin{equation}\label{relrhobar}
\bar{\rho}_{\e}(p(x),h)  = \rho_{\e}(x,h)
\end{equation}
with $\rho_\e : U_\e^\Pup(\mathfrak{b}_-)^{\rm cop} \times \Gamma_\e^\Q(\mathfrak{b}_+)$ the specialization of $\rho$ in \eqref{defrho}.
\smallskip

The Drinfeld double $\mathcal{D}(\bar{\tau}_\e)$ is the coalgebra $u_\e^\Q(\mathfrak{b}_-) \otimes u_\e^\Q(\mathfrak{b}_+)$ endowed with the product 
\begin{equation}\label{DrinfeldDoubleRel}
\quad (x^- \otimes x^+)(y^- \otimes y^+) = \sum_{(x^+),(y^-)} \bar\tau_\e\bigl( S(x^+_{(1)}), y^-_{(1)} \bigr) \bar\tau_\e\bigl( x^+_{(3)}, y^-_{(3)} \bigr)x^-y^-_{(2)} \otimes x^+_{(2)}y^+.
\end{equation}
Then $u_\e^\Q$ is isomorphic to the quotient of $\mathcal{D}(\bar{\tau}_\e)$ by the relations $1 \otimes K_i - K_i \otimes 1$ for all $i$. As for any Drinfeld double, the R-matrix of $\mathcal{D}(\bar{\tau}_\e)^{\otimes 2}$ is $\widetilde{R} = \textstyle \sum_i (1 \otimes h^+_i) \otimes (h_i^- \otimes 1)$, where $(h_i^\pm)$ are a pair of dual bases for $\bar{\tau}_\e$. The image of $\widetilde{R}$ through the quotient morphism $\mathcal{D}(\bar{\tau}_\e) \twoheadrightarrow u_\e^\Q$ is denoted $\overline{R} \in (u_\e^\Q)^{\otimes 2}$ and is the R-matrix that we fix for $u_\e^\Q$; for an explicit formula see \cite[App.\,A]{Lyub95}. We have morphisms of Hopf algebras $\bar \Phi^\pm_\e : (u_\e^\Q)^* \to (u_\e^\Q)^{\rm cop}$ given by
\begin{equation}\label{PhiPmeSmall}
\bar \Phi^+_\e(\varphi) = (\varphi \otimes \mathrm{id})(\overline{R}), \qquad \bar \Phi_\e^-(\varphi) = (\mathrm{id} \otimes \varphi)(\overline{R}^{-1}).
\end{equation}
By the very definition of $\overline{R}$, they factorize through $u_\e^\Q(\mathfrak{b}_\pm)^*$ and take values in $u_\e(\mathfrak{b}_\mp)$. Moreover they are easily seen to satisfy
\begin{equation}\label{wellnowndualityR}
\bar{\rho}_{\e}\bigl(\bar\Phi_\e^+(\alpha), h_+\bigr) = \alpha_+(h_+), \quad  \bar{\tau}_{\e}\bigl(\bar\Phi_\e^{-}(\alpha), h_-\bigr) = \alpha_-(h_-)
\end{equation}
for every $\alpha\in (u_\e^\Q)^*$ and $h_\pm\in u_\e^\Q(\mathfrak{b}_\pm)$. It follows that they descend to isomorphisms of Hopf algebras $\bar{\Phi}^\pm_\e : u_\e^\Q(\mathfrak{b}_\pm)^* \overset{\sim}{\to} u_\e^\Q(\mathfrak{b}_\mp)^{\rm cop}$.

\subsubsection{Relating (co)quasitriangular structures}\label{sectionlemmas} 

In this section we prove results that will be used in \S\ref{subsecStructResLgnEps}.

\smallskip

Recall the duality pairing $\langle \text{-}, \text{-} \rangle_{\AD}^\Pup : \Oo_{\AD} \times \Gamma_{\AD}^\Pup \to \AD$ from \eqref{defpairAD} which extends $\langle \text{-},\text{-} \rangle_A^\Q : \Oo_A \times \Gamma_A^\Q \to A$. The $\AD$ subscripts indicate extension of ground ring to $\mathbb{C}[\qD^{\pm 1}]$; recall that we specialize $\qD$ to the $l$-th root of unity $\eD$ fixed in \eqref{choixED}.
\begin{lem}\label{evalinnerprodA}
1. The surjection $p : U_\e^\Q \to u_\e^\Q$ in \eqref{defqmaps} can be extended to a morphism of Hopf algebras $\widetilde{p} : U_\e^\Pup \to u_\e^\Q$ (note that the target is unchanged), given by $\textstyle \widetilde{p}(L_i) = \prod_{s=1}^m (K_s^{\overline{D}})^{D\overline{a}_{si}}$ and $\textstyle \widetilde{p}(E_i)=E_i$, $\textstyle \widetilde{p}(F_i)=F_i$, where $(\overline{a}_{ij})$ is the inverse of the Cartan matrix $(a_{ij})$.
\\2. Let $\widetilde{i} : U_A^\Pup \hookrightarrow \Gamma_A^\Pup$ be the natural inclusion. For every $\alpha\in \Oo_A$ and $h\in U_A^\Pup$ we have
\[ \bigl( \bigl\langle \alpha,\widetilde{i}(h) \bigr\rangle_{\!\AD}^\Pup \bigr)_{\vert \eD} = \bigl\langle \alpha_{|\e}, \widetilde{p}(h_{|\e}) \bigr\rangle_\e^\Q = \pi(\alpha_{|\e})\bigl( \widetilde{p}(h_{|\e}) \bigr). \]
3. The pairings $\bar{\tau}_\e : u_\e^\Q(\mathfrak{b}_+)^{\rm cop} \times u_\e^\Q(\mathfrak{b}_-) \to \mathbb{C}$ from Lem.\,\ref{basicqpairingnov24} and $\tau_{\e} : U_\e^\Pup(\mathfrak{b}_+)^{\rm cop}\times \Gamma_\e^\Q(\mathfrak{b}_-)\to \mc$ from \eqref{identPhibrackete} satisfy
\[ \forall \, x \in U_\e^\Pup(\mathfrak{b}_+), \quad \forall \, h \in u_\e^\Q(\mathfrak{b}_-) \subset \Gamma_ \e^\Q(\mathfrak{b}_-), \quad \bar{\tau}_\e\bigl( \widetilde{p}(x),h \bigr) = \tau_{\e}(x,h). \]
The same property holds with the pairings $\rho_{\e}$ and $\bar{\rho}_\e$.
\end{lem}
\begin{remark}\label{tildepfactomars26}{\rm The morphism $\tilde{p}$ can be factorized as $\tilde{p}\colon U_\e^\Pup \twoheadrightarrow U_\e^\Pup/\mathcal{Z}_0^+(U_\e^\Pup)U_\e^\Pup \overset{\sim}{\longrightarrow} u_\e^\Q$. }
\end{remark}
\begin{proof}
Recall from \eqref{assumptionl} that $D$ and $l$ are assumed to be coprime, and we denote by $\overline{D}$ the multiplicative inverse of $D$ modulo $l$, so that $(\e^{\overline{D}})^D = \e$.

1. The version $U_\e^\Pup$ is obtained from $U_\e^\Q$ by adding the generators $L_i = K_{\varpi_i}$ for all $1 \leq i \leq m$. Remember that the matrix $(\overline{a}_{ij})$ has coefficients in $\textstyle \frac{1}{D}\mathbb{N}$. The relations $\textstyle \alpha_i = \sum_{j=1}^m a_{ji}\varpi_j$ for $1 \leq i \leq m$ are equivalent to $\textstyle \varpi_i = \sum_{j=1}^m \overline{a}_{ji}\alpha_j$. From the relation $K_i^l = 1$ in $u_\e^\Q$, we see that $K_i^{\overline{D}}$ can be used to play the role of $K_i^{1/D}$. We thus let $\textstyle \widetilde{p}$ be as in the statement. The only non-trivial thing to check is compatibility with the relations $L_iE_j = \epsilon_i^{\delta_{i,j}}E_jL_i$ (and similarly with $F_j$):
\[ \widetilde{p}(L_i) E_j = \prod_{s=1}^m (\epsilon^{\overline{D}})^{Dd_ja_{js}\overline{a}_{si}}E_j (K_s^{\overline{D}})^{D\overline{a}_{si}} = \epsilon^{d_j\overline{D}D\delta_{i,j}}E_j \widetilde{p}(L_i) = \epsilon_i^{\delta_{i,j}}E_j \widetilde{p}(L_i) \]
where we used that $K_sE_j = K_{\alpha_s}E_{\alpha_j} = q^{(\alpha_s,\alpha_j)} E_jK_s = q^{d_ja_{js}} E_jK_s$ and that $\epsilon^{\overline{D}D} = \epsilon$.

2. It suffices to check this when $h$ is a generator of $U_A^\Pup$, since $\langle \cdot, \cdot \rangle_{\AD}^\Pup$ is a Hopf pairing. We use again that $U_A^\Pup$ is an extension of $U_A^\Q$ by the elements $L_i$. If $h \in U_A^\Q$ we have $\widetilde{i}(h) = i(h)$ with $i : U_A^\Q \hookrightarrow \Gamma_A^\Q$ and thus
\[ \bigl( \langle \alpha,\widetilde{i}(h) \rangle_{\!\AD}^\Pup\bigr)_{\vert \eD} = \bigl( \langle \alpha, i(h) \rangle_A^\Q \bigr)_{|\e} = \langle \alpha_{|\e}, i(h)_{|\e} \rangle_\e^\Q \overset{\eqref{injectionBecomesProjection}}{=} \langle \alpha_{|\e}, p(h_{|\e}) \rangle_\e^\Q \overset{\eqref{piPdef}}{=} \pi(\alpha_{|\e})\bigl( p(h_{|\e}) \bigr). \]
Now consider the case $h = L_i$. By definition of $\mathcal{O}_A$ (\S\ref{integralO}), we can assume that $\alpha$ is a matrix coefficient $f(?\cdot v)$ of some $_A V \in \mathcal{C}_A$, with $f \in \mathrm{Hom}_A(_A V,A)$ and $v \in {_A V}$ a vector of weight $\mu \in P$. Then on one hand $\bigl( \langle \alpha,\tilde{i}(L_j) \rangle_{\AD}^\Pup\bigr)_{|\eD} = f(L_j \cdot v)_{|\eD} = \eD^{D(\varpi_j,\mu)}f(v)_{|\e}$ by \eqref{actionLi}. On the other hand, in the specialization $_\e V$, which is a $\Gamma_\e^\Q$-module, we have
\[ \textstyle\bigl\langle f(?\cdot v)_{|\e}, \widetilde{p}(L_i) \bigr\rangle_\e^\Q = f_{|\e}\left( \prod_{j=1}^m (K_j^{\overline{D}})^{D\overline{a}_{ji}} \cdot v_{|\e} \right) = \prod_{j=1}^m \epsilon^{\overline{D}D\overline{a}_{ji}(\alpha_j,\mu)}f(v)_{|\e} = \eD^{D(\varpi_i,\mu)}f(v)_{|\e} \]
where we used the relation between fundamental weights $(\varpi_i)$ and simple roots $(\alpha_j)$.

3. By definition of $\bar{\tau}_\e$, this is true for  the restriction $p$ of $\widetilde{p}$ to $U^\Q_\e(\mathfrak{b}_+)$. By the Hopf pairing property, it remains only to check this when $x = L_i$ and $h$ is a generator of $u_\e^\Q$. For $h=E_j$ and $h = F_j$ this is obvious from the definition of $\tau$ in \eqref{deftau}. For $h = K_i$ this follows by an argument which is completely similar to the ones in previous items.
\end{proof}

The ``quasitriangular datum'' (in the sense of \S\ref{subsecSubstQuasi})  $\Phi_\e^\pm$ from \eqref{Phi+-e} and $\bar{\Phi}^\pm_\e$ from \eqref{PhiPmeSmall} can now be related by means of the surjections $\pi$ from  \eqref{piPdef} and $\widetilde{p}$ from Lemma \ref{evalinnerprodA}:
\begin{lem}\label{commutephiate} We have a commutative diagram of Hopf algebra morphisms
\[ \xymatrix{\Oo_\e \ar[r]^{\Phi_\e^\pm\quad } \ar[d]_{\pi} & U_\e^\Pup(\mathfrak{b}_\mp)^{\rm cop} \ar[d]^{\widetilde{p}} \\ (u_\e^\Q)^* \ar[r]_{\ \bar \Phi_\e^{\pm}\quad } & u_\e^\Q(\mathfrak{b}_\mp)^{\rm cop}} \]
\end{lem}
\begin{proof} For every $h\in u_\e^\Q(\mathfrak{b}_-)$ and $\alpha \in \Oo_\e$ we have
\begin{align*}
\bar{\tau}_{\e}\bigl(\widetilde{p} \circ \Phi_\e^-(\alpha), h\bigr) & =  \tau_{\e}(\Phi_\e^-(\alpha), h) = \langle \alpha_, h\rangle_\e^\Q = \pi(\alpha) (h)= \bar{\tau}_{\e}\bigl(\bar\Phi_\e^{-} \circ \pi(\alpha), h\bigr).
\end{align*}
The first equality follow from Lemma \ref{evalinnerprodA}(3), the second from \eqref{identPhibrackete}, the third from \eqref{piPdef}, and the fourth from \eqref{wellnowndualityR}. Since $\bar{\tau}_{\e}$ is non degenerate, the desired equality follows. The proof for $\Phi^+_\e$ is similar, but now using the pairing $\bar{\rho}_\e$.
\end{proof}

Let $\mathcal{R}_{\e} : \mathcal{O}_\e \otimes \mathcal{O}_\e \to \mathbb{C}$ be the specialization to $q=\e$ (with $\qD = \eD = \e^{\overline{\D}}$, see \eqref{choixED}) of the co-R-matrix $\mathcal{R}_{\AD} : \mathcal{O}_{\AD} \otimes \mathcal{O}_{\AD} \to \AD$. Also, the R-matrix $\overline{R} \in (u_\e^\Q)^{\otimes 2}$ yields the co-R-matrix $\overline{\mathcal{R}} : (u_\e^\Q)^* \otimes (u_\e^\Q)^* \to \mathbb{C}$ given by $\overline{\mathcal{R}}(\alpha \otimes \beta) = (\alpha \otimes \beta)(\overline{R}) = \beta\bigl( \bar{\Phi}^+_\e(\alpha) \bigr)$. The next corollary will have important consequences in \S\ref{subsecStructResLgnEps}.

\begin{cor}\label{relationsCoRMatAtEps}
1. The projection $\pi : \Oo_\e \to (u_\e^\Q)^*$ defined in \eqref{piPdef} is a morphism of co-quasitriangular Hopf algebras, i.e.
\[ \forall \, \varphi, \psi \in \Oo_\e, \quad \mathcal{R}_{\e}(\varphi \otimes \psi) = \overline{\mathcal{R}}\bigl( \pi(\varphi) \otimes \pi(\psi) \bigr). \]
\noindent 2. For any $\varphi \in \Oo_\e$ and $\psi \in \mathcal{Z}_0(\Oo_\e)$, it holds
\[ \mathcal{R}_{\e}(\varphi \otimes \psi) = \mathcal{R}_{\e}(\psi \otimes \varphi) = \varepsilon_{\Oo_\e}(\varphi)\,\varepsilon_{\Oo_\e}(\psi) \]
where $\varepsilon_{\Oo_\e}$ is the counit of $\Oo_\e$, given by $\varepsilon_{\Oo_\e}(\varphi) = \langle \varphi, 1 \rangle_\e^\Q$.
\end{cor}
\begin{proof}
1. Let $\varphi', \psi' \in \Oo_A$ be such that $\varphi'_{|\e} = \varphi$ and $\psi'_{|\e} = \psi$. Then
\begin{align*}
\mathcal{R}_{\e}(\varphi \otimes \psi) &= \mathcal{R}_{\AD}(\varphi' \otimes \psi')_{|\eD} = \bigl\langle  \psi', \widetilde{i}\bigl(\Phi^+_A(\varphi') \bigr) \bigr\rangle_{\!\AD|\eD} = \pi(\psi)\bigl[ \, \widetilde{p}\bigl(\Phi^+_A(\varphi')_{|\e} \bigr) \bigr]\\
&= \pi(\psi)\bigl[ \, \widetilde{p}\bigl(\Phi^+_\e(\varphi) \bigr) \bigr] = \pi(\psi)\bigl[ \bar{\Phi}^+_\e\bigl(\pi(\varphi) \bigr) \bigr] = \overline{\mathcal{R}}\bigl( \pi(\varphi) \otimes \pi(\psi) \bigr)
\end{align*}
where the first equality is by definition of the specialization of a bilinear map (see \S\ref{generalRmkSpe}), the second is by definition of $\mathcal{R}_{\AD}$ in \eqref{integralCoRmat}, the third is by Lemma \ref{evalinnerprodA}(2), the fourth is by definition of $\Phi^+_\e$, the fifth is by Lemma \ref{commutephiate} and the last is by definition of $\overline{\mathcal{R}}$.
\\2. By Prop.\,\ref{pipropO}(2) we have $\pi(\psi) = \varepsilon_{\Oo_\e}(\psi)\varepsilon$ for $\psi \in \mathcal{Z}_0(\Oo_\e)$, where $\varepsilon$ is the counit of $u_\e^\Q$. Any co-R-matrix $\mathcal{R} : F \otimes F \to \Bbbk$ on a Hopf algebra $F$ satisfies $\mathcal{R}(1_F \otimes \text{-}) = \mathcal{R}(\text{-} \otimes 1_F) = \varepsilon_F$, and the previous item gives the result with $F=\Oo_\e$.
\end{proof}

\section{Quantum graph algebras at roots of unity}\label{secdefLgn}
Recall from \S\ref{subsecDefLgnH} the definition of the graph algebra $\mathcal{L}_{g,n}(H)$ associated to a Hopf algebra endowed with morphisms $\Phi^\pm : H^\circ \to H^{\rm cop}$ which are substitute for quasitriangularity. By a {\em quantum graph algebra} we mean a graph algebra obtained from a quantum group $U_q(\mathfrak{g})$. Here we introduce the specialization at a root of unity of such an algebra and one of our main results is that, in some sense, it can be decomposed in two parts: the finite-dimensional part $\mathcal{L}_{g,n}(u_\e^\Q)$ and the ``classical'' part $\mathcal{L}_{g,n}\bigl( U(\mathfrak{g}) \bigr)$.

\subsection{Quantum graph algebras and their specialization}\label{subsecQuantumGraphAlg}
Recall that $\qD$ is a formal variable such that $\qD^D = q$, where $D$ is the smallest integer such that $DP \subset Q$, see \S\ref{subsecLieAlg}.

\indent For the sake of precision, we introduce notations which indicate extension of scalars from $\mathbb{C}(q)$ to $\mathbb{C}(\qD)$:
\[ \Oo_q(\qD) := \Oo_q \otimes_{\mc(q)} \mc(\qD), \quad U_q^\Lam(\qD) :=U_q^\Lam \otimes_{\mc(q)} \mc(\qD) \]
where as usual $\Lambda$ is a lattice such that $Q \subset \Lambda \subset P$.

\indent By definition $\mathcal{O}_q$ is a restricted dual for $U_q^\Q$ based on type 1 $U_q^\Q$-modules (\S\ref{sec:Oq}). However, we noticed in \eqref{actionLi} that $U_q^\Pup(\qD)$ naturally acts on $U_q^\Q$-modules provided scalars are extended to $\mathbb{C}(\qD)$, and this yields a non-degenerate pairing $\langle \cdot,\cdot\rangle^\Pup : \mathcal{O}_q(\qD) \times U^\Pup_q(\qD) \to \mathbb{C}(\qD)$. As a result we can make the following definition:

\begin{defi}
We let $\mathcal{L}_{g,n}(\mathfrak{g})$ be the $\mathbb{C}(\qD)$-algebra $\mathcal{L}_{g,n}(H)$ for $H = U_q^\Pup(\qD)$ as in \S\ref{subsecDefLgnH}, except that we use $\mathcal{O}_q(\qD)$ instead of $H^\circ$. Very often we will write $\mathcal{L}_{g,n}$ instead of $\mathcal{L}_{g,n}(\mathfrak{g})$.
\end{defi}
\noindent Explicitly:
\begin{itemize}
\item $\mathcal{L}_{g,n}$ is $\mathcal{O}_q(\qD)^{\otimes (2g+n)}$ as a $\mathbb{C}(\qD)$-vector space.
\item The algebra structure on $\mathcal{L}_{g,n}$ from \S\ref{subsecDefLgnH} is here governed by the co-R-matrix $\mathcal{R} : \mathcal{O}_q(\qD) \otimes \mathcal{O}_q(\qD) \to \mathbb{C}(\qD)$ defined by $\mathcal{R}(\varphi \otimes \psi) = \langle \psi, \Phi^+(\varphi) \rangle^\Pup$ with $\Phi^+$ from \eqref{phipmdebut}. We have to make the definition of $\mathcal{L}_{g,n}$ with $U_q^\Pup$ because $\Phi^\pm$ takes values in it.
\end{itemize}
 The algebras $\mathcal{L}_{g,n}$ are finitely generated and were shown to be Noetherian domains in \cite{BFR}. When $g=0$ it is possible to work over $\mathbb{C}(q)$ because the appearances of the variable $\qD$ compensate each other \cite[Prop.\,6.2]{BR1}.

\indent The right coadjoint action of $U_q^\Q(\qD)$ on $\Oo_q(\qD)$ is defined by \eqref{defCoad}; so we have
\begin{equation}\label{coadUqOq}
\textstyle \mathrm{coad}^r(h)(\varphi) = \sum_{(h),(\varphi)} \bigl\langle \varphi_{(1)}, h_{(1)} \bigr\rangle^\Q \, \bigl\langle \varphi_{(3)}, S(h_{(2)}) \bigr\rangle^\Q \, \varphi_{(2)}.
\end{equation}
for all $h \in U_q^\Q(\qD)$ and $\varphi \in \Oo_q(\qD)$. Using the coproduct of $U_q(\qD)$ we have the diagonal action $\mathrm{coad}^r$ on $\mathcal{L}_{g,n}$, which makes it a $U_q^\Q(\qD)$-module-algebra (see \eqref{coadLgn}). Again due to \eqref{actionLi}, this action can be extended to $U^\Pup_q$. However the subalgebras of invariant elements \eqref{moduliAlgH} are the same under $U_q^\Q$ and $U_q^\Pup$. To avoid overweighted notations, we thus denote by $\mathcal{L}_{g,n}^{U_q}$ this subalgebra of invariant elements.

\smallskip

\indent Let $\mathcal{O}_{\AD} = \mathcal{O}_A \otimes_A \AD$, which is the integral form of $\mathcal{O}_q(\qD)$. We saw in \S\ref{integralO} that the co-R-matrix restricts to an $\AD$-bilinear form $\mathcal{R}_{\AD} : \mathcal{O}_{\AD} \otimes_{\AD} \mathcal{O}_{\AD} \to \AD$. As a result we can make the following definition:
\begin{defi}\label{defIntVersionLgn}
We let $\mathcal{L}_{g,n}^{\AD} = \mathcal{L}_{g,n}^{\AD}(\mathfrak{g})$ be the $\AD$-subalgebra $\mathcal{O}_{\AD}^{\otimes_{\AD} (2g+n)}$ in $\mathcal{L}_{g,n}$.
\end{defi}
\noindent $\mathcal{L}_{g,n}^{\AD}$ is a free $\AD$-module because $\Oo_A$ is a free $A$-module (\S\ref{integralO}) and extension of scalars preserves freeness as it commutes with direct sums. As already recalled just above, when $g=0$ we can work over $A$ instead of $\AD$. We thus allow ourselves to use the notation $\mathcal{L}_{0,n}^A$, which matches with \cite{BR1,BR2}. Of course, $\mathcal{L}_{0,n}^{\AD} = \mathcal{L}_{0,n}^A \otimes_A \AD$.

\indent The right coadjoint action $\mathrm{coad}^r : \Oo_q(\qD) \otimes U_q^\Q(\qD) \to \Oo_q(\qD)$ from \eqref{coadUqOq} restricts to a right action $\mathrm{coad}^r : \Oo_{\AD} \otimes_{\AD} \Gamma_{\AD}^\Q \to \Oo_{\AD}$. Note that $\Oo_{\AD}$ is stable under this action because it is a Hopf $\AD$-algebra. We denote by $\mathcal{L}_{g,n}^{\Gamma_A}$ the $\AD$-subalgebra of invariant elements under this action. It is an integral form of $\mathcal{L}_{g,n}^{U_q}$, meaning that $\mathcal{L}_{g,n}^{\Gamma_A} \otimes_{\AD} \mathbb{C}(\qD) = \mathcal{L}_{g,n}^{U_q}$. We note that $\mathcal{L}_{g,n}^{U_A} = \mathcal{L}_{g,n}^{\Gamma_A}$, where $U_A \subset \Gamma_A$ is the unrestricted integral form (\S\ref{integralFormsUq}, note that the choice of the lattice $\Lambda$ has no effect here).

\smallskip

\indent Recall that $\epsilon$ is a root of unity of order $l$, subject to the assumptions \eqref{assumptionl}, and that $\eD$ is chosen in \eqref{choixED} to be $\e^{\overline{D}}$ where $\overline{D}$ is the inverse of $D$ modulo $l$.

\begin{defi}\label{DefLgnesept25}
We let $\mathcal{L}_{g,n}^{\e} = \mathcal{L}_{g,n}^{\e}(\mathfrak{g})$ be the $\mathbb{C}$-algebra $\mathcal{L}_{g,n}^{\AD} \otimes_{\AD} \mathbb{C}_{\eD}$, that is the specialization of $\mathcal{L}_{g,n}^{\AD}$ at $\qD = \eD$.
\end{defi}

\noindent Note that we write $\mathcal{L}_{g,n}^{\e}$ instead of $\mathcal{L}_{g,n}^{\eD}$ because $\eD$ is defined in terms of $\e$ in \eqref{choixED}, contrarily to the variable $\qD$ which realizes a genuine extension of $\mathbb{C}(q)$. As a $\mathbb{C}$-vector space $\mathcal{L}_{g,n}^{\e}(\mathfrak{g})$ is $\Oo_\e(G)^{\otimes (2g+n)}$ and the algebra structure is given by the formulas of \S\ref{subsecDefLgnH}, but using the specialized co-R-matrix $\mathcal{R}_{\e} : \Oo_\e \otimes \Oo_\e \to \mathbb{C}$ (specialized in the sense of \S\ref{generalRmkSpe}).

\indent The right action $\mathrm{coad}^r : \mathcal{L}_{g,n}^{\AD} \otimes_{\AD} \Gamma_{\AD}^\Q \to \mathcal{L}_{g,n}^{\AD}$ gives by specialization a right action of $\Gamma^\Q_\e$ on  $\mathcal{L}_{g,n}^\e$ and we denote by $\mathcal{L}_{g,n}^{\Gamma_\e}$ the subalgebra of elements which are invariant under it. Said differently, it is the specialization at $\qD = \eD$ of the integral form $\mathcal{L}_{g,n}^{\Gamma_A}$. There is however another subalgebra of invariant elements which deserves interest:
\begin{equation}\label{defSmallInvEps}
\mathcal{L}_{g,n}^{u_\e} = \bigl\{ x \in \mathcal{L}_{g,n}^\e \, \big|\, \forall \, h \in u_\e^\Q \subset \Gamma_\e^\Q, \:\: \mathrm{coad}^r(h)(x) = \varepsilon(h)x \bigr\}.
\end{equation}
Here again there is no effect if one replaces $u_\e^\Q$ with $u_\e^\Lam$ for any lattice $Q \subset \Lambda \subset P$. Of course, $\mathcal{L}_{g,n}^{\Gamma_\e} \subset \mathcal{L}_{g,n}^{u_\e}$ but the converse is not true. Another notation for $\mathcal{L}_{g,n}^{u_\e}$ is $\mathcal{L}_{g,n}^{U_\e}$ because the inclusion $U_A^\Lam \subset \Gamma_A^\Lam$ factorizes after specialization as $U_\e^\Lam \twoheadrightarrow u_\e^\Lam \subset \Gamma_\e^\Lam$, see \eqref{injectionBecomesProjection} and \eqref{fundamentalDegeneracy}. But beware that $\mathcal{L}_{g,n}^{U_\e}$ is {\em not} the specialization of $\mathcal{L}_{g,n}^{U_A} = \mathcal{L}_{g,n}^{\Gamma_A}$, so we prefer to use $\mathcal{L}_{g,n}^{u_\e}$ to prevent confusion.

Finally, note that \eqref{genGammae}, and the isomorphism (i) below it, imply
\begin{equation}\label{factoraction}
\mathcal{L}_{g,n}^{\Gamma_\e} = (\mathcal{L}_{g,n}^{u_\e})^{U(\mathfrak{g})}.
\end{equation}
\subsection{Modified Alekseev morphism \texorpdfstring{$\widehat{\Phi}_{g,n}$}{ }}\label{subsecModifiedAlekseev}
For notational convenience and readability, extension of scalars from $\mathbb{C}(q)$ to $\mathbb{C}(\qD)$ is kept implicit in this subsection. For instance we write $\mathcal{O}_q$ (resp. $\mathcal{O}_A$) instead of $\mathcal{O}_q(\qD)$ (resp. $\mathcal{O}_{\AD}$).

\smallskip

\indent Let $\mathcal{H}_q = \mathcal{H}_q(\mathfrak{g})$ be the Heisenberg double of $\mathcal{O}_q(G)$; it is the $\mathbb{C}(\qD)$-vector space $\mathcal{O}_q \otimes U_q^\Pup$ with product such that $\mathcal{O}_q \otimes 1_{U_q} \cong \mathcal{O}_q$, $1_{\mathcal{O}_q} \otimes U^\Pup_q \cong U_q^\Pup$ are subalgebras and
\begin{equation}\label{relationEchangeHq}
\textstyle  \forall \, a \in U_q^\Pup, \:\: \forall \, \varphi \in \mathcal{O}_q, \quad \varphi a = \varphi \otimes a, \quad a\varphi = \sum_{(a),(\varphi)} \langle \varphi_{(2)}, a_{(1)} \rangle^\Pup \, \varphi_{(1)}a_{(2)}
\end{equation}
with the pairing $\langle\text{-},\text{-}\rangle^\Pup : \mathcal{O}_q \times U_q^\Pup \to \mathbb{C}(\qD)$ from \S\ref{sec:Oq}. Thus an element in $\mathcal{H}_q$ is a sum of elements $\varphi a$ with $\varphi \in \mathcal{O}_q$ and $a \in U_q^\Pup$.

\smallskip

Recall from \eqref{RSDphi} that the two maps $\Phi^\pm : \mathcal{O}_q \to U_q^\Pup$ in \eqref{phipmdebut} combine into a map $\Phi_{0,1}$. It can be written as follows (although $RR^{fl} \not\in U_q^\Pup \otimes U_q^\Pup$, this formula makes sense by the same argument as for \eqref{phipmdebut})
\begin{equation}\label{defPhi01janv25}
\fonc{\Phi_{0,1}}{\Oo_q}{U_q^\Pup}{\alpha}{(\alpha\otimes \mathrm{id})(RR^{fl})}
\end{equation}
where $R^{fl}$ means that tensorands are permuted. We denote by $U_q'$ the image of the map $\Phi_{0,1}$. From \eqref{elmtsChapeau} we have elements $\widehat{h'}$ which exist for all $h' \in U'_q$. Our goal is to construct an extension $\widehat{\mathcal{H}}_q(\mathfrak{g})$ of $\mathcal{H}_q(\mathfrak{g})$ which for all $h \in U_q^\Pup$ contains an element $\widehat{h}$ satisfying suitable commutation relations. This will allow us to define a modified (compared to \S\ref{subsecAlekseev}) Alekseev morphism $\widehat{\Phi}_{g,n} : \mathcal{L}_{g,n}(\mathfrak{g}) \to \widehat{\mathcal{H}}_q(\mathfrak{g})^{\otimes g} \otimes U_q^\Pup(\mathfrak{g})^{\otimes n}$.

\smallskip

\indent It is known that $\Phi_{0,1}$ is an embedding of $U_q^\Pup$-modules, where $\Oo_q$ is endowed with the right coadjoint action ${\rm coad}^r$ from \eqref{defCoad} and $U_q^\Pup$ with the right adjoint action ${\rm ad}^r$ from \eqref{defAdr}. Moreover $U'_q := \Phi_{0,1}(\Oo_q)$ equals $(U_q^\Pup)^{\rm lf}$, namely the (coideal) subalgebra of locally finite elements of $U_q^\Pup$ for the action ${\rm ad}^r$ (\text{i.e.}, the elements of $U_q^\Pup$ whose orbit under the right adjoint action \eqref{defAdr} of $U_q^\Pup$ is a  finite-dimensional subspace). A proof of these properties can be found in \cite{baumann} and \cite[Th.\,3.113 and Prop.\,3.116]{VY} up to slightly different conventions. The relation between $(U_q^\Pup)^{\mathrm{lf}}$ and $U_q^\Pup$ has been elucidated by Joseph--Letzter \cite{JL}. Let us introduce a few notations to explain it. Consider the multiplicative subset $\mathbb{T}_{2-} = \bigl\{ K_{-2\lambda} \, \big|\, \lambda \in P_+ \bigr\}$. It holds $\mathbb{T}_{2-} \subset (U_q^\Pup)^{\mathrm{lf}}$ and $\mathbb{T}_{2-}$ satisfies the Ore condition (obvious), so we have the localization $(U_q^\Pup)^{\mathrm{lf}}\mathbb{T}_{2-}^{-1}$. Consider moreover $\mathbb{T} = \bigl\{ K_\lambda \, \big|\, \lambda \in P \bigr\}$ and $\mathbb{T}_2 = \bigl\{ K_{2\lambda} \, \big|\, \lambda \in P \bigr\}$. Note that $\mathbb{T}/\mathbb{T}_2$ is an abelian group of cardinal $2^m$, $m = \mathrm{rank}(\mathfrak{g})$, and choose representatives $s_1, \ldots, s_{2^m}$. Then \cite[\S 6.4]{JL} asserts that $\textstyle U_q^\Pup = \bigoplus_{j=1}^{2^m} \bigl((U_q^\Pup)^{\mathrm{lf}}\mathbb{T}_{2-}^{-1}\bigr)s_j$. It is convenient to restate this as an isomorphism of algebras
\begin{equation}\label{relatingLfAndUq}
U_q^\Pup \cong (U_q^\Pup)^{\mathrm{lf}} \underset{\mathbb{C}(q)[\mathbb{T}_{2-}]}{\otimes} \mathbb{C}(q)[\mathbb{T}]
\end{equation}
with product $(h' \otimes K_\lambda)(h'' \otimes K_\eta) = h'\bigl( K_\lambda h''K_\lambda^{-1} \bigr) \otimes K_{\lambda + \eta}$ on the right-hand side; note that $\mathbb{C}(q)[\mathbb{T}]$ is just $U_q(\mathfrak{h})$. Of course all this remains true after extension of scalars to $\mathbb{C}(\qD)$.

\smallskip

\indent Since $K_{-2\lambda} \in U_q'$, we have for all $\lambda \in P_+$ an element $\widehat{K_{-2\lambda}} \in \mathcal{H}_q(\mathfrak{g})$ defined in \eqref{elmtsChapeau}. There is thus a natural right action of $\mathbb{C}(\qD)[\mathbb{T}_{2-}]$ on $\mathcal{H}_q$ given by $x \cdot K_{2\lambda} = x\,\widehat{K_{2\lambda}}$. As a result we may consider the following algebra.
\begin{defi}\label{defiHhatq} {\rm We let $\widehat{\mathcal{H}}_q(\mathfrak{g}) := \mathcal{H}_q(\mathfrak{g}) \underset{\mathbb{C}(\qD)[\mathbb{T}_{2-}]}{\otimes} \mathbb{C}(\qD)[\mathbb{T}]$ be the algebra with the associative product defined by the following rules (where $\otimes = \otimes_{\mathbb{C}(\qD)[\mathbb{T}_2]}$):
\smallskip

\indent \textbullet ~ $\mathcal{H}_q(\mathfrak{g}) \otimes K_{\boldsymbol{0}}$ and $1_{\mathcal{H}_q} \otimes \mathbb{C}(\qD)[\mathbb{T}]$ are subalgebras isomorphic to $\mathcal{H}_q(\mathfrak{g})$ and $\mathbb{C}(\qD)[\mathbb{T}]$.

\smallskip

\indent \textbullet ~ $\bigl( \varphi a \otimes K_{\boldsymbol{0}} \bigr) \bigl( 1_{\mathcal{H}_q} \otimes K_\lambda \bigr) = \varphi a \otimes K_\lambda$ and $\bigl( 1_{\mathcal{H}_q} \otimes K_\lambda \bigr) \bigl( \varphi a \otimes K_{\boldsymbol{0}} \bigr) = (\varphi \lhd K_{-\lambda})a \otimes K_\lambda$.}
\end{defi}

\noindent We often write $\widehat{\mathcal{H}}_q$ instead of $\widehat{\mathcal{H}}_q(\mathfrak{g})$. Given $h \in U_q^\Pup$, use \eqref{relatingLfAndUq} to decompose it as $h = \textstyle \sum_j h_j \otimes K_{\lambda_j}$ with $h_j \in(U_q^\Pup)^{\mathrm{lf}}$ and introduce the notation
\begin{equation}\label{elmtsChapeauHq}
\widehat{h} = \textstyle \sum_j  \widehat{h_j} \otimes K_{\lambda_j} \in \widehat{\mathcal{H}}_q
\end{equation}
with $\widehat{h_j}$ defined in \eqref{elmtsChapeau}; note in particular that $\varphi a \otimes_{\mathbb{C}(\qD)[\mathbb{T}_2]} K_\lambda = \varphi a \, \widehat{K_\lambda}$. With these definitions it holds in $\widehat{\mathcal{H}}_q(\mathfrak{g})$:
\begin{equation}\label{commutationHatHq}
\forall \, h,g \in U_q^\Pup, \:\: \forall \, \varphi x \in \mathcal{H}_q, \quad \widehat{hg} = \widehat{h}\,\widehat{g}, \quad \widehat{h} \, \varphi a = \sum_{(h)} \bigl( \varphi \lhd S^{-1}(h_{(2)}) \bigr)a \, \widehat{h_{(1)}}.
\end{equation}

\indent The relations \eqref{commutationHatHq} are formally the same as \eqref{2sidedproductHeisenberg}, replacing $\widetilde{h}$ by $\widehat{h}$. The modified Alekseev morphism is obtained by using elements $\widehat{h}$ instead of $\widetilde{h}$. Namely, let $\widehat{\mathsf{D}}_{g,n} : U_q^\Pup \to \widehat{\mathcal{H}}_q^{\otimes g} \otimes (U_q^\Pup)^{\otimes n}$ be the following morphism of $\mathbb{C}(\qD)$-algebras
\begin{equation}\label{dressingMap2}
\widehat{\mathsf{D}}_{g,n}(h) = \sum_ {(h)} \widehat{h_{(1)}}\,h_{(2)} \otimes \ldots \otimes \widehat{h_{(2g-1)}}\,h_{(2g)} \otimes h_ {(2g+1)} \otimes \ldots \otimes h_{(2g+n)}.
\end{equation}
Define $\widehat{\Phi}_{g,n} : \mathcal{L}_{g,n} \to \widehat{\mathcal{H}}_q^{\otimes g} \otimes (U_q^\Pup)^{\otimes n}$ inductively by (compare with \eqref{AlekseevValue1}--\eqref{AlekseevValue2}):
\begin{align}
\begin{split}\label{modifiedAlekseevValue}
\forall \, \varphi \in \mathcal{L}_{0,1}, \:\: \forall \, w \in \mathcal{L}_{0,n}, \quad\widehat{\Phi}_{0,n+1}(\varphi \otimes w) &=  \sum_{[\varphi]} \Phi_{0,1}(\varphi_{[2]}) \otimes \widehat{\mathsf{D}}_{0,n}\bigl( \Phi^+(\varphi_{[1]}) \bigr)\widehat{\Phi}_{0,n}(w),\\
\forall \, x \in \mathcal{L}_{1,0}, \:\: \forall \, w \in \mathcal{L}_{g,n}, \quad \widehat{\Phi}_{g+1,n}(x \otimes w) &=  \sum_{[x]} \Phi_{1,0}(x_{[2]}) \otimes \widehat{\mathsf{D}}_{g,n}\bigl( \Phi^+(x_{[1]}) \bigr) \widehat{\Phi}_{g,n}(w)
\end{split}
\end{align}
using the left coaction which is dual to $\mathrm{coad}^r$, see \eqref{defCoactCoad}. The base cases are $\widehat{\Phi}_{0,1} = \Phi_{0,1} : \mathcal{L}_{0,1} \to U_q^\Pup$ defined in \eqref{RSDphi} and $\widehat{\Phi}_{1,0} = \Phi_{1,0} : \mathcal{L}_{1,0} \to \mathcal{H}_q \subset \widehat{\mathcal{H}}_q$ defined in \eqref{defPhi10}.

\smallskip

\indent The algebra morphism $\widehat{\mathsf{D}}_{g,n}$ yields a right action of $U_q^\Pup$ on $\widehat{\mathcal{H}}_q$:
\begin{align*}
(\varphi a \otimes_{\mathbb{C}(\qD)[\mathbb{T}_2]} K_\lambda) \cdot h &:= \sum_{(h)} \widehat{\mathsf{D}}_{1,0}\bigl( S(h_{(1)}) \bigr) \,  (\varphi a \otimes_{\mathbb{C}(\qD)[\mathbb{T}_2]} K_\lambda) \, \widehat{\mathsf{D}}_{1,0}(h_{(2)})\\
&=\sum_{(h)} \bigl( S(h_{(2)}) \rhd \varphi \lhd h_{(3)} \bigr) S(h_{(1)})a h_{(6)} \, \widehat{S(h_{(4)}) K_\lambda h_{(5)}}
\end{align*}
where the second equality is obtained from \eqref{commutationHatHq} by a computation completely analogous to \eqref{computationProofQMM2SidedHeis}. It readily endows $\widehat{\mathcal{H}}_q$ with the structure of a right $U_q^\Pup$-module-algebra. By iterated coproduct $U_q^\Pup \to (U_q^\Pup)^{\otimes (g+n)}$ we get a right $U_q^\Pup$-module-algebra structure on $\widehat{\mathcal{H}}_q^{\otimes g} \otimes (U_q^\Pup)^{\otimes n}$, using the right adjoint action \eqref{defAdr} on the $(U_q^\Pup)^{\otimes n}$ part.

\begin{prop}\label{propModifAlekseev}
In the algebra $\widehat{\mathcal{H}}_q^{\otimes g} \otimes (U_q^\Pup)^{\otimes n}$ endowed with the usual tensor-wise product it holds
\[ \forall \, h \in U_q^\Pup, \:\: \forall \, x \in \mathcal{L}_{g,n}, \quad \widehat{\Phi}_{g,n}\bigl( \mathrm{coad}^r(h)(x) \bigr) = \sum_{(h)} \widehat{\mathsf{D}}_{g,n}\bigl( S(h_{(1)}) \bigr) \, \widehat{\Phi}_{g,n}(x) \, \widehat{\mathsf{D}}_{g,n}(h_{(2)}). \]
so that $\widehat{\Phi}_{g,n}$ is $U_q^\Pup$-linear. Moreover $\widehat{\Phi}_{g,n}$ is an algebra morphism which is injective.
\end{prop}
\begin{proof}
The formula and the algebra morphism claim are established by {\em exactly} the same computations than the ones used to prove Prop.\,\ref{propAlekseevMorph}. The only changes are typographical: $\Phi_{g,n}$ (resp. $\mathsf{D}_{g,n}$) must be replaced by $\widehat{\Phi}_{g,n}$ (resp. $\widehat{\mathsf{D}}_{g,n}$) and the ``tilda elements'' $\widetilde{x}$ must be replaced by the ``hat elements'' $\widehat{x}$. For the $U_q^\Pup$-linearity claim, it suffices to note that the action of $U_q^\Pup$ on $\widehat{\mathcal{H}}_q^{\otimes g} \otimes (U_q^\Pup)^{\otimes n}$ can be written as $w \cdot h = \textstyle \sum_{(h)} \widehat{\mathsf{D}}_{g,n}\bigl( S(h_{(1)}) \bigr) \, w \, \widehat{\mathsf{D}}_{g,n}(h_{(2)})$. For the injectivity, we know that $\Phi_{0,1}$ and $\Phi_{1,0}$ are injective \cite[Th.\,3.13]{BFR}. The same arguments as in Lemma \ref{lemmaInjAlekseev} (with only typographical changes) allow us to conclude that $\widehat{\Phi}_{g,n}$ is injective.
\end{proof}

\begin{remark}
{\rm By \eqref{elmtsChapeauHq} and Lemma \ref{lemmeChapeau} the representation $\rho : \mathcal{H}_q \to \mathrm{End}_{\mathbb{C}(\qD)}( \mathcal{O}_q )$ from \eqref{HeisenbergRep} satisfies $\widehat{\rho}(h')(\psi) = \psi \lhd S^{-1}(h')$ for all $h' \in (U_q^\Pup)^{\rm lf}$. Then $\rho$ can be extended to $\widehat{\rho} : \widehat{\mathcal{H}}_q \to \mathrm{End}_{\mathbb{C}(\qD)}( \mathcal{O}_q )$ simply by $\widehat{\rho}\bigl( \widehat{K_\lambda} \bigr)(\psi) = \psi \lhd K_{-\lambda}$ for all $\lambda \in P$. Another extension is $\widetilde{\rho} : \mathcal{HH}_q \to \mathrm{End}_{\mathbb{C}(\qD)}( \mathcal{O}_q)$ defined in general in \eqref{HeisenbergRepHH}, where $\mathcal{HH}_q = \mathcal{HH}( \mathcal{O}_q ) = \mathcal{H}_q \otimes U_q^\Pup$ is the two-sided Heisenberg double (\S\ref{subsecAlekseev}). We have a commutative diagram
\[ \xymatrix@C=5em@R=1.5em{
\mathcal{L}_{g,n} \ar[r]^-{\widehat{\Phi}_{g,n}} \ar[d]_-{\Phi_{g,n}} & \widehat{\mathcal{H}}_q^{\otimes g} \otimes (U_q^\Pup)^{\otimes n} \ar[d]^-{\widehat{\rho}^{\otimes g} \otimes \mathrm{id}^{\otimes n}}\\
\mathcal{HH}_q^{\otimes g} \otimes (U_q^\Pup)^{\otimes n} \ar[r]_-{\widetilde{\rho}^{\otimes g} \otimes \mathrm{id}^{\otimes n}} & \mathrm{End}_{\mathbb{C}(\qD)}( \mathcal{O}_q )^{\otimes g} \otimes (U_q^\Pup)^{\otimes n}
} \]
meaning that the usual Alekseev morphism $\Phi_{g,n}$ from \S\ref{subsecAlekseev} and its modified version $\widehat{\Phi}_{g,n}$ yield the same families of representations.}
\end{remark}

\indent The constructions above admit integral versions, \textit{i.e.} over $\AD = \mathbb{C}[\qD^{\pm 1}]$. Indeed, using the fact that $\mathcal{O}_A$ is a Hopf $A$-algebra and the integral versions $\Phi^\pm_A$ of $\Phi^\pm$ (\S\ref{integralO}), we get an integral version $\Phi_{0,1}^A : \mathcal{L}_{0,1}^A \to U_A^\Pup$ of $\Phi_{0,1}$. The following result is the analog over $A$ of the equality \eqref{relatingLfAndUq} from \cite[\S 6.4]{JL}. 
\begin{lem}\label{lemmaJLforUA}
Denote by $U'_A$ the image of $\Phi_{0,1}^A$ and $U_A^{\mathrm{lf}} := U_q^{\mathrm{lf}} \cap U_A^\Pup$. Given a family $(s_j)$ of representatives for $\mathbb{T}/\mathbb{T}_2$, we have
\[ \textstyle U_A^\Pup = \bigoplus_{j=1}^{2^m} \bigl( U'_A \mathbb{T}_{2-}^{-1} \bigr)s_j = \bigoplus_{j=1}^{2^m} \bigl( U_A^{\rm lf} \mathbb{T}_{2-}^{-1} \bigr)s_j. \]
This can also be written as $U_A^\Pup = U'_A \otimes_{A[\mathbb{T}_{2-}]} A[\mathbb{T}]$.
\end{lem}
\begin{proof}
Recall the pivotal element $\ell = K_{2\rho}$. We have $l^{-1} \in U'_A$ but $l \not\in U'_A$. Consider the subalgebra $U_A'[\ell] \subset U_A$ generated by $U_A'$ and $\ell$, which can also be seen as the localization of $U_A'$ at $\ell^{-1}$. We claim that $U'_A[\ell] = \mathbb{T}_{2-}^{-1}U_A'\mathbb{T}_{2-}^{-1}$. Indeed, we have $K_{-2\lambda} \in U_A'$ for all $\lambda \in P_+$ \cite[Prop.\,2.24]{BR2} so in particular $L_i^{-2} \in U_A'$ for all $1 \leq i \leq m$. Since $\ell = \textstyle \prod_{i=1}^m L_i^2$ we deduce that $L_i^2 \in U_A'[\ell]$ for all $i$. Therefore $\mathbb{T}_{2-}^{-1} \subset U_A'[\ell]$, whence the claim.
Now we use that there is an element $d \in \mathcal{L}^A_{0,1}$ such that $\Phi_{0,1}(d) = \ell^{-1}$. The multiplicative subset
$\bigl\{ d^n \, \big\vert\, n \in \mathbb{N} \bigr\} \subset \mathcal{L}^A_{0,1}$ satisfies the Ore condition \cite[Cor.\,2.23]{BR2}, so we denote by $\mathcal{L}_{0,1}^A[d^{-1}]$ the resulting localization. It is known from \cite[Prop.\,2.24]{BR2} that $\Phi_{0,1}\bigl( \mathcal{L}_{0,1}^A [d^{-1}] \bigr) = U_A^{\mathrm{lf}}\mathbb{T}_{2-}^{-1}$. As a result
\[ U_A^{\mathrm{lf}}\mathbb{T}_{2-}^{-1} = \Phi_{0,1}\bigl( \mathcal{L}_{0,1}^A[d^{-1}] \bigr) = \Phi_{0,1}(\mathcal{L}_{0,1}^A)\bigl[ \Phi_{0,1}(d^{-1}) \bigr] = U'_A[\ell] = U'_A \mathbb{T}_{2-}^{-1}. \]
Finally by \cite[Prop.\,2.24]{BR2} we have $\textstyle U_A^\Pup = \bigoplus_{j=1}^{2^m} \bigl(U_A^{\mathrm{lf}}\,\mathbb{T}_{2-}^{-1}\bigr)s_j$, thus proving the lemma.
\end{proof}
\begin{remark}\label{rmkTechnicalPropUA}
{\rm 1. It follows from Cor.\,\ref{cor:dec25PhiA} that $U_A'= U_A^{\rm lf}$. This fact gives another proof of Lemma \ref{lemmaJLforUA} (but less elementary as it relies on the quantum Killing form).

2. Note that, because $L_i^{-2} \in U_A'$ for all $1 \leq i \leq m$, Lemma \ref{lemmaJLforUA} implies that the centralizer of $U_A'$ in $U_A^\Pup$ is equal to $\mathcal{Z}(U_A^\Pup)$. In particular, $\mathcal{Z}(U_A') \subset \mathcal{Z}(U_A^\Pup)$.}
\end{remark}

\indent Now consider the $\AD$-module $\mathcal{H}_A = \mathcal{O}_A \otimes_{\AD} U_A^\Pup \subset \mathcal{H}_q$, which is an integral form of $\mathcal{H}_q$. By the existence of $\Phi_{0,1}^A$, for all $h' \in U'_A \subset U'_q$, the element $\widehat{h'} \in \mathcal{H}_q$ defined in \eqref{elmtsChapeau} actually lies in $\mathcal{H}_A$. Since $K_{-2\lambda} \in U'_A$ for all $\lambda \in P_+$ \cite[Prop.\,2.24]{BR2} we have $\widehat{K_{-2\lambda}} \in \mathcal{H}_A$ and thus it makes sense to define $$\widehat{\mathcal{H}}_A := \mathcal{H}_A \otimes_{\AD[\mathbb{T}_{2-}]} \AD[\mathbb{T}]$$ which is an integral version of $\widehat{\mathcal{H}}_q$ in \eqref{defiHhatq}. Then by Lemma \ref{lemmaJLforUA} for all $h \in U_A^\Pup$ there is an element $\widehat{h} \in \widehat{\mathcal{H}}_A$ defined as in \eqref{elmtsChapeauHq}. It follows that restriction of \eqref{dressingMap2} to $U_A^\Pup$ gives an integral version $\widehat{\mathsf{D}}_{g,n}^{\AD} : U_A^\Pup \to \widehat{\mathcal{H}}_A^{\otimes g} \otimes (U_A^\Pup)^{\otimes n}$.

\indent When restricted to $\mathcal{L}_{0,1}^{\AD} \subset \mathcal{L}_{0,1}$ and $\mathcal{L}_{1,0}^{\AD} \subset \mathcal{L}_{1,0}$, the morphisms $\Phi_{0,1}$ and $\Phi_{1,0}$ take values in $U_A^\Pup$ and $\mathcal{H}_A$. We deduce by induction from the definition of $\widehat{\Phi}_{g,n}$ that it has an integral version $\widehat{\Phi}_{g,n}^{\AD} : \mathcal{L}_{g,n}^{\AD} \to \widehat{\mathcal{H}}_A^{\otimes g} \otimes (U_A^\Pup)^{\otimes n}$; it is obviously injective as restriction of $\widehat{\Phi}_{g,n}$ (Prop.\,\ref{propModifAlekseev}).

\smallskip

\indent Let $\widehat{\Phi}_{g,n}^{\e} :  \mathcal{L}_{g,n}^{\e} \to \widehat{\mathcal{H}}_{\e}^{\otimes g} \otimes (U_\e^\Pup)^{\otimes n}$ be the specialization of $\widehat{\Phi}_{g,n}^{\AD}$ at the root of unity $\eD$ (again recall from \eqref{choixED} that $\eD$ depends on $\e$, whence the notation).

\begin{prop}\label{AlekseevInjRootOf1}
The morphism of algebras $\widehat{\Phi}_{g,n}^{\e}$ is injective.
\end{prop}
\begin{proof}
By the same argument as in Lemma \ref{lemmaInjAlekseev}, it suffices to check the cases $\Phi_{0,1}^{\e}$ and $\Phi_{1,0}^{\e}$. Injectivity of $\Phi_{0,1}^{\e}$ is proved in \cite[Cor.\,2.25]{BR2} (eventually the dependence of $\Phi_{0,1}^{\e}$ on $\eD$ reduces to $\e$). For $\Phi_{1,0}^{\e}$, the proof of \cite[Th.\,3.13]{BFR} for $\Phi_{1,0}$ can be adapted verbatim, by using properly defined weight spaces $_\lambda(\mathcal{O}_\e)_\mu$ of $\Oo_\e$. The details are given in App. \ref{AppPhiinj}.
\end{proof}

\begin{remark}
{\rm The faithful representation $\rho : \mathcal{H}_q \to \mathrm{End}_{\mathbb{C}(\qD)}(\mathcal{O}_q)$ specializes to a representation $\rho_\e : \mathcal{H}_\e \to \mathrm{End}_{\mathbb{C}}(\mathcal{O}_\e)$ which is no longer faithful: if $h \in \mathcal{Z}_0(U_\e^\Q)$ then $\rho_\e(h)(\varphi) = \textstyle \sum_{(h)} \langle \varphi_{(1)}, h \rangle_\e^\Pup \, \varphi_{(2)} = \varepsilon(h)\varphi$ by \eqref{fundamentalDegeneracy}. This fact implies that $\rho_\e$ factorizes through $\mathcal{O}_\e \,\otimes \, u_\e^\Q$, where $u_\e^\Q$ is the small quantum group defined in \S\ref{specializUq}. Also the morphism $F_{g,n}$ defined in \eqref{defFgn} is no longer injective when $\qD$ is specialized to $\eD$.}
\end{remark}

\subsection{Quantum moment map and invariant elements}\label{secQMMandInv}
Recall that $\Phi_{0,1} : \mathcal{L}_{0,1} \to U_q^\Pup$ is injective and we denote its image by $U'_q$. Then as in \eqref{defQMMGen} we have the quantum moment map
\[ \mu_{g,n} : U_q' \xrightarrow{\quad \Phi_{0,1}^{-1} \quad} \mathcal{L}_{0,1} \xrightarrow{\quad \mathfrak{d}_{g,n} \quad} \mathcal{L}_{g,n} \]
which is a morphism of $\mathbb{C}(\qD)$-algebras.
\begin{prop}\label{muinjq} $\mu_{g,n}$ is injective.
\end{prop}
\begin{proof} The morphism $\Phi_{1,0}$ is injective \cite[Th.\,3.13]{BFR}, so by Cor.\,\ref{rmkMu10Inj} it is enough to show that $\mathrm{coad}^r:U_q \to {\rm End}_{\mc(q)}(\Oo_q)$ is injective. Given an $U_q$-module $V$, denote by $\mathrm{Ann}(V)$ the set of elements $h\in U_q$ such that $h.V=\{0\}$. Since $\bigl( \mathcal{O}_q,\mathrm{coad}^r \bigr) = \textstyle\bigoplus_{\lambda \in P_+} V_\lambda \otimes V_\lambda^*$ as $U_q$-modules, we have to prove that $\textstyle \bigcap_{\lambda\in P_+} \mathrm{Ann}(V_\lambda \otimes V_\lambda^*)=\{0\}$. Given $\mu\in P_+$, if the vector space ${\rm Hom}_{U_q}(V_\lambda \otimes V_\lambda^*, V_\mu)\cong {\rm Hom}_{U_q}(V_\lambda , V_\mu\otimes V_\lambda)$ is non-zero then $\mu\in Q_+$. It is shown in \cite[Prop.\,17 and Def.\,5]{RS} that for generic $\lambda$ relative to $\mu$, this space is isomorphic to the $0$ ($=\lambda-\lambda$) weight subspace $V_\mu^0$ of $V_\mu$. Because $\mu \in Q_+$, we have $V_\mu^0\neq \{0\}$. Therefore, for every $\mu \in Q_+$ there exists $\lambda\in P_+$ such that ${\rm Hom}_{U_q}(V_\lambda \otimes V_\lambda^*, V_\mu)\ne \{0\}$. It follows that, if $\mathrm{coad}^r(h)=0$ then $h\in \textstyle \bigcap_{\mu\in Q_+} \mathrm{Ann}(V_\mu)$. By \cite[Lem.\,7.1.9]{Jos}, if we prove that $Q_+$ is Zariski dense in $\mathfrak{h}^*$ then $\textstyle \bigcap_{\mu\in Q_+} \mathrm{Ann}(V_\mu)=\{0\}$. Now, given $x\in U_q(\mathfrak{h})$, identified to $\mc(q)[\mathfrak{h}^*]$ in the usual way, if $x(\alpha)=0$ for every $\alpha\in Q_+$, then $x=0$, which implies that $Q_+$ is Zariski dense in $\mathfrak{h}^*$. This follows from the following basic fact: if $K$ is a field and $P\in K[X_1,\ldots,X_r]$ vanishes on a set $A_1\times \ldots \times A_r\subset K^r$ with all $A_i$ infinite, then $P=0$. Applying this to $K=\mc(q)$, $r=m$ (the rank of $\mathfrak{g}$) and $A_i = \mn \alpha_i$ for $i=1,\dots,m$, this concludes the proof.   
\end{proof}

The following result is the reason why we have introduced the modified Alekseev morphism $\widehat{\Phi}_{g,n}$ in \S\ref{subsecModifiedAlekseev}; compare with Remark \ref{remarqueProbleme}. Recall that $\widehat{\mathsf{D}}_{g,n}$ is defined in \eqref{dressingMap2}.

\begin{prop}\label{propAlekseevModifAndQMM}
It holds $\widehat{\Phi}_{g,n} \circ \mu_{g,n}(h') = \widehat{\mathsf{D}}_{g,n}(h')$ for all $h' \in U_q'$. In other words we have a commutative diagram
\[ \xymatrix@R=1.2em@C=3em{
\mathcal{L}_{0,1} \ar[r]^-{\mathfrak{d}_{g,n}} \ar[d]_-{\Phi_{0,1}} & \mathcal{L}_{g,n} \ar[d]^-{\widehat{\Phi}_{g,n}}\\
U'_q \ar[ur]^-{\mu_{g,n}} \ar[r]_-{\widehat{\mathsf{D}}_{g,n}}& \widehat{\mathcal{H}}_q^{\otimes g} \otimes (U_q^\Pup)^{\otimes n}
} \]
\end{prop}
\begin{proof}
For $g=0$ this is proven in \cite[Prop.\,6.18]{BR1}; note that in this case $\widehat{\Phi}_{0,n} = \Phi_{0,n}$ and $\widehat{\mathsf{D}}_{0,n}$ is the iterated coproduct $\Delta^{(n)}$. For $(g,n) = (1,0)$ we have $\widehat{\Phi}_{1,0}(\varphi) = \Phi_{1,0}(\varphi)$ for all $\varphi\in \Ll_{0,1}$, and thus the claim is proved by Prop.\,\ref{propCommDiag10}. It remains to do induction on $g$. For all $\varphi \in \mathcal{L}_{0,1}$ and with implicit summation for the coproduct we have
\begin{align*}
& \bigl(\widehat{\Phi}_{g+1,n} \circ \mathfrak{d}_{g+1,n})(\varphi) = \widehat{\Phi}_{g+1,n}\bigl( \mathfrak{d}_{1,0}(\varphi_{(1)}) \otimes \mathfrak{d}_{g,n}(\varphi_{(2)}) \bigr)\\
=\:&\widehat{\Phi}_{1,0}\bigl( \mathfrak{d}_{1,0}(\varphi_{(1)})_{[2]} \bigr) \otimes \widehat{\mathsf{D}}_{g,n}\bigl( \Phi^+\bigl(\mathfrak{d}_{1,0}(\varphi_{(1)})_{[1]} \bigr) \bigr) \widehat{\Phi}_{g,n}\bigl( \mathfrak{d}_{g,n}(\varphi_{(2)}) \bigr)\\
=\:& \widehat{\Phi}_{1,0}\bigl( \mathfrak{d}_{1,0}(\varphi_{(2)}) \bigr) \otimes \widehat{\mathsf{D}}_{g,n}\bigl( \Phi^+\bigl(S(\varphi_{(1)})\star \varphi_{(3)} \bigr) \bigr) \widehat{\Phi}_{g,n}\bigl( \mathfrak{d}_{g,n}(\varphi_{(4)}) \bigr)\\
=\:& \widehat{\mathsf{D}}_{1,0}\bigl( \Phi_{0,1}(\varphi_{(2)}) \bigr) \otimes \widehat{\mathsf{D}}_{g,n}\bigl[ \Phi^+\bigl(\varphi_{(1)}\star S(\varphi_{(3)}) \bigr)  \Phi_{0,1}(\varphi_{(4)}) \bigr]\\
=\:& \widehat{\mathsf{D}}_{1,0}\bigl( \Phi_{0,1}(\varphi_{(2)}) \bigr) \otimes \widehat{\mathsf{D}}_{g,n}\bigl[ \Phi^+(\varphi_{(1)})  \Phi^-(S(\varphi_{(3)})) \bigr]\\
=\:& \widehat{\mathsf{D}}_{1,0}\bigl( \Phi_{0,1}(\varphi)_{(1)} \bigr) \otimes \widehat{\mathsf{D}}_{g,n}\bigl( \Phi_{0,1}(\varphi)_{(2)}) \bigr) = \widehat{\mathsf{D}}_{g+1,n}\bigl( \Phi_{0,1}(\varphi) \bigr)
\end{align*}
where the first equality is by definition of $\mathfrak{d}_{g+1,n}$, the second is by \eqref{modifiedAlekseevValue}, the third uses that $\mathfrak{d}_{1,0}$ is a morphism of $U_q^\Pup$-modules $\mathcal{L}_{0,1} \to \mathcal{L}_{1,0}$ (see \eqref{pseudoQMMgnH}) hence it is a morphism of left $\mathcal{O}_q$-comodules meaning that $\mathfrak{d}_{1,0}(\alpha)_{[1]} \otimes \mathfrak{d}_{1,0}(\alpha)_{[2]} = \bigl( \alpha_{(1)} \star S(\alpha_{(3)}) \bigr) \otimes \mathfrak{d}_{1,0}(\alpha_{(2)})$ by \eqref{defCoactCoad}, the fourth uses the induction hypothesis, the fifth uses the definition of $\Phi_{0,1}$ in \eqref{RSDphi} and the fact that $\Phi^+$ is an algebra morphism to apply the antipode axiom, the sixth is by \eqref{coproduitSurPhi01} and the last is by definition of $\widehat{\mathsf{D}}_{g+1,n}$.
\end{proof}

\indent The restriction of $\Phi_{0,1}$ to the $A$-subalgebra $\mathcal{O}_A$ gives the integral version $\Phi_{0,1}^A : \mathcal{O}_A \hookrightarrow U_A^\Pup$. Similarly, $\mathfrak{d}_{g,n}$ defined in \S\ref{subsecQMMgeneral} has an integral version $\mathfrak{d}_{g,n}^{\AD} : \mathcal{L}_{0,1}^{\AD} \to \mathcal{L}_{g,n}^{\AD}$ because $\Oo_{A}$ is a Hopf $\AD$-algebra and $\mathcal{L}_{1,0}^{\AD}$ is an $\AD$-algebra (this last argument is for the ``$\mathfrak{d}_{1,0}$ part'' of the definition). Hence we have the integral version of the quantum moment map
\[ \mu_{g,n}^{\AD} : U_A' \xrightarrow{\quad (\Phi_{0,1}^A)^{-1} \quad} \mathcal{L}_{0,1}^{\AD} \xrightarrow{\quad \mathfrak{d}_{g,n}^{\AD} \quad} \mathcal{L}_{g,n}^{\AD} \]
where $U'_A$ is the image of $\Phi_{0,1}^A$. We denote its specialization by $U_\e'$, and let $\mu_{g,n}^{\e} : U'_{\e} \to \mathcal{L}_{g,n}^{\e}$ be the specialization of $\mu_{g,n}^{\AD}$.

\begin{teo}\label{centralizerMu}
1. The subalgebra of $\mathrm{coad}^r$-invariant elements $\mathcal{L}_{g,n}^{U_q}$ is the centralizer of $\mu_{g,n}(U_q')$ in $\mathcal{L}_{g,n}$.
\\2. The subalgebra of $\mathrm{coad}^r$-invariant elements $\mathcal{L}_{g,n}^{U_A}$ is the centralizer of $\mu_{g,n}^{\AD}(U_{A}')$ in $\mathcal{L}_{g,n}^{\AD}$.
\\3. The subalgebra of $\mathrm{coad}^r$-invariant elements $\mathcal{L}_{g,n}^{u_\e}$ under the action of $u_\e^\Q \subset \Gamma_\e^\Q$ is the centralizer of $\mu_{g,n}^{\e}(U_{\e}')$ in $\mathcal{L}_{g,n}^{\e}$.
\end{teo}
In particular note that item 3 implies:
\begin{cor} \label{ZinLue} We have $\mathcal{Z}(\mathcal{L}_{g,n}^\e) \subset \mathcal{Z}(\Ll_{g,n}^{u_\e})$.
\end{cor}
\begin{proof}[Proof of Theorem \ref{centralizerMu}.]
The proofs of the three items are completely similar. We note first thanks to the QMM property \eqref{mudef} that there is an obvious inclusion
\[ \bigl\{ \text{invariant elements for } \mathrm{coad}^r \bigr\} \subset \bigl\{ \text{elements commuting with } \mu_{g,n}(h') \text{ for all } h' \in H' \bigr\} \]
where $H$ means either $U_q$, $U_A$ or $U_\e$. In the root of unity case, the appearance of $u_\e$ instead of $U_\e$ is due to the degeneracy explained after \eqref{defSmallInvEps}. 

The converse inclusion, much harder, relies on the same arguments in all three cases; we discuss it in the case of item 3. Let $x \in \mathcal{L}_{g,n}^{\e}$ such that $x \mu_{g,n}^{\e}(h') = \mu_{g,n}^{\e}(h') x$ for all $h' \in U'_\e$. We consider $y = \widehat{\Phi}_{g,n}^{\e}(x) \in \widehat{\mathcal{H}}_\e^{\otimes g} \otimes (U_\e^\Pup)^{\otimes n}$, which by Prop.\,\ref{propAlekseevModifAndQMM} satisfies $y \widehat{\mathsf{D}}_{g,n}^{\e}(h') = \widehat{\mathsf{D}}_{g,n}^{\e}(h') y$ for all $h' \in U'_\e$, where $\widehat{\mathsf{D}}_{g,n}^{\e}$ is the specialization of the integral version $\widehat{\mathsf{D}}^{\AD}_{g,n}$ obtained after Lemma \ref{lemmaJLforUA}. The essential difference between $\mu_{g,n}^{\e}$ and $\widehat{\mathsf{D}}_{g,n}^{\e}$ is that $\widehat{\mathsf{D}}_{g,n}^{\e}$ is defined on the whole $U_\e^\Pup$ and not just on $U'_\e$. It is known \cite[Proof of Prop.\,2.24]{BR2} that
\begin{equation}\label{commutationYassumption}
K_{-2\lambda} \in U'_\e, \quad L_i^{-2}E_i \in U'_\e, \quad F_iK_iL_i^{-2} \in U'_\e \qquad (\forall \, \lambda \in P_+, \:\: \forall \, 1 \leq i \leq m).
\end{equation}
By assumption $y$ commutes with the images of these elements through $\widehat{\mathsf{D}}_{g,n}^{\e}$. Let us prove that $y$ actually commutes with $\widehat{\mathsf{D}}_{g,n}^{\e}(K_\lambda) = \bigl( \widehat{K_\lambda} \, K_\lambda \bigr)^{\otimes g} \otimes K_\lambda^{\otimes n}$ for all $\lambda \in P$. As an $A$-module, $\mathcal{H}_A = \mathcal{O}_A \otimes U_A^\Pup$. Since $\mathcal{O}_A$ has an $A$-basis $(\varphi_i)_{i \in \mathcal{I}}$ of weight vectors for the coadjoint action \cite[Th.\,2.8]{BR2} and $U_A^\Pup$ has an $A$-basis $(x_j)_{j \in \mathcal{J}}$ of weight vectors for the adjoint action, we have the $A$-basis $(\varphi_i x_j)_{i,j}$ which satisfies
\[ \forall \, \lambda \in P, \quad \widehat{K_{\lambda}}K_{\lambda}\varphi_i x_j = \bigl( K_\lambda \rhd \varphi_i \lhd K_{-\lambda} \bigr) K_\lambda x_j K_\lambda^{-1} \, \widehat{K_{\lambda}}K_{\lambda} = \qD^{D(\lambda,w_i+w_j)}\varphi_i x_j \widehat{K_{\lambda}}K_{\lambda} \]
where $w_i,w_j \in P$ are the weights of $\varphi_i$, $x_j$. It follows that $\mathcal{F} = \bigl( \varphi_i x_j \otimes_{\AD[\mathbb{T}_{2-}]} K_{2\mu}s_k \bigr)_{\mu \in P_+, i,j,k}$ is an $A$-basis of $\widehat{\mathcal{H}}_A$ whose elements $\qD$-commute with $\widehat{K_{\lambda}}K_{\lambda}$ for all $\lambda$, where we recall that $(s_k)_k$ is a family of representatives for $\mathbb{T}/\mathbb{T}_2$. Then $\mathcal{B} := \mathcal{F}^{\otimes g} \otimes \bigl[(x_j)_{j \in \mathcal{J}}\bigr]^{\otimes n}$ is a basis of $\widehat{\mathcal{H}}_A^{\otimes g} \otimes (U_A^\Pup)^{\otimes n}$ whose elements $\qD$-commute with $\widehat{\mathsf{D}}^{\AD}_{g,n}(K_{\lambda})$ for all $\lambda$. The same fact remains true after specialization at $\e$, with basis denoted by $\mathcal{B}_\e$. Assume that $b \in \mathcal{B}_\e$ satisfies $\widehat{\mathsf{D}}^{\e}_{g,n}(K_{2\lambda})\,b\,\widehat{\mathsf{D}}^{\e}_{g,n}(K_{-2\lambda}) = b$ for all $\lambda \in P_+$. By definition of $\mathcal{B}_\e$, there exists $\mu \in P$ such that $\widehat{\mathsf{D}}^{\e}_{g,n}(K_{\lambda})\,b\,\widehat{\mathsf{D}}^{\e}_{g,n}(K_{-\lambda}) = \eD^{D(\lambda,\mu)}b$ for all $\lambda$. It thus satisfies $2D(\lambda,\mu) \equiv 0 \pmod{l}$. Since $\mathrm{gcd}(D,l) = \mathrm{gcd}(2,l) = 1$ by assumption on $\e$, it holds $(\lambda,\mu) \equiv 0 \pmod{l}$ and thus $\widehat{\mathsf{D}}(K_{\lambda})$ commutes with $b$ for all $\lambda$. The conclusion is that if an element in $\widehat{\mathcal{H}}_\e^{\otimes g} \otimes (U_\e^\Pup)^{\otimes n}$ commutes with $\widehat{\mathsf{D}}^{\e}_{g,n}(K_{-2\lambda})$ for all $\lambda\in P_+$ then it commutes with $\widehat{\mathsf{D}}^{\e}_{g,n}(K_{\lambda})$ for all $\lambda\in P$. We deduce from \eqref{commutationYassumption} that $y = \widehat{\Phi}_{g,n}^{\e}(x)$ commutes with $\widehat{\mathsf{D}}_{g,n}^{\e}(K_{\lambda})$, $\widehat{\mathsf{D}}^{\e}_{g,n}(E_i)$ and $\widehat{\mathsf{D}}^{\e}_{g,n}(F_i)$ for all $\lambda$ and all $i$. Hence $y$ commutes pointwise with the whole $\widehat{\mathsf{D}}^{\e}_{g,n}(U_\e^\Pup)$. For all $h \in U_\e^\Pup$, Proposition \ref{propModifAlekseev} thus gives
\[ \widehat{\Phi}_{g,n}^{\e}\bigl( \mathrm{coad}^r(h)(x) \bigr) = \sum_{(h)} \widehat{\mathsf{D}}^{\e}_{g,n}\bigl( S(h_{(1)}) \bigr) \, y \, \widehat{\mathsf{D}}^{\e}_{g,n}(h_{(2)}) = \varepsilon(h)y = \widehat{\Phi}_{g,n}^{\e}\bigl( \varepsilon(h)x \bigr). \]
Since $\widehat{\Phi}_{g,n}^{\e}$ is injective (Prop.\,\ref{AlekseevInjRootOf1}) we conclude that $x$ is invariant under $\mathrm{coad}^r$.
\end{proof}

\begin{remark}
{\rm 1. The centralizer of $\mu_{g,n}(U_{\AD}')$ is the same thing as the centralizer of $\mathfrak{d}_{g,n}^{\AD}(\mathcal{L}_{0,1}^{\AD})$, with $\mathfrak{d}_{g,n}^{\AD}$ defined before Theorem \ref{centralizerMu}. In particular, for $g=0$, $\mathfrak{d}_{0,n} : \mathcal{L}_{0,1}^{\AD} \to \mathcal{L}_{0,n}^{\AD}$ is the iterated coproduct used in \cite[\S 6.5]{BR1} and Th.\,\ref{centralizerMu} recovers \cite[Prop.\,6.19]{BR1}.
\\2. Note that the difficulty in proving Th.\,\ref{centralizerMu} is due to the fact that $\mu_{g,n}$ is not defined on the whole $U_q^\Pup$ but only on $U'_q$. Otherwise it would suffice to rewrite \eqref{mudef} as $\mathrm{coad}^r(h)(x) = \textstyle \sum_{(h)} \mu_{g,n}(S(h_{(1)}))x\mu_{g,n}(h_{(2)})$ to conclude that elements in the centralizer of $\mu_{g,n}$ are invariant. But this equality does not make sense because $\Delta(U_q') \subset U_q' \otimes U_q^\Pup$ and $S(U'_q) \not\subset U'_q$. Fortunately this equality makes sense through $\widehat{\Phi}_{g,n}$ with $\widehat{\mathsf{D}}_{g,n}$ instead of $\mu_{g,n}$, which is the main idea of our proof (and the reason why we had to introduce $\widehat{\Phi}_{g,n}$ instead of using $\Phi_{g,n}$).}
\end{remark}

\subsection{Structure results for \texorpdfstring{$\mathcal{L}_{g,n}^\e$}{graph algebras at roots of unity}}\label{subsecStructResLgnEps}
We continue to make the assumptions \eqref{assumptionl} and \eqref{choixED} for the root of unity $\epsilon$. Recall that $\mathcal{Z}_0(\Oo_\e) \subset \Oo_\e$ is the image of the morphism of Hopf algebras $\mathbb{F}\mathrm{r}^*_\e : \mathcal{O}(G) \to \Oo_\e(G)$ which is dual to the Frobenius morphism $\mathbb{F}\mathrm{r}_\e : \Gamma_\e^\Q(\mathfrak{g}) \to U(\mathfrak{g})$; see \S\ref{subsecDefFrobGamma} and \S\ref{specialisOA}. Since $\mathcal{L}_{g,n}^\e$ is $\Oo_\e^{\otimes (2g+n)}$ as a $\mathbb{C}$-vector space, the following definition makes sense.
\begin{defi}\label{defZ0Lgn}
We let $\mathcal{Z}_0(\mathcal{L}_{g,n}^\e) = \mathcal{Z}_0(\Oo_\e)^{\otimes (2g+n)} \subset \mathcal{L}_{g,n}^\e$. 
\end{defi}
\indent The vector space $\Oo_\e$ is a $\mathbb{C}$-algebra, as specialization of the $A$-algebra $\mathcal{O}_A$. Hence $\Oo_\e^{\otimes (2g+n)}$ can be endowed with the usual structure on a tensor product of algebras; we denote this product by $x \star y$ for all $x,y \in \Oo_\e^{\otimes (2g+n)}$. But $\Oo_\e^{\otimes (2g+n)}$ has also the product of $\mathcal{L}_{g,n}^\e$, which we denote by $xy$.
\begin{prop}\label{Z0Lgn}
For all $x \in \mathcal{Z}_0(\mathcal{L}_{g,n}^\e)$ and $y \in \mathcal{L}_{g,n}^\e$ we have $xy = yx = x \star y$.
\\It follows that:

\noindent (1) $\mathcal{Z}_0(\Ll_{g,n}^{\e})$ is a central subalgebra of $\Ll_{g,n}^{\e}$, isomorphic to $\Oo(G)^{\otimes (2g+n)}$ through $(\mathbb{F}\mathrm{r}_\e^{*})^{\otimes (2g+n)}$. In particular $\mathcal{Z}_0(\Ll_{g,n}^{\e})$ is a Noetherian ring.

\noindent (2) As $\mathcal{Z}_0(\Ll_{g,n}^{\e})$-modules, $\Ll_{g,n}^{\e}$ and $\Oo_{\e}^{\otimes (2g+n)}$ coincide. 

\noindent (3) $\Ll_{g,n}^{\e}$ is a free $\mathcal{Z}_0(\Ll_{g,n}^{\e})$-module of rank $l^{(2g+n)\dim(\mathfrak{g})}$ and is a Noetherian ring, where $l$ is the order of the root of unity $\e$. In particular $\mathcal{Z}(\Ll_{g,n}^{\e})$ is a Noetherian ring.
\end{prop}
\begin{remark}\label{Z0submoduleLgn} {\rm It follows from item 3 that any subalgebra of $\Ll_{g,n}^\e$ containing  $\mathcal{Z}_0(\Ll_{g,n}^{\e})$ is a finitely generated $\mathcal{Z}_0(\Ll_{g,n}^{\e})$-submodule, and hence a Noetherian ring (since $\mathcal{Z}_0(\Ll_{g,n}^{\e})$ is a Noetherian ring). We will apply this to $\Ll_{g,n}^{u_\e}$ and its center $\mathcal{Z}(\Ll_{g,n}^{u_\e})$ in the proof of Th.\,\ref{mainDCKPsec6}.}
\end{remark}
\begin{proof}
First we have $x \star y = y \star x$ because $\mathcal{Z}_0(\Oo_\e) \subset \mathcal{Z}(\Oo_\e)$. Let us show that $xy = x \star y$. Consider the embeddings $\mathfrak{i}_{b_i}, \mathfrak{i}_{a_i}, \mathfrak{i}_{m_{g+j}} : \Oo_\e \to \Oo_\e^{\otimes (2g+n)}$ as in \eqref{labelCopiesHcirc}, and from which the product of $\mathcal{L}_{g,n}^\e$ is defined by the rules in \S\ref{subsecDefLgnH}. Note that the product of $\Oo_\e^{\otimes (2g+n)}$ is characterized by the rules:
\begin{align*}
&\varphi_1 \otimes \ldots \otimes \varphi_{2g+n} = \mathfrak{i}_{b_1}(\varphi_1) \star \mathfrak{i}_{a_1}(\varphi_2) \star \ldots \star \mathfrak{i}_{b_g}(\varphi_{2g-1}) \star \mathfrak{i}_{a_g}(\varphi_{2g})\\
&\qquad\qquad\qquad\qquad\:\:\star \mathfrak{i}_{m_{g+1}}(\varphi_{2g+1}) \star \ldots \star \mathfrak{i}_{m_{g+n}}(\varphi_{2g+n}),\\
&\mathfrak{i}_s(\varphi)\star \mathfrak{i}_s(\psi) = \mathfrak{i}_s(\varphi \star \psi), \qquad \mathfrak{i}_s(\varphi) \star \mathfrak{i}_t(\psi) = \mathfrak{i}_t(\psi) \star \mathfrak{i}_s(\varphi)
\end{align*}
for all $s,t \in \bigl\{ b_1,a_1, \ldots, b_g,a_g,m_{g+1}, \ldots, m_{g+n} \bigr\}$ with $s\neq t$. So it suffices to prove the result for every couple $(s,t)$, taking $x = \mathfrak{i}_s(\varphi)$ with $\varphi \in \mathcal{Z}_0(\Oo_\e)$ and $y = \mathfrak{i}_t(\psi)$ with $\psi \in \Oo_\e$. For instance, in the case $s=t$ we apply \eqref{fusionRelL01} and get:
\begin{align*}
&\mathfrak{i}_s(\varphi)\mathfrak{i}_s(\psi) = \sum_{(\varphi),(\psi)}\mathcal{R}_{\e}\bigl( \psi_{(1)} \otimes S(\varphi_{(3)}) \bigr) \mathcal{R}_{\e}\bigl( \psi_{(3)} \otimes \varphi_{(2)} \bigr) \, \mathfrak{i}_s\bigl(\varphi_{(1)} \star \psi_{(2)} \bigr)\\
=\:&\sum_{(\varphi),(\psi)} \varepsilon_{\Oo_\e}(\psi_{(1)}) \varepsilon_{\Oo_\e}\bigl( S(\varphi_{(3)}) \bigr) \varepsilon_{\Oo_\e}(\psi_{(3)}) \varepsilon_{\Oo_\e}(\varphi_{(2)})\mathfrak{i}_s\bigl(\varphi_{(1)} \star \psi_{(2)} \bigr) = \mathfrak{i}_s(\varphi \star \psi)  = \mathfrak{i}_s(\varphi) \star \mathfrak{i}_s(\psi)
\end{align*}
In the the second equality we used the fact that $\mathcal{Z}_0(\Oo_\e)$ is a Hopf subalgebra (and hence $S(\varphi_{(3)}), \varphi_{(2)} \in \mathcal{Z}_0(\Oo_\e)$) and Cor.\,\ref{relationsCoRMatAtEps}(2). Similarly, it is readily seen that the two exchange relations \eqref{L01prodsept25}-\eqref{braidedTensProdComm} in $\mathcal{L}_{g,n}^\e$ reduce to the desired commutativity relation between tensorands in $\Oo_\e^{\otimes (2g+n)}$.
\\(1) The first claim is due to the preliminary property and the fact that $\mathbb{F}\mathrm{r}_\e^*$ is injective. For the second claim,  note that $\mathcal{O}(G)$ is finitely generated as an algebra (by matrix coefficients of irreducible modules with highest weight $\varpi_i$, $i=1,\ldots,m$). Hence $\mathcal{O}(G)^{\otimes (2g+n)}$ is also a finitely generated commutative algebra, thus it is Noetherian by Hilbert's basis theorem.
\\(2) is obvious from the preliminary property.
\\(3) It is known that $\Oo_\e$ is a free $\mathcal{Z}_0(\Oo_\e)$-module of rank $l ^{\dim(\mathfrak{g})}$ (Th.\,\ref{DCLteo1} (2)). Hence, freeness and rank of $\mathcal{L}_{g,n}^\e$ over $\mathcal{Z}_0(\mathcal{L}_{g,n}^\e)$ are deduced from item 2. Noetherianity is due to this classical fact: if a ring $R$ is finitely generated as a module over a subring $S$ and if $S$ is a Noetherian ring, then $R$ is a Noetherian ring \cite[Prop.\,7.2]{AMacD}; apply with $R = \mathcal{L}_{g,n}^\e$ and $S = \mathcal{Z}_0(\mathcal{L}_{g,n}^\e)$. The same argument applies to  $\mathcal{Z}(\Ll_{g,n}^{\e})$.
\end{proof}

Now recall that $u_\e^\Q$ denotes the {\em small quantum group} (Def.\,\ref{defSmallUeps}) and that the inclusion $u_\e^\Q \subset \Gamma_\e^\Q$ gives by duality a surjection $\pi : \Oo_\e \to (u_\e^\Q)^*$ defined in \eqref{piPdef}. Since $u_\e^\Q$ is a finite-dimensional quasitriangular Hopf algebra, we have the algebra $\mathcal{L}_{g,n}(u_\e^\Q)$ obtained by taking $H = u_\e^\Q$ in \S\ref{subsecDefLgnH}. We have $H^\circ = (u_\e^\Q)^*$ by finite-dimensionality and the co-R-matrix is $\varphi \otimes \psi \mapsto (\varphi \otimes \psi)(\overline{R})$, where $\overline{R}$ is the $R$-matrix of $u_\e^\Q$ (\S\ref{subsecRmatSmallQG}). As a $\mathbb{C}$-vector space $\mathcal{L}_{g,n}(u_\e^\Q) = [(u_\e^\Q)^* ]^{\otimes (2g+n)}$, so the following result makes sense:

\begin{prop}\label{piMorphBetweenLgn}
$\pi^{\otimes (2g+n)} : \mathcal{L}_{g,n}^\e \to \mathcal{L}_{g,n}(u_\e^\Q)$ is a surjective morphism of algebras.
\end{prop}
\begin{proof}
Immediate from Lemma \ref{lemmaLgnIsFunctorial} and Corollary \ref{relationsCoRMatAtEps}(1).
\end{proof}

Let $\Oo^+(G) = \ker(\varepsilon) \cap \Oo(G)$ be the augmentation ideal of $\mathcal{O}(G)$, which consists of the functions $G \to \mathbb{C}$ vanishing at the unit element of $G$. 
\begin{teo}\label{centralextLgn} We have an exact sequence of algebras\footnote{The first term in this sequence is a non-unital algebra; exactness at the middle term means that ${\rm Ker}(\pi^{\otimes (2g+n)})$ is the ideal $(\mathbb{F}\mathrm{r}_\e^{*}(\Oo^+(G))^{\otimes (2g+n)}\mathcal{L}_{g,n}^{\e}$. This follows in part the notion of exact sequence of Hopf algebras introduced, e.g., in \cite{AD}.}
\begin{equation}\label{exactsequenceue}
0 \longrightarrow \Oo^+(G)^{\otimes (2g+n)} \xrightarrow{\quad(\mathbb{F}\mathrm{r}_\e^{*})^{\otimes (2g+n)}\quad} \mathcal{L}_{g,n}^{\e} \xrightarrow{\quad \pi^{\otimes (2g+n)} \quad} \mathcal{L}_{g,n}(u_\epsilon^\Q) \longrightarrow 0.
\end{equation}
Said differently, $\mathcal{L}_{g,n}^{\epsilon}$ is a central extension of the $l^{(2g+n)\dim(\mathfrak{g})}$-dimensional algebra $\mathcal{L}_{g,n}(u_\epsilon^\Q)$ by $\Oo^+(G)^{\otimes (2g+n)}$, where $l$ is the order of the root of unity $\e$.
\end{teo}
\begin{proof}
This is an immediate consequence of Prop.\,\ref{pipropO}, Prop.\,\ref{Z0Lgn}(1) and Prop.\,\ref{piMorphBetweenLgn}.
\end{proof}

\begin{prop}\label{propLgndomain}
The algebra $\mathcal{L}_{g,n}^{\e}$ is a domain. 
\end{prop}
\begin{proof}
The proof is contained in App.\,\ref{Lgndomain}, as it consists in an easy adaptation of arguments from \cite{BFR}. \end{proof}
\noindent Since $\mathcal{L}_{g,n}^{\e}$ is a domain, we can consider the central localization (see \S\ref{subsecCentralLoc})
\begin{equation}\label{QLe}
Q(\Ll_{g,n}^\e) := Q\bigl( \mathcal{Z}(\Ll_{g,n}^\e) \bigr)  \otimes_{\mathcal{Z}(\Ll_{g,n}^\e)} \Ll_{g,n}^\e.
\end{equation}
Note that since $\mathcal{L}_{g,n}^\e$ is finitely generated and free as $\mathcal{Z}_0(\mathcal{L}_{g,n}^\e)$-module (Prop.\,\ref{Z0Lgn}), Lemma \ref{lemTrucsGenerauxQR}(2) allows us to rewrite
\[ Q(\Ll_{g,n}^\e) \cong Q\bigl( \mathcal{Z}_0(\Ll_{g,n}^\e) \bigr)  \otimes_{\mathcal{Z}_0(\Ll_{g,n}^\e)} \Ll_{g,n}^\e. \]
Moreover, by Lemma \ref{lemQRCSA}, $Q(\Ll_{g,n}^\e)$ is a division algebra; it is thus in particular a central simple algebra and we wish to compute its PI-degree. Some technical preliminaries are in order.

\begin{prop}\label{irrep} $\mathcal{L}_{g,n}^{\e}$ has an irreducible representation of dimension $l^{g\,\mathrm{dim}(\mathfrak{g})+nN}$, where $l$ is the order of $\e$ and $N$ is the number of positive roots of $\mathfrak{g}$.
\end{prop}
\begin{proof}
Lyubashenko proved that the small quantum group $u_\e^\Q$ is factorizable when the order of $\e$ is odd  \cite[Cor.\,A.3.3]{Lyub95}. Hence by Prop.\,\ref{coroLgnFacto} we have an isomorphism of algebras
\begin{equation}\label{isoAlekseevSmallUq}
\mathcal{L}_{g,n}(u_\epsilon^\Q) \cong \mathrm{End}_{\mathbb{C}}\bigl( (u_\e^\Q)^* \bigr)^{\otimes g} \otimes (u_\e^\Q)^{\otimes n}.
\end{equation}
Moreover, it is known that $u_\e^\Q$ has an irreducible representation $V$ of dimension $l^N$ \cite[\S III.6.6]{BG}. Of course, there is the natural irreducible representation of $\mathrm{End}_{\mathbb{C}}\bigl( (u_\e^\Q)^* \bigr)$ on $(u_\e^\Q)^*$ which has dimensional $l^{\dim(\mathfrak{g})}$. Over an algebraically closed field, the external tensor product of irreducible representations is irreducible \cite[\S 3.10]{Et}. Hence \eqref{isoAlekseevSmallUq} gives an irreducible representation of $\mathcal{L}_{g,n}(u_\epsilon^\Q)$ on $\bigl( (u_\e^\Q)^* \bigr)^{\otimes g} \otimes V^{\otimes n}$, with desired dimension. The surjection $\mathcal{L}_{g,n}^\e \twoheadrightarrow \mathcal{L}_{g,n}(u_\e^\Q)$ from Th.\,\ref{centralextLgn} transports it to an irreducible representation of $\mathcal{L}_{g,n}^\e$.
\end{proof} 

\indent Denote by $[R:F] \in \mathbb{N}^* \cup \{\infty\}$ the dimension of a ring $R$ over a central subfield $F$.

\begin{lem}\label{lemmaSupMultDim}
Let $\Bbbk$ be a field, and $R',R''$ be commutative $\Bbbk$-algebras which are domains, and are such that the $\Bbbk$-algebra $R' \otimes_\Bbbk R''$ is a domain.
\\1. There is a natural inclusion $e : Q(R') \otimes_\Bbbk Q(R'') \hookrightarrow Q(R' \otimes_\Bbbk R'')$.
\\2. For any  subalgebras $S' \subset R'$ and $S'' \subset R''$, if $\bigl(r'_i\bigr)_i \subset Q(R')$ is free over $Q(S')$ and $\bigl(r''_j\bigr)_j \subset Q(R'')$ is free over $Q(S'')$ then $\bigl(e(r'_i \otimes_\Bbbk r''_j) \bigr)_{i,j} \subset Q(R' \otimes_\Bbbk R'')$ is free over $Q(S' \otimes_\Bbbk S'')$. In particular
\[ \bigl[ Q(R' \otimes_\Bbbk R'') : Q(S' \otimes_\Bbbk S'') \bigr] \geq \bigl[ Q(R') : Q(S') \bigr] \, \bigl[ Q(R'') : Q(S'') \bigr]. \]
\end{lem}
\begin{proof}
1. It is straightforward to check that the map $Q(R') \times Q(R'') \to Q(R' \otimes_\Bbbk R'')$ given by $\Bigl(\frac{p'}{q'}, \frac{p''}{q''} \Bigr) \mapsto \frac{p' \otimes p''}{q' \otimes q''}$ is well-defined and $\Bbbk$-bilinear. It thus yields the $\Bbbk$-linear map $e$ which satisfies $e\Bigl(\frac{p'}{q'} \otimes \frac{p''}{q''}\Bigr) = \frac{p' \otimes p''}{q' \otimes q''}$ and is readily seen to be a morphism of $\Bbbk$-algebras. Let us prove that $e$ is injective. Take an element $x = \sum_{c} \frac{p'_c}{q'_c} \otimes \frac{p''_c}{q''_c} \in Q(R') \otimes_\Bbbk Q(R'')$ and assume $e(x)=0$. Set $d = \textstyle \prod_c q'_c \otimes q''_c$ so that $dx \in R' \otimes_\Bbbk R'' \subset Q(R') \otimes_\Bbbk Q(R'')$. Then $ 0 = e(d)e(x) = e(dx) = dx$. But since $d$ is invertible in $Q(R') \otimes_\Bbbk Q(R'')$, we have $d^{-1}dx = x = 0$.
\\2. Assume that a finite sum vanishes: $\textstyle \sum_{i,j} f_{i,j} e(r'_i \otimes r''_j) = 0$, with $f_{i,j} \in Q(S' \otimes_\Bbbk S'')$. Collecting all denominators in the $f_{i,j}$'s, we get a non-zero element $P \in S' \otimes_\Bbbk S''$ such that $Pf_{i,j} \in S' \otimes_\Bbbk S''$ for all $i,j$. Let $(s'_k)$ be a basis of $S'$ over $\Bbbk$; then $(s'_k)$ is a free family over $\Bbbk$ in $Q(S')$ and it follows that $\bigl( s'_k r'_i \,\bigr)_{i,k}$ is a free family over $\Bbbk$ in $Q(R')$. Writing $Pf_{i,j} = \textstyle \sum_k s'_k \otimes x^{i,j}_k \in S' \otimes_\Bbbk S''$ and using that $e(Pf_{i,j})$ is $Pf_{i,j}$ viewed in $Q(R' \otimes_\Bbbk R'')$, we find
\[ \textstyle 0 = \sum_{i,j} Pf_{i,j} e(r'_i \otimes r''_j) = e\left(\sum_{i,j,k} s'_k r'_i \otimes x^{i,j}_k r''_j \right) = e\left( \sum_{i,k} s'_k r'_i \otimes \Bigl( \sum_j x^{i,j}_k r''_j \Bigr) \right).  \]
Since $e$ is injective, the freeness of $\bigl( s'_k r'_i \bigr)_{i,k}$ over $\Bbbk$ implies $\textstyle \sum_j x^{i,j}_k r''_j = 0$ for all $i,k$. But since $(r''_j)$ is free over $Q(S'')$, it follows in turn that $x^{i,j}_k=0$ for all $i,j,k$. Hence $Pf_{i,j} = 0$ for all $i,j$ and since $P \neq 0$ we conclude that $f_{i,j} = 0$ for all $i,j$.
\end{proof}

\begin{teo}\label{Lgnteo1}
The algebra $Q(\Ll_{g,n}^{\e})$ is a division algebra, and a central simple algebra of PI-degree $l^{g\,\mathrm{dim}(\mathfrak{g})+nN}$.
\end{teo}
\begin{proof}
We already noted that it is a division algebra and a central simple algebra. Denote by $\mathfrak{r}\in \mn$ its PI-degree; so $Q(\Ll_{g,n}^{\e})$ has dimension $\mathfrak{r}^2$ over $Q\bigl(\mathcal{Z}(\Ll_{g,n}^{\e})\bigr)$, and any irreducible representation of $\Ll_{g,n}^{\e}$ has dimension $\leq \mathfrak{r}$ by Th.\,\ref{unicityThm}(2). Hence Prop.\,\ref{irrep} gives $\mathfrak{r} \geq l^{g\,\mathrm{dim}(\mathfrak{g})+nN}$. 

To establish the other inequality, we make use of a central subalgebra $\hat{\mathcal{Z}_0}(\mathcal{L}_{g,n}^\e)$ larger than $\mathcal{Z}_0(\mathcal{L}_{g,n}^\e)$. In order to define it, recall the $\Gamma_A^\Q$-module $_AV_{\lambda}$ which is a free $A$-lattice in the highest weight $U_q^\Q$-module $V_\lambda$ for all $\lambda \in P_+$. Take an $A$-basis $(x_i)$ of weight vectors, with dual basis $(x^i)$, and let
\begin{equation}\label{casimirElmtsL01}
\textstyle \mathrm{qTr}_\lambda : \Gamma_A^\Q \to A, \quad h \mapsto \sum_i x^i\bigl( \ell h \cdot x_i \bigr) = \mathrm{Tr}\bigl( \mathrm{rep}_\lambda(\ell\,h) \bigr)
\end{equation}
where $\ell \in U_A^\Q \subset \Gamma_A^\Q$ is the pivotal element and $\mathrm{rep}_\lambda : \Gamma_A^\Q \to \mathrm{End}_A(_AV_\lambda)$ is the representation morphism. Then $\mathrm{qTr}_\lambda \in \mathcal{O}_A$, as an $A$-linear combination of matrix coefficients of $_AV_\lambda$. Moreover, $\mathrm{qTr}_\lambda$ is invariant under the coadjoint action \eqref{coadUqOq} of $\Gamma_A^\Q$ by cyclicity of the trace and the pivotal property; said differently $\mathrm{qTr}_\lambda \in \mathcal{L}_{0,1}^{\Gamma_A}$. As usual denote by $(\mathrm{qTr}_\lambda)_{|\e}$ the specialization of $\mathrm{qTr}_\lambda$ to $q=\e$. Now in that specialization let
\[ \hat{\mathcal{Z}}_0(\mathcal{L}_{0,1}^\e) := \mathcal{Z}_0(\mathcal{L}_{0,1}^\e)\bigl[ (\mathrm{qTr}_\lambda)_{|\e} \,\big|\, \lambda \in P_+ \bigr] \]
be the subalgebra obtained by extending $\mathcal{Z}_0(\mathcal{L}_{0,1}^\e)$ with the specializations of the elements $\mathrm{qTr}_\lambda$. Then, taking advantage of the vector space decomposition $\mathcal{L}_{g,n}^\e = \mathcal{L}_{g,0}^\e \otimes (\mathcal{L}_{0,1}^\e)^{\otimes n}$ we consider the subspace
\begin{equation}\label{defHatZ0Lgn}
\hat{\mathcal{Z}_0}(\mathcal{L}_{g,n}^\e) := \mathcal{Z}_0(\mathcal{L}_{g,0}^\e) \otimes \hat{\mathcal{Z}}_0(\mathcal{L}_{0,1}^\e)^{\otimes n} =  \mathcal{Z}_0(\mathcal{L}_{g,n}^\e)\bigl[ \mathfrak{i}_{g+j}(\mathrm{qTr}_\lambda)_{|\e} \,\big|\, \lambda \in P_+, \: 1 \leq j \leq n \bigr].
\end{equation}
By Prop.\,\ref{Z0Lgn}(1) and Lemma \ref{Z1general}(ii) it is a central subalgebra of $\mathcal{L}_{g,n}^\e$.
Of course, the inclusions $\mathcal{Z}_0(\mathcal{L}_{g,n}^\e) \subset \hat{\mathcal{Z}}_0(\mathcal{L}_{g,n}^\e) \subset \mathcal{Z}(\mathcal{L}_{g,n}^\e)$  still hold true at the level of fraction fields.  Moreover, it is known that $\bigl[ Q\bigl(\hat{\mathcal{Z}}_0(\mathcal{L}_{0,1}^\e)\bigr) : Q\bigl(\mathcal{Z}_0(\mathcal{L}_{0,1}^\e)\bigr) \bigr] = l^m$ by \cite[Rmk.\,5.3]{BR2}\footnote{The proof of \cite[Th.\,5.2]{BR2}, which is used in \cite[Rmk.\,5.3]{BR2}, implicitly uses Cor.\,\ref{Phi0Aiso} of the present paper.}. Hence
\begin{align}
\begin{split}\label{majorationDegZoverZ0}
&\bigl[ Q\bigl(\mathcal{Z}(\mathcal{L}_{g,n}^\e)\bigr) : Q\bigl(\mathcal{Z}_0(\mathcal{L}_{g,n}^\e)\bigr) \bigr]\\
=\:& \bigl[ Q\bigl(\mathcal{Z}(\mathcal{L}_{g,n}^\e)\bigr) : Q\bigl(\hat{\mathcal{Z}}_0(\mathcal{L}_{g,n}^\e)\bigr) \bigr] \bigl[ Q\bigl(\hat{\mathcal{Z}}_0(\mathcal{L}_{g,n}^\e)\bigr) : Q\bigl(\mathcal{Z}_0(\mathcal{L}_{g,n}^\e)\bigr) \bigr]\\
\geq\:&\bigl[ Q\bigl(\hat{\mathcal{Z}}_0(\mathcal{L}_{g,n}^\e)\bigr) : Q\bigl(\mathcal{Z}_0(\mathcal{L}_{g,n}^\e)\bigr) \bigr]\\
=\:&\bigl[ Q\bigl(\mathcal{Z}_0(\mathcal{L}_{g,0}^\e) \otimes \hat{\mathcal{Z}}_0(\mathcal{L}_{0,1}^\e)^{\otimes n}\bigr) : Q\bigl(\mathcal{Z}_0(\mathcal{L}_{g,0}^\e) \otimes \mathcal{Z}_0(\mathcal{L}_{0,1}^\e)^{\otimes n})\bigr) \bigr]\\
\geq\:&\bigl[ Q\bigl(\hat{\mathcal{Z}}_0(\mathcal{L}_{0,1}^\e)^{\otimes n}\bigr) : Q\bigl(\mathcal{Z}_0(\mathcal{L}_{0,1}^\e)^{\otimes n})\bigr) \bigr] \geq \bigl[ Q\bigl(\hat{\mathcal{Z}}_0(\mathcal{L}_{0,1}^\e)\bigr) : Q\bigl(\mathcal{Z}_0(\mathcal{L}_{0,1}^\e))\bigr) \bigr]^n = l^{mn}
\end{split}
\end{align}
where the two last inequalities are by Lemma \ref{lemmaSupMultDim}. Combining this with Prop.\,\ref{Z0Lgn}(3) we conclude that
\begin{align}
\begin{split}\label{majorationDegPILgn}
\mathfrak{r}^2 = \bigl[Q(\Ll_{g,n}^{\e}):Q(\mathcal{Z}(\Ll_{g,n}^{\e}))\bigr] & = \frac{\bigl[Q(\Ll_{g,n}^{\e}):Q\bigl(\mathcal{Z}_0(\Ll_{g,n}^{\e})\bigr)\bigr]}{\bigl[Q\bigl(\mathcal{Z}(\Ll_{g,n}^{\e})\bigr):Q\bigl(\mathcal{Z}_0(\Ll_{g,n}^{\e})\bigr)\bigr]}\\
&\leq \bigl[Q(\Ll_{g,n}^{\e}):Q\bigl(\mathcal{Z}_0(\Ll_{g,n}^{\e})\bigr)\bigr]l^{-mn}\\
&= l^{(2g+n)\mathrm{dim}(\mathfrak{g})} l^{-mn} = l^{2(g\,\mathrm{dim}(\mathfrak{g}) + nN)}
\end{split}
\end{align}
because $\dim(\mathfrak{g}) = 2N + m$. Therefore $\mathfrak{r}= l^{g\,\mathrm{dim}(\mathfrak{g})+nN}$.
\end{proof}

From the previous proof we can isolate facts which are interesting on their own:

\begin{cor}\label{coroIsolateFacts}
1. $\bigl[ Q\bigl(\mathcal{Z}(\Ll_{g,n}^{\e})\bigr):Q\bigl(\mathcal{Z}_0(\Ll_{g,n}^{\e})\bigr) \bigr] = l^{mn}$, where $m = \mathrm{rank}(\mathfrak{g})$.
\\2. $Q\bigl( \mathcal{Z}(\mathcal{L}_{g,n}^\e) \bigr) = Q\bigl( \hat{\mathcal{Z}}_0(\mathcal{L}_{g,n}^\e) \bigr)$, where the central subalgebra $\hat{\mathcal{Z}_0}(\mathcal{L}_{g,n}^\e)$ was defined in \eqref{defHatZ0Lgn}.
\end{cor}
\begin{proof}
1. The conclusion of the proof of Theorem \ref{Lgnteo1}  shows that the inequality in \eqref{majorationDegPILgn} is actually an equality, giving the claim.
\\2. By the previous item all inequalities in \eqref{majorationDegZoverZ0} are actually equalities. It follows that $\bigl[ Q\bigl(\mathcal{Z}(\mathcal{L}_{g,n}^\e)\bigr) : Q\bigl(\hat{\mathcal{Z}}_0(\mathcal{L}_{g,n}^\e)\bigr) \bigr] = 1$.
\end{proof}

\begin{remark}{\rm We will see in $\S$\ref{sec:proofLgn3} that $\mathcal{Z}(\mathcal{L}_{g,n}^\e) = \hat{\mathcal{Z}}_0(\mathcal{L}_{g,n}^\e)$. This answers \cite[Rk\,. 5.3]{BR2}. }
\end{remark}

\subsection{PI degree of \texorpdfstring{$Q( \mathcal{L}_{g,n}^{u_\e} )$}{the central localization}}\label{subsecPIdegInv}
Recall from \eqref{defSmallInvEps} that $\mathcal{L}_{g,n}^{u_\e}$ is the subalgebra of invariant elements under the action $\mathrm{coad}^r$ of the small quantum group $u_\e^\Q \subset \Gamma_\e^\Q$. By Proposition \ref{propLgndomain} it is a domain. Therefore we can consider its central localization (see \S\ref{subsecCentralLoc})
\[ Q( \mathcal{L}_{g,n}^{u_\e} ) = Q\bigl(\mathcal{Z}(\Ll_{g,n}^{u_\e})\bigr) \otimes_{\mathcal{Z}(\Ll_{g,n}^{u_\e})} \Ll_{g,n}^{u_\e}. \]
By Cor. \ref{ZinLue} we can write 
\[ \bigl[ Q( \mathcal{L}_{g,n}^{u_\e} ) : Q\bigl(\mathcal{Z}(\Ll_{g,n}^{u_\e})\bigr) \bigr] \leq \bigl[ Q( \mathcal{L}_{g,n}^{u_\e} ) : Q\bigl(\mathcal{Z}(\Ll^\e_{g,n})\bigr) \bigr] \leq \bigl[ Q( \mathcal{L}_{g,n}^\e) : Q\bigl(\mathcal{Z}(\Ll_{g,n}^\e)\bigr) \bigr] < \infty \]
since the last term was computed in Theorem \ref{Lgnteo1}. Lemma \ref{lemQRCSA} then indicates that $Q(\mathcal{L}_{g,n}^{u_\e})$ is a division algebra and a central simple algebra. Our goal is to compute its PI-degree. Note that by Lemma \ref{lemTrucsGenerauxQR}(2) we also have
\[ Q(\mathcal{L}_{g,n}^{u_\e}) \cong Q\bigl( \mathcal{Z}(\mathcal{L}_{g,n}^\e) \bigr) \otimes_{\mathcal{Z}(\mathcal{L}_{g,n}^\e)} \mathcal{L}_{g,n}^{u_\e} \cong Q\bigl( \mathcal{Z}_0(\mathcal{L}_{g,n}^\e) \bigr) \otimes_{\mathcal{Z}_0(\mathcal{L}_{g,n}^\e)} \mathcal{L}_{g,n}^{u_\e}. \]

\indent Recall the quantum moment map (QMM) $\mu_{g,n} : H' \xrightarrow{\Phi_{0,1}^{-1}} \mathcal{L}_{0,1}(H) \xrightarrow{\mathfrak{d}_{g,n}} \mathcal{L}_{g,n}(H)$ defined for a general ribbon Hopf algebra $H$ in \S\ref{subsecQMMgeneral} and further studied for $H = U_q^\Pup(\mathfrak{g})$ in \S\ref{secQMMandInv}; we saw in particular that it has a specialized version $\mu_{g,n}^\e : U_\e'\xrightarrow{(\Phi_{0,1}^\e)^{-1}} \mathcal{L}_{0,1}^{\e} \xrightarrow{\mathfrak{d}_{g,n}^{\e}} \mathcal{L}_{g,n}^{\e}$. We now show that $\mu_{g,n}^\e$ is injective by a ``reduction to the classical case'' based on the two first items in the following proposition:

\begin{prop}\label{QMMinExactSeq} 1. We have a commutative diagram
\[ \xymatrix@C=5em@R=1.5em{
\mathcal{O}(G) \ar[d]_-{\mathfrak{d}_{g,n}^{\mathrm{cl}}} \ar[r]^{\sim}_-{\mathbb{F}\mathrm{r}^*_\e} & \mathcal{Z}_0(\mathcal{L}_{0,1}^\e) \ar[d] \ar@{^{(}->}[r]^-{\mathrm{(subalg.)}} & \mathcal{L}_{0,1}^\e \ar[d]_-{\mathfrak{d}_{g,n}^{\e}} \ar@{->>}[r]^-{\pi} & \mathcal{L}_{0,1}(u_\e^\Q) \ar[d]_-{\mathfrak{d}_{g,n}^{(u_\e)}} \\
\mathcal{O}(G)^{\otimes (2g+n)} \ar[r]_-{(\mathbb{F}\mathrm{r}^*_\e)^{\otimes (2g+n)}}^{\sim} & \mathcal{Z}_0(\mathcal{L}_{g,n}^\e) \ar@{^{(}->}[r]_-{\mathrm{(subalg.)}} & \mathcal{L}_{g,n}^\e \ar@{->>}[r]_-{\pi^{\otimes (2g+n)}}& \mathcal{L}_{g,n}(u_\e^\Q)
} \]
where $\mathfrak{d}_{g,n}^{(u_\e)}$ is defined by taking $H = u_\e^\Q(\mathfrak{g})$ in \eqref{pseudoQMMgnH} while $\mathfrak{d}_{g,n}^{\mathrm{cl}}$ is defined by
\[ \mathfrak{d}_{g,n}^{\mathrm{cl}}(f) : \bigl( B_1, A_1, \ldots, B_g, A_g, M_1, \ldots, M_n \bigr) \mapsto f\Bigl[ \Bigl({\textstyle \prod_{i=1}^g} B_i A_i^{-1} B_i^{-1} A_i \Bigr) \Bigl({\textstyle \prod_{j=1}^n} M_j \Bigr) \Bigr] \]
for all $f \in \mathcal{O}(G)$, using  the identification $\mathcal{O}(G)^{\otimes (2g+n)} = \mathcal{O}\bigl(G^{\times (2g+n)}\bigr)$.
\\2. The morphism $\mathfrak{d}_{g,n}^{\mathrm{cl}} : \mathcal{O}(G) \to \mathcal{O}(G)^{\otimes (2g+n)}$ is injective.
\\3. The morphism $\mathfrak{d}_{g,n}^{\e}$ and the specialized QMM $\mu^\e_{g,n} : U_\e'\to \mathcal{L}_{g,n}^{\e}$ are injective.
\end{prop}
Therefore the three leftmost vertical maps in the diagram are embeddings, and $\mathcal{Z}_0(\mathcal{L}_{0,1}^\e)\subset \mathcal{Z}_0(\mathcal{L}_{g,n}^\e)$. 
\begin{proof}
1. We prove commutation of the left rectangle. Because of formula \eqref{pseudoQMMgnH} defining $\mathfrak{d}_{g,n}$ it suffices to do the case $\mathfrak{d}_{1,0}$. We need three preliminary facts:

\indent (i) Recall from \eqref{antipodeL01} the antipode $S_{\mathcal{L}}$ of $\mathcal{L}_{0,1}$. Thanks to \eqref{inverseu} and Corollary \ref{relationsCoRMatAtEps}(2) we see that $S_{\mathcal{L}} \circ \mathbb{F}\mathrm{r}_\e^* = \mathbb{F}\mathrm{r}_\e^* \circ S_{\mathcal{O}}$ with $S_{\mathcal{O}} : \mathcal{O}(G) \to \mathcal{O}(G)$ given by $S_{\mathcal{O}}(f)(g) = f(g^{-1})$. 

\indent (ii) Let $\mathfrak{i}^{\mathrm{cl}}_{b}, \mathfrak{i}^{\mathrm{cl}}_{a} :\mathcal{O}(G) \to \mathcal{O}(G)^{\otimes 2}$ given by $f \mapsto f \otimes \varepsilon_{\mathcal{O}}$ and $f \mapsto \varepsilon_{\mathcal{O}} \otimes f$ respectively, with $\varepsilon_{\mathcal{O}} : f \mapsto f(e_G)$ where $e_G$ is the unit element of $G$. Let $\mathfrak{i}^{\e}_b, \mathfrak{i}^{\e}_a : \mathcal{L}_{0,1}^\e \to \mathcal{L}_{1,0}^\e$ be defined similarly. By definition we have $\mathfrak{i}^\e_s \circ \mathbb{F}\mathrm{r}_\e^* = (\mathbb{F}\mathrm{r}_\e^*)^{\otimes 2} \circ \mathfrak{i}_s^{\mathrm{cl}}$ for $s \in \{a,b\}$.

\indent (iii) Let $\mathsf{v}_{\e} : \mathcal{O}_{\e} \to \mathbb{C}$ be the specialization of the coribbon element of $\mathcal{O}_q(\qD)$ (see \eqref{coribbonOq} and Lemma \ref{lemIntegralRandCorib}). We claim that $\mathsf{v}_{\e} \circ \mathbb{F}\mathrm{r}_\e^* = \varepsilon_{\mathcal{O}}$; said differently $\mathsf{v}_{\e}$ is $\varepsilon_{\mathcal{O}_\e}$ on $\mathcal{Z}_0(\mathcal{O}_\e)$. To see this note first that $\mathsf{v}_{\e}$ is multiplicative on the sub-Hopf algebra $\mathcal{Z}_0(\mathcal{O}_\e)$, thanks to the axiom \eqref{axiomCoribbonProduct} and Corollary \ref{relationsCoRMatAtEps}(2). Hence it suffices to check the claim on a subset of algebra generators. It is known that matrix coefficients of $\Gamma_\e^\Q$-modules of the form $_AV_{l\lambda} \otimes_A \mathbb{C}_\e$ generate $\mathcal{Z}_0(\mathcal{O}_\e)$ (Th.\,\ref{DCLteo1}). For such a matrix coefficient $c$, we have by \eqref{coribbonOq}:
\[ \mathsf{v}_{\e}(c) = \eD^{-D(l\lambda, l\lambda + 2\rho)}c(1) = (\eD^l)^{-D(\lambda,l\lambda+2\rho)}c(1) = c(1) = \varepsilon_{\mathcal{O}_\e}(c) \]
thanks to the choice of $\eD$ as an $l$-th root of unity in \eqref{choixED}.

\noindent As a result, for all $f \in \mathcal{O}(G)$ we find
\begin{align*}
&\mathfrak{d}_{1,0}^{\e} \circ \mathbb{F}\mathrm{r}_\e^*(f)\\
=\:& \sum_{(\mathbb{F}\mathrm{r}_\e^*(f))} \mathsf{v}_{\e}^2\bigl( \mathbb{F}\mathrm{r}_\e^*(f)_{(1)} \bigr) \mathfrak{i}^{\e}_b\bigl( \mathbb{F}\mathrm{r}_\e^*(f)_{(2)} \bigr) \mathfrak{i}^{\e}_a\bigl( S_{\mathcal{L}}\bigl(\mathbb{F}\mathrm{r}_\e^*(f)_{(3)} \bigr)\bigr)  \mathfrak{i}^{\e}_b\bigl( S_{\mathcal{L}}\bigl(\mathbb{F}\mathrm{r}_\e^*(f)_{(4)} \bigr)\bigr) \mathfrak{i}^{\e}_a\bigl( \mathbb{F}\mathrm{r}_\e^*(f)_{(5)} \bigr)\\
=\:& \sum_{(f)} \mathsf{v}_{\e}^2\bigl( \mathbb{F}\mathrm{r}_\e^*(f_{(1)}) \bigr) \mathfrak{i}^{\e}_b\bigl( \mathbb{F}\mathrm{r}_\e^*(f_{(2)}) \bigr) \mathfrak{i}^{\e}_a\bigl( S_{\mathcal{L}}\bigl(\mathbb{F}\mathrm{r}_\e^*(f_{(3)}) \bigr)\bigr)  \mathfrak{i}^{\e}_b\bigl( S_{\mathcal{L}}\bigl(\mathbb{F}\mathrm{r}_\e^*(f_{(4)}) \bigr)\bigr) \mathfrak{i}^{\e}_a\bigl( \mathbb{F}\mathrm{r}_\e^*(f_{(5)}) \bigr)\\
=\:& \sum_{(f)} (\mathbb{F}\mathrm{r}_\e^*)^{\otimes 2} \Bigl( \mathfrak{i}^{\mathrm{cl}}_b(f_{(1)}) \mathfrak{i}^{\mathrm{cl}}_a\bigl( S_{\mathcal{O}}(f_{(2)}) \bigr)  \mathfrak{i}^{\mathrm{cl}}_b\bigl( S_{\mathcal{O}}(f_{(3)}) \bigr) \mathfrak{i}^{\mathrm{cl}}_a(f_{(4)}) \Bigr)= (\mathbb{F}\mathrm{r}_\e^*)^{\otimes 2} \circ \mathfrak{d}_{1,0}^{\mathrm{cl}}(f)
\end{align*}
where the first equality is by definition \eqref{frakD10} of $\mathfrak{d}_{1,0}$, the second uses that $\mathbb{F}\mathrm{r}_\e^*$ is a Hopf algebra morphism, the third uses the previous remarks and the fact that $(\mathbb{F}\mathrm{r}_\e^*)^{\otimes 2} : \mathcal{O}(G) \to \mathcal{L}_{1,0}^\e$ is an algebra morphism (Prop.\,\ref{Z0Lgn}), and the last is by definition. 

The commutativity of the right square is established by similar arguments, now using Corollary \ref{relationsCoRMatAtEps}(1) and Proposition \ref{piMorphBetweenLgn}. One has to check the compatibility of the ribbon structures, {\it i.e.} $\pi(\varphi)(v) = \mathsf{v}_{\e}(\varphi)$ for all $\varphi \in \mathcal{O}_\e$, where $v \in u_\e^\Q$ is the ribbon element \cite[eq. (A.3.1)]{Lyub95}). This is a direct computation, using the expression of the $R$-matrix of $u_\e^\Q$ given by Lyubashenko \cite[App.]{Lyub95} and recalled in Th.\,\ref{teosmallqg} below.

\noindent 2. By the argument in the proof of Lemma \ref{lemmaInjQMMtrivial} it suffices to show that $\mathfrak{d}_{1,0}^{\mathrm{cl}}$ is injective. But this follows from the fact, proved in \cite{Ree}, that in a connected semi-simple algebraic group defined over an algebraically closed field, every element is a commutator.

\noindent 3. Since $\mu_{g,n}^\e = \mathfrak{d}_{g,n}^\e \circ (\Phi_{0,1}^\e)^{-1}$, it is enough to show that $\mathfrak{d}_{g,n}^\e$ is injective. By item 1 we know that the restriction $\mathfrak{d}_{g,n}^{\e}|_{\mathcal{Z}_0(\mathcal{L}^\e_{0,1})}$ takes values in $\mathcal{Z}_0(\mathcal{L}_{g,n}^\e)$, and this restriction is injective by item 2. As a result $\mathfrak{d}_{g,n}^{\e}\bigl( \mathcal{Z}_0(\mathcal{L}_{0,1}^\e) \setminus \{ 0 \} \bigr) \subset \mathcal{Z}_0(\mathcal{L}_{g,n}^\e) \setminus \{ 0 \} \subset \mathcal{Z}(\mathcal{L}_{g,n}^\e) \setminus \{ 0 \}$. Injectivity of $\mathfrak{d}_{g,n}^\e$ thus follows by taking $Z_0 = \mathcal{Z}_0(\mathcal{L}_{0,1}^\e)$ in Lemma \ref{lemmeTechniqueInjectivite}.
\end{proof}

Let us recall the statement of Theorem \ref{Lgn2}, and give its proof.
\begin{teo}\label{Lgnteo2} If $(g,n) \neq (0,1)$, the algebra $Q(\Ll_{g,n}^{u_\e})$ is a division algebra, and a central simple algebra of PI-degree $l^{g.\mathrm{dim}(\mathfrak{g})+N(n-1)-m}$, where $N$ is the number of positive roots of $\mathfrak{g}$ and $m$ is the rank of $\mathfrak{g}$.
\end{teo}
The end of this subsection is dedicated to the proof of this theorem. To keep clear the line of reasoning we postpone the proof of some arguments to \S \ref{sec:proofLgn3}.

We will use the following shorter notations:
\[ R := \mathcal{L}_{g,n}^{\e}, \quad S := \mathfrak{d}_{g,n}^{\e}(\mathcal{L}_{0,1}^\e) \cong \mathcal{L}_{0,1}^\e, \quad \mathcal{Z}_0(S) = \mathfrak{d}_{g,n}^{\e}\bigl( \mathcal{Z}_0(\mathcal{L}_{0,1}^\e) \bigr), \quad \mathcal{Z}_1(S) = \mathfrak{d}^{\e}_{g,n}\bigl( \mathcal{Z}_1(\mathcal{L}_{0,1}^\e) \bigr) \]
and $(\mathcal{Z}_0\cap \mathcal{Z}_1)(S) = \mathcal{Z}_0(S)\cap \mathcal{Z}_0(S)$, where $\mathfrak{d}_{g,n}^{\e} : \mathcal{L}_{0,1}^\e \to \mathcal{L}_{g,n}^{\e}$ from \S\ref{secQMMandInv} was proven injective in Prop.\,\ref{QMMinExactSeq}, and $\mathcal{Z}_1(\mathcal{L}_{0,1}^\e)$ is defined in \eqref{defZ1Lgn} below. Recall from Prop.\,\ref{QMMinExactSeq} that $\mathcal{Z}_0(S) \subset \mathcal{Z}_0(R) \subset \mathcal{Z}(R)$; thus $Q\bigl( \mathcal{Z}(R) \bigr)$ is a $\mathcal{Z}_0(S)$-module and the following lemma makes sense:

\begin{lem}\label{lemmaQZsimple}
$Q\bigl( \mathcal{Z}(R) \bigr) \otimes_{\mathcal{Z}_0(S)} S$ is a simple algebra whose center is $Q\bigl( \mathcal{Z}(R) \bigr) \otimes_{\mathcal{Z}_0(S)} \mathcal{Z}(S)$, and this center can also be written as $Q\bigl( \mathcal{Z}(R) \bigr) \otimes_{Q(\mathcal{Z}_0(S))} Q\bigl( \mathcal{Z}(S) \bigr)$.
\end{lem}
\begin{proof}
Observe that
\begin{align}
&Q\bigl( \mathcal{Z}(R) \bigr) \otimes_{\mathcal{Z}_0(S)} S \nonumber\\
&\cong Q\bigl( \mathcal{Z}(R) \bigr) \otimes_{Q(\mathcal{Z}_0(S))} Q\bigl( \mathcal{Z}_0(S) \bigr) \otimes_{\mathcal{Z}_0(S)} S \quad \text{\footnotesize (trick)} \nonumber\\
&\cong Q\bigl( \mathcal{Z}(R) \bigr) \otimes_{Q(\mathcal{Z}_0(S))} Q(S) \quad \text{\footnotesize by Lemma \ref{lemTrucsGenerauxQR}(3)} \label{chaineIsoQTilde}\\
&\cong \Bigl( Q\bigl( \mathcal{Z}(R) \bigr) \otimes_{Q(\mathcal{Z}_0(S))} Q\bigl(\mathcal{Z}(S) \bigr) \Bigr) \otimes_{Q(\mathcal{Z}(S))} Q(S) \quad \text{\footnotesize (trick)} \nonumber\\
&\cong \Bigl( Q\bigl( \mathcal{Z}(R) \bigr) \otimes_{Q(\mathcal{Z}_0(S))} Q\bigl(\mathcal{Z}_0(S) \bigr) \otimes_{\mathcal{Z}_0(S)} \mathcal{Z}(S) \Bigr) \otimes_{Q(\mathcal{Z}(S))} Q(S) \quad \text{\footnotesize by Lemma \ref{lemTrucsGenerauxQR}(3)} \nonumber\\
&\cong \Bigl( Q\bigl( \mathcal{Z}(R) \bigr) \otimes_{\mathcal{Z}_0(S)} \mathcal{Z}(S) \Bigr) \otimes_{Q(\mathcal{Z}(S))} Q(S) \nonumber
\end{align}
We claim that $Q\bigl( \mathcal{Z}(R) \bigr) \otimes_{\mathcal{Z}_0(S)} \mathcal{Z}(S)$ is a field (and in particular is a simple algebra), which by \cite[Th.\,1.7.27]{Rowen}\footnote{This theorem asserts: If $C$ is a field, $R$ is a central simple $C$-algebra and $R'$ is a simple $C$-algebra, then $R' \otimes_C R$ is a simple $C$-algebra.} implies that the last algebra in the above chain of isomorphisms is a simple algebra, as desired. To prove this claim we use the fact that the multiplication map $j : \mathcal{Z}_0(S) \otimes_{(\mathcal{Z}_0\cap \mathcal{Z}_1)(S)} \mathcal{Z}_1(S) \to \mathcal{Z}(S)$ is an isomorphism of algebras, which follows from Th.\,\ref{Lgnteo3}(1) applied to $\mathcal{L}_{0,1}^\e \cong S$, and the fact that the multiplication map $j_2 : \mathcal{Z}(R) \otimes_{(\mathcal{Z}_0\cap \mathcal{Z}_1)(S)} \mathcal{Z}_1(S) \to R$ is injective (actually taking values in $R^{u_\e}$) which is proven in Th.\,\ref{Lgnteo3}(2).\footnote{Th.\,\ref{Lgnteo3} is proved below in \textsection \ref{sec:proofLgn3}. We stress that surjectivity of $j$ uses Cor.\,\ref{coroIsolateFacts} and its injectivity uses arguments which are independent of Th.\,\ref{Lgnteo2}; the surjectivity of $j_2$ will use Th.\,\ref{Lgnteo2} but its injectivity is independent of it (hence there is no loop in the arguments). We found it clearer to put all the properties of $j$ and $j_2$ in \S\ref{sec:proofLgn3} below.} Thanks to these facts we find
\begin{align*}
\mathcal{Z}(R) \otimes_{\mathcal{Z}_0(S)} \mathcal{Z}(S) &\cong \mathcal{Z}(R) \otimes_{\mathcal{Z}_0(S)} \mathcal{Z}_0(S) \otimes_{(\mathcal{Z}_0 \cap \mathcal{Z}_1)(S)} \mathcal{Z}_1(S)\\
& \cong \mathcal{Z}(R) \otimes_{(\mathcal{Z}_0 \cap \mathcal{Z}_1)(S)} \mathcal{Z}_1(S) \xrightarrow{\:\text{injection}\:} R.
\end{align*}
Hence $\mathcal{Z}(R) \otimes_{\mathcal{Z}_0(S)} \mathcal{Z}(S)$ is a domain because so is $R$ (Prop.\,\ref{propLgndomain}). Therefore $Q\bigl( \mathcal{Z}(R) \bigr) \otimes_{\mathcal{Z}_0(S)} \mathcal{Z}(S)$ is a domain as well. Since $\bigl[ Q\bigl(\mathcal{Z}(S)\bigr) : Q\bigl(\mathcal{Z}_0(S)\bigr) \bigr] = l^m$ by Cor.\,\ref{coroIsolateFacts}(1), we get moreover that $Q\bigl( \mathcal{Z}(R) \bigr) \otimes_{\mathcal{Z}_0(S)} \mathcal{Z}(S) \cong Q\bigl( \mathcal{Z}(R) \bigr) \otimes_{Q(\mathcal{Z}_0(S))} Q\bigl( \mathcal{Z}(S) \bigr)$ has dimension $l^m$ over $Q\bigl( \mathcal{Z}(R) \bigr)$. But any finite-dimensional commutative domain is a field. In this way we have proved the first part of the lemma. 

The center is now computed as follows:
\begin{align*}
&\mathcal{Z}\Bigl( Q\bigl( \mathcal{Z}(R) \bigr) \otimes_{\mathcal{Z}_0(S)} S \Bigr) \cong \mathcal{Z}\Bigl( Q\bigl( \mathcal{Z}(R) \bigr) \otimes_{Q(\mathcal{Z}_0(S))} Q(S) \Bigr) \cong Q\bigl( \mathcal{Z}(R) \bigr) \otimes_{Q(\mathcal{Z}_0(S))} Q\bigl( \mathcal{Z}(S) \bigr) \bigr)\\
\cong\:\,& Q\bigl( \mathcal{Z}(R) \bigr) \otimes_{Q(\mathcal{Z}_0(S))} Q\bigl( \mathcal{Z}_0(S)  \bigr) \otimes_{\mathcal{Z}_0(S)} \mathcal{Z}(S)\cong Q\bigl( \mathcal{Z}(R) \bigr) \otimes_{\mathcal{Z}_0(S)} \mathcal{Z}(S)
\end{align*}
where the first isomorphism uses \eqref{chaineIsoQTilde}, the second uses the fact that the center of a tensor product of algebras over a field is the tensor product of their centers \cite[Cor.\,1.7.24]{Rowen} together with Lemma \ref{lemTrucsGenerauxQR}(2) and the third is by Lemma \ref{lemTrucsGenerauxQR}(3).
\end{proof}

\begin{proof}[Proof of Theorem \ref{Lgnteo2}.]
We already noted at the beginning of the present subsection that $Q(\Ll_{g,n}^{u_\e})$ is a division algebra. Its PI-degree is the square root of $\bigl[ Q(\Ll_{g,n}^{u_\e}) : Q\bigl( \mathcal{Z}(\Ll_{g,n}^{u_\e}) \bigr) \bigr]$. The key-point for the computation of this dimension is to use that $\mathcal{L}_{g,n}^{u_\e}=\mathbf{C}_R(S)$, namely the centralizer of $S := \mathfrak{d}_{g,n}^{\e}(\mathcal{L}_{0,1}^\e)$ in $R := \mathcal{L}_{g,n}^{\e}$ (Th.\,\ref{centralizerMu}). In particular $\mathcal{Z}(R) \subset \Ll_{g,n}^{u_\e}$. First, it is important to realize that the central localization $Q(\mathcal{L}_{g,n}^{u_\e})$ as defined in \S\ref{subsecCentralLoc} can be seen as a subalgebra of $Q(R)$. To see this, note by Prop.\,\ref{Z0Lgn}(3) that $R$ is Noetherian as a $\mathcal{Z}_0(R)$-module. It follows that the $\mathcal{Z}_0(R)$-submodule $\mathcal{Z}(\mathcal{L}_{g,n}^{u_\e})$ is finitely generated, whence $\bigl[ Q\bigl( \mathcal{Z}(\mathcal{L}_{g,n}^{u_\e}) \bigr) : Q\bigl( \mathcal{Z}(R) \bigr) \bigr] \leq \bigl[ Q\bigl( \mathcal{Z}(\mathcal{L}_{g,n}^{u_\e}) \bigr) : Q\bigl( \mathcal{Z}_0(R) \bigr) \bigr] < \infty$. Using Lemma \ref{lemTrucsGenerauxQR}(3), we thus obtain
\[ Q\bigl( \mathcal{L}_{g,n}^{u_\e} \bigr) \cong Q\bigl( \mathcal{Z}(R) \bigr) \otimes_{\mathcal{Z}(R)} \mathcal{L}_{g,n}^{u_\e} \xrightarrow{\:\text{inclusion}\:} Q\bigl( \mathcal{Z}(R) \bigr) \otimes_{\mathcal{Z}(R)} R \overset{\text{def}}{=} Q(R) \]
where the inclusion $\mathcal{L}_{g,n}^{u_\e} \subset R$ survives because $Q\bigl( \mathcal{Z}(R) \bigr)$ is a flat $\mathcal{Z}(R)$-module as any localization. Under this inclusion we have
\begin{equation}\label{InvAsCentralizer}
Q(\mathcal{L}_{g,n}^{u_\e}) = Q\bigl( \mathbf{C}_R(S) \bigr) = \mathbf{C}_{Q(R)}\bigl( Q_R(S) \bigr),
\end{equation}
with the subalgebra
\begin{equation}\label{centralizersept25}
Q_R(S) := \mathrm{span}_{Q(\mathcal{Z}(R))}\bigl\{ 1 \otimes_{\mathcal{Z}(R)} s \, | \, s \in S \bigr\} \subset Q\bigl( \mathcal{Z}(R) \bigr) \otimes_{\mathcal{Z}(R)} R \overset{\text{def}}{=} Q(R)
\end{equation}
whose definition makes sense because $Q(R)$ is naturally a $Q\bigl( \mathcal{Z}(R) \bigr)$-vector space.\footnote{This can also be written as $Q_R(S) = Q\bigl( \mathcal{Z}(R) \bigr) \otimes_{\mathcal{Z}(R)} \mathcal{Z}(R)S$, where $\mathcal{Z}(R)S$ is the $\mathcal{Z}(R)$-submodule generated by $S$ in $R$. Note that $Q\bigl( \mathcal{Z}(R) \bigr) \otimes_{\mathcal{Z}(R)} S$ does not make sense because $S$ is not a $\mathcal{Z}(R)$-module.} Since $Q_R(S)$ is of finite-dimension over $Q\bigl( \mathcal{Z}(R) \bigr)$ (because so is $Q(R)$) it is a division algebra, and in particular a simple subalgebra of the central simple algebra $Q(R)$ (Th.\,\ref{Lgnteo1}). As a result, the double centralizer theorem (see e.g. \cite[Th.\,7.1.9]{Rowen} or \cite[Th.\,7.1]{StackP2}) implies that $\mathbf{C}_{Q(R)}\bigl( Q_R(S) \bigr)$ is a central simple algebra and 
\[ \bigl[ Q(\Ll_{g,n}^{u_\e}) : Q\bigl( \mathcal{Z}(\Ll_{g,n}^{\e}) \bigr) \bigr] \overset{\eqref{InvAsCentralizer}}{=} \bigl[ \mathbf{C}_{Q(R)}\bigl( Q_R(S) \bigr) : Q\bigl( \mathcal{Z}(R) \bigr) \bigr]  = \frac{\bigl[ Q(R) : Q\bigl( \mathcal{Z}(R) \bigr) \bigr]}{\bigl[ Q_R(S) : Q\bigl( \mathcal{Z}(R) \bigr) \bigr]}. \]
The numerator equals $l^{2(g\,\mathrm{dim}(\mathfrak{g})+nN)}$ by Th.\,\ref{Lgnteo1}. To compute the denominator note that, since $\mathcal{Z}_0(S) \subset \mathcal{Z}_0(R) \subset \mathcal{Z}(R)$ by Prop.\,\ref{QMMinExactSeq}, there is a well-defined map
\begin{equation}\label{factQRS}
Q\bigl( \mathcal{Z}(R) \bigr) \otimes_{\mathcal{Z}_0(S)} S \to Q_R(S), \quad q \otimes_{\mathcal{Z}_0(S)} s \mapsto q \otimes_{\mathcal{Z}(R)} s.
\end{equation}
It is an isomorphism: surjectivity is clear while injectivity is due to the fact that the source is a simple algebra (Lemma \ref{lemmaQZsimple}), thus forcing the kernel to be $0$. Since $S$ is a free $\mathcal{Z}_0(S)$-module of rank $l^{\dim(\mathfrak{g})}$ by Prop.\,\ref{Z0Lgn}(3) applied to $S \cong \mathcal{L}_{0,1}^\e$, we deduce that $\bigl[ Q_R(S) : Q\bigl( \mathcal{Z}(R) \bigr) \bigr] = l^{\dim(\mathfrak{g})} = l^{2N+n}$.
As a result
\[ \bigl[ Q(\Ll_{g,n}^{u_\e}) : Q\bigl( \mathcal{Z}(\Ll_{g,n}^{\e}) \bigr) \bigr] = l^{2(g\,\mathrm{dim}(\mathfrak{g})+nN) - (2N+m)}= l^{2(g\,\mathrm{dim}(\mathfrak{g})+(n-1)N)-m}. \]
Another dimension computation is required in order to conclude the proof. Observe that
\begin{align*}
Q\bigl( \mathcal{Z}(\mathcal{L}_{g,n}^{u_\e}) \bigr) &= \mathcal{Z}\bigl( Q(\mathcal{L}_{g,n}^{u_\e}) \bigr)  = \mathcal{Z}\bigl[ \mathbf{C}_{Q(R)}\bigl( Q_R(S) \bigr) \bigr] = \mathbf{C}_{Q(R)}\bigl[ \mathbf{C}_{Q(R)}\bigl( Q_R(S) \bigr) \bigr] \cap \mathbf{C}_{Q(R)}\bigl( Q_R(S) \bigr)\\
&=Q_R(S) \cap \mathbf{C}_{Q(R)}\bigl( Q_R(S) \bigr) = \mathcal{Z}\bigl( Q_R(S) \bigr) \cong Q\bigl( \mathcal{Z}(R) \bigr) \otimes_{Q(\mathcal{Z}_0(S))} Q\bigl(\mathcal{Z}(S)\bigr)
\end{align*}
where the first equality is by Lemma \ref{lemTrucsGenerauxQR}(2), the second is by \eqref{InvAsCentralizer}, the third is trivial, the fourth is by the double centralizer theorem, the fifth is trivial and the sixth is by the isomorphism \eqref{factQRS} and Lemma \ref{lemmaQZsimple}. As a result
\begin{align*}
\bigl[ Q\bigl( \mathcal{Z}(\mathcal{L}_{g,n}^{u_\e}) \bigr) : Q\bigl( \mathcal{Z}(\mathcal{L}_{g,n}^{\e}) \bigr) \bigr] &= \bigl[ Q\bigl( \mathcal{Z}(R) \bigr) \otimes_{Q(\mathcal{Z}_0(S))} Q\bigl(\mathcal{Z}(S)\bigr) : Q\bigl( \mathcal{Z}(R) \bigr) \bigr]\\
&= \bigl[ Q\bigl(\mathcal{Z}(S)\bigr) : Q\bigl(\mathcal{Z}_0(S)\bigr) \bigr] = l^m
\end{align*}
where the last equality is by Cor.\,\ref{coroIsolateFacts} applied to $S \cong \mathcal{L}_{0,1}^\e$. We are now in position to conclude the proof:
\begin{align*}
\bigl[ Q(\Ll_{g,n}^{u_\e}) : Q\bigl( \mathcal{Z}(\Ll_{g,n}^{u_\e}) \bigr)  \bigr] & = \frac{\bigl[ Q(\Ll_{g,n}^{u_\e}) : Q\bigl( \mathcal{Z}(\Ll_{g,n}^{\e}) \bigr) \bigr]}{\bigl[ Q\bigl( \mathcal{Z}(\Ll_{g,n}^{u_\e}) \bigr) : Q\bigl( \mathcal{Z}(\Ll_{g,n}^{\e}) \bigr) \bigr]} \\
& = l^{2(g\,\mathrm{dim}(\mathfrak{g})+(n-1)N)-m}\,l^{-m} = l^{2(g\,\mathrm{dim}(\mathfrak{g})+(n-1)N-m)}. 
\end{align*}
This shows $Q(\Ll_{g,n}^{u_\e})$ has PI degree $l^{g\,\mathrm{dim}(\mathfrak{g})+(n-1)N-m}$, as claimed.
\end{proof}
Analogously to Corollary \ref{coroIsolateFacts}, the rank computation in the above proof implies:
\begin{cor}\label{rkjanv26} We have $\bigl[ Q\bigl(\mathcal{Z}(\Ll_{g,n}^{u_\e})\bigr):Q\bigl(\mathcal{Z}_0(\Ll_{g,n}^{\e})\bigr) \bigr] = l^{m(n+1)}$, where $m = \mathrm{rank}(\mathfrak{g})$.
\end{cor}

\section{Centers of graph algebras at roots of unity} \label{sec:proofLgn3} 
Recall the ``quantum trace'' elements $\mathrm{qTr}_\lambda \in \mathcal{L}_{0,1}^{\Gamma_A}$ defined in \eqref{casimirElmtsL01} for all $\lambda \in P_+$ and let $\mathcal{Z}_1(\mathcal{L}_{0,1}^\e)$ be the subalgebra of $\mathcal{L}_{0,1}^\e$ generated by their specializations $(\mathrm{qTr}_\lambda)_{|\e}$. Thanks to the vector space decomposition $\mathcal{L}_{g,n}^{\e} = (\mathcal{L}_{1,0}^{\e})^{\otimes g} \otimes (\mathcal{L}_{0,1}^{\e})^{\otimes n}$, define the subalgebra
\begin{equation}\label{defZ1Lgn}
\mathcal{Z}_1(\mathcal{L}_{g,n}^{\e}) := (1_{\mathcal{L}_{1,0}^{\e}})^{\otimes g} \otimes \mathcal{Z}_1(\mathcal{L}_{0,1}^{\e})^{\otimes n} \subset \mathcal{Z}(\mathcal{L}_{g,n}^{\e})
\end{equation}
where the inclusion is due to Lemma \ref{Z1general}(2). Also recall the morphism $\mathfrak{d}_{g,n}^{\e} : \Ll_{0,1}^\e \to \Ll_{g,n}^{\e}$ which enters in the definition of the quantum moment map (\S\ref{secQMMandInv}).
\smallskip

This section is mainly devoted to the proof of the following result:
\begin{teo}\label{Lgnteo3}
Let us use the notation $(\mathcal{Z}_0 \cap \mathcal{Z}_1)(\mathcal{L}_{g,n}^{\e}) = \mathcal{Z}_0(\mathcal{L}_{g,n}^{\e}) \cap \mathcal{Z}_1(\mathcal{L}_{g,n}^{\e})$.

\noindent (1) The multiplication map $j\colon \mathcal{Z}_0(\mathcal{L}_{g,n}^{\e}) \otimes_{(\mathcal{Z}_0 \cap \mathcal{Z}_1)(\mathcal{L}_{g,n}^{\e})} \mathcal{Z}_1(\mathcal{L}_{g,n}^{\e}) \to \mathcal{Z}(\Ll_{g,n}^{\e})$ is an isomorphism of algebras, and $\mathcal{Z}(\Ll_{g,n}^{\e})$ is a Noetherian and integrally closed ring.

\noindent (2) The multiplication map $j_2\colon \mathcal{Z}(\Ll_{g,n}^{\e}) \otimes_{\mathfrak{d}_{g,n}^{\e}\left((\mathcal{Z}_0\cap \mathcal{Z}_1)(\Ll_{0,1}^\e)\right)} \mathfrak{d}_{g,n}^{\e}\bigl( \mathcal{Z}_1(\Ll_{0,1}^\e) \bigr)\to \mathcal{Z}(\Ll_{g,n}^{u_\e})$ is an isomorphism of algebras, and $\mathcal{Z}(\mathcal{L}_{g,n}^{u_\e})$ is a Noetherian and integrally closed ring.
\end{teo}
\noindent The first claim in the theorem means in particular that $\mathcal{Z}(\Ll_{g,n}^{\e})$ is the subalgebra generated by $\mathcal{Z}_0(\mathcal{L}_{g,n}^{\e})$ and $\mathcal{Z}_1(\mathcal{L}_{g,n}^{\e})$ (it was denoted by $\hat{\mathcal{Z}}_0(\Ll_{g,n}^{\e})$ in the proof of Th.\,\ref{Lgnteo1}). The second claim implies that $\mathcal{Z}(\Ll_{g,n}^{u_\e})$ is a finite extension of $\mathcal{Z}(\Ll_{g,n}^{\e})$ by $\mathfrak{d}_{g,n}^{\e}\bigl( \mathcal{Z}_1(\Ll_{0,1}^\e) \bigr)$; we will see below that it has degree $l^m$.

\smallskip

The proof of Th.\,\ref{Lgnteo3} consists in reducing the statement to problems already solved by De Concini, Kac and Procesi in their proof that the multiplication map
\begin{equation}\label{factoCenterDCP}
\mathcal{Z}_0(U_\e^\Pup) \otimes_{(\mathcal{Z}_0 \cap \mathcal{Z}_1)(U_\e^\Pup)} \mathcal{Z}_1(U_\e^\Pup) \overset{\sim}{\longrightarrow} \mathcal{Z}(U_\e^\Pup)
\end{equation}
is an isomorphism of algebras (see \cite[\S 6.4]{DC-K-P1}, \cite[\S 21]{DCP}), where we denote $(\mathcal{Z}_0 \cap \mathcal{Z}_1)(U_\e^\Pup) = \mathcal{Z}_0 (U_\e^\Pup)\cap \mathcal{Z}_1(U_\e^\Pup)$. We recall from \eqref{Z0Uedefsep25} that $\mathcal{Z}_0(U_\e^\Pup)$ is the subalgebra of $U_\e^\Pup$ generated by $E_\beta^l$, $F_\beta^l$, $K_\lambda^l$ for $\beta \in \phi^+$ and $\lambda \in P$, while $\mathcal{Z}_1(U_\e^\Pup)$ is the specialization at $q=\e$ of the center $\mathcal{Z}(U_A^\Pup)$ of the simply-connected unrestricted integral form $U_A^\Pup$ defined in \S\ref{integralFormsUq}.

This reduction starts with the following observation. From the very definition of $\mathcal{Z}_1(\mathcal{L}_{g,n}^{\e})$ and the decomposition $\mathcal{Z}_0(\mathcal{L}_{g,n}^{\e}) = \mathcal{Z}_0(\mathcal{L}_{g,0}^{\e}) \otimes \mathcal{Z}_0(\mathcal{L}_{0,1}^{\e})^{\otimes n}$ we have 
\begin{equation}\label{factorZoct24}
\mathcal{Z}_0(\mathcal{L}_{g,n}^{\e}) \otimes_{(\mathcal{Z}_0 \cap \mathcal{Z}_1)(\mathcal{L}_{g,n}^{\e})} \mathcal{Z}_1(\mathcal{L}_{g,n}^{\e}) = \mathcal{Z}_0(\mathcal{L}_{g,0}^{\e}) \otimes \left(\mathcal{Z}_0(\Ll_{0,1}^\e) \otimes_{(\mathcal{Z}_0 \cap \mathcal{Z}_1)(\Ll_{0,1}^\e)} \mathcal{Z}_1(\mathcal{L}_{0,1}^{\e})\right)^{\otimes n}.
\end{equation}
The rightmost terms in \eqref{factorZoct24} are formally close to the factorization of $\mathcal{Z}(U_\e^\Pup)$ in \eqref{factoCenterDCP}. Indeed, recall that $\Phi_{0,1}^A\colon \Ll_{0,1}^A\to U_A^\Pup$ is an embedding of $\Gamma_A^\Pup$-module algebras (\S\ref{subsecModifiedAlekseev}). We have:
\begin{lem}\label{Z0Z1Phimai25}
The embedding $\Phi_{0,1}^\e\colon \Ll_{0,1}^\e \to U_\e^\Pup$ restricts to isomorphisms $\mathcal{Z}_1(\Ll_{0,1}^\e) \to \mathcal{Z}_1(U_\e^\Pup)$ and 
$\mathcal{Z}_0(\Ll_{0,1}^\e) \cap \mathcal{Z}_1(\Ll_{0,1}^\e) \to \mathcal{Z}_0(U_\e^\Pup) \cap \mathcal{Z}_1(U_\e^\Pup).$
\end{lem}

\begin{proof} The first claim is just obtained by specialization of Corollary \ref{Phi0Aiso} in App.\,\ref{sec:qKilling} below, as we noted in \eqref{speSurjAndIso} that specialization is a functor and thus preserves isomorphisms.

Consider the second claim. The inclusion $\Phi_{0,1}^{\e}\bigl( \mathcal{Z}_0(\mathcal{L}_{0,1}^\e) \bigr) \subset \mathcal{Z}_0(U_\e^\Pup)$ follows from results of \cite{DC-L} (see the proof of \cite[Prop.\,4.2]{BR2}). By the first claim, it is thus enough to show that $(\Phi_{0,1}^{\e})^{-1}\bigl(\mathcal{Z}_0(U_\e^\Pup)\bigr) \subset \mathcal{Z}_0(\mathcal{L}_{0,1}^\e)$. Take $\alpha \in (\Phi_{0,1}^{\e})^{-1}\bigl(\mathcal{Z}_0(U_\e^\Pup)\bigr)$. According to Lemma \ref{lemAutreDefZ0}, we must prove that $\alpha$ vanishes on $\ker(\mathbb{F}\mathrm{r}_\e) \subset \Gamma^\Q_\e$. But by specialization of formula \eqref{KillingPhi01Int} in App.\,\ref{sec:qKilling} below it holds $\langle \alpha,y \rangle^\Q_{\e} = \bigl( \Phi_{0,1}^\e(\alpha) \,\big|\, y \bigr)_{\eD}$ for all $y \in \Gamma_\e^\Q$, where $(\:\,|\,\:)_{\eD} : U_\e^\Pup \times \Gamma_\e^\Pup \to \mathbb{C}$ is the specialization of the integral quantum Killing form $(\:\,|\,\:)_{\AD}$ defined in \eqref{integralKillingForm} of the same appendix. The desired vanishing then immediately follows from Lemma \ref{lemmaKerFrZ0}, which asserts that the subspaces $\mathcal{Z}_0(U_\e^\Pup)$ and $\ker(\mathbb{F}\mathrm{r}_\e)$ are orthogonal for $(\:\,|\,\:)_{\eD}$.
\end{proof}

The proof of Theorem \ref{Lgnteo3} will also use the next two results. The following Lemma is classical:
\begin{lem}\label{lemTechnicalCommRings}
Let $A$, $B$ and $C$ be commutative rings.

1. Assume that $A\subset B$, that $B$ is a domain and a Noetherian $A$-module, that $A$ is integrally closed, and that $Q(A)=Q(B)$. Then $A=B$.

2. Assume that $B$ and $C$ are domains, and that we have ring morphisms $f\colon A\to B$ and $g\colon B\to C$ such that $f$ and $g\circ f$ are injective, and $B$ is integral over $A$ via $f$. Then $g$ is injective.
\end{lem}

\begin{proof}
1. Let $b\in B$. Since $B$ is Noetherian over $A$, the $A$-submodule $A[b]$ is finitely generated, and therefore $b$ is integral over $A$ \cite[Prop.\,5.1]{AMacD}. Since $A$ is integrally closed and $b\in Q(B)=Q(A)$, it follows $b\in A$. 

\noindent 2. Since $C$ is a domain, ${\rm Ker}(g)$ is a prime ideal of $B$, and since $g\circ f$ is injective, ${\rm Ker}(g)\cap f(A) = \{0\}$, {\it i.e.}\, the null ideal $0f(A)$ of $f(A)$. Moreover, $0f(A)$ and $0B$ are prime ideals of $f(A)$ and $B$, respectively, since these are domains (for $f(A)$, because  $f$ and $g\circ f$ are injective). Now, as $B$ is integral over $f(A)$ there are no inclusions between prime ideals of $B$ lying over a given prime ideal of $f(A)$ \cite[Th.\,9.3 (ii)]{Matsumura}. Therefore the inclusion $0B\subset {\rm Ker}(g)$ is an equality, whence $g$ is injective. 
\end{proof}

The following Proposition combines \eqref{factorZoct24} and Lemma \ref{Z0Z1Phimai25} with results contained in \cite[\S 6]{DC-K-P1}, \cite[\S 21]{DCP}. For completeness we will recall briefly some of their arguments.

\begin{prop}\label{mainDCKPsec6} The ring $\mathcal{Z}_0(\mathcal{L}_{g,n}^{\e}) \otimes_{(\mathcal{Z}_0 \cap \mathcal{Z}_1)(\mathcal{L}_{g,n}^{\e})} \mathcal{Z}_1(\mathcal{L}_{g,n}^{\e})$ is an integrally closed domain.
\end{prop}

\begin{proof} Let $V:=\mathrm{MaxSpec}\bigl(\mathcal{Z}_0(\mathcal{L}_{g,n}^{\e}) \otimes_{(\mathcal{Z}_0 \cap \mathcal{Z}_1)(\mathcal{L}_{g,n}^{\e})} \mathcal{Z}_1(\mathcal{L}_{g,n}^{\e})\bigr)$. We will show below that $V$ is a complete intersection affine variety which is smooth outside a subvariety $Z$ of codimension $\geq 2$. This will conclude the proof as follows. Since $V$ is an irreducible algebraic set, $\Oo(V)$ is a domain, and $V$ being a complete intersection, $\Oo(V)$ is a complete intersection ring and hence a Cohen-Macaulay ring. This and ${\rm codim}(Z, V)\geq 2$ imply that $\Oo(V)$ satisfies both conditions of Serre's normality criterium (see \cite[Th.\,23.8]{Matsumura} and the comments above it), and thus it is a normal ring in the sense of \cite[Rk. on p.\,64]{Matsumura}.  Being also a domain (hence with the $0$ ideal its unique minimal prime), it follows that $\Oo(V)$ is integrally closed. 

We now show that $V$ is a complete intersection variety which is smooth outside of a subvariety $Z$ of codimension $\geq 2$.

We need to recall some results from \cite[$\S$21.2-21.3]{DCP}, \cite[$\S$6]{DC-K-P1}. In these papers, De Concini--Kac--Procesi showed that $\mathrm{MaxSpec}\bigl( \mathcal{Z}_0(U_\e^\Pup)\otimes_{(\mathcal{Z}_0 \cap \mathcal{Z}_1)(U_\e^\Pup)}\mathcal{Z}_1(U_{\e}^\Pup) \bigr)$ is an open subset in an unramified covering of the fiber product (where we used the notations of \textsection 
\ref{subsecLieAlg} for $T$, $W$ and $G$)
\begin{equation}\label{defXG}
X_G = G \times_{T/W} T/W, 
\end{equation}
defined by the quotient map $G\to G/\!/G \cong T/W$ and the map $T/W \to T/W$ induced by the $l$-th power map $t\mapsto t^l$. More precisely, they proved that $X_G$ is a variety, smooth away from a subvariety of codimension $\geq 2$. Moreover $\mathrm{MaxSpec}\bigl( \mathcal{Z}_0(U_\e^\Pup)\bigr)\cong G^*$, the dual Poisson-Lie group of $G$, which has an unramified covering map $\sigma\colon G^*\to G^0$ of degree $2^m$, where $G^0$ is the big cell $G^0\subset G$ (see \cite[Th.\,2.27]{BR2} in the present notations). Then
\begin{equation}\label{MaxspecXG}
\mathrm{MaxSpec}\bigl( \mathcal{Z}_0(U_\e^\Pup)\otimes_{(\mathcal{Z}_0 \cap \mathcal{Z}_1)(U_\e^\Pup)} \mathcal{Z}_1(U_{\e}^\Pup) \bigr) \cong G^* \times_{T/W} T/W,
\end{equation}
where the left map in this fiber product is the composition $G^* \stackrel{\sigma}{\to} G^0 \hookrightarrow G\to G/\!/G \cong T/W$, and the map $T/W \to T/W$ is as above.  

Also, De Concini--Kac--Procesi showed that: (a)  $\mathcal{Z}_1(U_\e^\Pup)$ is a complete intersection over its subring $(\mathcal{Z}_0 \cap \mathcal{Z}_1)(U_\e^\Pup)$, and (b) $(\mathcal{Z}_0 \cap \mathcal{Z}_1)(U_\e^\Pup) \cong \Oo(G)^{G}$ as a subalgebra of $\mathcal{Z}_0(U_\e^\Pup)\cong \Oo(G^*)$ via the quantum Harish-Chandra isomorphism at $q=\e$ (see Rk.\,\ref{rkZ1OAjanv26}). Now $\Oo(G^*)$ is a complete intersection ring (because $G^*$ is a smooth algebraic subgroup of $G\times G$), and $\Oo(G^*)$ is a free (hence flat) $\Oo(G)^G$-module (since $\Oo(G)$ is so by Richardson's theorem \cite{Richardson}, $\Oo(G^0)$ is a localization of $\Oo(G)$ and hence a free $\Oo(G)^{G}$-module as well, and $\Oo(G^*)$ is a free $\Oo(G^0)$-module (of rank $2^m$)). These two properties together with (a) and (b) imply that $\mathcal{Z}_0(U_\e^\Pup)\otimes_{(\mathcal{Z}_0 \cap \mathcal{Z}_1)(U_\e^\Pup)}\mathcal{Z}_1(U_{\e}^\Pup)$ is a complete intersection ring (by, e.g., the arguments of \cite[Lem.\,10.135.1]{StackP}, taking $S = \mathcal{Z}_0(U_\e^\Pup)$ and replacing the field $k$ by $(\mathcal{Z}_0 \cap \mathcal{Z}_1)(U_\e^\Pup)$ and $K$ by $\mathcal{Z}_0(U_{\e}^\Pup)$).

We are now ready to conclude the proof. Consider the embedding $\Oo(G)\to \Oo(G^*)$ dual to the dominant map $G^* \stackrel{\sigma}{\to} G^0 \hookrightarrow G$. We know that $\Oo(G)$ is isomorphic to $\mathcal{Z}_0(\mathcal{L}_{0,1}^\e)$, and $\Phi_{0,1}^\e$ embeds $\mathcal{Z}_0(\mathcal{L}_{0,1}^\e)$ into $\mathcal{Z}_0(U_\e^\Pup)\cong \Oo(G^*)$; in fact the inclusion $\Phi_{0,1}^\e(\mathcal{Z}_0(\mathcal{L}_{0,1}^\e))\subset \mathcal{Z}_0(U_\e^\Pup)$ is naturally identified with the embedding $\Oo(G)\to \Oo(G^*)$ (see \cite[Lem.\,4.2 and Rk.\,4.3]{BR2}). By using Lemma \ref{Z0Z1Phimai25} to lift via $\Phi_{0,1}^\e$ the inclusion map $(\mathcal{Z}_0 \cap \mathcal{Z}_1)(U_\e^\Pup) \hookrightarrow \Phi_{0,1}^\e(\mathcal{Z}_0(\mathcal{L}_{0,1}^\e)) \subset \mathcal{Z}_0(U_\e^\Pup)$, similarly as (b) above we find that $(\mathcal{Z}_0 \cap \mathcal{Z}_1)(\Ll_{0,1}^\e) \cong \Oo(G)^G$ as a subalgebra of $\mathcal{Z}_0(\Ll_{0,1}^\e)\cong \Oo(G)$. It is a complete intersection ring, so as in the previous discussion, using (a) one obtains that the ring $\mathcal{Z}_0(\mathcal{L}_{0,1}^\e)\otimes_{(\mathcal{Z}_0 \cap \mathcal{Z}_1)(\Ll_{0,1}^\e)} \mathcal{Z}_1(\mathcal{L}_{0,1}^{\e})$ is a complete intersection. Then \eqref{factorZoct24} implies immediately that $\mathcal{Z}_0(\mathcal{L}_{g,n}^{\e}) \otimes_{(\mathcal{Z}_0 \cap \mathcal{Z}_1)(\mathcal{L}_{g,n}^{\e})} \mathcal{Z}_1(\mathcal{L}_{g,n}^{\e})$ is a complete intersection ring as well. 

Moreover, using the embedding $\Oo(G)\to \Oo(G^*)$ the same arguments as \eqref{MaxspecXG} imply that we have an isomorphism 
\begin{equation}\label{isoMaxspecsep25}
\mathrm{MaxSpec}\bigl(\mathcal{Z}_0(\mathcal{L}_{0,1}^\e)\otimes_{(\mathcal{Z}_0 \cap \mathcal{Z}_1)(\Ll_{0,1}^\e)} \mathcal{Z}_1(\mathcal{L}_{0,1}^{\e}) \bigr) \cong X_G.
\end{equation}
By \eqref{factorZoct24} it follows that
\[ V=\mathrm{MaxSpec}\bigl( \mathcal{Z}_0(\mathcal{L}_{g,n}^{\e}) \otimes_{(\mathcal{Z}_0 \cap \mathcal{Z}_1)(\mathcal{L}_{g,n}^{\e})} \mathcal{Z}_1(\mathcal{L}_{g,n}^{\e}) \bigr)\cong G^{\times 2g}\times  X_G^{\times n}. \]
By the properties of $X_G$ this variety is smooth outside of a subvariety of codimension $\geq 2$. As explained at the beginning of the proof, by Serre's normality criterium we can conclude that $\mathcal{Z}_0(\mathcal{L}_{g,n}^{\e}) \otimes_{(\mathcal{Z}_0 \cap \mathcal{Z}_1)(\mathcal{L}_{g,n}^{\e})} \mathcal{Z}_1(\mathcal{L}_{g,n}^{\e})$ is an integrally closed domain.\end{proof}

\begin{proof}[Proof of Theorem \ref{Lgnteo3}.] In this proof we write $\mathcal{Z} := \mathcal{Z}(\mathcal{L}_{g,n}^{\e})$, $\mathcal{Z}_0 := \mathcal{Z}_0(\mathcal{L}_{g,n}^{\e})$, and $\mathcal{Z}_1 := \mathcal{Z}_1(\mathcal{L}_{g,n}^{\e})$. Also we write $\hat{\mathcal{Z}}_0 = j(\mathcal{Z}_0\otimes_{\mathcal{Z}_0 \cap \mathcal{Z}_1} \mathcal{Z}_1)$, where $j\colon \mathcal{Z}_0 \otimes_{\mathcal{Z}_0 \cap \mathcal{Z}_1} \mathcal{Z}_1 \to \mathcal{Z}$ is the multiplication map.

\indent (1) That the multiplication map $j\colon \mathcal{Z}_0 \otimes_{\mathcal{Z}_0 \cap \mathcal{Z}_1} \mathcal{Z}_1 \to \mathcal{Z}$ is injective follows from Lemma \ref{lemTechnicalCommRings}(2), taking $A:=\mathcal{Z}_0$, $B:=\mathcal{Z}_0 \otimes_{\mathcal{Z}_0 \cap \mathcal{Z}_1} \mathcal{Z}_1$, $C:= \mathcal{Z}$, $f$ equal to the identification map of $\mathcal{Z}_0$ with $\mathcal{Z}_0 \otimes_{\mathcal{Z}_0 \cap \mathcal{Z}_1} 1 \subset \mathcal{Z}_0\otimes_{\mathcal{Z}_0 \cap \mathcal{Z}_1} \mathcal{Z}_1$, and $g:=j$. Let us show that all hypothesis of Lemma \ref{lemTechnicalCommRings}(2) are satisfied. First $j_{\vert A}$ is obviously injective, and $\mathcal{Z}_0\otimes_{\mathcal{Z}_0 \cap \mathcal{Z}_1} \mathcal{Z}_1$ and $\mathcal{Z}$ are domains (Prop.\,\ref{mainDCKPsec6} and \ref{propLgndomain}). Moreover $\mathcal{Z}_1(U_\e^\Pup)$ is a free module of rank $l^m$ over its subring $(\mathcal{Z}_0 \cap \mathcal{Z}_1)(U_\e^\Pup)$ (\cite[$\S$21.2-21.3]{DCP}, \cite[$\S$6]{DC-K-P1}). By Lemma \ref{Z0Z1Phimai25} the $\mathcal{Z}_0(\Ll_{0,1}^\e)$-module $\mathcal{Z}_0(\Ll_{0,1}^\e) \otimes_{(\mathcal{Z}_0 \cap \mathcal{Z}_1)(\Ll_{0,1}^\e)} \mathcal{Z}_1(\mathcal{L}_{0,1}^{\e})$ is therefore free of rank $l^m$, and from \eqref{factorZoct24} we deduce that $\mathcal{Z}_0\otimes_{\mathcal{Z}_0 \cap \mathcal{Z}_1} \mathcal{Z}_1$ is a free $\mathcal{Z}_0$-module of finite rank (=$l^{mn}$). Since $\mathcal{Z}_0 \cong \Oo(G^{2g+n})$ is Noetherian (by Prop.\,\ref{Z0Lgn}(1) and because $\mathcal{O}(G^{\times (2g+n)})$ is finitely generated), this last fact implies that $\mathcal{Z}_0\otimes_{\mathcal{Z}_0 \cap \mathcal{Z}_1} \mathcal{Z}_1$ is integral over $\mathcal{Z}_0$ by the argument in Lemma \ref{lemTechnicalCommRings}(1). This concludes the proof of injectivity. 

The surjectivity of $j$ follows from Lemma \ref{lemTechnicalCommRings}(1), taking  $A:=\hat{\mathcal{Z}}_0$ and $B:=\mathcal{Z}$. Indeed, we already know that $\mathcal{Z}$ is a domain, and by injectivity of $j$ (proved above) and Proposition \ref{mainDCKPsec6} the ring $\hat{\mathcal{Z}}_0$ is integrally closed. Since $\mathcal{Z}_0$ is a Noetherian ring and $\mathcal{Z}_0\otimes_{\mathcal{Z}_0 \cap \mathcal{Z}_1} \mathcal{Z}_1$ is a finite $\mathcal{Z}_0$-module (see above), $\mathcal{Z}_0 \otimes_{\mathcal{Z}_0 \cap \mathcal{Z}_1} \mathcal{Z}_1$ is a Noetherian ring, and therefore $\hat{\mathcal{Z}}_0$ is a Noetherian ring as well. But $\mathcal{Z}$ is a finite $\mathcal{Z}_0$-module (Prop.\,\ref{Z0Lgn}(3)), and $\mathcal{Z}_0 = j(\mathcal{Z}_0\otimes 1)\subset \hat{\mathcal{Z}}_0$, so $\mathcal{Z}$ is also a finite $\hat{\mathcal{Z}}_0$-module, and therefore a Noetherian $\hat{\mathcal{Z}}_0$-module. We noted in Corollary \ref{coroIsolateFacts} that $Q\bigl( \mathcal{Z} \bigr) = Q\bigl( \hat{\mathcal{Z}}_0 \bigr)$. Therefore Lemma \ref{lemTechnicalCommRings}(1) applies, and proves that $\mathcal{Z} = \hat{\mathcal{Z}}_0$. 

\smallskip

\indent (2) Since $\mathfrak{d}_{g,n}^{\e} : \mathcal{L}_{0,1}^\e \to \mathcal{L}_{g,n}^{\e}$ is injective (Prop.\,\ref{QMMinExactSeq}), the conclusion of item (1) above and a similar reasoning imply that $\mathcal{Z} \otimes_{\mathfrak{d}_{g,n}^{\e}\left( (\mathcal{Z}_0 \cap \mathcal{Z}_1)(\Ll_{0,1}^\e) \right)} \mathfrak{d}_{g,n}^{\e}\bigl( \mathcal{Z}_1(\Ll_{0,1}^\e) \bigr)$ is integrally closed, and that $j_2$ is injective. To prove surjectivity, we use again Lemma \ref{lemTechnicalCommRings}(1), taking $A= \hat{\mathcal{Z}}_0(\Ll_{g,n}^{u_\e}):= \mathrm{Im}(j_2)$ and $B=\mathcal{Z}(\Ll_{g,n}^{u_\e})$. It remains to show that $\mathcal{Z}(\Ll_{g,n}^{u_\e})$ is a Noetherian $\hat{\mathcal{Z}}_0(\Ll_{g,n}^{u_\e})$-module and that $Q\bigl(\mathcal{Z}(\Ll_{g,n}^{u_\e})\bigr) = Q\bigl( \hat{\mathcal{Z}}_0(\Ll_{g,n}^{u_\e}) \bigr)$.

For the first claim we know that $\mathcal{Z}(\Ll_{g,n}^{u_\e})$ is a finite Noetherian $\mathcal{Z}_0$-module and a Noetherian ring (Rk.\,\ref{Z0submoduleLgn}), and $\mathcal{Z}_0 = j_2(\mathcal{Z}_0\otimes 1) \subset \hat{\mathcal{Z}}_0(\Ll_{g,n}^{u_\e}) \subset \mathcal{Z}(\Ll_{g,n}^{u_\e})$. Therefore $\mathcal{Z}(\Ll_{g,n}^{u_\e})$ is a finitely generated $\hat{\mathcal{Z}}_0(\Ll_{g,n}^{u_\e})$-module. Also $\mathcal{Z} \otimes_{\mathfrak{d}_{g,n}^{\e}\left( (\mathcal{Z}_0 \cap \mathcal{Z}_1)(\Ll_{0,1}^\e) \right)} \mathfrak{d}_{g,n}^{\e}\bigl( \mathcal{Z}_1(\Ll_{0,1}^\e) \bigr)$ is finitely generated over $\mathcal{Z}_0\otimes 1$, so it is a Noetherian ring. Therefore its image by $j_2$, $\hat{\mathcal{Z}}_0(\Ll_{g,n}^{u_\e})$, is also a Noetherian ring. As we saw $\mathcal{Z}(\Ll_{g,n}^{u_\e})$ is finite over $\hat{\mathcal{Z}}_0(\Ll_{g,n}^{u_\e})$, it is thus a Noetherian $\hat{\mathcal{Z}}_0(\Ll_{g,n}^{u_\e})$-module, as expected. 

Finally let us check that $Q\bigl(\mathcal{Z}(\Ll_{g,n}^{u_\e})\bigr) = Q\bigl( \hat{\mathcal{Z}}_0(\Ll_{g,n}^{u_\e}) \bigr)$. We have $Q\bigl(\mathcal{Z}(\Ll_{g,n}^{u_\e})\bigr)) =  \mathcal{Z}\bigl(Q(\Ll_{g,n}^{u_\e})\bigr)$ by Lemma \ref{lemTrucsGenerauxQR}(2). Let us write
\[ S := \mathfrak{d}_{g,n}^{\e}(\mathcal{L}_{0,1}^\e), \quad \mathcal{Z}_1(S) := \mathfrak{d}_{g,n}^{\e}\bigl(\mathcal{Z}_1(\mathcal{L}_{0,1}^\e)\bigr), \quad (\mathcal{Z}_0\cap \mathcal{Z}_1)(S) := \mathfrak{d}_{g,n}^{\e}\bigl(\mathcal{Z}_0(\mathcal{L}_{0,1}^\e) \cap \mathcal{Z}_1(\mathcal{L}_{0,1}^\e)\bigr). \]
We have 
\[ Q\bigl( \mathcal{Z} \otimes_{(\mathcal{Z}_0\cap \mathcal{Z}_1)(S)} \mathcal{Z}_1(S)\bigr) =  Q\bigl( \mathcal{Z} \bigr) \otimes_{(\mathcal{Z}_0\cap \mathcal{Z}_1)(S)} \mathcal{Z}_1(S) \]
since the latter is a field (it embeds by $j_2$ in the domain $\Ll_{g,n}^{u_\e}$). Then
\[ Q\bigl( \hat{\mathcal{Z}}_0(\Ll_{g,n}^{u_\e}) \bigr) = j_2\bigl( Q\bigl( \mathcal{Z} \otimes_{(\mathcal{Z}_0\cap \mathcal{Z}_1)(S)} \mathcal{Z}_1(S)\bigr)\bigr) = j_2\bigl( Q\bigl( \mathcal{Z} \bigr) \otimes_{(\mathcal{Z}_0\cap \mathcal{Z}_1)(S)} \mathcal{Z}_1(S)\bigr). \]
We saw at the end of the proof of Theorem \ref{Lgnteo2} that this is $\mathcal{Z}(Q(\Ll_{g,n}^{u_\e}))$, so $Q\bigl(\mathcal{Z}(\Ll_{g,n}^{u_\e})\bigr) = Q\bigl( \hat{\mathcal{Z}}_0(\Ll_{g,n}^{u_\e}) \bigr)$ and the proof of surjectivity of $j_2$ is completed.
\end{proof}

Let $\mathfrak{A}$ be either $\Ll_{g,n}^{\e}$ or $\Ll_{g,n}^{u_\e}$. The following result is of the same flavour as \cite[Prop.\,20.2(ii)]{DCP}.

\begin{cor}\label{corZZ0} The ring $\mathcal{Z}(\mathfrak{A})$ coincides with the trace ring of $\mathfrak{A}$, and it is a free $\mathcal{Z}_0(\mathcal{L}_{g,n}^{\e})$-module of rank $l^{mn}$ (for $\mathfrak{A}=\Ll_{g,n}^{\e}$) or $l^{m(n+1)}$ (for $\mathfrak{A}=\Ll_{g,n}^{u_\e}$). 
\end{cor}

\begin{proof} The first claim follows from Th.\,\ref{Lgnteo3}  ({\it i.e.}, that $\mathcal{Z}(\mathfrak{A})$ is integrally closed and Noetherian) and Th.\,\ref{Lgnteo1}. Indeed, if a domain $\mathfrak{A}$ has Noetherian and integrally closed center $\mathcal{Z}$ and is such that $Q(\mathfrak{A})$ has finite dimension over $Q(\mathcal{Z})$, then $\mathcal{Z}$ coincides with the {\it trace ring} of $\mathfrak{A}$ (see \cite[Th.\,10.1]{Reiner}). Denote by $t_{\rm red}\colon Q(\mathfrak{A}) \rightarrow Q(\mathcal{Z}\bigl( \mathfrak{A}) \bigr)$ the reduced trace map of the central simple algebra $Q(\mathfrak{A})$ (see eg. \cite[\S 9]{Reiner}). Because $\mathcal{Z}(\mathfrak{A})$ is the trace ring of $\mathfrak{A}$, we have $\mathcal{Z}(\mathfrak{A}) = t_{\rm red}(\mathfrak{A})$. Trivially the inclusion map $i\colon \mathcal{Z}(\mathfrak{A})\rightarrow Q(\mathfrak{A})$ satisfies $t_{\rm red} \circ i = \mathrm{id}$, so $\mathcal{Z}(\mathfrak{A})$ is a direct summand of $\mathfrak{A}$ as a $\mathcal{Z}_0(\mathcal{L}_{g,n}^{\e})$-module. But $\mathfrak{A}$ is free over $\mathcal{Z}_0(\mathfrak{A})$, so $\mathcal{Z}(\mathfrak{A})$ is a projective $\mathcal{Z}_0(\mathcal{L}_{g,n}^{\e})$-module. Since $\mathcal{Z}_0(\mathcal{L}_{g,n}^{\e}) \cong \mathcal{O}\bigl( G^{\times (2g+n)} \bigr)$ and $G^{\times (2g+n)}$ is simply-connected and semisimple, as at the end of the proof of \cite[Th.\,4.9]{BR2} we can use Bass' cancellation theorem and Marlin's computation of $K_0(G^{\times (2g+n)})$ in \cite{Marlin} to deduce that the module is free. The rank computation follows from $[Q(\mathcal{Z}(\mathcal{L}_{g,n}^{\e})):Q(\mathcal{Z}_0(\mathcal{L}_{g,n}^{\e}))]=l^{mn}$ (Cor.\,\ref{coroIsolateFacts}) and $\bigl[ Q\bigl(\mathcal{Z}(\Ll_{g,n}^{u_\e})\bigr):Q\bigl(\mathcal{Z}_0(\Ll_{g,n}^{\e})\bigr) \bigr] = l^{m(n+1)}$ (Cor.\,\ref{rkjanv26}). 
\end{proof}

In the rest of this section we describe the $\Gamma_\e$-invariant elements of the centers in Th.\,\ref{Lgnteo3}. Recall that if $A$ is a module-algebra over a Hopf algebra (or a group) $H$, then we denote by $A^H$ the subalgebra of $H$-invariant elements in $A$. More generally, if $B \subset A$ is any subalgebra then the subspace $B^H$ is actually a subalgebra (of $B$).

\indent Recall that $\mathcal{O}(G)^{G}$ is the polynomial algebra on the $m$ generators $Q_j$ defined by
\begin{equation}\label{carclassmai25}
Q_j(a) := \mathrm{Tr}\bigl(\rho_{\varpi_j}(a)\bigr),
\end{equation} where $\rho_{\varpi_j}\colon G \to \mathrm{End}(V_{\varpi_j})$ is the $j$-th fundamental representation of $G$ (see \cite[Th.\,6.1]{Steinberg}). Also, since $\mathbb{F}\mathrm{r}_\e$ is a morphism of Hopf algebras with values in $U(\mathfrak{g})$, we have
\begin{align}
\begin{split}\label{equivUresmai25}
&\bigl\langle \mathrm{coad}^r(h)\big(\mathbb{F}\mathrm{r}_\e^*(\varphi)\bigr), x \bigr\rangle^\Q_\e \overset{\eqref{defCoad}}{=} \bigl\langle \varphi, \mathbb{F}\mathrm{r}_\e(h)_{(1)}\mathbb{F}\mathrm{r}_\e(x) S\bigl( \mathbb{F}\mathrm{r}_\e(h)_{(2)} \bigr) \bigr\rangle_{\mathrm{cl}}\\
=\:\,&\bigl\langle \varphi, \mathbb{F}\mathrm{r}_\e(x) S\bigl( \mathbb{F}\mathrm{r}_\e(h)_{(2)} \bigr) \mathbb{F}\mathrm{r}_\e(h)_{(1)}\bigr\rangle_{\mathrm{cl}} = \varepsilon(h)\bigl\langle \mathbb{F}\mathrm{r}_\e^*(\varphi), x \bigr\rangle^\Q_\e
\end{split}
\end{align}
for every $\varphi\in \mathcal{O}(G)^{G}$ and $h,x\in \Gamma_\e^\Q$, where $\langle \text{-},\text{-} \rangle_{\mathrm{cl}} : \mathcal{O}(G) \times U(\mathfrak{g}) \to \mathbb{C}$ is the usual duality pairing; the second equality is by the trace-like property of $\varphi$ while the last equality uses that $U(\mathfrak{g})$ is co-commutative.

The equality \eqref{equivUresmai25} shows that $\mathbb{F}\mathrm{r}_\e^*$ embeds $\mathcal{O}(G)^{G}$ in $\Oo_\e^{\Gamma_\e} = \mathcal{L}_{0,1}^{\Gamma_\epsilon}$. 

\indent Denote
$$\mathcal{O}\bigl[ \mathcal{X}_G(\Sigma_{g,n}^{\circ}) \bigr]  := \bigl( \mathcal{O}(G)^{\otimes (2g+n)} \bigr)^{G}$$
where the invariant elements are defined with respect to the diagonal coadjoint action of $G$. The notation $\mathcal{X}_G(\Sigma_{g,n}^\circ)$ stands for the character variety $\mathrm{Hom}_{\mathrm{Grp}}\bigl( \pi_1(\Sigma_{g,n}^\circ),G) \bigr)/\!/G$ of the oriented surface of genus $g$ with $n+1$ boundary components, and the above equality is the definition of its coordinate ring. We moreover define
\[ \mathcal{Z}_0^{\rm inv}  := (\mathbb{F}\mathrm{r}_{\epsilon}^*)^{(2g+n)}\bigl( \mathcal{O}\bigl[ \mathcal{X}_G(\Sigma_{g,n}^{\circ}) \bigr] \bigr) \subset \mathcal{Z}_0(\Ll_{g,n}^{\e}). \]
Also, let us write $\mathcal{Z} := \mathcal{Z}(\mathcal{L}_{g,n}^{\e})$, $\mathcal{Z}_0 := \mathcal{Z}_0(\mathcal{L}_{g,n}^{\e})$, $\mathcal{Z}_1 := \mathcal{Z}_1(\mathcal{L}_{g,n}^{\e})$ and  $(\mathcal{Z}_0 \cap\mathcal{Z}_1) (\mathcal{L}_{0,1}^\e) = \mathcal{Z}_0(\mathcal{L}_{0,1}^\e)  \cap\mathcal{Z}_1(\mathcal{L}_{0,1}^\e)$. 
\begin{cor}\label{FrisoG-inv} (i) The following multiplication maps are isomorphisms of algebras 
\begin{align*}
\mathcal{Z}_0^{\mathrm{inv}}  \otimes_{\mathcal{Z}_0^{\mathrm{inv}} \cap \mathcal{Z}_1}  \mathcal{Z}_1 & \to \mathcal{Z}^{\Gamma_\e}\\
\mathcal{Z}^{\Gamma_\epsilon} \otimes_{\mathfrak{d}_{g,n}^\e((\mathcal{Z}_0 \cap\mathcal{Z}_1) (\mathcal{L}_{0,1}^\e))} \mathfrak{d}_{g,n}^\e(\mathcal{Z}_1(\Ll_{0,1}^{\e})) & \to \mathcal{Z}(\Ll_{g,n}^{u_\e})^{\Gamma_\e}.
\end{align*}
(ii) We have 
\[ (\mathcal{Z}_0 \cap\mathcal{Z}_1) (\mathcal{L}_{0,1}^\e) =\mathbb{F}\mathrm{r}_{\epsilon}^*\bigl(\mathcal{O}(G)^{G}\bigr) = \mathbb{F}\mathrm{r}_{\epsilon}^*\bigl(\mc[Q_1,\ldots, Q_m]\bigr) \]
and the equalities $\mathcal{L}_{0,1}^{u_\epsilon}= \mathcal{Z}(\mathcal{L}_{0,1}^{\e})$ and  $\mathcal{L}_{0,1}^{\Gamma_\epsilon} =  \mathcal{Z}_1(\Ll_{0,1}^\e)$.

(iii) The elements $(\mathrm{qTr}_{l\lambda})_{\vert \e}$, $\lambda\in P_+$, defined after \eqref{casimirElmtsL01} form a basis of $(\mathcal{Z}_0 \cap\mathcal{Z}_1) (\mathcal{L}_{0,1}^\e)$.
\end{cor} 

\begin{proof} (i) As $\mathcal{Z}_0 = (\mathbb{F}\mathrm{r}_{\epsilon}^*)^{(2g+n)}\bigl(\mathcal{O}(G)^{\otimes (2g+n)}\bigr)$, the definition of $\mathbb{F}\mathrm{r}_\e$ and the equivariance \eqref{equivUresmai25} imply
$$\mathcal{Z}_0^{\Gamma_\e} = (\mathbb{F}\mathrm{r}_{\epsilon}^*)^{(2g+n)}\left(\bigl( \mathcal{O}(G)^{\otimes (2g+n)} \bigr)^{U(\mathfrak{g})} \right) = \mathcal{Z}_0^{\rm inv}.$$
Also $\mathcal{Z}_1$ is $(\Gamma_l)_\e$-invariant (since it is the specialization of the $U_q$-invariant algebra $1^{\otimes 2g} \otimes \mathcal{Z}(\Ll_{0,n})$), and using the isomorphism $j$ of Th.\,\ref{Lgnteo3} it is immediate that $\mathcal{Z}$ is a $\Gamma_\e$-module algebra. These facts imply that the first multiplication map is an isomorphism. Since $\mathfrak{d}_{g,n}^\e$ is $\Gamma_\e$-equivariant, one obtains from Theorem \ref{Lgnteo3}(2) that the second multiplication map is an isomorphism.

(ii) The second equality follows from \eqref{carclassmai25}. Let us identify the maximal torus $T\subset G$ with ${\rm Hom}_{\rm Grp}(2lP,\mc^\times)$, and thus $\mc[2lP]$ with the coordinate ring $\Oo(T)$. The Chevalley isomorphism, \textit{i.e.} the restriction map $\mathcal{O}(G)^{G}\to \Oo(T)^W$, is an isomorphism. Let us identify $U_\e^\Pup(\mathfrak{h})$ with the group algebra $\mc[P]$ via the map $K_\mu\mapsto \mu$ for all $\mu\in P$. Then $\mathcal{Z}_0\bigl( U_\e^\Pup(\mathfrak{h}) \bigr) := U_\e^\Pup(\mathfrak{h}) \cap \mathcal{Z}_0\bigl( U_\e^\Pup\bigr) = \mc[K_{l\mu}, \mu\in P]$ is identified with $\mc[lP]$, and it is proved in \cite[Prop.\,6.3]{DC-K-P1} (or \cite[\S 21.2]{DCP}) that $\mathcal{Z}_0 (U_\e^\Pup) \cap\mathcal{Z}_1(U_\e^\Pup)$ is identified with $\mc[2lP]^W = \Oo(T)^W = \mathcal{O}(G)^{G}$ as a subalgebra of $\mathcal{Z}_0(U_\e^\Pup)$. Since $\Phi_{0,1}^\e\colon \mathcal{Z}_0(\mathcal{L}_{0,1}^\e) \cap\mathcal{Z}_1 (\mathcal{L}_{0,1}^\e) \to \mathcal{Z}_0 (U_\e^\Pup) \cap\mathcal{Z}_1(U_\e^\Pup)$ is $\Gamma_\e$-equivariant and an isomorphism (Lemma \ref{Z0Z1Phimai25}), the equality $(\mathcal{Z}_0 \cap\mathcal{Z}_1) (\mathcal{L}_{0,1}^\e)= \mathbb{F}\mathrm{r}_{\epsilon}^*\bigl(\mathcal{O}(G)^{G}\bigr)$ follows.

One obtains the equality $\mathcal{L}_{0,1}^{u_\epsilon} = \mathcal{Z}(\mathcal{L}_{0,1}^{\e})$ from the equivariant embedding $\Phi_{0,1}\colon \mathcal{L}_{0,1}^\e \to U_\e^\Pup$ (as usual the source and target are given the actions ${\rm coad}^r$ \eqref{defCoad} and ${\rm ad}^r$ \eqref{defAdr}) and the fact that $\mathcal{Z}(U_\e^\Pup) = (U_\e^\Pup)^{U_\e^\Pup}$. Since $\mathbb{F}\mathrm{r}_{\epsilon}^*\bigl(\mathcal{O}(G)^{G}\bigr) = (\mathcal{Z}_0 \cap\mathcal{Z}_1) (\mathcal{L}_{0,1}^\e)\subset \mathcal{Z}_1 (\mathcal{L}_{0,1}^\e)$, the equality $\mathcal{L}_{0,1}^{u_\epsilon} = \mathcal{Z}(\mathcal{L}_{0,1}^{\e})$ and the fact that the multiplication map $\mathcal{Z}_0^{\rm inv}  \otimes_{\mathcal{Z}_0^{\rm inv} \cap \mathcal{Z}_1}  \mathcal{Z}_1 \to \mathcal{Z}^{\Gamma_\e}$ is an isomorphism eventually give $$\mathcal{L}_{0,1}^{\Gamma_\epsilon}  = \mathcal{Z}(\mathcal{L}_{0,1}^\e)^{\Gamma_\e} = \mathcal{Z}_1(\mathcal{L}_{0,1}^{\e}).$$

(iii) From App.\,\ref{sec:qKilling} we have $(\mathcal{P}_\rho \circ \Phi_{0,1})(\mathrm{qTr}_{ l\lambda}) = \mathrm{ch}^{(2)}_{l\lambda}$. The elements $\mathrm{ch}^{(2)}_{l\lambda}$, $\lambda\in P_+$, form a basis of the ring $\mc[K_{2 l \mu}, \mu\in P]^W$, and $\mathcal{Z}_0 (U_\e^\Pup) \cap\mathcal{Z}_1(U_\e^\Pup) = \mc[K_{2 l \mu}, \mu\in P]^W$ (see (ii) above). The conclusion follows from Lemma \ref{Z0Z1Phimai25}. 

This concludes the proof. \end{proof}

\begin{remark}\label{annoncemod5}{\rm Recall that $\mathcal{Z}_1(\Ll_{0,1}^\e)$ is the specialization to $q=\e$ of the central algebra $\mathcal{Z}_1(\Ll_{0,1}^A) = (\Phi_{0,1}^{A})^{-1}\bigl( \mathcal{Z}(U_A^\Pup) \bigr)$ (see Lemma \ref{Z0Z1Phimai25}). Since $\mathcal{L}_{0,1}^A = \mathcal{O}_A$ as $\Gamma_A$-modules and $\Phi_{0,1}^A$ is an algebra embedding, the inclusion $\Phi_{0,1}^A\bigl( \mathcal{Z}(\mathcal{L}_{0,1}^A) \bigr) \subset \mathcal{Z}(U_A^\Pup)$ from Remark \ref{rmkTechnicalPropUA}(2) together with Cor.\,\ref{Phi0Aiso} below imply that $\mathcal{Z}(\Ll_{0,1}^A) = \mathcal{Z}_1(\Ll_{0,1}^A) = \Ll_{0,1}^{\Gamma_A}$. Thus we can reformulate the last claim of Cor.\,\ref{FrisoG-inv} as the fact that
$$\mathcal{L}_{0,1}^{\Gamma_\epsilon} = (\Ll_{0,1}^{\Gamma_A})_{\vert q=\e}.$$
More generally, for any $g$, $n$ one can prove that  $\mathcal{L}_{g,n}^{\Gamma_\epsilon} =  (\Ll_{g,n}^{\Gamma_A})_{\vert \qD=\eD}$. This will be developed in \cite{BF}.}
\end{remark}

\appendix

\section{The quantum Killing form and \texorpdfstring{$\Phi_{0,1}$}{its applications}}\label{sec:qKilling} This appendix justifies some properties of $\Phi_{0,1}$ used in the proof of Lemma \ref{Z0Z1Phimai25}.  They rely on the quantum Killing form.
 
\smallskip

Recall the pairing $\rho : U_q^\Pup(\mathfrak{b}_-)^{\rm cop} \times U_q^\Pup(\mathfrak{b}_+) \to \mathbb{C}(q_\D)$ from \eqref{defrho}. Observe that any element of $U_q^\Pup$ can be written as a linear combination of elements of the form $S(X_-)K_\lambda X_+$ with $X_\pm \in U_q(\mathfrak{n}_\pm)$ and $\lambda \in P$; the same remark applies with elements of the form $S(X_+)K_\lambda X_-$.\footnote{Note that this is different from the triangular decompositions $U_q^\Pup = U_q(\mathfrak{n}_-)U_q^\Pup(\mathfrak{h})U_q(\mathfrak{n}_+)$ and $U_q^\Pup = U_q(\mathfrak{n}_+)U_q^\Pup(\mathfrak{h})U_q(\mathfrak{n}_-)$ because the subalgebras $U_q(\mathfrak{n}_\pm)$ are not stable by the antipode $S$.} Extend the ground field $\mc(\qD)$ by an element $\qD^{1/2}$ such that $(\qD^{1/2})^2 = \qD$. We define a bilinear map $(\:\,|\,\:) : U_q^\Pup \times U_q^\Pup \to \mc\bigl(\qD^{1/2}\bigr)$ by 
\begin{equation}\label{qKillingnous}
\bigl( S(X_-) K_\lambda X_+ \,\big|\, S(Y_+) K_\mu Y_- \bigr) = \qD^{D(\lambda,\mu)/2} \, \rho\bigl(X_-, Y_+\bigr) \, \rho\bigl(Y_-,X_+\bigr)
\end{equation}
for all $X_\pm,Y_\pm\in U_q(\mathfrak{n}_\pm)$, and $\lambda, \mu\in P$. Recall that $\qD$ is such that $(\qD)^D  = q$ and we have to use this extension of $\mathbb{C}(q)$ because $\textstyle (\lambda,\mu) \in \frac{1}{D}\mathbb{Z}$ in general.
\smallskip

\indent The bilinear map $(\:\,|\,\:)$ is a variant of the {\em quantum Killing form} $\kappa$ defined in \cite[Def.\,3.78]{VY}. We have to introduce a different version of this bilinear form because here we use the map $\Phi_{0,1} : \mathcal{O}_q \to U_q^\Pup$, $\alpha \mapsto (\alpha \otimes \mathrm{id})(RR^{fl})$ while \cite{VY} uses the map $I : \mathcal{O}_q \to U_q^\Pup$, $\alpha \mapsto (\alpha \otimes \mathrm{id})(R^{-1}(R^{fl})^{-1})$. The defining formula \eqref{qKillingnous} is indeed designed so that the following theorem holds:
\begin{teo}\label{teoqKillingnous}
(i) For all $\alpha \in \mathcal{O}_q$ and $x \in U_q^\Pup$ we have
\begin{equation}\label{relationEntrePhi01EtKilling}
\bigl( \Phi_{0,1}(\alpha) \,\big|\, x \bigr) = \langle \alpha, x \rangle^\Pup
\end{equation}
where we recall that $\langle \alpha, - \rangle^\Pup$ denotes evaluation of $\alpha$ extended from $U_q^\Q$ to $U_q^\Pup$, see \S\ref{sec:Oq}.
\\(ii) The form $(\:\,|\,\:)$ is non degenerate and is invariant under the adjoint action, in the following sense:
\[ \forall \, h,x,y\in U_q^\Pup, \quad \bigl(\mathrm{ad}^r(h)(x) \,|\, y \bigr) = \bigl( x \,\big|\, \mathrm{ad}^l(h)(y) \bigr) \]
where $\mathrm{ad}^r(h)(x) = \textstyle \sum_{(h)} S(h_{(1)})xh_{(2)}$ and $\mathrm{ad}^l(h)(x) = \textstyle \sum_{(h)} h_{(1)}xS(h_{(2)})$.
\\(iii)  Recall that $(U_q^\Pup)^{\mathrm{lf}}$ denotes the subspace of elements in $U_q^\Pup$ which have finite-dimensional orbit under $\mathrm{ad}^r$. There is a linear map $J_q\colon (U_q^\Pup)^{\rm lf} \to \Oo_q$ uniquely characterized by
\[ \forall \, x \in (U_q^\Pup)^{\mathrm{lf}}, \:\: \forall \, y \in U_q^\Pup, \quad \langle J_q(x),y\rangle^\Pup = (x \,|\, y). \]
The map $J_q$ is the inverse of $\Phi_{0,1}\colon \Oo_q\to (U_q^\Pup)^{\rm lf}$, and hence an isomorphism of $U_q^\Pup$-modules $\bigl( (U_q^\Pup)^{\mathrm{lf}}, \mathrm{ad}^r \bigr) \overset{\sim}{\to} \bigl( \mathcal{O}_q, \mathrm{coad}^r \bigr)$, with $\mathrm{coad}^r$ from \eqref{defCoad}
\end{teo}
\begin{proof}
(i) We use the notation $X_\pm^{\boldsymbol{s}}$ for PBW monomials of $U_q(\mathfrak{n}_\pm)$ introduced in Lemma \ref{lemRhoVsRMat}, and the notation $\boldsymbol{1}_\lambda$ for projection on weight spaces introduced in the proof of item 2 of that lemma. In the second equality below we take advantage of the fact that $R = S^{\otimes 2}(R) = S^{\otimes 2}(\Theta\hat{R}) = S^{\otimes 2}(\hat{R})\Theta$. Let $\boldsymbol{a}, \boldsymbol{b} \in \mathbb{N}^N$ and recall from \eqref{expressionRho} the orthogonality of PBW monomials with respect to $\rho$, \textit{i.e.} $\rho(X_-^{\boldsymbol{s}}, X_+^{\boldsymbol{a}}) = \delta_{\boldsymbol{s},\boldsymbol{a}} \rho(X_-^{\boldsymbol{a}}, X_+^{\boldsymbol{a}})$. We thus have
\begin{align*}
&\bigl( \Phi_{0,1}(\alpha) \,\big|\, S(X_+^{\boldsymbol{a}})K_\mu X_-^{\boldsymbol{b}} \bigr) = \bigl( (\alpha \otimes \mathrm{id})(RR^{fl}) \,\big|\, S(X_+^{\boldsymbol{a}})K_\mu X_-^{\boldsymbol{b}} \bigr) \quad \text{\footnotesize by def of $\Phi_{0,1}$}\\
=\:& \bigl( (\alpha \otimes \mathrm{id})\bigl( S^{\otimes 2}(R)R^{fl}\bigr) \,\big|\, S(X_+^{\boldsymbol{a}})K_\mu X_-^{\boldsymbol{b}} \bigr) \\
=\:& \bigl( (\alpha \otimes \mathrm{id})\bigl( S^{\otimes 2}(\hat{R})\Theta^2 \hat{R}^{fl} \bigr) \,\big|\, S(X_+^{\boldsymbol{a}})K_\mu X_-^{\boldsymbol{b}} \bigr) \quad \text{\footnotesize by remark above}\\
=\:&\sum_{\boldsymbol{s},\boldsymbol{t},\lambda}\alpha\bigl( S(X_+^{\boldsymbol{s}}) \boldsymbol{1}_\lambda X_-^{\boldsymbol{t}} \bigr) \frac{ \bigl( S(X_-^{\boldsymbol{s}}) K_{2\lambda} X_+^{\boldsymbol{t}} \,\big|\, S(X_+^{\boldsymbol{a}})K_\mu X_-^{\boldsymbol{b}} \bigr)}{\rho(X_-^{\boldsymbol{s}}, X_+^{\boldsymbol{s}})\, \rho(X_-^{\boldsymbol{t}}, X_+^{\boldsymbol{t}})} \quad \text{\footnotesize by Lemma \ref{lemRhoVsRMat}(1)}\\
=\:&\sum_{\boldsymbol{s},\boldsymbol{t},\lambda}\alpha\bigl( S(X_+^{\boldsymbol{s}}) \boldsymbol{1}_\lambda X_-^{\boldsymbol{t}} \bigr) \, \qD^{D(\lambda,\mu)} \frac{\rho(X_-^{\boldsymbol{s}}, X_+^{\boldsymbol{a}})\rho(X_-^{\boldsymbol{b}},X_+^{\boldsymbol{t}})}{\rho(X_-^{\boldsymbol{s}}, X_+^{\boldsymbol{s}})\rho(X_-^{\boldsymbol{t}}, X_+^{\boldsymbol{t}})} \quad \text{\footnotesize by \eqref{qKillingnous}}\\
=\:&\sum_{\lambda}\alpha\bigl( S(X_+^{\boldsymbol{a}}) \boldsymbol{1}_\lambda X_-^{\boldsymbol{b}} \bigr) \, \qD^{D(\lambda,\mu)} \frac{\rho(X_-^{\boldsymbol{a}}, X_+^{\boldsymbol{a}})\rho(X_-^{\boldsymbol{b}},X_+^{\boldsymbol{b}})}{\rho(X_-^{\boldsymbol{a}}, X_+^{\boldsymbol{a}})\rho(X_-^{\boldsymbol{b}}, X_+^{\boldsymbol{b}})} \quad \text{\footnotesize by orthogonality}\\
=\:& \alpha\!\left[ S(X_+^{\boldsymbol{a}}) \left({\textstyle \sum_\lambda} \qD^{D(\lambda,\mu)} \boldsymbol{1}_\lambda \right) X_-^{\boldsymbol{b}} \right ] = \alpha\bigl( S(X_+^{\boldsymbol{a}}) K_\mu X_-^{\boldsymbol{b}} \bigr) \quad \text{\footnotesize by def of $\boldsymbol{1}_\lambda$.}
\end{align*}
(ii) The non-degeneracy of $(\:\,|\,\:)$  can be deduced from the non-degeneracy of $\rho$ \cite[Th.\,3.92]{VY}. The ad-invariance is obtained by direct computations similar to the ones in \cite[Prop.\,3.79]{VY}. 

\indent (iii) The formula $\langle J_q(x),y\rangle^\Pup := (x \,|\, y)$ for all $x,y\in U_q^\Pup$ defines a linear map $J_q: U_q^\Pup \to (U_q^\Pup)^*$. It is injective, since $(\:\,|\,\:)$ is non degenerate. Moreover, due to the ad-invariance property of $(\:\,|\,\:)$, the map $J_q$ intertwines the $H$-actions $\mathrm{ad}^r$ and $\mathrm{coad}^r$:
\begin{align*}
\bigl\langle \mathrm{coad}^r(h)\bigl(J_q(x)\bigr),y \bigr\rangle^\Pup &= \bigl\langle J_q(x), \mathrm{ad}^l(h)(y) \bigr\rangle^\Pup\\
&= \bigl( x \,\big|\, \mathrm{ad}^l(h)(y) \bigr) = \bigl( \mathrm{ad}^r(h)(x) \,\big|\, y \bigr) = \bigl\langle J_q\bigl(\mathrm{ad}^r(h)(x)\bigr), y \bigr\rangle^\Pup.
\end{align*}
This fact allows one to deduce that $J_q$ actually takes values in $\mathcal{O}_q \subset (U_q^\Pup)^*$, as done in \cite[Prop.\,3.116]{VY}. Let us give some details within our conventions for completeness. The formula \eqref{qKillingnous} shows that for all $\lambda\in P_+$ the linear form $J_q(K_{-2\lambda})$ vanishes on elements $S(Y_+)K_\mu Y_-$ such that $Y_+$ or $Y_-$ belongs to $\ker(\varepsilon_{U_q})$. Recall that for $\lambda \in P_+$, $V_{-w_0(\lambda)}$ is the simple $U_q^\Q$-module with highest weight $-w_0(\lambda)$, hence with lowest weight $-\lambda \in -P_+$. Let $v$ be a (unique up to scalar) lowest weight vector. Define $\psi_{-\lambda}\in \Oo_q$ by $\psi_{-\lambda}(x) = v^*(x\cdot v)$ for every $x\in U_q^\Q$, where $v^*\in V_{-w_0(\lambda)}^*$ is such that $v^*(v)=1$ and $v^*$ vanishes on all other weight spaces in $V_{-w_0(\lambda)}$. Then we have
\[ \forall \, \mu \in P, \quad \bigl\langle J_q(K_{-2\lambda}), K_\mu \bigr\rangle^\Pup = ( K_{-2\lambda} \,|\, K_\mu) = q^{-(\lambda,\mu)} =  \langle \psi_{-\lambda}, K_\mu \rangle^\Pup. \]
Also, it is readily seen by very definition of $\psi_{-\lambda}$ that
\begin{align*}
\begin{split}
\bigl\langle \psi_{-\lambda}, S(Y_+)K_\mu Y_-\bigr\rangle^\Pup &= \varepsilon(Y_+) \psi_{-\lambda}(K_\mu)\varepsilon(Y_-)\\
&= \bigl( K_{-2\lambda} \,\big|\, S(Y_+)K_\mu Y_- \bigr) = J_q(K_{-2\lambda})\bigl( S(Y_+)K_\mu Y_- \bigr)
\end{split}
\end{align*}
for all $Y_\pm\in U_q(\mathfrak{n}_\pm)$ and $\mu \in P$. Therefore $J_q(K_{-2\lambda}) = \psi_{-\lambda} \in \mathcal{O}_q$. Now recall that we have $(U_q^\Pup)^{\rm lf} = \textstyle \bigoplus_{\lambda\in P_+} \mathrm{ad}^r(U_q^\Pup)(K_{-2\lambda})$ (see \cite[Th.\,3.113]{VY}, where the different adjoint action implies that $K_{-2\lambda}$ is replaced with $K_{2\lambda}$ in their formula). Hence, by equivariance of $J_q$,
\[ J_q\bigl( (U_q^\Pup)^{\rm lf} \bigr) = \textstyle \bigoplus_{\lambda\in P_+} \mathrm{coad}^r(U_q^\Pup)\bigl( J_q(K_{-2\lambda}) \bigr) = \bigoplus_{\lambda\in P_+} \mathrm{coad}^r(U_q^\Pup)\bigl( \psi_{-\lambda} \bigr) \subset \mathcal{O}_q. \]
It now makes sense to check that $J_q$ and $\Phi_{0,1}$ are inverse each other, using non-degeneracy of the evaluation pairing and of the quantum Killing form,
\begin{align*}
&\bigl\langle J_q\bigl( \Phi_{0,1}(\alpha) \bigr), y \bigr\rangle^\Pup = \bigl( \Phi_{0,1}(\alpha) \,\big|\, y \bigr) = \langle \alpha, y \rangle^\Pup,\\
&\bigl( \Phi_{0,1}\bigl(J_q(x)\bigr) \,\big|\, y \bigr) = \bigl\langle J_q(x), y \bigr\rangle^\Pup = (x \,|\,y)
\end{align*}
for all $x,y\in U_q^\Pup$ and $\alpha \in \mathcal{O}_q$.
\end{proof}

For the next result, let $\mathcal{P}_\rho :  U_q^\Pup \to U_q^\Pup(\mathfrak{h})$ be the projection given by
\[ \forall \, X_\pm \in U_q^\Pup(\mathfrak{n}_\pm), \:\: \forall \, \mu \in P, \quad \mathcal{P}_\rho(X_+ K_\mu X_-) = \varepsilon(X_+)\varepsilon(X_-)q^{-(\rho,\mu)}K_\mu \]
where $\rho$ here denotes the half-sum of positive roots. Also let $\mathcal{O}_q^{U_q}$ be the subspace of invariant elements under the right action $\mathrm{coad}^r$ defined in \eqref{defCoad}. 

\begin{prop}\label{propPropietesPhi01App}
(i) $\Phi_{0,1}\colon \Oo_q^{U_q}\to \mathcal{Z}(U_q^\Pup)$ is an algebra isomorphism, where $\Oo_q^{U_q} \subset \Oo_q$ is the subspace of invariant elements under the coadjoint action ${\rm coad}^r$ from \eqref{defCoad}. 
\\(ii) The quantum trace elements $\mathrm{qTr}_\lambda $ ($\lambda \in P_+$) defined in \eqref{casimirElmtsL01} form a basis of $\Oo_q^{U_q}$.
\\(iii) The map $\mathcal{P}_\rho$ restricts to an algebra embedding $\mathcal{Z}(U_q^\Pup) \hookrightarrow U_q^\Pup(\mathfrak{h})$ and we have
\[ (\mathcal{P}_\rho \circ \Phi_{0,1})(\mathrm{qTr}_\lambda) = \sum_{\mu \in P(V_\lambda)} K_{2\mu}. \]
for all $\lambda \in P_+$, where the sum runs over the set of weights $P(V_\lambda)$ of $V_\lambda$ counted with multiplicity.  
\end{prop}

\begin{proof} For item (i) recall that $\mathcal{Z}(U_q^\Pup) = (U_q^\Pup)^{U_q} = ((U_q^\Pup)^{\rm lf})^{U_q}$, which is easy to see. By Th.\,\ref{teoqKillingnous}(iii), $\Phi_{0,1}$ is an equivariant isomorphism $\Oo_q \overset{\sim}{\to} (U_q^\Pup)^{\rm lf}$. It thus establishes a vector space isomorphism between $\Oo_q^{U_q}$ and $\mathcal{Z}(U_q^\Pup)$. This is moreover an algebra morphism by formula \eqref{productL01InTermsOfCoaction} and the fact that $\Phi_{0,1}$ is an algebra morphism $\mathcal{L}_{0,1} \to U_q$.
\\Item $(ii)$ is proved as \cite[Lem.\,3.119]{VY}, and item $(iii)$ is part of \cite[Th.\,3.122]{VY} where $\mathcal{P}_\rho$ is denoted by $\gamma^{-1} \circ \mathcal{P}$ (in \cite{VY} the last formula uses $K_{-2\mu}$ because of their different conventions on $\Phi_{0,1}$).
\end{proof}

For all $\lambda \in P_+$ let 
\[ \textstyle \mathrm{ch}^{(2)}_\lambda :=\sum_{\mu \in P(V_\lambda)} K_{2\mu} \in \mathbb{Z}[K_{2\varpi_1}^{\pm 1},\ldots, K_{2\varpi_m}^{\pm 1}] \subset U_q^{\scalebox{0.6}{2P}}(\mathfrak{h}). \]
Note that the ring $\mathbb{Z}[K_{2\varpi_1}^{\pm 1},\ldots, K_{2\varpi_m}^{\pm 1}]$ is isomorphic to $\mathbb{Z}[2P]$, and hence has a natural action of the Weyl group $W$. The elements $\mathrm{ch}^{(2)}_\lambda$ are clearly $W$-invariant. Therefore Proposition \ref{propPropietesPhi01App} is entirely summed-up by the following commutative diagram
\[ \xymatrix{
\mathbb{C}(q)[P_+] \ar[r]^-{\mathrm{ch}^{(2)}} \ar[d]_-{\mathrm{qTr}}^-{\cong\qquad\qquad\qquad\scalebox{1.5}{$\circlearrowleft$}} & \mathbb{C}(q)[K_{2\varpi_1}^{\pm 1},\ldots, K_{2\varpi_m}^{\pm 1}]^W = U_q^{\scalebox{0.6}{2P}}(\mathfrak{h})^W\\
\mathcal{O}_q^{U_q} = \mathcal{Z}(\mathcal{L}_{0,1}) \ar[r]_-{\Phi_{0,1}}^-{\cong} & \mathcal{Z}(U_q^\Pup) \ar[u]_-{\mathcal{P}_\rho}
}\]
where $\mathbb{C}(q)[P_+]$ is the algebra of the abelian monoid $P_+$ with coefficients in $\mathbb{C}(q)$. The vertical arrow on the right is called {\em quantum Harish-Chandra map}. By a classical argument \cite[\S 22.5]{Hum} we have:
\begin{equation}\label{AbaseHum}
\text{the family } \bigl\{ \mathrm{ch}^{(2)}_\lambda \,\big|\,\lambda\in P_+ \bigr\} \text{ is a } \mathbb{Z} \text{-basis of the ring } \mathbb{Z}[K_{2\varpi_1}^{\pm 1},\ldots, K_{2\varpi_m}^{\pm 1}]^W.
\end{equation}
It follows that the map $\mathrm{ch}^{(2)}$ in the above diagram is an algebra isomorphism. We conclude that
\[ \mathcal{P}_\rho : \mathcal{Z}(U_q^\Pup) \longrightarrow \mathbb{C}(q)[K_{2\varpi_1}^{\pm 1},\ldots, K_{2\varpi_m}^{\pm 1}]^W = U_q^{\scalebox{0.6}{2P}}(\mathfrak{h})^W \]
is an algebra isomorphism, which is the quantum version of Harish-Chandra's  theorem.\footnote{Note that this proof, originally due to Baumann \cite{baumann2}, is very different from the one given, e.g., in \cite[\S 18]{DCP}, which reconstructs an element $z\in \mathcal{Z}(U_q^\Pup)$ from $\mathcal{P}_\rho(z)$ by induction, using a complicated algorithm.}

To be on the safe side, we record that the commutation of the above diagram means that
\begin{equation}\label{Phicentresept25}
\textstyle \mathcal{P}_\rho^{-1 }\left( \sum_{\lambda \in P_+} c_\lambda \, \mathrm{ch}^{(2)}_\lambda \right) =  \Phi_{0,1}\left(\sum_{\lambda \in P_+}  c_\lambda \, \mathrm{qTr}_\lambda\right)
 \end{equation} 
where the coefficients $c_\lambda \in \mathbb{C}(q)$ are equal to $0$ except for a finite number of them.

\smallskip

\indent The properties of $\Phi_A^\pm$ discussed in \S\ref{integralO} readily imply that the restriction of $\Phi_{0,1}$ to $\mathcal{O}_A$ takes values in $U_A^\Pup$. Hence we have an integral version
\[ \Phi_{0,1}^A : \Oo_A \to (U_A^\Pup)^{\mathrm{lf}} := U_A^\Pup\cap (U_q^\Pup)^{\rm lf}, \quad \alpha \mapsto (\alpha \otimes \mathrm{id})(RR^{fl}). \]
As a restriction of $\Phi_{0,1}$, it is an embedding of $\Gamma_A^\Pup$-modules. This allows for an integral version of Proposition \ref{propPropietesPhi01App}(i):
\begin{cor}\label{Phi0Aiso} Denote by $\mathcal{Z}_1(\Oo_A)$ the $A$-subalgebra of $\Oo_A^{\Gamma_A}$ generated by the quantum trace elements $\mathrm{qTr}_\lambda$, $\lambda \in P_+$. The restriction of $\Phi_{0,1}^A$ to $\mathcal{Z}_1(\Oo_A)$ yields an algebra isomorphism $\mathcal{Z}_1(\Oo_A) \overset{\sim}{\to} \mathcal{Z}(U_A^\Pup)$, and therefore $\mathcal{Z}_1(\Oo_A)=\Oo_A^{\Gamma_A}$.
\end{cor}

\begin{proof} This restriction is injective because so is $\Phi_{0,1}^A$. For surjectivity, let $z\in \mathcal{Z}(U_A^\Pup)$. By extension of scalars from $\mathbb{Z}$ to $A$ in \eqref{AbaseHum}, the elements $\textstyle \mathrm{ch}^{(2)}_\lambda$ form an $A$-basis of the $A$-algebra $A[K_{2\varpi_1}^{\pm 1},\ldots, K_{2\varpi_m}^{\pm 1}]^W$, so we can write $\textstyle \mathcal{P}_\rho(z) =  \sum_\lambda  c_\lambda\, \mathrm{ch}^{(2)}_\lambda$ with a finite set of coefficients $c_\lambda\in A$ for $\lambda\in P_+$. From \eqref{Phicentresept25} it follows that $\textstyle \Phi_{0,1}^A\bigl(\sum_\lambda  c_\lambda\, \mathrm{qTr}_\lambda\bigr) = z$.

For the second claim, recall that $\Phi_{0,1}^A$ is an $U_A^\Pup$-linear injection $(\mathcal{O}_A, \mathrm{coad}^r) \to (U_A^\Pup, \mathrm{ad}^r)$. Since $\mathcal{Z}(U_A^\Pup) = (U_A^\Pup)^{U_A}$ (which is true for any Hopf algebra), we deduce from the first claim that $\mathcal{Z}_1(\mathcal{O}_A) = \mathcal{O}_A^{U_A} = \mathcal{O}_A^{\Gamma_A}$.
\end{proof}
\begin{remark}\label{rkZ1OAjanv26} {\rm The corollary implies that the quantum Harish-Chandra map $\mathcal{P}_\rho$ restricts to an isomorphism $\mathcal{Z}(U_A^\Pup) \to A[K_{2\varpi_1}^{\pm 1},\ldots, K_{2\varpi_m}^{\pm 1}]^W \subset U_A^\Pup(\mathfrak{h})$. This was proved somewhat implicitly in \cite[Th. 18.3 and \S 21]{DCP} , by showing that their algorithm for the computation of $\mathcal{P}_\rho^{-1}$ does not introduce denominators.}
\end{remark}

To conclude we discuss an integral form $(\:\,|\,\:)_{\AD}$ of the pairing \eqref{qKillingnous} and its specialization at $\eD$. First, at the level of Cartan parts, let us define $(\:\,|\,\:)_{\AD} : U_A^\Pup(\mathfrak{h}) \times \Gamma_A^\Pup(\mathfrak{h}) \to \mathbb{C}\bigl[ \qD^{\pm 1/2} \bigr]$ by
\begin{equation}\label{IntKillingCartan}
\biggl( K_\mu \,\bigg| \, K_s \:{\textstyle \prod}_{i=1}^m \, K_i^{-\sigma(t_i)}(K_i; t_i)_{q_i} \biggr)_{\!\!\!\AD} = \qD^{D(\mu,s)/2} \prod_{i=1}^m q^{-(\mu,\alpha_i)\sigma(t_i)/2}\begin{pmatrix} (\mu,\alpha_i)/2 \\ t_i \end{pmatrix}_{\!\!q_i}
\end{equation}
where $s$ belongs to a chosen set $S \subset P$ of representatives for the cosets in $P/Q$ (we use the basis of $\Gamma_A^\Pup(\mathfrak{h})$ explained in \S\ref{integralFormsUq}). Now, using the notations introduced in \eqref{defRootVectorsUnrestricted} and \eqref{defDivPowerBeta} for the generators of $U_A^\Pup$ and $\Gamma_A^\Pup$, observe from the expression of the pairing $\rho$ in \eqref{expressionRho} that
\begin{align}
\rho\!\left(\underline{F}_{\beta_N}^{s_N}\ldots \underline{F}_{\beta_1}^{s_1}, E_{\beta_N}^{(r_N)}\ldots E_{\beta_1}^{(r_1)}\right) &= \textstyle \prod_{i=1}^N \delta_{s_i,t_i} q_{\beta_i}^{-s_i(s_i-1)/2},\label{intDrinfeldOnNilpotent1}\\
\rho\!\left(F_{\beta_N}^{(s_N)}\ldots F_{\beta_1}^{(s_1)}, \underline{E}_{\beta_N}^{t_N}\ldots \underline{E}_{\beta_1}^{t_1}\right) & = \textstyle \prod_{i=1}^N \delta_{s_i,t_i} q_{\beta_i}^{-s_i(s_i-1)/2}.\label{intDrinfeldOnNilpotent2}
\end{align}
By definition of the integral forms $U_A^\Pup$ and $\Gamma_A^\Pup$ and their basis in \S\ref{integralFormsUq}, we get two pairings
\begin{align*}
&\rho'_A : U_A(\mathfrak{n}_-) \times \Gamma_A(\mathfrak{n}_+) \to A \quad \text{given by \eqref{intDrinfeldOnNilpotent1}}\\
&\rho''_A : \Gamma_A(\mathfrak{n}_-) \times U_A(\mathfrak{n}_+) \to A \quad \text{given by \eqref{intDrinfeldOnNilpotent2}}
\end{align*}
We can thus define $(\:\,|\,\:)_{\AD} : U_A^\Pup \times \Gamma_A^\Pup \to \mathbb{C}\bigl[ \qD^{\pm 1/2} \bigr]$ by
\begin{equation}\label{integralKillingForm}
\bigl( S(X_-) K_\nu X_+ \,|\, S(Y_+) M Y_- \bigr)_{\!\AD} = (K_\nu \,|\, M)_{\AD}\, \rho'_A(X_-, Y_+) \, \rho''_A(Y_-,X_+)
\end{equation}
for all $X_\pm \in U_A(\mathfrak{n}_\pm)$, $Y_\pm \in \Gamma_A(\mathfrak{n}_\pm)$, and $M\in \Gamma_A^\Pup(\mathfrak{h})$. Note that $(\:\,|\,\:)_{\AD}$ is non-degenerate as it recovers $(\:\,|\, \:)$ from \eqref{qKillingnous} after extending scalars to $\mathbb{C}(\qD^{1/2})$. It is moreover ad-invariant in the sense of Theorem \ref{teoqKillingnous}, as a restriction of $(\:\,|\,\:)$.

\begin{cor}\label{cor:dec25PhiA} $\Phi_{0,1}^A :  \Oo_A\to (U_A^\Pup)^{\rm lf}$ is an isomorphism of $\Gamma_A^\Pup$-modules, whose inverse is the restriction $J_A$ of $J_q$ to $(U_A^\Pup)^{\rm lf}$.
\end{cor}
\begin{proof} 
By restriction to integral forms in Theorem \ref{teoqKillingnous}, we have 
\begin{equation}\label{KillingPhi01Int}
\forall \, \alpha\in \Oo_A, \:\: \forall \, y\in \Gamma_A^{\Pup}, \quad \langle \alpha, y\rangle_{\AD}^\Pup = \bigl(\Phi_{0,1}^{A}(\alpha)) \,\big|\, y\bigr)_{\AD} \in \AD
\end{equation}
and $J_q$ restricts to $J_A : (U_A^\Pup)^{\mathrm{lf}} \to \Oo_A$ characterized by $\bigl\langle J_A(x),y \bigr\rangle_{\AD}^\Pup= (x \,|\, y )_{\AD} $ for all $x \in (U_A^{\Pup})^{\mathrm{lf}}$. Since $(\:\,|\,\:)_{\AD}$ is non degenerate, we have the result.
\end{proof}
\begin{remark}{\rm This corollary implies in particular that $\Phi_{0,1}^A(\Oo_A) = (U_A^\Pup)^{\rm lf}$, which answers positively a question asked before Prop.\,2.24 in \cite{BR2}.}
\end{remark}
Choose an arbitrary square root $\eD^{1/2}$ of the $l$-th root of unity $\eD$ defined in \eqref{choixED} and denote by $(\:\,|\,\:)_{\eD} : U_\e^\Pup \times \Gamma_\e^\Pup \to \mathbb{C}$ the specialization of the pairing $(\:\,|\,\:)_{\AD}$ at $\qD^{1/2} \mapsto \eD^{1/2}$. We record the following property, which is a key-point in the proof of Lemma \ref{Z0Z1Phimai25}:

\begin{lem}\label{lemmaKerFrZ0}
The subspaces $\mathcal{Z}_0(U^\Pup_\e) \subset U^\Pup_\e$ and $\mathrm{ker}(\mathbb{F}\mathrm{r}_\e) \subset \Gamma_\e^\Q$ are orthogonal with respect to the pairing $(\:\,|\,\:)_{\eD}$, {\it i.e.} 
\[ \forall \, z_0 \in \mathcal{Z}_0(U^\Pup_\e),\:\:\forall \, b \in \mathrm{ker}(\mathbb{F}\mathrm{r}_\e), \quad (z_0\,|\,b)_{\eD} = 0 \]
\end{lem}
\begin{proof} Using the formulas in \eqref{intDrinfeldOnNilpotent1} and \eqref{intDrinfeldOnNilpotent2}, let us compute the pairing $(\:\,|\,\:)_{\eD}$ when the first argument belongs to $\mathcal{Z}_0(U^\Pup_\e)$, and the second argument is one of the basis elements of $\mathbb{F}\mathrm{r}_\e$ in Remark \ref{rkbaseuedec25}:
\begin{align*}
\textstyle \biggl( S\!\left(\prod_{k=N}^1 F_{\beta_k}^{lp_k}\right) K_\mu^l \left(\prod_{t=N}^1 E_{\beta_t}^{ls_t}\right) & \,\bigg|\, S\!\textstyle \left(\prod_{i=N}^1 E_{\beta_i}^{(n_i)} \right) M \left(\prod_{j=N}^1 F_{\beta_j}^{(r_j)}\right) \biggr)_{\!\!\eD} \\
=\:\,&\textstyle (K_{l\mu}\,|\,M)_{\eD} \, \prod_{i=1}^N \delta_{n_i,lp_i} \prod_{j=1}^N \delta_{r_j,l s_j}.
\end{align*}
This is $0$ when some $n_i$ or $r_j$ are not divisible by $l$, or $M$ is in the ideal of $\Gamma_\e^\Pup(\mathfrak{h})$ generated by the elements $K_i - 1 = K_i - K_0$ (since $(K_{l\mu}\,|\,K_i - 1)_{\eD}=\eD^{D(\alpha_i,l\mu)} - 1 = \e^{l(\alpha_i,\mu)} - 1 = 0$).
\end{proof}

\section{Lusztig's modified quantum group and canonical basis}\label{appLusztigModif}

The aim of this appendix is to provide the details of the arguments in Lemma \ref{nondegpairepsD}, using Lusztig's modified quantum group; the key-point is Lemma \ref{embedGammaeps}.

\subsection{Definitions and relationships with \texorpdfstring{$\Oo_A$ and $\Gamma_\e^\Lam$}{integral quantum groups}} The Lusztig {\it modified quantum group} $\stackrel{.}{\mathbf{U}}$ associated to the simply-connected root datum \cite[Chaps.\,2, 3 and 23]{Lusztig} can be defined as follows (actually $\stackrel{.}{\mathbf{U}}$ here is obtained by replacing $Y= \check{Q}$ with $Q$ in Lusztig's $\stackrel{.}{\mathbf{U}}$). The elements of $\stackrel{.}{\mathbf{U}}$ are finite $\mathbb{C}(q)$-linear combinations of elements $x1_\lambda$ with $x \in U_q^\Q$ and $\lambda \in P$; of course we have $(x_1 + cx_2)1_\lambda = (x_11_\lambda) + c(x_2 1_\lambda)$ for all $c \in \mathbb{C}(q)$. The product in $\stackrel{.}{\mathbf{U}}$ is entirely determined by the product in $U_q^\Q$ together with the following relations
\[ 1_{\lambda}1_{\mu} = \delta_{\lambda,\mu} 1_{\lambda}, \qquad K_{\alpha}1_{\lambda} = 1_{\lambda}K_{\alpha} = q^{(\alpha,\lambda)} 1_{\lambda}, \qquad w 1_{\lambda} = 1_{\lambda + \nu}w \]
for all $\alpha \in Q$, $x,y \in U_q^\Q$ and $w \in U_q^\Q$ which has weight $\nu$ under the adjoint action (\textit{i.e.} $K_{\alpha}wK_{\alpha}^{-1} = q^{(\alpha,\nu)}w$).

One can regard $1_\lambda$ as a substitute for the projector onto the weight subspaces of weight $\lambda$ of the $U_q^\Q$-modules of type $1$, and $\stackrel{.}{\mathbf{U}}$ as obtained from $U_q^\Q$ by replacing $U_q^\Q(\mathfrak{h})$ with $\textstyle \bigoplus_{\lambda \in P} \mathbb{C}(q)1_{\lambda}$.

Note the decompositions $\textstyle \stackrel{.}{\mathbf{U}}\, = \bigoplus_{\lambda \in P} U_q^\Q 1_{\lambda}$ and $\textstyle \stackrel{.}{\mathbf{U}}\, = \bigoplus_{\lambda',\lambda'' \in P} 1_{\lambda'} U_q^\Q 1_{\lambda''}$ as vector spaces, corresponding to weight spaces for the right and left actions of $U_q^\Q$ on $\stackrel{.}{\mathbf{U}}$. There is also a substitute of the PBW basis: given bases $\{b^\pm\}$ of $U_q(\mathfrak{n}_\pm)$, the set of elements $b^+1_\lambda b^-$ (or $b^-1_\lambda b^+$, or $b^+b^-1_\lambda$), where $\lambda\in P$, is a basis of $\stackrel{.}{\mathbf{U}}$. 

The integral form $\stackrel{.}{\mathbf{U}}_A$ is the $A$-algebra generated by the elements $E_i^{(p)}1_\lambda$  and $F_i^{(p)}1_\lambda$, for all $i\in \{1,\ldots,m\}$, $p\in \mathbb{N}$, $\lambda \in P$. Given bases $\{b_A^\pm\}$ of $\Gamma_A(\mathfrak{n}_\pm)$, the set of elements $b_A^+1_\lambda b_A^-$ (or $b_A^-1_\lambda b_A^+$, or $b_A^+b_A^-1_\lambda$), where $\lambda\in P$, is a basis of $\stackrel{.}{\mathbf{U}}_A$. It inherits from $\stackrel{.}{\mathbf{U}}$ a structure of $\Gamma_A^\Q$-bimodule; by extending the ground ring to $A_\D$ and using \eqref{actionLi}, it becomes a $\Gamma_A^\Pup$-bimodule. By construction $\stackrel{.}{\mathbf{U}}_A\!\!\!{} \textstyle \,=\bigoplus_{\lambda',\lambda''\in P} 1_{\lambda'}\Gamma_A^\Q 1_{\lambda''} = \bigoplus_{\lambda\in P}\Gamma_A^\Q 1_{\lambda}$.

The algebra $\stackrel{.}{\mathbf{U}}_A$ has a canonical $A$-basis $\stackrel{.}{\mathbf{B}}$ having many remarkable properties \cite[Chap.\,25 and 29]{Lusztig}. In particular, its elements are weight vectors under the left and right actions of $\Gamma_A^\Q(\mathfrak{h})$, and there is a canonical partition into finite sets $$\textstyle \stackrel{.}{\mathbf{B}} = \coprod_{\lambda\in P_+} \stackrel{.}{\mathbf{B}}\!\![\lambda].$$

The algebra $\stackrel{.}{\mathbf{U}}_A$ has no unit, and it is not a Hopf algebra in a strict sense but is an instance of multiplier Hopf algebra in the sense of \cite{VD1}. The canonical basis provides a very natural completion $\hat{{\mathbf{U}}}_A$, introduced by Lusztig: it is the set of (possibly infinite) sums of the form $\textstyle \sum_{b\in\stackrel{.}{\mathbf{B}} }n_b b$, with $n_b\in A$. It is an algebra \cite{Lusztig3}, with unit $\textstyle \sum_{\lambda\in P} 1_\lambda$. If we moreover denote by $\stackrel{.}{\mathbf{U}}_A\!\!{}^{\hat{\otimes} 2}$ the set of (possibly infinite) sums of the form $\textstyle \sum_{b,b'\in\stackrel{.}{\mathbf{B}} }n_{b,b'} b\otimes b'$ with $n_{b,b'}\in A$, then we can define $\Delta :\: \stackrel{.}{\mathbf{U}}_A \,\to\, \stackrel{.}{\mathbf{U}}_A\!\!{}^{\hat{\otimes} 2}$, $S : \:\stackrel{.}{\mathbf{U}}_A \,\to\, \stackrel{.}{\mathbf{U}}_A$ and $\varepsilon :\: \stackrel{.}{\mathbf{U}}_A \,\to A$ by
\[ \Delta(x1_{\lambda}) = \sum_{(x),\:\mu \in P} x_{(1)}1_{\mu} \otimes x_{(2)}1_{\lambda - \mu}, \quad S(x1_{\lambda}) = 1_{-\lambda}S(x), \quad \varepsilon(x1_{\lambda}) = \varepsilon(x)\delta_{\lambda,0} \]
for all $x \in \Gamma_A^\Q$ and $\lambda \in P$, using the Hopf structure of $\Gamma_A^\Q$ on $x \in \Gamma_A^\Q$.
\smallskip

For every lattice $Q\subset \Lambda\subset P$ there is a natural algebra morphism $\Gamma_A^\Lam \to \hat{{\mathbf{U}}}_A \otimes_A \AD$ given by
\begin{equation}\label{embedGamma}
x^+ y x^{-} \mapsto \sum_{\mu\in P}x^+ y _\mu x^{-}, \qquad \text{where } x^{\pm}\in \Gamma_A(\mathfrak{n}^{\pm}), \:\: y\in \Gamma_A^\Lam(\mathfrak{h}), \:\: y_\mu= y 1_{\mu}\in \AD 1_\mu.
\end{equation}
Indeed, by a reasoning on weights it can  be shown that the sum is indeed of the form $\textstyle \sum_{b\in\stackrel{.}{\mathbf{B}} }n_b b$, with $n_b\in \AD$.

The morphism \eqref{embedGamma} is injective, and the proof of this fact is completely similar to the one of Lemma \ref{embedGammaeps} below, which treats the specialized version. 

The category $\mathcal{C}_A$ (\S\ref{integralO}) is ribbon equivalent to the category of finite rank {\it unital} $\stackrel{.}{\mathbf{U}}_A$-modules $M$ \cite[\S 23.1.4, \S 31.1.5-9, Th.\,32.1.5]{Lusztig}. Here, unital means that each vector of $M$ is annihilated by all but a finite number of the idempotents $1_\lambda\in \stackrel{.}{\mathbf{U}}_A$, and $\textstyle \sum_{\lambda\in P} 1_\lambda$ acts as the identity. The weight space decomposition $M = \textstyle \oplus_{\lambda \in P} M^\lambda$ is given by $M^\lambda = 1_\lambda M$; the action of $x\in \Gamma_A\Q$ is given by $xm = (x1_\lambda)m$ for any $\lambda\in P$, $m\in M^\lambda$.

\indent This permits us to consider $\Oo_A$ as the dual space of $\stackrel{.}{\mathbf{U}}_A$. More precisely, in \cite[\S 29.5]{Lusztig} Lusztig showed that $\Oo_A$ is a free $A$-module, with a basis $\stackrel{.}{\mathbf{B}}{\!\!^*}:=\{\varphi_b\}_{b\in \stackrel{.}{\mathbf{B}}}$ in duality with $\stackrel{.}{\mathbf{B}}$. Therefore the bilinear form defined by 
\begin{equation}\label{pairOAnov24}
\langle \text{-},\text{-} \rangle_A^{\smallbullet}:{\mathcal O}_A\times \stackrel{.}{\mathbf{U}}_A\to A\ ,\ \langle \varphi_b,b'\rangle_A^{\smallbullet}:=\delta_{b,b'}
\end{equation}
is a perfect pairing, that is: the maps $\Oo_A\to {\rm Hom}_A(\stackrel{.}{\mathbf{U}}_A,A)$, $\varphi\mapsto \langle \varphi,\text{-} \rangle_A^{\smallbullet}$, and $\stackrel{.}{\mathbf{U}}_A \to {\rm Hom}_A(\Oo_A,A)$, $x\mapsto \langle \text{-},x \rangle_A^{\smallbullet}$, are bijective. Moreover $\langle \text{-},\text{-} \rangle_A^{\smallbullet}$ is a (multiplier) Hopf pairing. Note that, by very definition, the pairing $\langle \text{-},\text{-} \rangle_A^{\smallbullet}$ can be extended to $\mathcal{O}_A \times \hat{\mathbf{U}}_A \to A$. Using the embedding \eqref{embedGamma} for $\Lambda = Q$, the bilinear forms $\langle \text{-},\text{-} \rangle_A^\Q$ and $\langle \text{-},\text{-} \rangle_A^{\smallbullet}$ are related as follows: given $x^{\pm}\in{\Gamma_A}(\mathfrak{n}^{\pm})$, $y\in \Gamma_A^\Q(\mathfrak{h})$ we have (the sum has a finite number of non zero terms):
\begin{equation}\label{relpairingdotA}
\langle \varphi , x^+ y x^-\rangle_{A}^\Q=\sum_{\mu\in P}\langle \varphi , x^+ y_\mu x^- \rangle_A^{\smallbullet}, \qquad \text{with } y_{\mu} = y1_{\mu} \in A 1_{\mu}.
\end{equation}
By duality the elements of $\stackrel{.}{\mathbf{B}}{\!\!^*}$ are weight vectors for the left and right coregular actions of $\Gamma_A(\mathfrak{h})$. There is a decomposition into weight subspaces 
\begin{equation}\label{OAweightdecomp} \Oo_A = \bigoplus_{\mu,\nu\in P} {}_\mu\!\left(\Oo_A\right)_{\nu},
\end{equation}
where 
\begin{multline}\label{weightspaceOA}
{}_\mu\!\left(\Oo_A\right)_{\nu} :=\left\lbrace f\in \Oo_\e\mid  K_i\rhd f =q_i^{(\check{\alpha}_i,\mu)}f,  ( K_i; l)_{q_i}\rhd f = ( q_i^{(\check{\alpha}_i,\mu)}; l)_{q_i}f,\right. \\ \left. f \lhd K_i =q_i^{(\check{\alpha}_i,\nu)}f,  f \lhd ( K_i; l)_{q_i} = ( q_i^{(\check{\alpha}_i,\nu)}; l)_{q_i}f, i=1,\ldots,m\right\rbrace.
\end{multline}

Using $\stackrel{.}{\mathbf{B}}{\!\!^*}$ Lusztig proved in \cite[Prop.\,3.3]{Lusztig3} that $\Oo_A$ is generated as an algebra by certain matrix coefficients of the modules ${}_AV_{\varpi_i}$, with highest weights the fundamental weights $\varpi_i$, $i=1,\ldots,m$.

\indent Clearly the existence of the dual basis $\stackrel{.}{\mathbf{B}}{\!\!^*}$ implies the identification $\hat{{\mathbf{U}}}_A \cong {\rm Hom}_A(\Oo_A,A)$. Also, $\Oo_A$ inherits from $\Oo_q$ a bimodule structure over $\Gamma_A^\Q$ (and $\hat{{\mathbf{U}}}_A$), with actions $\rhd$, $\lhd$.
\smallskip

The next lemma uses the decomposition $\Gamma^{\Lam}_{\epsilon}(\mathfrak{g}) = \Gamma_{\epsilon}(\mathfrak{n}^+)\Gamma^\Lam_{\epsilon}(\mathfrak{h})\Gamma_{\epsilon}(\mathfrak{n}^-)$, obtained by specializing to $q = \epsilon$ the decomposition of $\Gamma_A^\Lam$ explained in \S\ref{integralFormsUq}.
\begin{lem}\label{embedGammaeps}
For every lattice $Q\subset \Lambda \subset P$, the morphism \eqref{embedGamma} specializes to $\qD = \eD$ into a monomorphism
\begin{equation*}
\iota^\Lam_\e\colon\Gamma_\e^\Lam\hookrightarrow \hat{{\mathbf{U}}}_\e, \quad x^+ y x^{-} \mapsto \sum_{\mu\in P}x^+ y _\mu x^{-}\in \hat{{\mathbf{U}}}_\e
\end{equation*}
where $x^{\pm}\in \Gamma_\e(\mathfrak{n}^{\pm})$, $y\in \Gamma_\e^\Lam(\mathfrak{h})$ and $y_\mu:= y 1_{\mu}\in \mathbb{C}1_\mu \subset \:\hat{\mathbf{U}}_\e$.
\end{lem}
\begin{proof}
Let $\{ b^{\pm}_i\}_{i \in I}$ be the canonical basis of $\Gamma_A(\mathfrak{n}_{\pm})$. Then the elements $b^+_i k_s b^-_j$ form a basis of $\Gamma_A^\Lam$, where $\{k_s\}_s$ is a basis of $\Gamma_A^\Lam(\mathfrak{h})$. There is a trigonal transformation with coefficients in $A$ from the collection of elements $b^+_i 1_\mu b^-_j \in \:\stackrel{.}{\mathbf{U}}_A$ to the canonical basis $\stackrel{.}{\mathbf{B}}$ of $\stackrel{.}{\mathbf{U}}_A$ \cite[\S 25.2]{Lusztig}. It follows that any element in $\hat{\mathbf{U}}_A$ can be written uniquely as an infinite sum $\textstyle \sum_{i,j\in I, \: \mu \in P} c_{i,j,\mu} b^+_i 1_\mu b^-_j$ with $c_{i,j,\mu} \in A$, and the same is true in the specialization to $\epsilon$. With these notations, the value of the morphism $\Gamma_\e^\Lam \to \hat{{\mathbf{U}}}_\e$ is
$\textstyle \sum_{i,j, s} x_{i,j,s} b^+_i k_s b^-_j \mapsto \sum_{i,j, s} \sum_{\lambda\in P} x_{i,j,s} b^+_i (k_s1_\lambda) b^-_j$, where the sum on the left is finite. If we assume this image is $0$, the uniqueness implies that $\textstyle \sum_s x_{i,j,s} k_s1_{\lambda} = 0$ for all $i,j,\lambda$. We are thus left to show that the restriction $\textstyle \Gamma_{\epsilon}^\Lam(\mathfrak{h}) \to \prod_{\lambda \in P} \mathbb{C}1_{\lambda}$ is injective. We first do the case $\Lambda = Q$, which relies on an argument completely similar to \cite[Th.\,3.1]{DC-L}. Note that on the basis elements $Y(\mathbf{0},\mathbf{t},\mathbf{0})_\e$ of $\Gamma_{\epsilon}(\mathfrak{h}) = \Gamma^\Q_{\epsilon}(\mathfrak{h})$ from \eqref{basegamma}, we have
\begin{equation}\label{mapYnov25}
\iota_\e^\Lam\left(Y(\mathbf{0},\mathbf{t},\mathbf{0})_{\epsilon}\right) = \sum_{\lambda \in P} \left( \prod_{i=1}^m \epsilon^{-\sigma(t_i)(\check{\alpha}_i,\lambda)}_i\left(q_i^{(\check{\alpha}_i,\lambda)}; t_i\right)_{\epsilon_i} \right) 1_\lambda \qquad \forall \, \mathbf{t} = (t_1,\ldots,t_m) \in \mathbb{N}^m
\end{equation}
where the $\epsilon_i$ subscript denotes specialization. Assume that an element $\textstyle Y = \sum_{\mathbf{t} \in \mathbf{T}} c_{\mathbf{t}} Y(\mathbf{0},\mathbf{t},\mathbf{0})_{\epsilon}$ $\neq 0$ is sent to $0$. Here $\varnothing \neq \mathbf{T} \subset \mathbb{N}^m$ is such that $c_{\mathbf{t}} \neq 0$ for all $\mathbf{t} \in \mathbf{T}$. Let $\mathbf{t}' = (t'_1, \ldots, t'_m) \in \mathbf{T}$ be the minimal element (here minimality refers to the lexicographic order on $\mathbb{N}^m$) and consider $\textstyle \lambda = \sum_{i=1}^m t'_i\varpi_i \in P$. For all $\mathbf{t} \in \mathbf{T}$, there at least one $1 \leq i \leq m$ such that $t'_i < t_i$. Hence $\textstyle \prod_{i=1}^m\bigl(q_i^{(\check{\alpha}_i,\lambda)}; t_i\bigr)_{q_i} = \prod_{i=1}^m\left(q_i^{t'_i}; t_i\right)_{q_i} = 0$, and the same is of course true after specialization. We conclude that $\textstyle \iota_\e^\Lam(Y)1_\lambda = c_{\mathbf{t}'}\prod_{i=1}^m \epsilon^{-\sigma(t'_i)(\check{\alpha}_i,\lambda)}_i\left(q_i^{t'_i}; t'_i\right)_{\epsilon_i}1_\lambda =  c_{\mathbf{t}'}\prod_{i=1}^m \epsilon^{-\sigma(t'_i)t'_i}_i1_\lambda $ which equals $0$ by assumption on $Y$. Hence $c_{\mathbf{t}'} = 0$, a contradiction. Finally we prove the claim for $\Lambda = P$, since any $\Gamma^\Lam_{\epsilon}(\mathfrak{h})$ is a subalgebra of $\Gamma^\Pup_{\epsilon}(\mathfrak{h})$. Take a system of representatives $\{ \mu_0, \ldots, \mu_s \}$ of $P/Q$ where $\mu_0 = 0$; then the elements $Y(\mathbf{0},\mathbf{t},\mathbf{0})K_{\mu_i}$ form an $A$-basis of $\Gamma^\Pup_{\epsilon}(\mathfrak{h})$. If $\textstyle Y = \sum_{i,\mathbf{t}} c_{i,\mathbf{t}} Y(\mathbf{0},\mathbf{t},\mathbf{0})K_{\mu_i}$ is sent to $0$ in $\hat{\mathbf{U}}_\e$, the result for $Q$ just proven yields $\textstyle \sum_{i=0}^s c_{i,\mathbf{t}}\eD^{D(\mu_i,\lambda)} = 0$ for all $\mathbf{t}$ and $\lambda \in P$. By direct inspection of root systems \cite[Chap.\,6]{Bourbaki}, we see that the elements $\mu_1,\ldots,\mu_s$ can be choosen among the fundamental weights: $\mu_i = \varpi_{j_i}$ for some $j_i$. Fix an arbitrary $\mathbf{t}$; for $\lambda = 0$ we get $\textstyle \sum_{i=0}^s c_{i,\mathbf{t}} = 0$ while for $\lambda = \alpha_{j_k}$ we get $\textstyle \e^{d_{j_k}}c_{k,\mathbf{t}} + \sum_{i\neq k}c_{i,\mathbf{t}} = 0$. The unique solution to this system of $s+1$ linear equations with $s+1$ variables is $c_{i,\mathbf{t}} = 0$ for all $i$.
\end{proof}

\subsection{Proof of Lemma \ref{nondegpairepsD}.}\label{sec:pflemnondegpair} By specialization of the perfect pairing \eqref{pairOAnov24}, the specialized basis elements $\varphi_b^{\epsilon} := \varphi_b \otimes_A 1_\e$ of ${\mathcal O}_\epsilon$ and $b_{\epsilon}:=  b\otimes_A 1_\e$ of $\stackrel{.}{\mathbf{U}}_\epsilon$, where $b\in \stackrel{.}{\mathbf{B}}$,  are in duality. Therefore there is a perfect (multiplier) Hopf pairing $\langle \text{-},\text{-} \rangle_\epsilon^{\smallbullet}:{\mathcal O}_\epsilon\,\times \stackrel{.}{\mathbf{U}}_\epsilon \: \to {\mathbb C}$ which can be extended to ${\mathcal O}_\epsilon\times \hat{\mathbf{U}}_\e$ in the obvious way. Using the monomorphism $\Gamma_\e^\Pup\hookrightarrow \hat{{\mathbf{U}}}_\e$ from Lemma \ref{embedGammaeps}, it is related to the pairing $\langle.,.\rangle_\e^\Pup : \mathcal{O}_\e \times \Gamma_\e^\Pup \to \mathbb{C}$ as follows:
\begin{equation}\label{relpairnov25}\langle \varphi , x^+ y x^-\rangle_\e^\Pup =\sum_{\mu\in P}\langle \varphi , x^+ y_\mu x^- \rangle_\e^{\smallbullet}, \quad \text{with } y_\mu = y1_\mu \in \mathbb{C}1_\mu.
\end{equation}
Then, given $X\in \Gamma_\epsilon^\Q$ in the right kernel of $\langle \text{-},\text{-} \rangle_\epsilon^\Pup$, embed it as an element $\textstyle \hat{X}=\sum_{b\in \stackrel{.}{\mathbf{B}}} X_b b_{\epsilon} \in \hat{{\mathbf{U}}}_\e$. Note that $\langle\varphi_{b}^{\epsilon}, \hat{X}\rangle^{\smallbullet}_\e = X_b$. Hence each coefficient $X_b$ is zero, so is $\hat{X}$, and therefore $X=0$.

\section{Injectivity of \texorpdfstring{$\Phi_{1,0}^{\epsilon_{\scalebox{0.5}{D}}}$}{the morphism to the Heisenberg double}}\label{AppPhiinj}

This appendix proves the last argument of Proposition \ref{AlekseevInjRootOf1}:

\begin{prop}\label{Phigneinjproof} The morphism $\Phi_{1,0}^{\e} : \mathcal{L}_{1,0}^{\e} \to \mathcal{H}_{\e}$ is an embedding. 
\end{prop}

\begin{proof}
We proceed similarly as for \cite[Th.\,3.11]{BFR}, which treats the case of generic $q$. The only novel point is to use the specialization of \eqref{OAweightdecomp} and\eqref{weightspaceOA}, to get weight spaces ${}_\mu\!\left(\Oo_\e\right)_{\nu}$ and a decomposition $\textstyle \Oo_\e = \bigoplus_{\mu,\nu\in P} {}_\mu\!\left(\Oo_\e\right)_{\nu}$. It implies decompositions
\[ \mathcal{L}_{1,0}^{\e} = \bigoplus_{\lambda, \sigma \in P} {_{\lambda}(\mathcal{O}_{\e})_{\sigma}} \otimes  \mathcal{L}_{0,1}^{\e}\ , \ \mathcal{H}_{\e} = \bigoplus_{\lambda,\sigma \in P} {_{\lambda}(\mathcal{O}_{\e})_{\sigma}} \otimes U_\e^\Pup. \]
Let us spell out the main steps to conclude. In the notations of \cite[Th.\,3.13]{BFR}, the map $\Phi_{1,0}^{\e} \colon \mathcal{L}_{1,0}^{\e} \rightarrow \mathcal{H}_{\e}$, which is obtained by restriction of $\Phi_{0,1}$ from $\mc(\qD)$ to $A_\D$ and then specialization to $q =\e$ with $\qD = \e^{\overline{D}}$ as in \eqref{choixED}, satisfies
\[ \Phi_{1,0}^{\e} (\beta  \otimes 1) = \e^{(\lambda,\sigma)} \beta \otimes  K_{\lambda + \sigma} + \bigoplus_{(\lambda',\sigma') < (\lambda,\sigma)} {_{\lambda'}(\mathcal{O}_\e)_{\sigma'}} \otimes U_\e^\Pup,\quad \forall\beta\in {_{\lambda}(\mathcal{O}_{\e})_{\sigma}},\]
where $\leq$ is the partial order on $P^2$ defined by  $(\lambda',\sigma') \leq (\lambda,\sigma)$ if and only if $\lambda - \lambda' \in Q_+$ and $\sigma - \sigma' \in Q_+$. For any non-zero $x \in \mathcal{L}_{1,0}^{\e}$ one can find $(\lambda,\sigma)$ maximal for the order $\leq$, and a finite set $I_{\lambda,\sigma}$, such that 
\[ \Phi_{1,0}^{\e}(x) \in \e^{(\lambda,\sigma)} \sum_{i \in I_{\lambda,\sigma}} \beta_{\lambda, \sigma,i} \otimes  K_{\lambda + \sigma}\Phi_{0,1}^{\e}(\alpha_{\lambda,\sigma,i}) + \bigoplus_{(\lambda',\sigma') \neq (\lambda,\sigma)} {_{\lambda'}(\mathcal{O}_\e)_{\sigma'}} \otimes U_\e^\Pup, \]
where $\bigl(\beta_{\lambda,\sigma,i}\bigr)_{i \in I_{\lambda, \sigma}}$ is a family of linearly independent elements and $\alpha_{\lambda,\sigma,i} \neq 0$ for at least one $i\in I_{\lambda,\sigma}$. Since $\Phi_{0,1}^{\e}$ is injective \cite[Cor.\,2.25]{BR2} and $K_{\lambda+\sigma}$ is invertible, we conclude that $\Phi_{1,0}^{\e}(x) \neq 0$. This proves injectivity of $\Phi_{1,0}^{\e}$.
\end{proof}

\section{Proof that \texorpdfstring{$\Ll_{g,n}^{\e}$}{specialized graph algebra} is a domain}\label{Lgndomain}

This appendix proves Proposition \ref{propLgndomain}. Since the morphism $\widehat{\Phi}_{g,n}^{\e} :  \mathcal{L}_{g,n}^{\e} \to \widehat{\mathcal{H}}_{\e}^{\otimes g} \otimes (U_\e^\Pup)^{\otimes n}$ is injective (Prop.\,\ref{AlekseevInjRootOf1}), it is enough to show:
\begin{cor}\label{HUemaxorder} The algebra $\widehat{\mathcal{H}}_{\e}^{\otimes g} \otimes (U_\e^\Pup)^{\otimes n}$ is a domain.
\end{cor}
\begin{proof} It is an adaptation of the proof of \cite[Prop.\,5.7]{BFR}, which treats the case of generic $q$. We construct an algebra filtration $\Tt$ on $\widehat{\mathcal{H}}_{\e}^{\otimes g} \otimes (U_\e^\Pup)^{\otimes n}$ such that the associated graded algebra $\mathrm{gr}_{\Tt}(\widehat{\mathcal{H}}_{\e}^{\otimes g} \otimes (U_\e^\Pup)^{\otimes n})$ is quasi-polynomial over $\Oo_\e^{\otimes g}\cong \Oo_\e(G^{\times g})$. The latter algebra is a domain, by the same arguments as for Theorem \ref{DCLteo1} (2). Then  the claim follows from general results (see e.g. \cite[\S 1.2.9]{MC-R}). 

The construction of $\Tt$ is the same as in \cite[Prop.\,5.7]{BFR}, replacing the two-sided Heisenberg doubles $\mathcal{HH}_q$ with the modified ones $\widehat{\mathcal{H}}_{q}$ from Def.\,\ref{defiHhatq}, and taking specialization at $\qD=\eD$.

Let us spell out the changes. In \cite[Prop.\,5.7]{BFR} we produced an algebra filtration $\Gg$ of $\mathcal{H}_q$, such that the associated graded algebra $\mathrm{gr}_{\Gg}(\mathcal{H}_q)$ is the quasi-polynomial ring over $\mathcal{O}_q$ generated over $\mathcal{O}_q$ by $\overline{E_{\beta_i}}$, $\overline{F_{\beta_i}}$, $\overline{K_{\nu}}$ (with $1 \leq i \leq N$, $\nu \in P$) satisfying the relations 
\begin{equation}\label{commutationGrDCK}
\begin{array}{lll}
\overline{E_{\beta_i}} \, \overline{E_{\beta_j}} = q^{(\beta_i,\beta_j)} \overline{E_{\beta_j}} \, \overline{E_{\beta_i}}, & \overline{F_{\beta_i}} \, \overline{F_{\beta_j}} = q^{(\beta_i,\beta_j)} \overline{F_{\beta_j}} \, \overline{F_{\beta_i}} \ \  \mathrm{for}\ i<j, & \overline{E_{\beta_i}} \, \overline{F_{\beta_j}} = \overline{F_{\beta_j}} \, \overline{E_{\beta_i}},\\[.5em]
\overline{K_{\mu}} \, \overline{E_{\beta_j}} = q^{(\mu,\beta_j)} \overline{E_{\beta_j}} \, \overline{K_{\mu}}, & \overline{K_{\mu}} \, \overline{F_{\beta_j}} = q^{-(\mu,\beta_j)} \overline{F_{\beta_j}} \, \overline{K_{\mu}}, & \overline{K_{\mu}} \, \overline{K_{\nu}} = \overline{K_{\mu + \nu}}.
\end{array}
\end{equation}
 and 
\begin{equation}\label{commutationGrDCK2}
\overline{E_{\beta_k}} \, {_V\phi^i_j} = {_V\phi^i_j} \, \overline{E_{\beta_k}}, \quad \overline{F_{\beta_k}} \, {_V\phi^i_j} = q^{-(\beta_k, \epsilon_j)} {_V\phi^i_j} \, \overline{F_{\beta_k}}, \quad \overline{K_{\nu}} \, {_V\phi^i_j} = q^{(\nu, \epsilon_j)} {_V\phi^i_j} \, \overline{K_{\nu}}. 
\end{equation}
By keeping the same notation, $\Gg$ extends trivially to a filtration $\widehat{\Gg}$ on $\textstyle \widehat{\mathcal{H}}_{q}$. Indeed, by Def.\,\ref{defiHhatq} it has the additional generators $\widehat{K_\lambda}$; we give them degree $0$ in $\widehat{\Gg}$, and then the third relation in \eqref{commutationHatHq} implies 
\begin{equation}\label{relhatsept25}
\widehat{K_{\nu}} \, {_V\phi^i_j}x = \qD^{-(\nu, \epsilon_i)} {_V\phi^i_j}x \, \widehat{K_{\nu}},
\end{equation}
and the analogous relations for the cosets $ \overline{\widehat{K_{\nu}}}$ of the elements $\widehat{K_{\nu}}$ in $\mathrm{gr}_{\widehat{\Gg}}( \widehat{\mathcal{H}}_{\qD})$. All this is valid also for the specialization $\widehat{\Hh}_{\e}$. 

As in \cite[Prop. 5.7]{BFR} one can then combine the filtration $\Gg$ on the copies of $\widehat{\mathcal{H}}_{\e}$ with a filtration $\mathcal{F}_{\mathrm{DCK}}$ on the copies of $U_\e^\Pup$ to get an algebra filtration $\Tt$ on $\widehat{\mathcal{H}}_{\e}^{\otimes g} \otimes (U_\e^\Pup)^{\otimes n}$ such that $\mathrm{gr}_{\Tt}\bigl(\widehat{\mathcal{H}}_{\e}^{\otimes g} \otimes (U_\e^\Pup)^{\otimes n}\bigr) = \mathrm{gr}_{\widehat{\Gg}}( \widehat{\mathcal{H}}_{\e})^{\otimes g} \otimes \mathrm{gr}_{\mathcal{F}_{\mathrm{DCK}}}(U_\e^\Pup)^{\otimes n}$. By \eqref{commutationGrDCK}-\eqref{commutationGrDCK2}-\eqref{relhatsept25}, $\mathrm{gr}_{\Tt}\bigl(\widehat{\mathcal{H}}_{\e}^{\otimes g} \otimes (U_\e^\Pup)^{\otimes n}\bigr)$ is quasi-polynomial over $\Oo_\e^{\otimes 2g}$, as desired.
\end{proof}

\begin{remark}\label{rqmaxorder09/25}{\rm We used above that the algebra $\Oo_\e^{\otimes (2g)}\cong \Oo_\e(G^{\times (2g)})$ does not have non-trivial zero divisors. It it is also a Noetherian ring (since it is a finite $\mathcal{Z}_0(\Oo_\e)$-module and $\mathcal{Z}_0(\Oo_\e)\cong \Oo(G)$ is a Noetherian ring), and it is a maximal order (\cite{DC-L}, see Theorem \ref{DCLteo1}). Then the fact proved above that $\mathrm{gr}_{\Tt}\bigl(\widehat{\mathcal{H}}_{\e}^{\otimes g} \otimes (U_\e^\Pup)^{\otimes n}\bigr)$ is quasi-polynomial over $\Oo_\e^{\otimes (2g+n)}\cong \Oo_\e(G^{\times (2g+n)})$ and the following result (see,{\it e.g.}, \cite[Th.\,5.1.6]{MC-R}) imply that $\widehat{\mathcal{H}}_{\e}^{\otimes g} \otimes (U_\e^\Pup)^{\otimes n}$ is a maximal order: if $\mathfrak{A}$ be a filtered ring whose associated graded ring ${\rm gr}(\mathfrak{A})$ is a Noetherian domain which is a maximal order in its quotient ring, then $\mathfrak{A}$ has the same properties.}\end{remark}

\section{Primitive \texorpdfstring{$Dl$-th}{ } roots of unity}\label{appDlRoot} Recall from \S\ref{subsecQuantumGraphAlg} that $\mathcal{L}_{g,n} = \mathcal{L}_{g,n}(\mathfrak{g})$ is an algebra over $\mathbb{C}(\qD)$, where $\qD$ is a formal variable such that $\qD^D = q$ ($D$ being the smallest integer such that $DP \subset Q$, see \S\ref{subsecLieAlg}). Let $\epsilon$ be a root of unity whose order $l$ satisfies the assumptions \eqref{assumptionl}. There are two options regarding the specializations of $\mathcal{L}_{g,n}$:
\begin{itemize}
\item We can decide that $\qD$ is the main player. Then from the beginning we choose a primitive $Dl$-th root of unity $\textstyle \eD:= \exp(\frac{2i\pi k}{Dl})$, where $k$ and $Dl$ are coprime, and set $\textstyle \e := \eD^D = \exp(\frac{2i\pi k}{l})$.
\item We can decide that $q$ is the main player. Then we specialize $q$ to $\textstyle \e:= \exp(\frac{2i\pi k}{l})$, where $k$ and $l$ are coprime. In \eqref{choixED} we took $\eD=\e^{\overline{D}} = \textstyle \exp(\frac{2i\pi k \overline{D}}{l}) = \exp(\frac{2i\pi k \overline{D}D}{Dl})$ among the $D$-th roots of $\e$. It is no longer a $Dl$-th primitive root of unity.

\end{itemize}
The second option was used everywhere in the present paper (also in the work \cite{GJS}). The goal of this appendix is to give an overview of what happens with the first option. We will then indicate specialization to $\qD = \eD$ by a subscript or a superscript $\eD$, because now $\eD$ does not depend on $\e$. There are intermediate situations that could be considered but that we do not adress in this Appendix, namely the cases where $\eD$ is an arbitrary, not primitive, $Dl$-th root of unity.
\smallskip

Similarly as in Def.\,\ref{DefLgnesept25}, we define $\mathcal{L}_{g,n}^{\eD}= \mathcal{L}_{g,n}^{\eD}(\mathfrak{g})$ as the $\mathbb{C}$-algebra $\mathcal{L}_{g,n}^{\AD} \otimes_{\AD} \mathbb{C}_{\eD}$, that is the specialization of $\mathcal{L}_{g,n}^{\AD}$ at the root of unity $\qD = \eD$ as above (in Def.\,\ref{DefLgnesept25} we used $\eD= \e^{\bar{D}}$). The results of \S \ref{subsecModifiedAlekseev}-\ref{secQMMandInv} hold true verbatim for $\mathcal{L}_{g,n}^{\eD}$. By the same proof as Prop.\,\ref{propLgndomain} the algebra $\mathcal{L}_{g,n}^{\eD}$ is a domain.

\subsection{Preliminaries: the quantum function algebra \texorpdfstring{$\Oo_\e^\Lam$}{for a sublattice}}  (See \cite{BG}) Let us fix a sublattice $\Lambda$ of $P$ containing $Q$. It has rank $m$, and the semigroup $\Lambda_+ = P_+\cap \Lambda$ is finitely generated. In the following sections we will use only $\Lambda=P$ or $\Lambda=Q$. \\
\indent Recall from $\S$\ref{sec:Uq} and $\S$\ref{integralO} the categories $\mathcal{C}$ and $\mathcal{C}_A$.

\begin{defi}{\rm We denote by ${\mathcal C}^\Lam$ the full subcategory of ${\mathcal C}$ consisting of the modules all of whose weights lie in $\Lambda$ (so ${\mathcal C} = {\mathcal C}^\Pup$). The quantum function algebra associated to the lattice $\Lambda$ is $\textstyle {\mathcal O}_q^\Lam:= \bigoplus_{\mu\in \Lambda_+}C(\mu).$}
\end{defi}
The category ${\mathcal C}^\Lam$ is monoidal, rigid, and semisimple, with simple objects up to isomorphisms the modules $V_\mu$ where $\mu\in\Lambda_+$. The $\mc(q)$-vector space ${\mathcal O}_q^\Lam$ is a Hopf subalgebra of $\Oo_q$. It is generated as an algebra by the matrix coefficients of the modules $V_{\lambda_1^+},\ldots , V_{\lambda_r^+}$, where $\lambda_1^+,\ldots, \lambda_r^+$ generate $\Lambda_+$.
\begin{defi}{\rm We denote by ${\mathcal C}^\Lam_A$ the full subcategory of ${\mathcal C}_A$ consisting of the modules all of whose weights lie in $\Lambda$ (so ${\mathcal C}_A = {\mathcal C}^\Pup_A$). The integral quantum function algebra associated to $\Lambda$ is $\Oo_A^\Lam := \Oo_q^\Lam \cap \Oo_A$.}
\end{defi}
The category ${\mathcal C}^\Lam_A$ is monoidal and rigid, non semisimple, with simple objects up to isomorphisms the modules ${}_AV_\lambda$ with $\lambda\in \Lambda_+$. The $A$-module $\Oo_A^\Lam$ is a Hopf $A$-subalgebra of $\Oo_A$. It is spanned over $A$ by the matrix coefficients of the modules in $\mathcal{C}_A^\Lam$, and generated as an $A$-algebra by the matrix coefficients of the modules ${}_AV_{\lambda_1^+},\ldots , {}_AV_{\lambda_r^+}$, where $\lambda_1^+,\ldots, \lambda_r^+$ generate $\Lambda_+$.
\smallskip

 Many of the properties of $\Oo_A$ transfer to $\Oo_A^\Lam$. In particular we need:
\begin{prop}\label{PhirresLam} The restriction of $\Phi_A^\pm$ to $\Oo_A^\Lam$ gives an isomorphism $\Oo_A^\Lam(B_\pm) \overset{\sim}{\to} U_A^\Lam(\mathfrak{b}_\mp)^{\mathrm{cop}}$.
\end{prop}
\begin{proof} One proceeds exactly as in \cite[Prop.\,4.2]{DC-L}, which proves that $\Phi_A^\pm$ is an isomorphism. The point is to show that any module $M$ from ${\mathcal C}^\Lam_A$ is naturally a $U_A^\Lam(\mathfrak{b}_\mp)^{cop}$-comodule. Once the case of $\Lambda=P$ is done, one uses the fact $M$ has an $A$-basis of vectors $m_{\lambda,i}$ of weights $\lambda\in \Lambda$, and conclude by the same computations. \end{proof}
\medskip

As usual, by specializing $q$ to $\e$ we define $\Oo_\e^\Lam := \Oo_A^\Lam\otimes_A \mc_\e$; so $\Oo_\e= \Oo_\e^\Pup$.

\subsection{Central subalgebras}\label{sec:centsubalg} Let us fix again a sublattice $\Lambda$ of $P$ containing $Q$. Recall the notation $G^\Lam$ from \S\ref{subsecLieAlg}. The function algebra $\mathcal{O}(G^\Lam)$ is the $\mc$-linear subspace of $\mathcal{O}(G^\Pup)$ generated by the matrix coefficients of the finite dimensional $U(\mathfrak{g})$-modules with weights in $\Lambda$. Recall the dual Frobenius morphism $\mathbb{F}\mathrm{r}_\e^* : \mathcal{O}(G^\Pup) \hookrightarrow \mathcal{O}_\e$ from \S\ref{specialisOA}. We define 
\begin{equation}\label{LambdaZmai25}
 {\mathcal Z}_0^\Lam({\mathcal O}_\e) := \mathbb{F}\mathrm{r}_\e^{*}\bigl( \mathcal{O}(G^\Lam) \bigr) \subset \mathcal{O}_\e \quad \text{and} \quad \mathcal{Z}_0^{\Lam,+} := \ker\bigl(\varepsilon\vert_{\mathcal{Z}_0^\Lam(\Oo_\e)} \bigr).
 \end{equation}
Here are some facts. Since ${\mathcal Z}_0^\Lam({\mathcal O}_\e)\subset {\mathcal Z}_0^\Pup ({\mathcal O}_\e) = {\mathcal Z}_0({\mathcal O}_\e)$ is a central Hopf subalgebra of $\Oo_\e$, $\mathcal{Z}_0^{\Lam,+}$ is a Hopf ideal of $\Oo_\e$. The natural $P/\Lambda$-grading of $P$ gives a $P/\Lambda$-grading on the Peter-Weyl decomposition $\Oo(G^\Pup)=\textstyle \oplus_{\mu\in P^+} C(\mu)$, where $\mathcal{O}(G^\Lam)=\textstyle \oplus_{\lambda\in \Lambda^+} C(\lambda)$ is the homogeneous component of degree $0$. This grading of $\Oo(G^\Pup)$ transfers via $\mathbb{F}\mathrm{r}_\e^{*}$, so that $\mathcal{Z}_0^\Pup({\mathcal O}_\e)$ is a free $\mathcal{Z}_0^\Lam(\mathcal{O}_\e)$-module of rank $\vert P/\Lambda \vert$. It follows from Theorem \ref{DCLteo1}(1) that $\mathcal{Z}_0^\Lam({\mathcal O}_\e)$ is generated as an algebra by certain matrix coefficients of simple $\Gamma_\e^\Q$-modules of highest weights $l\lambda$, where $\lambda\in \Lambda_+$.

\begin{defi}\label{defZ0gnoct24} We let $\mathcal{Z}_0(\Ll_{g,n}^{\eD})  := \mathcal{Z}_0^{\Q}(\Oo_{\e})^{\otimes 2g}\otimes \mathcal{Z}_0(\Oo_{\e})^{\otimes n} \subset \Ll_{g,n}^{\eD}$.
\end{defi}
Note the order of the tensor factors: each factor $\mathcal{Z}_0^{\Q}(\Oo_{\e})^{\otimes 2}$ corresponds to one factor $\mathcal{L}_{1,0}^{\eD}$ in $\mathcal{L}_{g,n}^{\eD}=(\mathcal{L}_{1,0}^{\eD})^{\widetilde{\otimes }g}\widetilde{\otimes} (\mathcal{L}_{0,1}^\e)^{\widetilde{\otimes} n}$.

The following result is analogous to Prop.\,\ref{Z0Lgn} for $\Ll_{g,n}^{\eD}$.
\begin{prop}\label{Z0LgneD} \label{Z0L0n} For all $x \in \mathcal{Z}_0(\mathcal{L}_{g,n}^{\eD})$ and $y \in \mathcal{L}_{g,n}^{\eD}$ we have $xy = yx = x \star y$.
\\It follows that:

\noindent (1) $\mathcal{Z}_0(\Ll_{g,n}^{\eD})$ is a central subalgebra of $\Ll_{g,n}^{\eD}$, isomorphic to $\Oo(G^\Q)^{\otimes 2g}\otimes \Oo(G)^{\otimes n}$ through $(\mathbb{F}\mathrm{r}_\e^{*})^{\otimes (2g+n)}$. In particular $\mathcal{Z}_0(\Ll_{g,n}^{\eD})$ is a Noetherian ring.

\noindent (2) As $\mathcal{Z}_0(\Ll_{g,n}^{\eD})$-modules, $\Ll_{g,n}^{\eD}$ and $\Oo_{\e}^{\otimes (2g+n)}$ coincide.

\noindent (3) $\Ll_{g,n}^{\eD}$ is a free $\mathcal{Z}_0(\Ll_{g,n}^{\eD})$-module of rank $\vert P/Q\vert^{2g}.l^{(2g+n).\mathrm{dim}(\mathfrak{g})}$ and a Noetherian ring, where $l$ is the order of the root of unity $\e$.
\end{prop}

\begin{proof}
Item (1) and (2) are immediate consequences of the first claim. For item (3), Noetherianity is proved as in Prop.\,\ref{Z0Lgn} (3). For freeness and rank, combine the facts that $\Oo_{\e}$ is a free $\mathcal{Z}_0(\Oo_\e)$-module of rank $l^{\mathrm{dim}(\mathfrak{g})}$ (Th.\,\ref{DCLteo1} (2)), and $\mathcal{Z}_0(\Oo_{\e})$ is a free module of rank $\vert P/\Lambda\vert$ over $\mathcal{Z}_0^\Lam(\Oo_{\e})$ (see the comments below \eqref{LambdaZmai25}). Then taking $\Lambda=Q$ the conclusion follows from (2).

For the first claim we need to reconsider in detail the computations in the proof Prop.\,\ref{Z0Lgn}, for all kinds of relations in $\Ll_{g,n}^{\eD}$ (\S \ref{subsecDefLgnH}). Centrality with respect to the exchange relation \eqref{braidedTensProdComm} and the fusion relation \eqref{fusionRelL01} holds true for arbitrary $\eD$, as proved in \cite[Prop.\,4.1]{BR2} (note that the arguments in Prop.\,\ref{Z0Lgn} based on Cor.\,\ref{relationsCoRMatAtEps} apply only when $\eD = \e^{\bar{D}}$). We explain centrality with respect to the other exchange relation, namely \eqref{L01prodsept25}.

It is enough to show that the subalgebras $1\otimes \mathcal{Z}_0^\Q(\Oo_{\e})$ and $\mathcal{Z}_0^\Q(\Oo_{\e})\otimes 1$ of $\Ll_{1,0}^{\eD}$ are central. Let $\alpha\in \Oo_A$ be such that its specialization $\alpha_{\vert \e}$ is in $\mathcal{Z}_0^\Q(\Oo_{\e})$. By Prop.\,\ref{PhirresLam} and \cite[Prop. 6.5]{DC-L} (see \cite[Th.\,2.29(2)]{BR2} in the present notations) we have $\Phi_\e^\pm(\alpha_{\vert \e})\in \mathcal{Z}_0\bigl( U_{\e}^\Q(\mathfrak{b}_\mp) \bigr)$. From this and \cite[Lem.\,2.28]{BR2} it follows that
\[ \forall \beta\in \Oo_A,\quad \bigl\langle \beta,\Phi^\pm_A(\alpha))\bigr\rangle_{A_\D\vert \eD}^\Pup = \varepsilon(\alpha_{\vert \e})\varepsilon(\beta_{\vert \e}) \]
where the subscript $_{\vert\eD}$ indicates specialization $\qD \mapsto \eD$. On the left-hand side, $\Phi^{\pm}_A(\alpha)$ belongs to $U_A^\Pup(\mathfrak{b}_\mp)$, so the inclusion $U_A^\Pup(\mathfrak{b}_\mp) \hookrightarrow \Gamma_A^\Pup(\mathfrak{b}_\mp)$ is implicitly used. The point of the above equality is that the specialization of this inclusion factorizes through the ideal generated by $\mathcal{Z}_0\bigl( U_{\e}^\Q(\mathfrak{b}_\mp) \bigr) \cap \ker(\varepsilon)$, as was explained in \S\ref{specializUq} for the other choice of $\eD$. Now, using \eqref{integralCoRmat} (which is the integral version of \eqref{PhipmInverses} and \eqref{coRMatSign}), it is readily checked that for any $\beta_{\vert \e}\in \Oo_\e$ the relation \eqref{L01prodsept25} gives
\begin{align*}
&(1\otimes \alpha_{\vert \e}) (\beta_{\vert \e}\otimes 1) \\
=\:& \sum_{(\alpha),(\beta)} (\beta_{(3)|\e} \otimes \alpha_{(3)|\e})\ \bigl\langle \beta_{(1)},\Phi_A^-\bigl( S^{2}(\alpha_{(5)})\bigr)\bigr\rangle_{A_\D\vert \eD}^\Pup \bigl\langle \beta_{(2)}, \Phi_A^-\bigl( S(\alpha_{(1)}) \bigr) \bigr\rangle_{A_\D\vert \eD}^\Pup\times \\
& \hspace*{4cm} \times \bigl\langle \beta_{(4)}, \Phi_A^+(\alpha_{(2)}) \bigr\rangle_{A_\D\vert \eD}^\Pup \bigl\langle \beta_{(5)}, \Phi_A^-\bigl( S(\alpha_{(4)}) \bigr)\bigr\rangle_{A_\D\vert \eD}^\Pup \notag \\
=\:& \beta_{\vert \e}\otimes \alpha_{\vert \e}  = (\beta_{\vert \e}\otimes 1)(1\otimes \alpha_{\vert \e}).
\end{align*}
One proves similarly that $(1\otimes \alpha_{\vert \e})(\beta_{\vert \e}\otimes 1) = (\beta_{\vert \e}\otimes 1)(1\otimes \alpha_{\vert \e})$ if $\alpha_{\vert \e}\in \Oo_\e$, $\beta_{\vert \e}\in \mathcal{Z}_0^\Q(\Oo_{\e})$.
\end{proof}
\begin{remark}\label{centralextsl2eD}{\rm One can check by direct computation that in the case $\mathfrak{g} = \mathfrak{sl}_2$ the space $\mathcal{Z}_0(\Oo_\e)^{\otimes (2g+n)}$, which strictly contains $\mathcal{Z}_0(\Ll_{g,n}^{\eD})$, is also a central subalgebra in $\Ll_{g,n}^{\eD}$. We strongly expect that for any $\mathfrak{g}$ one can define a larger $\mathcal{Z}_0$ than the one in Def.\,\ref{defZ0gnoct24}, of the form $\mathcal{Z}_0^{\Lam}(\Oo_{\e})^{\otimes 2g}\otimes \mathcal{Z}_0(\Oo_{\e})^{\otimes n}$ for some sublattice $Q \subset \Lambda \subset P$ which depends on $D$, such that Prop.\,\ref{Z0LgneD} still holds (after the obvious changes).
}
\end{remark}

\subsection{The central extension} The aim of this section is to prove the analog of Th.\,\ref{centralextLgn} when $\eD=\e^{\bar{D}}$ is replaced by an arbitrary primitive $Dl$-th root of $1$. This will be achieved in \S\ref{sec:analogextfev26} (see Cor.\,\ref{centralextLgneD}) after having done the following preliminary steps. First, because of the appearance of the lattice $Q$ in Def.\, \ref{defZ0gnoct24} and Prop.\,\ref{Z0LgneD}(1), this result will use a map $\pi^\Q$ analogous to $\pi$ in Prop.\,\ref{pipropO}; this map is studied (in some greater generality, replacing $Q$ with an arbitrary sublattice $\Lambda\subset P$ containing $Q$) in Prop.\,\ref{pipropOLam} below. Next, we will need an $R$-matrix $R_l^\Pup$ on the small quantum group $u_\e^\Pup$ associated to the lattice $P$; $u_\e^\Pup$ is described below, and $R_l^\Pup$ in $\S$\ref{subsec:qtsuefev26}. Finally, in $\S$\ref{sec:relatingfev26} we study the relationships between this $R$-matrix and the maps $\Phi_\e^\pm$, in a way similar to $\S$\ref{sec:qtsroot1}.
\smallskip

Given a sublattice $Q \subset \Lambda \subset P$ containing $Q$, its {\em dual lattice} $Y=Y(\Lambda)$ is defined by
\begin{equation}\label{duallatticesept25}
Y := \{\zeta\in P \: \vert \: \forall \lambda\in \Lambda, \:\: (\zeta,\lambda) \in \mz\}.
\end{equation}
So we have inclusions $Q \subset Y \subset P$. Recall from Rmk.\,\ref{comparaisonRestrictedVersions} that there is an inclusion $U_A^\Y \hookrightarrow \Gamma_A^\Y$ of the unrestricted to the restricted integral forms. 
\begin{defi}\label{defueY} We denote by $u_\e^\Y$ the image of the specialization to $q=\e$ of the inclusion map $U_A^\Y \hookrightarrow \Gamma_A^\Y$.
\end{defi}
By the same arguments as for $u_\e^\Q$ (see the comments after \eqref{dlambdae}) we have 
\begin{equation}\label{reltotouY} u_\e^\Y  \cong U_{\epsilon}^\Y\big/[{\rm relations}\ \eqref{toto}].
\end{equation}
Recall that this map is no longer injective; $u_\e^\Y\subset \Gamma_\e^\Y$ is a Hopf subalgebra of dimension
\begin{equation}\label{dimuYmars26}{\rm dim}_\mc(u_\e^\Y) = |Y/Q|l^{\dim(\mathfrak{g})},
\end{equation}
and we can consider its Hopf algebra dual $(u_\epsilon^\Y)^{*}$. The proof of Lemma \ref{nondegpairepsD} works also for the specialization to $\qD=\eD$, and therefore the Hopf pairing $\langle \text{-},\text{-} \rangle_{\eD}^\Pup \colon {\mathcal O}_\e\times \Gamma_\e^\Pup \to \mc$ is non-degenerate. Define a map $\pi^\Lam: {\mathcal O}_\epsilon \rightarrow (u_\epsilon^\Y)^{*}$ by  restriction, {\it i.e.}
\begin{equation}\label{piLamdef}
\forall \, \varphi \in \mathcal{O}_\epsilon, \:\: \forall \, x\in u_\epsilon^\Y \subset \Gamma_\epsilon^\Y \subset \Gamma_\e^\Pup, \quad \pi^\Lam(\varphi)(x):=\langle \varphi,x\rangle_{\eD}^\Pup.
\end{equation}
Proposition \ref{pipropO} (which was for $\Lambda = P$ and $\eD$ of order $l$) generalizes to the present setup:
\begin{prop}\label{pipropOLam}
(1) The map $\pi^\Lam$ is a surjective morphism of Hopf algebras.

\noindent (2) We have  $\ker(\pi^\Lam)={\mathcal O}_\epsilon \mathcal{Z}_{0}^{\Lam,+}$.
\end{prop}

\begin{proof}
(1) is proven exactly as in Prop.\,\ref{pipropO}.

(2) Let us show that $\mathcal{O}_\e\mathcal{Z}_{0}^{\Lam,+} \subset \ker(\pi^\Lam)$. Due to the Hopf pairing property, it suffices to check that $\langle \varphi , - \rangle_{\eD}^\Pup$ vanishes on the generators $E_i,F_i, K_{\zeta}$ for $\zeta \in Y$, of $u_\epsilon^\Y$, for all $\varphi \in \mathcal{Z}_0^{\Lam,+}$. Only the case of $K_{\zeta}$ is non-trivial. By Theorem \ref{DCLteo1} (1) we know that ${\mathcal Z}_0^\Lam({\mathcal O}_\e)$ is generated as a subalgebra by certain matrix coefficients of simple $\Gamma_\e$-modules of highest weights $l\lambda$, $\lambda\in \Lambda_+$. Let $f(?\cdot v) : x \mapsto f(x \cdot v)$ be such a matrix coefficient. We can assume that $v$ is a weight vector, and an important point for us is that all weights in such modules are multiples of $l$, thus of the form $l(\lambda + \alpha)$ for some $\alpha \in Q$ (see \cite[Remark on p.\,357]{CP}). From \eqref{actionLi} it follows
\[ f(K_{\zeta} \cdot v) = \e_\D^{Dl(\zeta, \, \lambda + \alpha)}f(v) = f(v) \]
because $(\zeta, \, \lambda + \alpha) \in \mathbb{Z}$ by definition of the lattice $Y$ which is dual to $\Lambda$.\footnote{Note that we write $\e_\D^{Dl(\zeta, \, \lambda + \alpha)}$ instead of $\e^{l(\zeta, \, \lambda + \alpha)}$ because in general $(\zeta, \, \lambda + \alpha)$ belongs to $\textstyle \frac{1}{D}\mathbb{Z}$ if $\zeta$ is not in $Y$. It is at this point that we see the relevance of the dual lattice $Y$.} Since $K_{\zeta}$ is grouplike, this fact on generators remains true for any element $\varphi \in {\mathcal Z}_{0}^{\Lam}$, \textit{i.e.} $\langle \varphi, K_{\lambda} \rangle_{\epsilon_\D}^\Pup = \varphi(1) = \varepsilon(\varphi)=0$.

\indent Now, since ${\mathcal O}_\epsilon {\mathcal Z}_{0}^{\Lam,+}\subset \ker(\pi^\Lam)$, the quotient morphism $\bar{\pi}^\Lam: {\mathcal O}_\e/{\mathcal O}_\epsilon {\mathcal Z}_{0}^{\Lam,+}\rightarrow (u_\epsilon^\Y)^{*}$ is still surjective. From \cite[\S III.7.7]{BG}, the algebra ${\mathcal O}_\e/{\mathcal O}_\epsilon {\mathcal Z}_{0}^{\Pup,+}$ has dimension $l^{\dim(\mathfrak g)}$. We have
\[ \begin{array}{l}
\dim({\mathcal O}_\e/{\mathcal O}_\epsilon {\mathcal Z}_{0}^{\Lam,+}) = \dim({\mathcal O}_\e/{\mathcal O}_\epsilon {\mathcal Z}_{0}^{\Pup,+}) \dim({\mathcal O}_\epsilon {\mathcal Z}_{0}^{\Pup,+}/{\mathcal O}_\epsilon {\mathcal Z}_{0}^{\Lam,+}),\\[.3em]
\dim({\mathcal O}_\epsilon {\mathcal Z}_{0}^{\Pup,+}/{\mathcal O}_\epsilon {\mathcal Z}_{0}^{\Lam,+}) = \vert Y/Q\vert
\end{array} \]
since ${\mathcal Z}_0^\Pup({\mathcal O}_\e)$ is a free ${\mathcal Z}_0^\Lam({\mathcal O}_\e)$-module of rank $\vert P/\Lambda \vert$, and $\vert P/\Lambda\vert = \vert Y/Q\vert$ by the duality map $P \to {\rm Hom}(Q,\mz)$, $\lambda \mapsto (- \ ,\lambda )$. But $\dim(u_\e^\Y)^{*}= \vert Y/Q\vert . l^{\dim {\mathfrak g}}$ by \eqref{dlambdae}, so $\bar{\pi}^\Lam$ is an isomorphism and $\ker(\pi^\Lam)={\mathcal O}_\epsilon {\mathcal Z}_{0}^{\Lam,+}.$
\end{proof}
\medskip

\subsubsection{Quasi-triangular structures on \texorpdfstring{$u_\e^\Pup$}{the small quantum group} and factorisability}\label{subsec:qtsuefev26} Because of Prop.\,\ref{Z0LgneD}(1), only the map $\pi^\Lam$ of Prop.\,\ref{pipropOLam} for the lattices $\Lambda=Q$ and $\Lambda=P$ will eventually be relevant for the generalization of Th.\,\ref{centralextLgn} when $\eD=\e^{\bar{D}}$ is replaced by an arbitrary primitive $Dl$-th root of $1$.\\
We note that $u_\e^\Pup$ has always the $R$-matrix $\bar{R}$ used in $\S$\ref{subsecRmatSmallQG} (also denoted $R_l^\Q$ below); the point is that we are looking for $R$-matrices with larger Cartan support, so that they may provide factorisability.

\smallskip

\indent The following theorem summarizes a combination of results proved in \cite[App.\,A1-A3]{Lyub95} and \cite{LN,LO}.

\begin{teo}\label{teosmallqg} (i) The Hopf algebra $u_\e^\Pup$ has a ribbon structure, with an $R$-matrix $R_l^\Pup$ defined in \eqref{RmatlP} below, and pivotal element $\ell=K_{2\rho}$. It is factorizable if, and only if, $D$ is odd.

(ii) The image of $\Phi_{0,1}\colon (u_\e^{\Pup})^* \to u_\e^{\Pup}$ is the subalgebra of $u_\e^{\Pup}$ with basis elements \begin{equation}\label{basisPhi01}\textstyle \left(\prod_{j=N}^1 F_{\beta_j}^{t_j}\right) \left(\prod_{j=N}^1 K_{\beta_j}^{t_j}\right) K_{2\bar{\lambda}} \left(\prod_{j=N}^1 E_{\beta_j}^{s_j}\right),
\end{equation}
where $0\leq t_j, s_j< l$ and $\bar{\lambda}\in P$ spans a set of representatives of the elements of $P/lQ$. Therefore, if $D$ is even $\Phi_{0,1}\colon (u_\e^{\Pup})^* \to u_\e^{\Pup}$ is not injective, and of course not surjective.
\end{teo}
Note that, contrary to $\Phi_{0,1}\colon (u_\e^{\Pup})^* \to u_\e^{\Pup}$ above, the morphism $\Phi_{0,1}^{\e}\colon \Ll_{0,1}^\e \to U_\e^{\Pup}$ is injective (see Prop.\,\ref{AlekseevInjRootOf1}).

\begin{proof}
(i) Let us first recall the case of $\Lambda=Q$, which is treated by Lyubashenko in \cite{Lyub95}. Consider the specialization at $q=\e$ of the restriction to $Q\times Q$ of the Drinfeld pairing $\tau$ in \eqref{deftau}; it is the symmetric bilinear form $\tau_{\vert \Q}:Q \times Q \to \mc^\times$ given by $\tau_{\vert \Q}(\beta,\alpha) := \e^{-(\beta, \alpha)}$. Denote by ${\rm rad}(\tau_{\vert \Q}) := \{\alpha\in Q\ \vert\  \forall \beta\in Q, \: \tau_{\vert \Q}(\beta,\alpha)=1\}$ its radical, and let
\[ \Gamma_\e^{\tau_{\vert \Q}} := \Gamma_\e \,\big/ \bigl\langle K_\alpha-1, \: \alpha\in {\rm rad}(\tau_{\vert \Q}) \bigr\rangle. \]
Also let $\mathbb{G}:= Q/{\rm rad}(\tau_{\vert \Q})$ (which is isomorphic to the Cartan subalgebra of $\Gamma_\e^{\tau_{\vert \Q}}$), and $u_\e^{\tau_{\vert \Q}}$ be the Hopf subalgebra of $\Gamma_\e^{\tau_{\vert \Q}}$ generated by the elements $E_{\beta}$, $F_{\beta}$ for $\beta\in \phi^+$ and $K_\alpha$, $\alpha\in Q$. Lyubashenko proved in \cite[App.\,A.1--A.3]{Lyub95}, for a root of unity $\e$ of arbitrary order $l$, that $\Gamma_\e^{\tau_{\vert \Q}}$ is quasitriangular and ribbon, with pivotal element $\ell=K_{2\rho}$, with an $R$-matrix supported on $(u_\e^{\tau_{\vert \Q}})^{\otimes 2}$, thus making $u_\e^{\tau_{\vert \Q}}$ quasi triangular and ribbon as well. Moreover $u_\e^{\tau_{\vert \Q}}$ is factorizable if, and only if, $2\mathbb{G} = \mathbb{G}$. The $R$-matrix on ($\Gamma_\e^{\tau_{\vert \Q}}$ and) $u_\e^{\tau_{\vert \Q}}$ has the form 
\begin{equation}\label{Rmatl}
R_l^\Q = \Theta_l \hat{R}_l,
\end{equation}
where (see \cite[Thm.\,A.1.6 and eq.\,(A.2.1)]{Lyub95}):
\begin{align*}
\Theta_l & := \frac{1}{\vert \mathbb{G} \vert}\sum_{\bar{\lambda},\bar{\mu}\in \mathbb{G}} \e^{-(\lambda,\mu)} K_{\bar{\lambda}} \otimes K_{\bar{\mu}}\\
\hat{R}_l & := \sum_{0\leq t_1,\ldots,t_N<l} \ \prod_{r=N}^1 \e_{\beta_r}^{\frac{1}{2}t_r(t_r+1)} \frac{(1-\e_{\beta_r}^{-2})^{t_r}}{[t_r]_{\e_{\beta_r}} !} (E_{\beta_r})^{t_r} \otimes (F_{\beta_r})^{t_r}.
\end{align*}
Here $\bar{\lambda}$ denotes the class in $\mathbb{G}$ of an element $\lambda\in Q$, and $K_{\bar{\lambda}}$ one representative of the set $\{K_\mu, \mu\in \bar{\lambda}\}$. Notice that $\hat{R}_l=(p^\Y\otimes {\rm id})(\hat{R}_\e)\in (u_\e^{\tau_{\vert \Q}})^{\otimes 2}$, where $p^\Y\colon U_\e^\Y \to u_\e^\Y$ is the quotient map from \eqref{reltotouY}, and $\hat{R}_\e$ is obtained from $\hat{R}$ in \eqref{Rhat} by specialization to $q=\e$ ({\it e.g.}, regarding $\hat{R}$ as an element of $\stackrel{.}{\mathbf{U}}_A\!\!\!{}^{\hat{\otimes} 2}$ through the embeddings $\Gamma_A^\Lam\hookrightarrow \hat{{\mathbf{U}}}_A$ analogous to the one of Lemma \ref{embedGammaeps}). A key point of the proof that $R_l^\Q$ is an $R$-matrix is to show that $\Theta_l$ is an $R$-matrix for the Cartan subalgebra of $u_\e^{\tau_{\vert \Q}}$; in particular the property $(\Delta\otimes {\rm id})(\Theta_l) = (\Theta_l)_{13}(\Theta_l)_{23}$ follows from the identity
$$\sum_{\bar{\lambda}\in \mathbb{G}} \e^{-(\lambda,\mu)}  = \vert \mathbb{G} \vert \cdot \delta_{\mu\in {\rm rad}(\tau_{\vert \Q})},$$
which is a consequence of the non-degeneracy of the pairing induced by ${\tau_{\vert \Q}}$ on $\mathbb{G}$. But
\begin{align*}
{\rm rad}(\tau_{\vert \Q}) & = Q\cap l\mz\{d_i^{-1}\varpi_i, i=1,\ldots,m\} \\ & = Q\cap lP = \left\lbrace\sum_{j=1}^m c_j\alpha_j\in Q\ \left\vert \ c_j\in \mz \text{ are such that }\sum_{j=1}^m a_{ij}c_j \equiv 0 \: (\mathrm{mod} \: l) \text{ for all } i\right.\right\rbrace .
\end{align*}
In the second equality we use our assumption ${\rm gcd}(l,d_i)=1$, $i=1,\ldots,m$. Therefore ${\rm rad}(\tau_{\vert \Q})=lQ$ if, and only if, ${\rm gcd}(l,{\rm det}(a_{ij})) =1$, which is always satisfied in the present cases of $\mathfrak{g}$, since $l$ satisfies the assumptions \eqref{assumptionl}. By \eqref{toto} it follows that $\Gamma_\e^{\tau_{\vert \Q}}=\Gamma_\e$, and $u_\e^{\tau_{\vert \Q}} = u_\e^\Q$, exactly when ${\rm gcd}(l,D)=1$. Since $l$ is odd and $\mathbb{G} = Q/lQ$, we have $2\mathbb{G} = \mathbb{G}$, and hence Lyubashenko's results recalled above prove that $u_\e^\Q$ is ribbon and factorizable.

The case $\Lambda=P$ has been treated by \cite{LN,LO}. The above arguments extend verbatim to $u_\e^{\Pup}$ by considering the form $\tau_{\vert \Pup}:P \times P \to \mc^\times$ given by $\tau_{\vert \Pup}(\lambda,\mu) := \e^{-(\lambda, \mu)}$, where as usual $\e^{-(\lambda,\mu)} := \eD^{-D(\lambda,\mu)}$. Indeed, replacing $\Gamma_\e^{\tau_{\vert \Q}}$  with $\Gamma_\e^{\tau_{\vert \Pup}} := \Gamma_\e^\Pup/(K_\alpha-1,\alpha\in {\rm rad}(\tau_{\vert \Pup}))$ and using that $\tau_{\vert \Pup}$ is symmetric and non degenerate on $\mathbb{G}^\Pup:= P/{\rm rad}(\tau_{\vert \Pup})$ (whose group algebra is the Cartan subalgebra of $\Gamma_\e^{\tau_{\vert \Pup}}$), the arguments of \cite[Thm.\,A.1.6]{Lyub95} yield an $R$-matrix on $\Gamma_\e^{\tau_{\vert \Pup}}$ and $u_\e^{\tau_{\vert \Pup}} := U_\e^\Pup/(E_{\beta}^l=F_{\beta}^l=0, \forall \beta\in \phi^+, \;K_\lambda=1, \forall \lambda\in {\rm rad}(\tau_{\vert \Pup}))$. Since $\eD$ is a primitive $Dl$-th root of $1$, ${\rm gcd}(l,d_i)=1$, and ${\rm gcd}(l,D)=1$, we have ${\rm rad}(\tau_{\vert \Pup}) := \{\mu\in P\ \vert\  \forall \lambda\in P, \tau_{\vert \Pup}(\lambda,\mu)=1\}=lQ$, and therefore $u_\e^\Pup$ coincides with $u_\e^{\tau_{\vert \Pup}}$. The obtained $R$-matrix is
\begin{equation}\label{RmatlP}
R_l^\Pup = \Theta_l^\Pup \hat{R}_l,
\end{equation}
with $\hat{R}_l$ as above and $\Theta_l^\Pup$ of the form
$$\Theta_l^\Pup = \frac{1}{\vert P/lQ \vert}\sum_{\bar{\lambda},\bar{\mu}\in P/lQ} \eD^{-D(\lambda,\mu)} K_{\bar{\lambda}} \otimes K_{\bar{\mu}}.$$
Note that the map $u_\e(\mathfrak{b}_+)^*\to u_\e(\mathfrak{b}_-)$, $\alpha\mapsto (\alpha\otimes {\rm id})(R_l^\Pup)$, is a bijection, so $R_l^\Pup$ puts $u_\e(\mathfrak{b}_+)$ and $u_\e(\mathfrak{b}_-)$ in duality. Also, note that for $D$ odd we have $\vert P/lQ \vert =Dl^m$. A pivotal element is obviously $\ell=K_{2\rho}$, as for $u_\e^\Q$. \\
\indent Since $\Theta_l^\Pup$ is a symmetric tensor, the algebra $u_\e^\Pup$ is factorizable if and only if $\vert P/lQ \vert = \vert P/Q \vert\cdot l^m$ is odd \cite[Ex.\,5.11]{LO}. Since $l$ is odd, this amounts to $\vert P/Q \vert$ odd and proves (i).

(ii) Let us consider the analogous situation for $\Lambda=Q$. The computation of the image of $\Phi_{0,1}\colon (u_\e^{\Q})^* \to u_\e^\Q$ is an immediate consequence of \cite[Prop.\,A.3.1]{Lyub95}. The proof is in two steps, first commuting the two first factors of $R_l^\Q(R_l^\Q)' = \Theta_l \hat{R}_l \Theta_l (\hat{R}_l)'$ (where $(R_l^\Q)'=\sigma(R_l^\Q)$ and $\sigma\colon x\otimes y\mapsto y\otimes x$, extended linearly), so as to get an expression of the form 
\begin{multline}\label{commutRR'}
R_l^\Q(R_l^\Q)' = \sum_{{\substack{t_r,s_r=0 \\ r=1,\ldots,N}}}^{l-1} \ \sum_{(\Theta_l^2)}\ \left(\prod_{r=1}^N \epsilon_{\beta_r}^{\frac{1}{2}t_r(t_r+1)} \frac{(1-\epsilon_{\beta_r}^{-2})^{t_r}}{[t_r]_{\epsilon_{\beta_r}} !}\right)  \left(\prod_{r=1}^N \epsilon_{\beta_r}^{\frac{1}{2}s_r(s_r+1)} \frac{(1-\epsilon_{\beta_r}^{-2})^{s_r}}{[s_r]_{\epsilon_{\beta_r}} !} \right)\\   \left(\prod_{r=N}^1 (E_{\beta_r}K_{\beta_r}^{-1})^{t_r}\right) (\Theta_l^2)_{(1)} \left(\prod_{r=N}^1  F_{\beta_r}^{s_r} \right)
\otimes  \left(\prod_{r=N}^1 (K_{\beta_r}F_{\beta_r})^{t_r}\right) (\Theta_l^2)_{(2)} \left(\prod_{r=N}^1 E_{\beta_r}^{s_r}\right),
\end{multline}
with $\Theta_l^2  = \textstyle \sum_{(\Theta_l^2)} (\Theta_l^2)_{(1)} \otimes (\Theta_l^2)_{(2)}\in u_\e^\Q(\mathfrak{h}) \otimes u_\e^\Q(\mathfrak{h}) \cong \mc[\mathbb{G}]\otimes \mc[\mathbb{G}]$ (where as above $\mathbb{G}=Q/{\rm rad}(\tau_{\vert \Q}) = Q/lQ$ in our situation, and $\mc[\mathbb{G}]$ is identified with the group algebra of the elements $K_{\bar \lambda}$, $\bar \lambda\in \mathbb{G})$). Then one shows that the smallest subspace $A \subset \mc[\mathbb{G}]$ such that $\Theta_l  \Theta_l \in A^{\otimes 2}$ is $A=\mc[2\mathbb{G}]$ . 

The very same arguments give the image of $\Phi_{0,1}\colon (u_\e^{\Pup})^* \to u_\e^\Pup$.\end{proof}

\begin{remark}\label{exRmatfev26}{\rm As shown in \cite{LN,LO} the Hopf algebra $u_\e^\Q$ has the unique $R$-matrix $R_l^\Q$ satisfying the so-called Lusztig's ansatz, whereas $u_\e^\Pup$ has a finite number of such $R$-matrices, including $R_l^\Pup$ in \eqref{RmatlP}. Our following results do not depend on the choice of this $R$-matrix. } 
\end{remark}

\subsubsection{Relating the quasitriangular structures}\label{sec:relatingfev26} The analysis of this section goes parallel to that of $\S$\ref{sec:qtsroot1}. The main result is Lemma \ref{commutephiateD}, the analog of Lemma \ref{commutephiate} when $\eD$ is a primitive $Dl$-th root of $1$. 

As usual, let $\Lambda\subset P$ be a sublattice containing $Q$, and let $Y$ be the dual lattice, see \eqref{duallatticesept25}. Recall $\mathcal{Z}_0^+(U_\e^\Q)$ defined after \eqref{Z0Uedefsep25}. By definition $u_\e^\Lam  \cong U_{\epsilon}^\Lam/({\rm relations}\ \eqref{toto})$, so $u_\e^\Y \cong U_\e^\Y/\mathcal{Z}_0^+(U_\e^\Q)U_\e^\Y$. Denote the quotient map
\begin{equation}\label{defqmapssept25}
p^\Y\colon U_\e^\Y\to u_\e^\Y.
\end{equation}
As in \eqref{fundamentalDegeneracy}-\eqref{piPdef} we have
\begin{equation}\label{fundamentalDegeneracyeD}\forall \, \varphi \in \mathcal{O}_A, \:\: \forall \, h \in U_A^\Y \subset \Gamma_A^\Y, \quad \bigl(\langle \varphi, h \rangle_{\AD}^\Pup\bigr)_{|\eD} =  \bigl\langle \varphi_{|\e}, p^\Y(h_{|\e}) \bigr\rangle_{\eD}^\Pup = \pi^\Lam(\varphi_{|\e})(p^\Y(h_|{\e})).
\end{equation}
with the map $\pi^\Lam$ from \eqref{piLamdef}. By extending scalars to $A_\D$ the pairings \eqref{tauA} yield Hopf pairings
\begin{equation}\label{tauAD}
\tau_{\AD}^\Y\colon U_A^\Y(\mathfrak{b}_+)^{\mathrm{cop}}\times \Gamma_A^\Y(\mathfrak{b}_-)\to \AD \ , \quad \rho_{\AD}^\Y\colon U_A^\Y(\mathfrak{b}_-)^{\mathrm{cop}}\times \Gamma_A^\Y(\mathfrak{b}_+)\to \AD.
\end{equation}
These pairings are still non degenerate, since they recover $\tau$ and $\rho$ after tensoring with $\mc(\qD)$. Also, they satisfy the identities \eqref{identPhibracket}, taking $x_\pm\in \Gamma_A^\Y(\mathfrak{b}_\pm)$ and replacing $\langle \text{-},\text{-} \rangle_A^\Q$ with the extended pairing $\langle \text{-}, \text{-} \rangle_{\AD}^\Pup$. By specialization, similarly as in \eqref{taue} one obtains non degenerate Hopf pairings $\tau_{\eD}^\Y$, $\rho_{\eD}^\Y$, satisfying
\begin{equation}\label{identPhibracketeD}
\rho_{\eD}^\Y\bigl(\Phi_\e^+(\alpha),x_+\bigr) = \langle \alpha,x_+\rangle_{\eD}^\Pup\ ,\ \tau_{\eD}^\Y\bigl(\Phi_\e^-(\alpha),x_-\bigr) = \langle \alpha,x_-\rangle_{\eD}^\Pup
\end{equation}
for every $\alpha\in \Oo_\e^\Y$, and $x_\pm\in \Gamma_\e^\Y(\mathfrak{b}_\pm)$.

Recall the algebras $\mathcal{Z}_0(U_\e^\Lam)$ from \eqref{Z0Uedefsep25}. Put $\mathcal{Z}_0^+(U_\e^\Lam):= \ker(\varepsilon) \cap \mathcal{Z}_0(U_\e^\Lam)$ (a Hopf ideal of $U_\e^\Lam$), and $\mathcal{Z}_0^+\bigl( U_\e^\Lam(\mathfrak{b}_+) \bigr) = \mathcal{Z}_0^+(U_\e) \cap U_\e^\Lam(\mathfrak{b}_+)$. The following lemma is proved like Lemma \ref{basicqpairingnov24}.

\begin{lem}\label{basicqpairingsept25} For every $x\in \mathcal{Z}_0^+(U_\e^\Q(\mathfrak{b}_+))$ and $h\in u_\e^\Y(\mathfrak{b}_-)\subset \Gamma_\e^\Y(\mathfrak{b}_-)$ we have $\tau_{\eD}^\Y(x,h)=0$. Therefore we have a well-defined pairing
$$\bar{\tau}_{\eD}^\Y: u_\e^\Y(\mathfrak{b}_+)^{cop}\times u_\e^\Y(\mathfrak{b}_-)\to \mc, \ \bar{\tau}_{\eD}^\Y(p^\Y(x),h) := \tau_{\eD}^\Y(x,h).$$
Moreover $\bar{\tau}_{\eD}^\Y$ is a non degenerate Hopf pairing.
\end{lem}
Denote by $\bar{\rho}_{\eD}^\Y: u_\e^\Y(\mathfrak{b}_-)^{cop}\times u_\e^\Y(\mathfrak{b}_+)\to \mc$ the (non-degenerate) pairing defined by $$\bar{\rho}_{\eD}^\Y(x,y) := \bar{\tau}_{\eD}^\Y(y,S^{-1}(x)) = \bar{\tau}_{\eD}^\Y(S(y),x).$$
For every $x\in U_\e^\Y(\mathfrak{b}_-)$, $h\in u_\e^\Y(\mathfrak{b}_+)$ we have $\bar{\rho}_{\eD}^\Y(p^\Y(x),h)  = \rho_{\eD}^\Y(x,h)$.

We now relate these pairings and the isomorphisms $\Phi_\e^\pm$ in \eqref{Phi+-e} with analogous ones defined on $u_\e^\Pup$ from the quasitriangular structure described in Th.\,\ref{teosmallqg}. Recall the $R$-matrix $R_l^\Pup$ from \eqref{RmatlP}. For notational convenience let us set
$$\Upsilon:= R_l^\Pup.$$
By the general properties of $R$-matrices, the maps (where $\Upsilon^-=\sigma(\Upsilon)^{-1}$ with $\sigma\colon x\otimes y\mapsto y\otimes x$, extended linearly):
\begin{equation}\label{phipm}
\fonc{\bar\Phi_\e^{\pm, \Upsilon}}{u_\e^\Pup(\mathfrak{b}_\pm)^*}{u_\e^\Pup(\mathfrak{b}_\mp)^{cop}}{\alpha}{(\alpha \otimes id)(\Upsilon^\pm)}
\end{equation}
are morphisms of Hopf algebras. As we saw in the proof of Theorem \ref{teosmallqg}, $R_l^\Pup$ puts $u_\e^\Pup(\mathfrak{b}_+)$ and $u_\e^\Pup(\mathfrak{b}_-)$ in duality. Therefore the maps $\bar\Phi_\e^{\pm, \Upsilon}$ are isomorphisms, and one can define non degenerate Hopf pairings
\begin{equation*}
\bar{\rho}_{\eD}^\Upsilon: u_\e^\Pup(\mathfrak{b}_-)^{cop}\times u_\e^\Pup(\mathfrak{b}_+)\to \mc\ ,\ \bar{\tau}_{\eD}^\Upsilon: u_\e^\Pup(\mathfrak{b}_+)^{cop}\times u_\e^\Pup(\mathfrak{b}_-)\to \mc, 
\end{equation*} 
by
\begin{equation}\label{rhoRdefsept25}
\bar{\rho}_{\eD}^\Upsilon(\bar\Phi_\e^{+,\Upsilon}(\alpha_+), h_+) = \alpha_+(h_+)\ ,\  \bar{\tau}_{\eD}^\Upsilon(\bar\Phi_\e^{-,\Upsilon}(\alpha_-), h_-) = \alpha_-(h_-)
\end{equation}
for every $\alpha_\pm\in u_\e^\Pup(\mathfrak{b}_\pm)^*$, $h_\pm\in u_\e^\Pup(\mathfrak{b}_\pm)$, where $\bar{\rho}_{\eD}^\Upsilon (x,y) := \bar{\tau}_{\eD}^\Upsilon (y,S^{-1}(x))$. Note that we have
\[ (\alpha_+\otimes \alpha_-)(\Upsilon) = \bar{\tau}_{\eD}^\Upsilon\bigl( \bar\Phi_\e^{-,\Upsilon}(\alpha_-), \bar\Phi_\e^{+,\Upsilon}(\alpha_+) \bigr). \] 

Since $\bar{\tau}_{\eD}^\Pup$ is a non degenerate Hopf pairing (Lem.\,\ref{basicqpairingsept25}), the map
\[ \phi_\e^{+,\Pup}\colon u_\e^\Pup(\mathfrak{b}_+)^{\mathrm{cop}} \to \bigl(u_\e^\Pup(\mathfrak{b}_-)\bigr)^*, \quad x\mapsto \bar{\tau}_{\eD}^\Pup(x,\cdot) \]
is an isomorphism of Hopf algebras. In particular, by definition of $\tau_{\eD}^\Pup$ we have
$$\left\langle \phi_\e^{+,\Pup}(K_{\bar{\lambda}}),K_{\bar{\mu}} \right\rangle = \eD^{-D(\lambda,\mu)}$$
where $\langle \text{-}, \text{-} \rangle$ is the evaluation pairing. Also, since $\bar{\tau}_{\eD}^\Upsilon$ is a non degenerate Hopf pairing, the map $\phi_\e^{+, \Upsilon}\colon u_\e^\Pup(\mathfrak{b}_+)^{\mathrm{cop}} \to \bigl(u_\e^\Pup(\mathfrak{b}_-)\bigr)^*$, $x\mapsto \bar{\tau}_{\eD}^\Upsilon(x,\cdot)$, is an isomorphism of Hopf algebras, and by \eqref{rhoRdefsept25} it coincides with the inverse map $(\bar\Phi_\e^{-,\Upsilon})^{-1}$. Using this, the formula \eqref{RmatlP}, and the identity 
$$\sum_{\bar{\lambda}\in P/lQ} \eD^{D(\lambda,\mu)} = \vert P/lQ\vert\cdot \delta_{\bar{\mu},0},$$
one checks that $\left\langle \phi_\e^{+, \Upsilon}(K_{\bar{\lambda}}), K_{\bar{\mu}} \right\rangle = \eD^{-D(\lambda,\mu)}$ and hence $\phi_\e^{+, \Upsilon} = \phi_\e^{+,\Pup}$. Similarly, by using the pairings $\bar{\rho}_{\eD}^\Pup$ and $\bar{\rho}_{\eD}^\Upsilon$ one obtains associated maps $\phi_\e^{-, \Pup}, \phi_\e^{-, \Upsilon} \colon u_\e^\Pup(\mathfrak{b}_-)^{cop} \to (u_\e^\Pup(\mathfrak{b}_+))^*$ satisfying $\phi_\e^{-, \Upsilon} = \phi_\e^{-,\Pup}$. 

Recall the quotient map $p^\Pup\colon U_\e^{\Pup}\to u_\e^{\Pup}$ from \eqref{defqmapssept25}. The identities $\phi_\e^{\pm, \Upsilon} = \phi_\e^{\pm,\Pup}$ imply
\begin{equation}\label{newqpairingnov24} \bar{\tau}_{\eD}^\Upsilon\bigl( p^{\Pup}(x),h \bigr) := \tau_{\eD}^\Pup(x,h), \quad \bar{\rho}_{\eD}^\Upsilon\bigl(p^{\Pup}(x),h\bigr)  = \rho_{\eD}^\Pup(x,h).
\end{equation}
Recall the morphisms $\bar\Phi_\e^{\pm}$ from $\S$\ref{sec:qtsroot1}. The statement analogous to Lemma \ref{commutephiate} for $\eD$ a primitive $Dl$-th root of $1$ is:
\begin{lem}\label{commutephiateD} We have commutative diagrams:
$$\xymatrix{\Oo_\e \ar[r]^{\Phi_\e^\pm\quad } \ar[d]^{\pi^\Q} & U_\e^\Pup(\mathfrak{b}_\mp)^{cop} \ar[d]^{p^{\Pup}} \\ (u_\e^\Pup)^* \ar[r]^{\ \bar \Phi_\e^{\pm,\Upsilon}\quad } & u_\e^\Pup(\mathfrak{b}_\mp)^{cop}}\quad \quad \xymatrix{\Oo_\e^\Q \ar[r]^{\Phi_\e^\pm\quad } \ar[d]^{\pi^\Pup} & U_\e^\Q(\mathfrak{b}_\mp)^{cop} \ar[d]^{p^{\Q}} \\ (u_\e^\Q)^* \ar[r]^{\ \bar \Phi_\e^{\pm}\quad } & u_\e^\Q(\mathfrak{b}_\mp)^{cop}.}$$
\end{lem}
\begin{proof} Consider the left diagram; the proof is similar for the right one, where on the top arrow we use Prop.\,\ref{PhirresLam}. For every $h_{\vert \e}\in u_\e^\Pup(\mathfrak{b}_-)\subset \Gamma_\e^\Pup(\mathfrak{b}_-)$ and $\alpha_{\vert \e}\in \Oo_\e$ we have
\begin{align*} \bar{\tau}_{\eD}^\Upsilon\bigl( p^{\Pup}\circ \Phi_\e^-(\alpha_{\vert \e}), h_{\vert \e} \bigr) & =  \tau_{\eD}^\Pup\bigl( \Phi_\e^-(\alpha_{\vert \e}), h_{\vert \e} \bigr) = \langle \alpha_{\vert \e}, h_{\vert \e}\rangle_{\eD}^\Pup\\
& = \pi^\Q(\alpha_{\vert \e}) (h_{\vert \e}) = \bar{\tau}_{\eD}^\Upsilon\bigl( \bar\Phi_\e^{-,\Upsilon}(\pi^\Q(\alpha_{\vert \e})), h_{\vert \e} \bigr).
\end{align*}
The first equality follows from \eqref{newqpairingnov24}, the second from \eqref{identPhibracketeD}, the third is by the definition \eqref{piLamdef}, and the fourth follows from \eqref{rhoRdefsept25}. Since $\bar{\tau}_{\eD}^\Upsilon$ is non degenerate, we get $p^\Pup\circ \Phi_\e^-(\alpha_{\vert \e}) = \bar \Phi_\e^{-,\Upsilon}(\pi^\Q(\alpha_{\vert \e}))$. This proves the diagram for $\Phi^-_\e$ is commutative. For $\Phi^+_\e$ we proceed similarly, using the pairing $\bar{\rho}_{\eD}^\Upsilon$. \end{proof}

\subsubsection{The analog of Th.\,\ref{centralextLgn} for $\eD$}\label{sec:analogextfev26} We can now relate $\mathcal{L}_{g,n}^{\eD}$ with a graph algebra for small quantum groups.

As in \S \ref{subsecDefLgnH} we can define graph algebras $\mathcal{L}_{g,0}(u_\e^\Pup)$ for the datum $(H,\Phi^\pm) := (u_\epsilon^\Pup,\bar\Phi_\e^{\pm,\Upsilon})$, with $\bar\Phi_\e^{\pm,\Upsilon}$  as in \eqref{phipm}, and $\mathcal{L}_{0,n}(u_\e^\Q)$ for the datum $(u_\epsilon^\Q,\bar\Phi_\e^{\pm})$ as in $\S$\ref{sec:qtsroot1}. Set
\begin{equation}\label{defLgnueRlp}
\mathcal{L}_{g,n}(u_\e) := \mathcal{L}_{g,0}(u_\e^\Pup)\ \widetilde{\otimes}\ \mathcal{L}_{0,n}(u_\e^\Q),
\end{equation}
where $\widetilde{\otimes}$ is the braided tensor product of $u_\e^\Pup$-module algebras, defined by means of the $R$-matrix $\Upsilon$ on $u_\e^\Pup$. Some comments are necessary for this construction to hold. First, the choice $H=u_\e^\Pup$ in the ``handle part'' $\mathcal{L}_{g,0}(H)$, and $H=u_\e^\Q$ in the ``puncture part'' $\mathcal{L}_{0,n}(H)$, will be justified in Th.\,\ref{LgntoLgnue} below. On the latter the $R$-matrix $R_l^\Q$ is the Lyubashenko's one; for the handle part $\mathcal{L}_{g,0}(u_\e^\Pup)$ we use $\Upsilon$. Finally, it is easily checked that the structure of $u_\e^\Q$-module algebra of $\mathcal{L}_{0,n}(u_\e^\Q)$ extends to a structure of $u_\e^\Pup$-module algebra (using the Alekseev map $\Phi_{0,n}$, which gives an isomorphism of $u_\e^\Q$-module algebras $\mathcal{L}_{0,n}(u_\e^\Q) \cong (u_\e^\Q)^{\otimes n}$ (see \eqref{isoAlekseevSmallUq}), this translates into the fact that $(u_\e^\Q)^{\otimes n}\subset (u_\e^\Pup)^{\otimes n}$ is stable under the adjoint action of $u_\e^\Pup$). Then $\mathcal{L}_{g,n}(u_\e)$ makes sense as a braided tensor product of $u_\e^\Pup$-module algebras.
\smallskip

Consider the epimorphisms $\pi^\Lam: {\mathcal O}_\epsilon \rightarrow (u_\epsilon^\Y)^{*}$ from \eqref{piLamdef} for $\Lambda=Q$ or $P$. The following result is the analog of Prop.\,\ref{piMorphBetweenLgn} for the specialization to $\qD=\eD$.
\begin{teo}\label{LgntoLgnue} The map $(\pi^\Q)^{\otimes 2g}\otimes (\pi^\Pup)^{\otimes n}\colon \Oo_\e^{\otimes (2g+n)} \rightarrow ((u_\epsilon^\Pup)^*)^{\otimes 2g}\otimes ((u_\epsilon^\Q)^*)^{\otimes n}$ yields a morphism of algebras $\mathcal{L}_{g,n}^{\eD}\rightarrow \mathcal{L}_{g,n}(u_\epsilon)$.
\end{teo}
\begin{proof} It is enough to show the result for $\Ll_{0,1}^{\eD}$, $\Ll_{1,0}^{\eD}$ and $\Ll_{0,2}^{\eD}$, as these cases cover all the defining relations \eqref{fusionRelL01}, \eqref{L01prodsept25} and \eqref{braidedTensProdComm} of $\Ll_{g,n}^{\eD}$. We give the details for $\Ll_{0,1}^{\eD}$ and $\Ll_{1,0}^{\eD}$, the case of $\Ll_{0,2}^{\eD}$ being similar.

Consider first the case of $\Ll_{0,1}^{\eD}$. By \eqref{fusionRelL01} we can write the product of $\alpha_{|\e},\beta_{|\e}\in \Ll_{0,1}^{\eD}$ as 
\begin{equation}\label{basicprode}
\alpha_{|\e}\beta_{|\e} = \sum_{(\alpha),(\beta)} \alpha_{(1)|\e}\star \beta_{(2)|\e}\ \bigl\langle \alpha_{(2)},\Phi_A^+(\beta_{(3)}) \bigr\rangle_{\AD\vert \eD}^\Pup \, \bigl\langle S(\alpha_{(3)}), \Phi_A^+(\beta_{(1)})\bigr\rangle_{\AD\vert \eD}^\Pup.
\end{equation}
Let us assume for a while that $\beta\in \Ll_{0,1}^A$ belongs to $\Oo_A^\Q$. This situation is covered by Lem.\,\ref{lemmaLgnIsFunctorial} and Cor.\,\ref{relationsCoRMatAtEps} as in Prop.\,\ref{piMorphBetweenLgn}, but for the generalization to follow it is useful to rewrite the proof. We have $\Phi_A^+(\beta_{(1)}), S(\Phi_A^+(\beta_{(3)})) \in U_A^\Q$ (see Prop.\,\ref{PhirresLam}), and then \eqref{fundamentalDegeneracyeD} and Lemma \ref{commutephiateD} imply
\begin{align} \langle \alpha_{(2)},\Phi_A^+(\beta_{(3)})\rangle_{\AD\vert \eD}^\Pup& =  \pi^\Pup(\alpha_{(2)|\e})\bigl(p^\Q(\Phi_\e^+(\beta_{(3)|\e}))\bigr)\notag\\ & = \pi^\Pup(\alpha_{(2)|\e})\bigl(\bar{\Phi}^{+}_\e (\pi^\Pup(\beta_{(3)|\e}))\bigr) \label{eqprodPhi1}\\ \langle S(\alpha_{(3)}), \Phi_A^+(\beta_{(1)})\rangle_{\AD\vert \eD}^\Pup & = \pi^\Pup(S(\alpha_{\e|(3)}))\bigl(p^\Q(\Phi_\e^+(\beta_{(1)|\e}))\bigr) \notag\\ & = \pi^\Pup(S(\alpha_{(3)|\e}))\bigl(\bar{\Phi}_\e^{+}(\pi^\Pup(\beta_{(1)|\e}))\bigr) \label{eqprodPhi2}.
\end{align}
 Now, denote by $\cdot$ the usual product of $(u_\e^\Q)^*$. Using the fact that $\pi^\Pup$ is a morphism of Hopf algebras, from \eqref{basicprode} it follows
\begin{multline}\label{PiPexrelmars26}\pi^\Pup(\alpha_{|\e}\beta_{|\e}) = \\ \sum_{(\alpha),(\beta)} \pi^\Pup(\alpha_{|\e})_{(1)}\cdot \pi^\Pup(\beta_{|\e})_{(2)}\ \pi^\Pup(\alpha_{|\e})_{(2)}\bigl(\bar{\Phi}_\e^{+} (\pi^\Pup(\beta_{|\e})_{(3)})\bigr) S(\pi^\Pup(\alpha_{|\e})_{(3)})\bigl(\bar{\Phi}_\e^{+}(\pi^\Pup(\beta_{|\e})_{(1)})\bigr).
\end{multline}
This is the very definition of the product $\pi^\Pup(\alpha_{|\e})\pi^\Pup(\beta_{|\e})$ in $\Ll_{0,1}(u_\e^\Q)$, so $\pi^\Pup\colon \Oo_\e \to (u_\e^\Q)^*$ restricts to a morphism of algebras $\pi^\Pup\colon \Ll_{0,1}^{\Q,\eD} \to \Ll_{0,1}(u_\e^\Q)$, where $\Ll_{0,1}^{\Q,\eD}$ is the subspace of $\Ll_{0,1}^{\eD}$ spanned by the matrix elements in $\Oo_\e^\Q$.

Next we extend this result to the whole $\Ll_{0,1}^{\eD}$. Let $\beta\in \Ll_{0,1}^A$ be arbitrary. The elements $\Phi_A^+(\beta_{(1)}), S(\Phi_A^+(\beta_{(3)})) \in U_A^\Pup$ may not belong to $U_A^\Q$, and therefore the first equalities in \eqref{eqprodPhi1} and \eqref{eqprodPhi2} do not hold true in general, even by replacing the quotient map $p^\Q$ with $\tilde{p}\colon U_\e^\Pup\to u_\e^\Q$ from Lemma \ref{evalinnerprodA} (1). However there are global compensations. Let us show this. Since every $\gamma_{|\e} \in  \mathcal{Z}_0(\Oo_\e)$ is central in $\Oo_\e$, and satisfies $\Oo_\e\gamma_{|\e} = \Oo_\e\star \gamma_{|\e}$ pointwise (Proposition \ref{Z0LgneD} for $(g,n)=(0,1)$), Prop.\,\ref{pipropOLam} gives
$$\pi^\Pup(\alpha_{|\e}\gamma_{|\e}\Oo_\e) = \pi^\Pup(\alpha_{|\e}\Oo_\e\star \gamma_{|\e}) = \pi^\Pup(\alpha_{|\e}\Oo_\e)\cdot\pi^\Pup(\gamma_{|\e})=0\quad {\rm if}\ \gamma_{|\e}\in  \mathcal{Z}_0^+(\Oo_\e).$$
This shows $\pi^\Pup\bigl( \alpha_{|\e}(\beta_{|\e}+\mathcal{Z}_0^+(\Oo_\e)\Oo_\e) \bigr) = \pi^\Pup(\alpha_{|\e} \beta_{|\e})$ meaning that $\pi^\Pup(\alpha_{|\e}\beta_{|\e})$ eventually depends only on the class of $\beta_{|\e}$ modulo $\mathcal{Z}_0^+(\Oo_\e)\Oo_\e$. Since $\mathcal{Z}_0^+(\Oo_\e)\Oo_\e$ is a Hopf ideal of $\Oo_\e$, and $\Phi_\e^+(\Oo_\e) = U_\e^\Pup(\mathfrak{b}_-)$ and $\Phi_\e^+(\mathcal{Z}_0^+(\Oo_\e)) = \mathcal{Z}_0^+\bigl( U_\e^\Pup(\mathfrak{b}_-) \bigr)$, the map $(\Phi_\e^+\otimes {\rm id} \otimes \Phi_\e^+)\circ \Delta^{(3)}$ sends the class of $\beta_{|\e}$ mod $\mathcal{Z}_0^+(\Oo_\e)\Oo_\e$ to the class of $\textstyle \sum_{(\beta)} \Phi_\e^+(\beta_{(1)|\e})\otimes \beta_{(2)|\e}\otimes \Phi_\e^+(\beta_{(3)|\e})$ mod $(\mathcal{Z}_0^+(U_\e^\Pup(\mathfrak{b}_-))U_\e^\Pup(\mathfrak{b}_-)\otimes \Oo_\e\otimes U_\e^\Pup(\mathfrak{b}_-)) + (U_\e^\Pup(\mathfrak{b}_-) \otimes \mathcal{Z}_0^+(\Oo_\e)\Oo_\e\otimes U_\e^\Pup(\mathfrak{b}_-))+(U_\e^\Pup(\mathfrak{b}_-) \otimes \Oo_\e\otimes \mathcal{Z}_0^+(U_\e^\Pup(\mathfrak{b}_-))U_\e^\Pup(\mathfrak{b}_-))$. Now the Hopf epimorphism $\tilde{p}\colon U_\e^\Pup\to u_\e^\Q$ yields by restriction isomorphisms of Hopf algebras $U_\e^\Pup(\mathfrak{b}_\pm)/\mathcal{Z}_0^+(U_\e^\Pup(\mathfrak{b}_\pm))U_\e^\Pup(\mathfrak{b}_\pm) \to u_\e^\Q(\mathfrak{b}_\pm)$ (see Rmk.\,\ref{tildepfactomars26}). From this discussion and \eqref{basicprode} we can write
\begin{align*}
&\pi^\Pup (\alpha_{|\e}  \beta_{|\e}) \\
=\:&\sum_{(\alpha),(\beta)} \Bigl(\pi^\Pup(\alpha_{(1)|\e})\cdot \pi^\Pup(\beta_{(2)|\e})\Bigr) \pi^\Pup(\alpha_{(2)|\e})\bigl((\tilde{p} \circ \Phi_\e^+)(\beta_{(3)|\e})\bigr) \pi^\Pup(S(\alpha_{(3)|\e}))\bigl((\tilde{p}\circ \Phi_\e^+)(\beta_{(1)|\e})\bigr).
\end{align*}
By using Lemma \ref{commutephiate} (where $\pi=\pi^\Pup$ in the present notations) we conclude as about \eqref{PiPexrelmars26} that $\pi^\Pup(\alpha_{|\e}\beta_{|\e})=\pi^\Pup(\alpha_{|\e})\pi^\Pup(\beta_{|\e})$ in $\Ll_{0,1}(u_\e^\Q)$. Therefore $\pi^\Pup\colon \Ll_{0,1}^\e \to \Ll_{0,1}(u_\e^\Q)$ is a morphism of algebras.

Consider now the case of $\Ll_{1,0}^{\eD}$. We have to show that $(\pi^\Q)^{\otimes 2}\colon \Oo_\e \otimes \Oo_\e \to (u_\e^\Pup)^*\otimes (u_\e^\Pup)^*$ defines an algebra morphism $(\pi^\Q)^{\otimes 2}\colon \Ll_{1,0}^{\eD}\to \Ll_{1,0}(u_\e^\Pup)$ for every $\mathfrak{g}$. By \eqref{monomialRelLgn}, for every $\alpha,\beta \in \mathcal{L}_{0,1}^A$ we have $ \beta \otimes \alpha = (\beta\otimes 1)(1\otimes \alpha)$  (product in $\Ll_{1,0}^A$), whence $(\pi^\Q)^{\otimes 2}((\beta_{|\e}\otimes 1)(1\otimes \alpha_{|\e})) = \pi^\Q(\beta_{|\e}) \otimes \pi^\Q(\alpha_{|\e}) = (\pi^\Q(\beta_{|\e})\otimes 1)(1\otimes \pi^\Q(\alpha_{|\e}))$ (product in $\Ll_{1,0}(u_\e^\Pup)$). Therefore, in order to conclude it remains to evaluate $(\pi^\Q)^{\otimes 2}$ on products of the form $(1\otimes \alpha_{|\e})(\beta_{|\e}\otimes 1)$. 

Using \eqref{PhipmInverses} and \eqref{coRMatSign} it is readily checked that the relation \eqref{L01prodsept25} gives 
\begin{align*}
&(1\otimes \alpha_\e)  ( \beta_\e\otimes 1) \\
=\:& \sum_{(\alpha),(\beta)} (\beta_{(3)|\e} \otimes \alpha_{(3)|\e})\ \bigl\langle S(\alpha_{(5)}),\Phi_A^+(\beta_{(1)})\bigr\rangle_{A_\D\vert \eD}^\Pup \bigl\langle \alpha_{(1)}, \Phi_A^+(\beta_{(2)})\bigr\rangle_{A_\D\vert \eD}^\Pup \\[-1em]
& \hspace*{4.5cm} \times\ \bigl\langle S^{-1}(\alpha_{(2)}), \Phi_A^-(\beta_{(4)})\bigr\rangle_{A_\D\vert \eD}^\Pup \bigl\langle \alpha_{(4)}, \Phi_A^+(\beta_{(5)})\bigr\rangle_{A_\D\vert \eD}^\Pup\\
=\:& \sum_{(\alpha),(\beta)}  (\beta_{|\e(3)} \otimes \alpha_{|\e(3)})\ \Bigl\langle \pi^\Q\bigl(S(\alpha_{|\e (5)})\bigr), (p^\Pup\circ \Phi_\e^+)(\beta_{|\e (1)}) \Bigr\rangle \, \Bigl\langle \pi^\Q(\alpha_{|\e (1)}),(p^\Pup\circ \Phi_\e^+)(\beta_{|\e (2)}) \Bigr\rangle\\[-1em]
& \hspace*{3cm} \times \ \Bigl\langle\pi^\Q\bigl(S^{-1}(\alpha_{|\e (2)})\bigr), (p^\Pup\circ \Phi_\e^-)(\beta_{|\e (4)}) \Bigr\rangle \, \Bigl\langle \pi^\Q(\alpha_{|\e (4)}), (p^\Pup\circ  \Phi_\e^+)(\beta_{|\e (5)}) \Bigr\rangle \\
=\:& \sum_{(\alpha),(\beta)}  (\beta_{|\e(3)} \otimes  \alpha_{|\e(3)})\ \Bigl\langle \pi^\Q\bigl(S(\alpha_{|\e (5)})\bigr), \bar{\Phi}_\e^{+,\Upsilon}\bigl(\pi^\Q(\beta_{|\e (1)})\bigr) \Bigr\rangle \, \Bigl\langle \pi^\Q(\alpha_{|\e (1)}), \bar{\Phi}_\e^{+,\Upsilon}\bigl(\pi^\Q(\beta_{|\e (2)}) \bigr) \Bigr\rangle \\[-1em]
 & \hspace*{3cm} \times \ \Bigl\langle \pi^\Q\bigr(S^{-1}(\alpha_{|\e (2)})\bigl), \bar{\Phi}_\e^{-,\Upsilon}\bigl(\pi^\Q(\beta_{|\e (4)})\bigr) \Bigr\rangle \, \Bigl\langle \pi^\Q(\alpha_{|\e (4)}), \bar{\Phi}_\e^{+,\Upsilon}\bigl(\pi^\Q(\beta_{|\e (5)})\bigr) \Bigr\rangle
\end{align*}
where the second equality follows from \eqref{fundamentalDegeneracyeD}, the third from Lemma \ref{commutephiateD} and $\langle \text{-},\text{-} \rangle$ is the evaluation pairing $(u_\e^\Pup)^* \times u_\e^\Pup \to \mathbb{C}$. Using that $\pi^\Q$ is a morphism of Hopf algebras we get
\begin{align*}
&(\pi^\Q)^{\otimes 2}\bigl((1\otimes \alpha_{|\e}) (\beta_{|\e}\otimes  1)\bigr) = \\
& \sum_{(\alpha),(\beta)} \!\! \left(\pi^\Q(\beta_{|\e(3)}) \otimes  \pi^\Q(\alpha_{|\e(3)})\right) \Bigl\langle \pi^\Q\bigl(S(\alpha_{|\e (5)})\bigr), \bar{\Phi}_\e^{+,\Upsilon}\bigl(\pi^\Q(\beta_{|\e (1)})\bigr) \Bigr\rangle \, \Bigl\langle \pi^\Q(\alpha_{|\e (1)}), \bar{\Phi}_\e^{+,\Upsilon}\bigl(\pi^\Q(\beta_{|\e (2)})\bigr) \Bigr\rangle \\[-.8em]
& \qquad \quad \times \ \Bigl\langle\pi^\Q\bigl(S^{-1}(\alpha_{|\e (2)})\bigr), \bar{\Phi}_\e^{-,\Upsilon}\bigl(\pi^\Q(\beta_{|\e (4)})\bigr)\Bigr\rangle \, \Bigl\langle \pi^\Q(\alpha_{|\e (4)}), \bar{\Phi}_\e^{+,\Upsilon}\bigl( \pi^\Q(\beta_{|\e (5)}) \bigr)\Bigr\rangle
\end{align*}
which coincides with the very definition of the product $\textstyle \bigl(1\otimes \pi^\Q(\alpha_{|\e})\bigr)\bigl(\pi^\Q(\beta_{|\e})\otimes 1\bigr)$ in $\mathcal{L}_{1,0}(u_\e^\Pup)$. Therefore $(\pi^\Q)^{\otimes 2}\colon \Ll_{1,0}^{\eD}\to \mathcal{L}_{1,0}(u_\e^\Pup)$ is a morphism of algebras, which concludes the proof.
\end{proof}

As usual put $\Oo^+(G^\Lam) = \ker(\varepsilon) \cap \Oo(G^\Lam)$, and let us set $\pi_{u_\e}:= (\pi^\Q)^{\otimes 2g}\otimes (\pi^\Pup)^{\otimes n}$. The next result proves the analog of Thm.\,\ref{centralextLgn} for $\eD$ an arbitrary primitive $Dl$-th root of $1$.
\begin{cor}\label{centralextLgneD}  We have an exact sequence of algebras
\begin{equation}\label{exactsequenceueApp}
\xymatrix{0 \ar[r] & \Oo^+(G^\Q)^{\otimes 2g}\otimes \Oo^+(G)^{\otimes n} \ar[rr]^{\hspace*{1.7cm} ({\mathbb F}r_\e^{*})^{\otimes (2g+n)}} & &\mathcal{L}_{g,n}^{\eD} \ar[rr]^{\pi_{u_\e}} & & \mathcal{L}_{g,n}(u_\epsilon) \ar[r] & 0}.
\end{equation}
\end{cor}

\begin{proof}
This is an immediate consequence of Prop.\,\ref{pipropOLam}, Prop.\,\ref{Z0LgneD} (1) and Thm.\,\ref{LgntoLgnue}.
\end{proof}

Since $\Ll_{g,n}^{\eD}$ is a domain (see the beginning of App.\,\ref{appDlRoot}), the central localization $Q(\Ll_{g,n}^{\eD})$ is well-defined and a division algebra, and Prop.\,\ref{Z0LgneD} implies that $Q(\Ll_{g,n}^{\eD})$ is a central simple algebra. 
\smallskip

\indent Let us conclude with a number of facts about the PI degree of $Q(\mathcal{L}_{g,n}^{\eD})$ (see Th.\,\ref{Lgnteo1eD} below). Recall the inductive definition \eqref{AlekseevValue1}-\eqref{AlekseevValue2} of the Alekseev morphism on $\Ll_{g,n}(H)$. A similar definition applies to get a morphism on $\mathcal{L}_{g,n}(u_\e)$, with the decomposition \eqref{defLgnueRlp}, by taking in \eqref{AlekseevValue1} the morphisms $\Phi_{0,1}$ and $\Phi^+$ defined from the $R$-matrix $R_l^\Q$ of $u_\e^\Q$, and in \eqref{AlekseevValue2} the morphisms $\Phi_{1,0}$ and $\Phi^+$ defined from the $R$-matrix $\Upsilon$ of $u_\e^\Pup$. By the same proof as Prop.\,\ref{propAlekseevMorph} we get an algebra morphism $$\Phi_{g,n}(u_\e) \colon \mathcal{L}_{g,n}(u_\e) \to \Hh\Hh\bigl((u_\e^\Pup)^*\bigr)^{\otimes g}\otimes (u_\e^\Q)^{\otimes n}$$
and hence
$$F_{g,n}(u_\e) : \mathcal{L}_{g,n}(u_\e) \to \mathrm{End}_\mc\bigl((u_\e^\Pup)^*\bigr)^{\otimes g} \otimes (u_\e^\Q)^{\otimes n}$$ as in \eqref{defFgn}.

\indent Set
\[ {\bf u}_{\e (2)}^\Pup :=\Phi_{0,1}(u_\e)\bigl( \mathcal{L}_{0,1}(u_\e^\Pup) \bigr). \]
Note that a basis of the algebra ${\bf u}_{\e (2)}^\Pup$ was recalled in Th.\,\ref{teosmallqg} (ii).
\begin{prop}\label{Phi10image} We have $\Phi_{1,0}(u_\e)\bigl(\mathcal{L}_{1,0}(u_\e^\Pup)\bigr) =(u_\e^\Pup)^* \,\#\, {\bf u}_{\e (2)}^\Pup$.
\end{prop}
\begin{proof} We use arguments from the proof of \cite[Thm.\,3.13]{BFR}. Let us give some details. The map
\[ \fonc{I}{\mathcal{H}\bigl( (u_\e^\Pup)^* \bigr)}{(u_\e^\Pup)^* \otimes u_\e^\Pup}{\varphi \,\#\, x}{\sum_{(\Upsilon^1),(\Upsilon^2),(\Upsilon^3)}} \bigl(S\bigl( \Upsilon^1_{(2)} \Upsilon^2_{(1)} \bigr) \rhd \varphi \lhd \Upsilon^1_{(1)}\Upsilon^3_{(1)}\bigr) \otimes \bigl(S\bigl( \Upsilon^3_{(2)}\Upsilon^2_{(2)} \bigr)x\bigr) \]
is a linear isomorphism, where $\Upsilon$ is the R-matrix of $u_\e^\Pup$. It satisfies
\begin{equation}\label{IPhi10}
I \circ \Phi_{1,0}(u_\e)\left(\beta \otimes \alpha\right) = \sum_{(\beta)} \beta_{(1)} \otimes \Phi_{0,1}(u_\e)\bigl( \beta_{(2)}\alpha \bigr)
\end{equation}
where $\textstyle \sum_{(\beta)} \beta_{(1)} \otimes \beta_{(2)}$ is the coproduct of $\beta$ in $(u_\e^\Pup)^*$ and $\beta_{(2)}\alpha$ is the product of $\beta_{(2)}$ and $\alpha$ in $\mathcal{L}_{0,1}(u_\e^\Pup)$. Moreover, there is a map $S_{\mathcal{L}_{0,1}} : \mathcal{L}_{0,1}(u_\e^\Pup) \to \mathcal{L}_{0,1}(u_\e^\Pup)$ such that $\textstyle \sum_{(\varphi)} \varphi_{(1)} S_{\mathcal{L}_{0,1}}(\varphi_{(2)}) = \varepsilon(\varphi)1_{\mathcal{L}_{0,1}}$ for all $\varphi \in \mathcal{L}_{0,1}(u_\e^\Pup)$; it is given by
\[ S_{\mathcal{L}_{0,1}}(\varphi) = \sum_{(\Upsilon)} S_{(u_\e^\Pup)^*}\bigl( S(\Upsilon_{(1)}) \rhd \varphi \lhd \Upsilon_{(2)}u_\Upsilon^{-1} \bigr) \]
where $S_{(u_\e^\Pup)^*}$ is the antipode of $(u_\e^\Pup)^*$ and $\textstyle u_\Upsilon = \sum_{(\Upsilon)}S(\Upsilon_{(2)})\Upsilon_{(1)}$ is the Drinfeld element. Now take $\psi \otimes \Phi_{0,1}(u_\e)(\gamma) \in (u_\e^\Pup)^* \otimes {\bf u}_{\e (2)}^\Pup$. By \eqref{IPhi10} and the property of $S_{\mathcal{L}_{0,1}}$ we find
\begin{align*} 
&I \circ \Phi_{1,0}(u_\e)\biggl( \sum_{(\psi)} \psi_{(1)} \otimes S_{\mathcal{L}_{0,1}}(\psi_{(2)})\gamma \biggr)\\
=\:&\sum_{(\psi)} \psi_{(1)} \otimes \Phi_{0,1}(u_\e)\bigl(\psi_{(2)} S_{\mathcal{L}_{0,1}}(\psi_{(3)})\gamma\bigr) = \psi \otimes \Phi_{0,1}(u_\e)(\gamma)
\end{align*}
which proves that $\mathrm{Im}(I \circ \Phi_{1,0}(u_\e)) = (u_\e^\Pup)^* \otimes {\bf u}_{\e (2)}^\Pup$. Since $I$ is a linear isomorphism the result follows.
\end{proof}
\begin{cor}\label{cor24mars26} When $D$ is odd $F_{g,n}(u_\e)\colon \mathcal{L}_{g,n}(u_\epsilon)\to {\rm End}((u_\epsilon^\Pup)^*)^{\otimes g} \otimes (u_\epsilon^\Q)^{\otimes n}$ is an isomorphism of algebras. When $D$ is even $F_{g,n}$ is not surjective (whence not injective). 
\end{cor}
\begin{proof} Both claims follow from Th.\,\ref{teosmallqg} (i), Prop.\,\ref{Phi10image}, Lem.\,\ref{lemmaInjAlekseev}, using the fact that the Heisenberg representation $\rho\colon (u_\e^\Pup)^*  \,\#\, u_\e^\Pup \to {\rm End}\bigl( (u_\e^\Pup)^* \bigr)$ from \eqref{HeisenbergRep} is an isomorphism.
\end{proof}

\begin{lem}\label{irrepeD} When $D$ is odd $\mathcal{L}_{g,n}(u_\epsilon)$ has an irreducible representation of dimension $D^gl^{g.\mathrm{dim}(\mathfrak{g})+nN}$.
\end{lem}
\begin{proof} Since $u_\e^\Q$ has an irreducible representation $V$ of dimension $l^N$ \cite[\S III.6.6]{BG} and $u_\e^\Pup$ is an irreducible representation of ${\rm End}\bigl( (u_\epsilon^\Pup)^* \bigr)$ of dimension $\vert P/Q\vert.l^{\mathrm{dim}(\mathfrak{g})}$ (see \eqref{dimuYmars26}), it is a classical fact \cite[Thm.\,3.10.2]{Et} that $\bigl( (u_\epsilon^\Pup)^* \bigr)^{\otimes g} \otimes V^{\otimes n}$ is an irreducible representation of ${\rm End}\bigl( (u_\epsilon^\Pup)^* \bigr)^{\otimes g} \otimes (u_\epsilon^\Q)^{\otimes n}$ of dimension $\vert P/Q\vert^gl^{g.\dim(\mathfrak{g})+nN}$. When $D$ is odd $\vert P/Q\vert = D$ and $F_{g,n}(u_\e)$ is an isomorphism by Cor.\,\ref{cor24mars26}, so the conclusion follows.
\end{proof}

Recall that $Q(\Ll_{g,n}^{\eD})$ is a division algebra, and a central simple algebra (see the comments after Cor.\,\ref{centralextLgneD}). Since $[Q(\Ll_{g,n}^{\eD}):Q(\mathcal{Z}_0(\Ll_{g,n}^{\eD})] = \vert P/Q\vert^{2g}l^{(2g+n).\mathrm{dim}(\mathfrak{g})}$ (Prop.\,\ref{Z0LgneD}), $[Q(\mathcal{Z}(\Ll_{g,n}^{\eD}):Q(\mathcal{Z}_0(\Ll_{g,n}^{\eD})] \geq l^{mn}$ (same proof as in Th.\,\ref{Lgnteo1}), and the PI-degree is bounded from below by the dimensions of irreducible representations, Lem.\,\ref{irrepeD} implies:
\begin{teo}\label{Lgnteo1eD} When $D$ is odd the algebra $Q(\Ll_{g,n}^{\eD})$ has PI-degree $D^gl^{g.\mathrm{dim}(\mathfrak{g})+nN}$.
\end{teo}
We propose the following conjecture when $D$ is even. By Rk.\,\ref{centralextsl2eD} and the construction of irreducible representations of $\Ll_{g,n}^{\eD}$, when $\mathfrak{g}= \mathfrak{sl}_2$, done in \cite{BFRsl2} (whence the computation of their dimension), the same method shows that the conjecture holds true in this case; see also Rmk.\,\ref{rk:comparLW}. A proof for arbitrary $\mathfrak{g}$ would require a generalization of Lem.\,\ref{irrepeD} to all even $D$, and a caracterization of the maximal subalgebra in $\mathcal{Z}_0(\Oo_\e)^{\otimes (2g+n)}$ which is central in $\Ll_{g,n}^{\eD}$ (as in Rmk.\,\ref{centralextsl2eD}).
\begin{conj}\label{Lgnteo1eDconj} When $D$ is even the algebra $Q(\Ll_{g,n}^{\eD})$ has PI-degree $(D/2)^g.l^{g.\mathrm{dim}(\mathfrak{g})+nN}$.
\end{conj}
\begin{remark}\label{rk:comparLW}{\rm For $n=0$ and $\mathfrak{g}=\mathfrak{sl}_{m+1}$ the PI-degree above could be deduced from the computation in \cite[Thm.\,1.4]{KaruoWang} of the rank over its center of the stated skein algebra $\mathcal{S}_{m+1}(\Sigma_{g,n}^{\circ,\bullet},\mathbf{v})$, where $\mathbf{v}:=\eD^2$, using the specialization of the (integral form of the) holonomy isomorphism $\Ll_{g,n}\to \mathcal{S}_{m+1}(\Sigma_{g,n}^{\circ,\bullet},\mathbf{v})$ of \cite{BFR} (see \cite{BF} for this last point). To compare the notations, use $D=m+1$ and the condition ($*$) in \cite[\S 2.1]{KaruoWang}; it follows that $q^2 := \mathbf{v}^{4D} = \e^8$ and $\mathbf{v}$ are primitive roots of unity of order $l$ and $\mathfrak{d}_\D l$ respectively, where $\mathfrak{d}_\D =D$ if $D$ is odd and $D/2$ otherwise, and then $d:={ord}(\mathbf{v})/{\rm ord}(q^2) = \mathfrak{d}_\D$. Finally, in \cite[Thm.\,1.4]{KaruoWang} the order $l$ is denoted $m$, $\mathrm{dim}(\mathfrak{g})$ is $n^2-1$, $2g = r(\Sigma_{g,n}^{\circ,\bullet})$, and $t=0$.}
\end{remark}

\end{document}